\documentclass[times,final]{elsarticle}

\usepackage{jcomp}

\usepackage{framed,multirow}

\usepackage{amssymb,amsmath,amsthm,bm}
\usepackage{latexsym}
\usepackage{mathtools}
\usepackage{algorithmic} 
\usepackage{algorithm} 
\usepackage{subfigure}
\usepackage{cleveref}
\usepackage{subcaption}
\usepackage{graphicx}

\usepackage{booktabs}
\usepackage{makecell}
\usepackage{enumitem} 

\usepackage{url}
\usepackage{xcolor}

\allowdisplaybreaks[3]

\definecolor{newcolor}{rgb}{.8,.349,.1}

\makeatletter

\newcommand{\Rmnum}[1]{\uppercase\expandafter{\romannumeral #1}}
\makeatother

\newtheorem{theorem}{Theorem}[section]
\newtheorem{lemma}[theorem]{Lemma}

\newtheorem{corollary}{Corollary}[section]

\theoremstyle{definition}

\newtheorem{example}{Example}[section]

\theoremstyle{remark}
\newtheorem{remark}{Remark}[]

\allowdisplaybreaks

\journal{Journal of Computational Physics}
\begin{document}

\verso{Huihui Cao, Manting Peng, Kailiang Wu}
\begin{frontmatter}

\title{{\bf Robust Discontinuous Galerkin Methods Maintaining Physical Constraints for General Relativistic Hydrodynamics}\tnoteref{tnote1}}
	

\tnotetext[tnote1]{This work is partially supported by Shenzhen Science and Technology Program (Grant No.~RCJC20221008092757098) and 
	National Natural Science Foundation of China (Grant No.~12171227).}

\author[1]{Huihui {Cao}}
\ead{caohh@sustech.edu.cn}
\author[1]{Manting {Peng}}
\ead{12232855@mail.sustech.edu.cn}
\author[1,2]{Kailiang {Wu}\corref{cor1}}
\cortext[cor1]{Corresponding author.}
\ead{wukl@sustech.edu.cn}

\address[1]{Department of Mathematics, Southern University of Science and Technology, Shenzhen 518055, China}
\address[2]{Shenzhen International Center for Mathematics, Southern University of Science and Technology, Shenzhen 518055, China}


\begin{abstract}
Numerically simulating general relativistic hydrodynamics (GRHD) presents significant challenges, including handling curved spacetime, achieving non-oscillatory shock-capturing and high-order accuracy, and maintaining essential physical constraints (such as positive density and pressure, and subluminal fluid velocity) under strong nonlinear coupling. This paper develops high-order accurate, physical-constraint-preserving, oscillation-eliminating discontinuous Galerkin (PCP-OEDG) schemes with the Harten--Lax--van Leer flux for GRHD. To suppress spurious oscillations near discontinuities, we incorporate an oscillation-eliminating (OE) procedure after each Runge--Kutta stage. This OE procedure, based on the exact solution operator of a novel linear damping equation, is computationally efficient and avoids the need for complicated characteristic decomposition. It ensures effective oscillation suppression while preserving the high-order accuracy and conservation properties of the DG method. To further enhance the stability and robustness of the DG method, we develop fully physical-constraint-preserving (PCP) schemes. First, we utilize the W-form of GRHD equations, which reformulates the $(3+1)$ Arnowitt--Deser--Misner formalism via the Cholesky decomposition of the spatial metric. This addresses the challenge of the non-equivalence of admissible state sets at different points in curved spacetime, enabling the construction of provably PCP schemes via convexity techniques. Second, we rigorously prove the PCP property of cell averages using highly technical estimates and the Geometric Quasi-Linearization (GQL) approach [K. Wu and C.-W. Shu, {\em SIAM Review}, 65:1031--1073, 2023], which equivalently casts complex nonlinear constraints into linear ones by introducing auxiliary variables. Our proof shows that, with the enforcement of a simple PCP limiter, the updated cell averages of the OEDG solutions remain physically admissible  throughout the simulation. 
Finally, we present provably convergent PCP iterative algorithms for the robust recovery of primitive variables, ensuring that these variables, approximately recovered from the evolved variables, satisfy the physical constraints throughout the iterative process. The resulting PCP-OEDG method is validated through extensive numerical experiments, including classical test problems in flat spacetime, axisymmetric ultra-relativistic jet flows, and accretion onto rotating black holes in the Kerr metric. These results demonstrate our method's robustness, accuracy, and ability to handle extreme GRHD scenarios involving strong shocks, high Lorentz factors, and strong gravitational fields.
\end{abstract}


\begin{keyword}	
	{\bf Keywords:} 
	General relativistic hydrodynamics; 
	Physical-constraint-preserving;
	Discontinuous Galerkin method;
    Oscillation-eliminating procedure;
    Geometric quasilinearization; 
    W-form 
\end{keyword}

\vspace{-1mm}

\end{frontmatter}

\section{Introduction}
In astrophysics and high-energy physics, scenarios often arise where the velocity of fluid approaches the speed of light and strong gravitational fields significantly influence hydrodynamics. These extreme conditions necessitate the consideration of special and general relativistic effects in the modeling of fluid dynamics. Relativistic hydrodynamics (RHD) extends beyond the capabilities of Newtonian mechanics, offering explanations for phenomena that are inaccessible to non-relativistic theories. Additionally, RHD  describes a variety of critical astrophysical processes, such as accretion onto black holes, the coalescence of neutron stars, and core-collapse supernovae. As such, RHD is indispensable in the study of a wide range of astrophysical phenomena, from stellar to galactic scales, as well as in laboratory plasma experiments.

The equations governing RHD are highly nonlinear, incorporating complex spacetime curvature and relativistic effects. Due to these complexities, obtaining analytical solutions is often extraordinarily challenging, making numerical simulations the primary and most powerful tool for exploring the physical mechanisms in RHD. Early work in this field can be traced back to the pioneering efforts of May, White, and Wilson \cite{MW1996, MW1967, W1972}, who applied the finite difference method combined with artificial viscosity techniques to solve the general relativistic hydrodynamics (GRHD) equations in either Lagrangian or Eulerian coordinates. 
Since the 1990s, there has been significant development in high-resolution shock-capturing schemes for RHD. These include, but are not limited to, finite difference methods \cite{DW1995, ZM2006, RR2012, HT2012, RRG2014, WT2015}, finite volume methods \cite{MB2005, TMN2007, BK2016, CW2022, LDT2019}, adaptive moving mesh methods \cite{HT2012, HT20122}, and discontinuous Galerkin (DG) methods \cite{ZT2013,zanotti2015solving,T2016, QSY2016, fambri2018ader}. For further information and additional methods, interested readers are referred to papers such as \cite{W2021, YHT2011, MM2003, MM2015, F2008} and the references therein.

The DG method is a class of ﬁnite element methods whose trial and test spaces consist of discontinuous piecewise polynomials. It was first introduced by Reed and Hill \cite{RH1973} to solve a steady linear hyperbolic equation in 1973. A major development of DG method was coupled with the Runge--Kutta (RK) time discretization for solving nonlinear hyperbolic conservation laws by Cockburn, Shu, and their collaborators \cite{CHS1990,CLS1989,CS1989,CS1998}. The foundational framework of the explicit RKDG method in one dimension was established in \cite{CS1989, CLS1989}, with subsequent extensions to multidimensional equations and systems in \cite{CHS1990, CS1998}.  
Despite the widespread popularity of DG methods for solving a variety of hyperbolic equations due to their notable features---such as compactness, uniformly high-order accuracy, and high parallel efficiency---several key technical challenges remain in developing a robust high-order DG method for GRHD equations:

(i) \textbf{Robust Primitive Variables Recovery:} The flux and the eigenvalues/eigenvectors of its Jacobian matrix cannot be explicitly expressed in terms of the conserved variables, due to the nonlinear coupling introduced by relativistic effects. 
Therefore, it is essential to recover the primitive variables (such as rest-mass density, velocity, and pressure) from the conserved variables.
This recovery process typically requires solving complex nonlinear algebraic equations, either analytically or numerically. Several analytical algorithms for recovering the primitive variables were developed in \cite{AS1972,ED2005,RCC2006}. However, these methods may pose risks of low accuracy and high computational cost. As for numerical method, iterative algorithms are usually used to compute an intermediate variable, which is then used to determine the remaining variables. For instance, a root-finding algorithm using density as the intermediate variable was introduced in \cite{DW1995}. The Newton--Raphson algorithms, based on other intermediate variables like pressure, velocity, and Lorentz factors, were proposed in \cite{RD2008}. 
Unfortunately, the convergence of these iterative methods was not yet rigorously proven or guaranteed in theory. 
In the special relativistic case, recent studies  have explored robust linearly convergent iterative recovery algorithms in \cite{CW2022} and proposed three efficient Newton--Raphson methods with provable convergence in \cite{CQW2024}. In the context of GRHD, the added complexity of spacetime curvature and stronger nonlinear couplings further complicates the recovery process. Extending the provably convergent recovery algorithms \cite{CQW2024,CW2022} to the GRHD case are highly desirable but have not been explored yet.

(ii) \textbf{Controlling Spurious Oscillations:} The Gibbs phenomenon associated with high-order polynomials in the DG method can lead to spurious oscillations when solving nonlinear hyperbolic systems involving discontinuities such as shock waves. These oscillations can cause nonphysical wave structures, numerical instability, or even blowup solutions. Controlling such spurious oscillations is critical for designing robust DG methods for RHD equations. Over the years, various strategies have been developed to mitigate these oscillations, including monotonic schemes \cite{G1959}, different types of limiters \cite{S2009, QS2005, ZS2013}, and artificial viscosity terms \cite{ZGMP2013, HM2014, YH2020, HC2020}. However, these methods have limitations: monotonic schemes are only first-order accurate, limiter selection can be problem-dependent, and artificial diffusion terms often require parameter tuning. 
In recent work, Lu, Liu, and Shu introduced numerical damping terms into the semi-discrete DG formulation, resulting in the oscillation-free DG (OFDG) methods \cite{LLS2021, LLS2022} to control spurious oscillations. This approach retains many desirable properties of the classical DG method, such as conservation, $L^2$-boundedness, and high accuracy. However, the OFDG method may suffer from excessive smearing or persistent spurious oscillations in problems involving large or small scales (wave speeds). In cases with strong discontinuities and highly stiff damping terms, the (modified) exponential RK methods \cite{HS2018} are necessary due to the very restrictive time step sizes required by explicit RK discretization \cite{LLS2022}. Additionally, characteristic decomposition is often required for the OFDG method to solve hyperbolic systems. More recently, motivated by the damping technique of OFDG methods, the oscillation-eliminating DG (OEDG) method, equipped with a scale-invariant and evolution-invariant damping operator, was proposed in \cite{PSW2024}. This method incorporates an oscillation-eliminating (OE) procedure after each RK stage, effectively suppressing nonphysical oscillations for problems with various scales and wave speeds. The OE procedure is independent of the RK stage update, enabling seamless integration into existing DG codes as a standalone module. Compared to the OFDG method, the OEDG method offers reduced computational cost and greatly simplifies implementation, as it does not require characteristic decomposition and allows for normal time step sizes. Extensions and applications of the OEDG method to ideal magnetohydrodynamics (MHD) and the Kapila five-equation two-phase flow model have been explored in \cite{LW2024, YAW2024}. In this work, we further develop high-order robust OEDG schemes preserving physical constraints for GRHD in curved spacetime. 

(iii) \textbf{Physical-Constraint-Preserving:} 
In RHD, several critical physical constraints must be satisfied, such as the positivity of density and pressure and the requirement that fluid velocity remains subluminal. These constraints are crucial for ensuring the physical plausibility and relativistic causality of the solutions, as well as for maintaining the stability of numerical methods. Violating these constraints during numerical computations can result in the eigenvalues of the Jacobian matrix becoming imaginary, leading to an ill-posed discrete problem and simulation failures. Recent progress has been made in developing high-order methods that preserve such bound constraints in hyperbolic conservation laws through two types of limiters. 
The first type is flux-correction limiters \cite{HAS2013, X2013, WT2015}, which modify the high-order flux by incorporating a first-order bound-preserving flux, resulting in fluxes that maintain both physical constraints and high-order accuracy. The second type is the simple scaling limiter, applicable to high-order finite volume and DG schemes. This approach was first proposed by Zhang and Shu in \cite{ZS2010, ZS20102}, where a scaling limiter is applied to the numerical solution to ensure pointwise constraint preservation,  if it can be proven that the updated cell averages remain within the admissible state set.

For RHD equations, most existing physical-constraint-preserving (PCP) research has been primarily focused on the special relativistic case. Wu and Tang \cite{WT2015} were the first to provide an explicit characterization of the admissible state set and proposed high-order PCP finite difference weighted essentially non-oscillatory (WENO) schemes with flux-corrected PCP limiters for special RHD. A PCP finite volume WENO method for special RHD on unstructured meshes was later developed in \cite{CW2022}, which also introduced three robust algorithms for recovering the primitive variables. Bound-preserving DG methods for special RHD were designed in \cite{QSY2016}. Additionally, \cite{W2021} established a minimum principle for specific entropy and developed high-order accurate invariant region-preserving schemes for special RHD with an ideal equation of state. PCP central DG schemes were constructed in \cite{WT2017} for special RHD with a general equation of state.
In addition, a flux-limiting approach that preserves the positivity of the rest-mass density was introduced in \cite{RRG2014}. A reconstruction technique enforcing the subluminal constraint on fluid velocity was proposed in \cite{BalsaraKim2016}. A flux-limiting entropy-viscosity approach for RHD, based on a measure of the entropy generated by the solution, was developed in \cite{guercilena2017entropy}. 
Furthermore, PCP schemes have also been investigated for special relativistic magnetohydrodynamics (MHD) in \cite{WT20172, WS2021}, and for non-relativistic MHD in \cite{W2018, WS2018, WS2019}. A robust adaptive-order positivity-preserving finite difference scheme for general relativistic MHD was recently proposed in \cite{DKT2023}. 
More recently, inspired by the works in \cite{WT20172, W2018, WS2018, WS2019, WS2021}, a general framework known as Geometric Quasi-Linearization (GQL) was established in \cite{WS2023}. This framework transforms complex nonlinear constraints into equivalent linear forms, offering a novel and powerful method for studying the preservation of nonlinear constraints.

Developing a genuinely PCP numerical method for GRHD equations presents several technical challenges, particularly in three key aspects:
i) Establishing a provable PCP property for cell averages throughout the evolution process; ii) Ensuring that the values of the polynomial numerical solution at critical points satisfy physical constraints; 
iii) Maintaining the physical constraints of the primitive variables during their recovery from the conserved variables. 
The main challenges arise from the nonlinear implicit relation between the conserved variables (the evolved variables of the numerical schemes) and the primitive variables (which directly define the physical constraints). This feature, absent in the non-relativistic case, significantly complicates the analysis and design of PCP schemes for RHD. Moreover, the curved spacetime in the GRHD case further complicates the PCP study.  
High-order PCP methods for GRHD were only investigated in \cite{W2017}. However, the schemes studied in \cite{W2017} were limited to the simple Lax--Friedrichs numerical flux, and PCP convergent algorithms for recovering primitive variables have not yet been explored in the general relativistic case.

The aim of this paper is to design and analyze high-order PCP-OEDG methods with the Harten--Lax--van Leer (HLL) numerical flux for the GRHD equations. The distinctive advantages of the proposed methods include their essentially non-oscillatory nature and its provable PCP properties, including preservation of physical constraints in cell-average updates, pointwise solution values, and primitive variable recovery. The contributions, novelty, and significance of this work are outlined as follows:

\begin{itemize}
    \item \textbf{Robust DG Method for GRHD Equations in W-form:} We develop a high-order robust DG method for the W-form of GRHD equations in arbitrary spacetimes. 
    The W-form was initially proposed in \cite{W2017}, leveraging the Cholesky decomposition of the induced metric on each spacelike hypersurface to reformulate the $(3+1)$ Arnowitt--Deser--Misner (ADM) formalism of the GRHD equations.  
    The admissible state set for the evolved variables in the $(3+1)$ ADM formulation is spacetime-dependent, varying at different points in curved spacetime. This variation invalidates the key convexity of the set required for the analysis and design of PCP DG schemes. 
    The W-form effectively addresses this challenge by ensuring that the admissible state set is invariant with respect to spacetime. 
    This invariance is essential for constructing high-order provably PCP-OEDG schemes, guaranteeing robustness and stability. 

    \item \textbf{Suppression of Spurious Oscillations:} To mitigate spurious oscillations in DG solutions near discontinuities, we incorporate an OE procedure after each RK stage. 
    The OE procedure is based on the solution operator of a novel linear damping equation, 
    which can be solved exactly without the need for numerical discretization. Consequently, the OE procedure is simple to implement and computationally efficient, avoiding the need for characteristic decomposition, which is particularly complex and computationally expensive for GRHD equations. 
    Furthermore, the OE procedure is non-intrusive to the DG discretization and can be integrated into existing GRHD DG solvers as a separate, modular component. 
    Since the OE procedure only involves the solution data from the target and adjacent cells, 
    the resulting OEDG method maintains the compactness and parallel efficiency of the original DG method, which is crucial for GRHD simulations in highly complex spacetimes, such as those involving black holes or complicated geometries.
    
    \item \textbf{PCP Theory and Techniques:} 
    We develop fully PCP, high-order DG schemes through the following three approaches:
    \begin{itemize}
        \item {\bf Prove the PCP Property of Cell Averages via GQL}: We employ the GQL approach to rigorously prove that high-order OEDG schemes with appropriate HLL wave speeds satisfy the weak PCP property, meaning they preserve the PCP property of cell averages if the DG solution from the previous time step meets the required physical constraints. This ensures that the cell averages of the evolved variables remain within the physically admissible state set. The proof is quite nontrivial due to the strong nonlinearity of the GRHD equations. 
        
        \item {\bf  PCP Limiter for Point Values}: Building upon the weak PCP property, we present a PCP limiter to enforce the PCP property for the values of the OEDG solution polynomials at certain quadrature points. This limiter guarantees the satisfaction of conditions necessary for the cell averages to preserve the PCP property at the next time step. 
        
        \item {\bf PCP Recovery of Primitive Variables}: Although the evolved variables are ensured to remain within the admissible state set, it is also necessary to ensure that the primitive variables, approximately recovered from the evolved variables, satisfy the physical constraints. 
        To address this, we develop provably convergent, PCP iterative algorithms for robust recovery of primitive variables. These algorithms guarantee that the approximate primitive variables always respect the physical constraints throughout the recovery process.
    \end{itemize}
    As a result, the numerical solutions evolved by the OEDG method remain in the admissible state set and satisfy the physical constraints throughout the entire simulation. 

    \item \textbf{Validation and Simulation:} To validate the robustness, accuracy, and effectiveness of our PCP-OEDG method, we perform extensive numerical experiments, including classical test problems such as the double Mach reflection and shock-bubble interaction in the Minkowski spacetime, and axisymmetric ultra-relativistic jet flows in cylindrical coordinates. 
    Additionally, we accurately simulate the accretion process onto rotating black holes within the Kerr metric and study the properties of the event horizon. The results demonstrate that the PCP-OEDG method is capable of successfully simulating extreme GRHD problems involving strong shocks, extremely low densities/pressures, high Lorentz factors, and/or strong gravitational fields, providing a reliable numerical tool for studying complex physical phenomena in both special and general relativity.
\end{itemize}

This paper is organized as follows. In Section \ref{sec2}, we introduce the governing equations of GRHD. Section \ref{sec3} presents the high-order PCP-OEDG method for GRHD equations in the W-form, including the semi-discrete DG formulation in Subsection \ref{subsec3-1}, the fully-discrete OEDG method in Subsection \ref{subsec3-2}, and the PCP theory and techniques in Subsection \ref{subsec3-3}. In Section \ref{sec4}, we illustrate our PCP-OEDG method for RHD equations in various metrics and coordinate systems. Section \ref{sec5} provides extensive numerical tests to validate the effectiveness of the PCP-OEDG method. Finally, we draw conclusions in Section \ref{sec6}.

Throughout the paper, we employ the Einstein summation convention over repeated indices, and the geometrized unit system so that the speed of light in vacuum and the gravitational constant are equal to one.

\section{Governing Equations of Relativistic Hydrodynamics}\label{sec2}

The $(d+1)$-dimensional spacetime GRHD equations can be formulated in a covariant form as \cite{F2008}:
\begin{align}\label{eq: form1}
\begin{cases}
\partial_{\mu}(\rho u^{\mu}) = 0,\\
\partial_{\mu}T^{\mu\nu} = 0,
\end{cases}
\end{align}
which represent the local conservation laws of baryon number density (the continuity equation) and the stress-energy tensor $T^{\mu \nu}$ (the conservation of energy and momentum). The Greek indices $\mu$ and $\nu$ in Eqs.~\eqref{eq: form1} run from $0$ to $d$. Here, $\rho$ denotes the rest-mass density, and $u^{\mu} = \gamma(1, v^1, \ldots, v^d)$ represents the $(d+1)$-velocity of the fluid, where $(v^1, \ldots, v^d)$ is the $d$-dimensional spatial velocity of the fluid, and $\gamma$ is the Lorentz factor, to be defined later. The operator $\partial_{\mu}$ denotes the covariant derivative associated with the $(d+1)$-dimensional spacetime metric $g_{\mu\nu}$, with the line element given by $\mathrm{d}s^2 = g_{\mu\nu} \mathrm{d}x^{\mu} \mathrm{d}x^{\nu}$. The coordinates in the $(d+1)$-dimensional spacetime are $x^{\mu} = (t, x^1, \ldots, x^d)$.

For an ideal fluid, neglecting the effects of viscosity and heat transfer, the stress-energy tensor is given by
\begin{align}\label{stress-energy}
T^{\mu\nu} = \rho h u^{\mu} u^{\nu} + p g^{\mu\nu},
\end{align}
where $p$ is the pressure, $\rho$ is the rest-mass density, and $h$ is the relativistic specific enthalpy, defined by
\begin{align}\label{EOS}
h = 1 + e + \frac{p}{\rho},
\end{align}
with $e$ representing the specific internal energy. The metric tensor $g^{\mu\nu}$ is the inverse of $g_{\mu\nu}$, and it satisfies the relation
\begin{align}\label{contravariant}
g^{\mu\nu}g_{\nu\lambda} = \delta^{\mu}_{\lambda},
\end{align}
where $\delta^{\mu}_{\lambda}$ is the Kronecker delta:
\begin{align*}
\delta^{\mu}_{\lambda} =
\begin{cases}
1, & \text{if } \mu = \lambda,\\
0, & \text{if } \mu \neq \lambda.
\end{cases}
\end{align*}
The system is closed by an equation of state that relates the fundamental thermodynamic variables. In this study, we adopt the ideal equation of state:
\begin{align}\label{ideal EOS}
e = \frac{p}{(\Gamma - 1)\rho},
\end{align}
where $\Gamma$ is the adiabatic index, which satisfies $\Gamma \in (1,2]$. The condition $\Gamma \le 2$ ensures that the speed of sound remains less than the speed of light, thereby maintaining relativistic causality \cite{WT2015}.

In scenarios where the fluid's self-gravity is negligible compared to the background gravitational field, the dynamics of the system are fully described by Eqs.~\eqref{eq: form1} and \eqref{ideal EOS}. If the self-gravity cannot be neglected, the GRHD equations must be solved in conjunction with the Einstein field equations:
\begin{align*}
G^{\mu\nu} = 8 \pi T^{\mu\nu},
\end{align*}
which relate the curvature of spacetime to the energy-momentum distribution. 
In this paper, we assume that the spacetime metric $g_{\mu\nu}$ and its derivatives $\partial g_{\mu\nu} / \partial x^{\sigma}$ are either given or can be computed numerically using a solver for the Einstein equations at each time step. Furthermore, the metric tensor $g_{\mu\nu}$ is assumed to be a real, symmetric tensor with a signature $(-, +, \dots, +)$.

In numerical computations, it is often advantageous to reformulate the covariant equations \eqref{eq: form1} using the $(3+1)$ formalism, also known as the ADM  decomposition \cite{ADM1962}. In this formalism, spacetime is decomposed into a series of non-intersecting spacelike hypersurfaces, each characterized by a unit normal vector given by $(1/\alpha, -\beta^i/\alpha)$. The lapse function $\alpha > 0$ measures the proper time between adjacent hypersurfaces, and the shift vector $\beta^i$ describes the displacement of the hypersurfaces relative to each other over time. 
The line element in the $(3+1)$ formalism is expressed as
\begin{align*}
\mathrm{d}s^2 = -(\alpha^2 - \beta_i \beta^i)\mathrm{d}t^2 + 2 \beta_i \mathrm{d}x^i \mathrm{d}t + \chi_{ij} \mathrm{d}x^i \mathrm{d}x^j,
\end{align*}
where $\chi_{ij}$ is the induced metric on each spacelike hypersurface. The components $\beta_i$ are the covariant components of the shift vector, defined by $\beta_i = \chi_{ij} \beta^j$. 
The GRHD equations \eqref{eq: form1} can then be rewritten as a first-order hyperbolic system \cite{BFIMM1997}:
\begin{align}\label{eq: form2}
\frac{1}{\sqrt{-g}}\left(\frac{\partial \sqrt{\chi}\mathbf{U}}{\partial t} + \frac{\partial \sqrt{-g}\mathbf{G}^i(\mathbf{U})}{\partial x^i}\right) = \mathbf{R}(\mathbf{U}),
\end{align}
where $g = \mathrm{det}(g_{\mu\nu})$ satisfies $\sqrt{-g} = \alpha \sqrt{\chi}$ with $\chi = \mathrm{det}(\chi_{ij})$. For curved spacetime, the source terms $\mathbf{R}(\mathbf{U})$ in Eq.~\eqref{eq: form2} arise from the spacetime geometry. The conservative vector, flux vectors, and source term are given by
\begin{align}
\mathbf{U} &= (D,\,{\bm m},\,E)^{\top},\label{con 1}\\
\mathbf{G}^i &= (D\tilde{v}^i,\,{\bm m}\tilde{v}^i + p\mathbf{e}_{i},\,E\tilde{v}^i + pv^i)^{\top},\quad i = 1,\,\ldots,\,d, \label{flux 1}\\
\mathbf{R} &= \left(0,\,T^{\mu\nu}g_{\nu\sigma}\Gamma_{\mu i}^{\sigma},\,T^{\mu 0}\partial_{\mu}\alpha - \alpha T^{\mu\nu}\Gamma_{\mu\nu}^{0}\right)^{\top},\label{source 1}
\end{align}
where $\mathbf{e}_i$ denotes the $i$th row of the identity matrix of size $d$, $\tilde{v}^i = v^i - \beta^i/\alpha$, and the fluid velocity vector ${\bm v} = (v^1,\,\ldots,\,v^d)$ is defined by
\begin{align}
v^i = \frac{u^i}{\alpha u^{0}} + \frac{\beta^i}{\alpha}.
\end{align}
In \eqref{source 1}, $\Gamma_{\mu\nu}^{\sigma}$ are the $(d+1)$-dimensional Christoffel symbols, defined by
\begin{align}\label{Christoffel symbols}
\Gamma_{\mu\nu}^{\sigma} = \frac{1}{2}g^{\sigma\lambda}\left(\partial_{\mu}g_{\nu\lambda} + \partial_{\nu}g_{\mu\lambda} - \partial_{\lambda}g_{\mu\nu}\right).
\end{align}
The conserved variables $D,\,{\bm m} = (m_1,\,\ldots,\,m_d)$, and $E$ represent the mass density, the momentum vector, and the total energy, respectively. Let $\mathbf{Q} = (\rho,\,{\bm v},\,p)^{\top}$ denote the primitive variable vector. Then, $\mathbf{U}$ can be calculated from $\mathbf{Q}$ by
\begin{align}\label{pri2con}
D = \rho \gamma,\qquad
{\bm m} = Dh\gamma {\bm v},\qquad
E = Dh\gamma - p,
\end{align}
where the Lorentz factor $\gamma = \alpha u^0 = 1/\sqrt{1 - \|{\bm v}\|_{\mathbf{X}}^2}$, with $\|{\bm v}\|_{\mathbf{X}}^2 = \chi_{ij}v^i v^j$. The matrix $\mathbf{X} = (\chi^{ij})_{1 \le i,j \le d}$ is symmetric and positive definite, typically associated with the coordinates $(t, x^i)$. For simplicity, we denote $\|{\bm v}\|_{\mathbf{I}} =: \|{\bm v}\|$, where $\mathbf{I}$ is the identity matrix of size $d$.

In addition to ensuring the positivity of density and pressure (internal energy), as in the non-relativistic case, another physical constraint is that the fluid velocity must remain subluminal. The physically admissible state set of $\mathbf{U}$ is defined by
\begin{align}\label{adm state 1}
\mathcal{G}^{(0)} = \left\{\mathbf{U} = (D,\,{\bm m},\,E)^{\top}\,:\, D>0,\,p(\mathbf{U})>0,\,\|{\bm v}(\mathbf{U})\|_{\mathbf{X}}<1\right\}.
\end{align}
Preserving the conserved variables within the admissible state set is necessary for maintaining the hyperbolicity of the system. This ensures the existence of $(d+2)$ real eigenvalues for the Jacobian matrix $\partial(n_i\mathbf{G}^i)/\partial \mathbf{U}$ \cite{F2008}:
\begin{align}\label{G-eigenvalues}
\lambda_{\bf n}^{(0)} &= \frac{1}{1 - \|{\bm v}\|_{\mathbf{X}}^2c_s^2}\left(n_iv^i(1-c_s^2) - c_s\gamma^{-1}\sqrt{(1 - \|{\bm v}\|_{\mathbf{X}}^2c_s^2)(n_i n^i) - \left(1 - c_s^2\right)(n_iv^i)^2}\right) - \frac{n_i\beta^i}{\alpha},\\
\lambda_{\bf n}^{(1)} &= \ldots = \lambda_{\bf n}^{(d)} = n_iv^i - \frac{n_i\beta^i}{\alpha},\\
\lambda_{\bf n}^{(d+1)} &= \frac{1}{1 - \|{\bm v}\|_{\mathbf{X}}^2c_s^2}\left(n_iv^i(1-c_s^2) + c_s\gamma^{-1}\sqrt{(1 - \|{\bm v}\|_{\mathbf{X}}^2c_s^2)(n_i n^i) - \left(1 - c_s^2\right)(n_iv^i)^2}\right) - \frac{n_i\beta^i}{\alpha},\label{G-eigenvalues2}
\end{align}
where ${\bf n} = (n_1,\,\ldots,\,n_d)$ is any non-zero vector, and the local sound speed for the ideal equation of state is defined by $c_s = \sqrt{\frac{\Gamma p}{\rho h}}$. 

Due to the strong nonlinearity, $p(\mathbf{U})$ and ${\bm v}(\mathbf{U})$ in the last two constraints in \eqref{adm state 1} are implicit functions of $\mathbf{U}$ without explicit formulas. Consequently, verifying these constraints for the solutions of a numerical scheme is challenging. This complexity makes it very difficult to study the properties of $\mathcal{G}^{(0)}$ and to develop PCP schemes for \eqref{eq: form2} that ensure the numerical solution remains within $\mathcal{G}^{(0)}$. Our previous work \cite{W2017} discovered an equivalent but much simpler representation of $\mathcal{G}^{(0)}$, which is explicitly expressed in terms of $\mathbf{U}$:
\begin{align}\label{adm state 2}
\mathcal{G}_{\chi}^{(1)} = \left\{\mathbf{U} = (D,\,{\bm m},\,E)^{\top}\,:\,D>0,\,q_{\chi}(\mathbf{U})  = E - \sqrt{D^2 + {\bm m}\mathbf{X}{\bm m}^{\top}}>0\right\},
\end{align}
From the equivalent representation $\mathcal{G}_{\chi}^{(1)}$ of $\mathcal{G}^{(0)}$, one can discover that the admissible state set \eqref{adm state 2} is spacetime-dependent and varies across different points in a curved spacetime. Specifically, if ${\bf X} \neq \hat{\bf X}$, then $\mathcal{G}_{\chi}^{(1)} \neq \mathcal{G}_{\hat{\chi}}^{(1)}$. Consequently, the basic convexity of $\mathcal{G}_{\chi}^{(1)}$, which is typically required by existing PCP frameworks, cannot be utilized when studying PCP schemes for GRHD. This is because convexity techniques are not applicable between inequivalent admissible state sets. 
To address this challenge, we follow our previous work \cite{W2017} and introduce a linear operator in the local spacetime that maps elements of $\mathcal{G}_{\chi}^{(1)}$ to a spacetime-independent set $\mathcal{G}^{(2)}$, given by
\begin{align}
\mathbf{W} = \sqrt{\chi}\mathbf{L}\mathbf{U} \in \mathcal{G}^{(2)},\qquad \forall \,\mathbf{U}\in \mathcal{G}_{\chi}^{(1)},
\end{align}
where
\begin{align}\label{adm state 3}
\mathcal{G}^{(2)}  = \left\{\mathbf{W} = (W_0,\,\ldots,\,W_{d+1})^{\top}\,:\,W_0>0,\,q(\mathbf{W}) = W_{d+1} - \left(\sum_{i = 0}^dW_i^2\right)^{1/2}>0\right\},
\end{align}
is a spacetime-independent admissible state set expressed in terms of $\mathbf{W}$. Here, the square matrix $\mathbf{L} = \mathrm{diag}\{1,\,\mathbf{\Theta},\,1\}$, and $\mathbf{\Theta}$ is the Cholesky decomposition of the matrix $\mathbf{X}$, satisfying $\mathbf{\Theta}^{\top}\mathbf{\Theta} = \mathbf{X}$; see \cite{W2017} for more details.

The spacetime-independent admissible state set $\mathcal{G}^{(2)}$ has the following properties. 

\begin{lemma}[Convexity \cite{W2017}]\label{lem1}
The function $q(\mathbf{W})$ defined in \eqref{adm state 3} is concave with respect to $\mathbf{W}$. The admissible set $\mathcal{G}^{(2)}$ is an open convex set. 
\end{lemma}

Next, we introduce an equivalent form of the GRHD equations, called the W-form \cite{W2017}, by multiplying Eq.~\eqref{eq: form2} by the invertible matrix $\mathbf{L}$ from the left. This transformation ensures that the admissible evolved variables in the W-form exactly match the set $\mathcal{G}^{(2)}$:
\begin{align}\label{eq: form3}
\frac{\partial \mathbf{W}}{\partial t} + \frac{\partial \mathbf{H}^i(\mathbf{W})}{\partial x^i} = \mathbf{S}(\mathbf{W}),
\end{align}
where
\begin{align}
&\mathbf{W} = \sqrt{\chi}(D,\,{\bm m}\mathbf{\Theta}^{\top},\,E)^{\top},\label{GRHDcon}\\
&\mathbf{H}^i(\mathbf{W}) = \sqrt{-g}\mathbf{L}\mathbf{G}^i, \qquad i = 1,\,\ldots,\,d,\\
&\mathbf{S}(\mathbf{W}) = \sqrt{\chi}\frac{\partial \mathbf{L}}{\partial t}\mathbf{U} + \sqrt{-g}\left(\mathbf{L}\mathbf{R} + \frac{\partial \mathbf{L}}{\partial x^i}\mathbf{G}^i\right).
\end{align}
The spectral radius of the Jacobian matrix $\partial(n_i\mathbf{H}^i)/\partial \mathbf{W}$ for any non-zero vector ${\bf n} = (n_1,\,\ldots,\,n_d)$ is defined by
\begin{align}\label{spectral radius}
\eta_{\bf n} = \frac{\alpha}{1 - \|{\bm v}\|_{\mathbf{X}}^2c_s^2}\left\{|n_iv^i|(1-c_s^2) + c_s\gamma^{-1}\sqrt{(1-\|{\bm v}\|_{\mathbf{X}}^2c_s^2)(n_i n^i) - (1-c_s^2)(n_iv^i)^2}\right\} + |n_i\beta^i|.
\end{align}
In flat spacetime, the line element is given by
\[
\mathrm{d}s^2 = -\mathrm{d}t^2 + (\mathrm{d}x^1)^2 + \cdots + (\mathrm{d}x^d)^2,
\]
where the spacetime metric $g_{\mu\nu} = \mathrm{diag}\{-1, 1, \ldots, 1\}$ is the Minkowski tensor. In this case, the $(3+1)$ ADM formulation and the W-form of the special RHD equations coincide, both reducing to a system of conservation laws:
\begin{align}\label{eq:form4}
\frac{\partial \mathbf{U}}{\partial t} + \frac{\partial \mathbf{F}_i(\mathbf{U})}{\partial x^i} = \mathbf{0},
\end{align}
where the conservative vector and flux vectors are defined as
\begin{align}
\mathbf{U} &= (D,\,m_1,\,\ldots,\,m_d,\,E)^{\top},\label{con}\\
\mathbf{F}_i &= (Dv_i,\,m_1v_i + pe_{1i},\,\ldots,\,m_dv_i + pe_{di},\,m_iv_i)^{\top}, \qquad i = 1,\,\ldots,\,d,\label{flux}
\end{align}
with $e_{ij}$ denoting the element in the $i$-th row and $j$-th column of the identity matrix of size $d$.

\section{Robust PCP-OEDG Schemes for GRHD Equations in W-Form}\label{sec3}
In this section, we introduce the high-order PCP-OEDG method for solving the GRHD system in the W-form:
\begin{align}\label{Wform}
\frac{\partial \mathbf{W}}{\partial t} + \nabla \cdot \mathbf{H}(\mathbf{W}) = \mathbf{S}(\mathbf{W}),
\end{align}
where $\mathbf{H} = (\mathbf{H}^1, \ldots, \mathbf{H}^d)$ denotes the fluxes. Without loss of generality, we consider numerical methods for a $d$-dimensional spatial domain $\Omega$, where $d = 1, 2, \text{or } 3$. The spatial coordinate vector in the computational domain $\Omega$ is represented by $\mathbf{x} = (x^1, \ldots, x^d)$.

\subsection{Semi-Discrete DG Formulation}\label{subsec3-1}
Let $\mathcal{T}_h$ represent a partition of $\Omega$ into a finite number of cells, where cells are polygons for $d = 2$ and polyhedra for $d = 3$, with disjoint interiors satisfying $\overline{\Omega} = \cup_{K \in \mathcal{T}_h} K$. We denote by $\mathbf{n}_{\mathcal{E}} = \left(n_1, n_2, \ldots, n_d\right)$ the outward unit normal vector on the cell interface $\mathcal{E}$. For any subset $D \subseteq \Omega$, the notation $|D|$ denotes the $d$-dimensional Lebesgue measure of $D$.

The $\mathbb{P}^m$-based discontinuous finite element space on the mesh $\mathcal{T}_h$ is defined as
\begin{align*}
V_h^m := \left\{v(\mathbf{x}) \in L^2(\Omega)\,:\,v(\mathbf{x})|_K \in \mathbb{P}^m(K),\,\forall \, K \in \mathcal{T}_h\right\},
\end{align*}
where $\mathbb{P}^m(K)$ represents the space of polynomials on cell $K$ of degree at most $m$, specifically:
\begin{align*}
\mathbb{P}^m(K) = \text{span}\left\{ \varphi_K^{(\mathbf{a})}(\mathbf{x}),~  |\mathbf{a}| \le m \right\}
\end{align*}
with $\{\varphi_K^{(\mathbf{a})}(\mathbf{x})\}$ denoting the local orthogonal basis functions on $K$. 
Here, 
the vector $\mathbf{a} = (a_1,\,\ldots,\,a_d)$ is a multi-index with $a_j$ denoting the degree of $\varphi_K^{(\mathbf{a})}(\mathbf{x})$ in the variable $x^j$, and $|\mathbf{a}| = \sum_{j=1}^d a_j$ denotes the total degree. 
In particular, we set $\varphi_K^{(\mathbf{0})}(\mathbf{x}) \equiv 1$ for all $\mathbf{x} \in K$. The dimensionality of $\mathbb{P}^m(K)$ is denoted by $N$. 

Multiplying equation \eqref{Wform} by a test function $v_h(\mathbf{x}) \in V_h^m$, integrating by parts over each cell $K$, and approximating the exact solution $\mathbf{W}$ by $\mathbf{W}_h \in [V_h^m]^{d+2}$, we obtain the semi-discrete DG formulation:
\begin{align}\label{DG form1}
\int_K \frac{\partial \mathbf{W}_h(t,\mathbf{x})}{\partial t} v_h(\mathbf{x}) \, \mathrm{d}\mathbf{x} &= \int_K \mathbf{S}(\mathbf{W}_h(t,\mathbf{x})) v_h(\mathbf{x}) \, \mathrm{d}\mathbf{x} + \int_K \mathbf{H}(\mathbf{W}_h(t,\mathbf{x})) \cdot \nabla v_h(\mathbf{x}) \, \mathrm{d}\mathbf{x} \notag \\
&\quad - \sum_{\mathcal{E} \in \partial K} \int_{\mathcal{E}} \widehat{\mathbf{H}}(\mathbf{W}_h(t,\mathbf{x}); \mathbf{n}_{\mathcal{E}}) v_h(\mathbf{x}) \, \mathrm{d}s, \quad \forall v_h \in V_h^m,
\end{align}
where $\widehat{\mathbf{H}}(\mathbf{W}_h(t,\mathbf{x}); \mathbf{n}_{\mathcal{E}})$ represents a suitable numerical flux, often taken as
\[
\widehat{\mathbf{H}}(\mathbf{W}_h^{K^-}, \mathbf{W}_h^{K^+}; \mathbf{n}_{\mathcal{E}}),
\]
with $\mathbf{W}_h^{K^-}$ and $\mathbf{W}_h^{K^+}$ denoting the values of $\mathbf{W}_h$ on either side of the interface $\mathcal{E}$ from the interior and exterior of $K$, respectively. For clarity, we may omit the time dependence in subsequent expressions when it does not lead to confusion. In this work, we utilize the HLL flux:
\begin{align}\label{HLL flux}
\widehat{\mathbf{H}}^{\text{HLL}}(\mathbf{W}^-, \mathbf{W}^+; \mathbf{n}_{\mathcal{E}}) := \frac{ \left( s^+ \mathbf{H}(\mathbf{W}^-) - s^- \mathbf{H}(\mathbf{W}^+) \right) \cdot \mathbf{n}_{\mathcal{E}} + s^+ s^- (\mathbf{W}^+ - \mathbf{W}^-) }{s^+ - s^-},
\end{align}
with
\begin{align}
s^-(\mathbf{W}^-, \mathbf{W}^+; \mathbf{n}_{\mathcal{E}}) &= \alpha\min\left\{\lambda_{\mathbf{n}_{\mathcal{E}}}^{(0)}(\mathbf{W}^-),  \lambda_{\mathbf{n}_{\mathcal{E}}}^{(0)}(\mathbf{W}^+), 0 \right\}, \label{s-}\\
s^+(\mathbf{W}^-, \mathbf{W}^+; \mathbf{n}_{\mathcal{E}}) &= \alpha\max\left\{ \lambda_{\mathbf{n}_{\mathcal{E}}}^{(d+1)}(\mathbf{W}^-), \lambda_{\mathbf{n}_{\mathcal{E}}}^{(d+1)}(\mathbf{W}^+), 0 \right\},\label{s+}
\end{align}
where $\alpha$ denotes the lapse function, and $\lambda_{\mathbf{n}_{\mathcal{E}}}^{(0)}$ and $\lambda_{\mathbf{n}_{\mathcal{E}}}^{(d+1)}$ represent the smallest and largest eigenvalues of the Jacobian matrix $\frac{\partial \mathbf{G}\cdot{\mathbf{n}_{\mathcal{E}}}}{\partial \mathbf{U}}$ in the direction of $\mathbf{n}_{\mathcal{E}}$, given by \eqref{G-eigenvalues} and \eqref{G-eigenvalues2}, respectively.


Let $\mathbf{W}_h(t,\mathbf{x}) = \left(W_{h}^{(0)}, W_{h}^{(1)}, \ldots, W_{h}^{(d+1)}\right)^{\top}$ represent the DG solution. Let $\boldsymbol{\phi}^{(q)}_K(\mathbf{x}) \in \mathbb{R}^{N_q \times 1}$ be the column vector formed by the basis functions $\left\{{\varphi}_K^{(\mathbf{a})}(\mathbf{x}) : |\mathbf{a}| = q\right\}$ of total degree $q$, where $N_q$ is the cardinality of the index set $\{\mathbf{a} : |\mathbf{a}| = q\}$. Each component $W_{h}^{(r)}(t, \mathbf{x}) \in V_h^m$ can be expanded as
\begin{align*}
W_{h}^{(r)}(t, \mathbf{x}) = \sum_{q = 0}^m \mathbf{w}_{r}^{(q)}(t)\boldsymbol{\phi}_K^{(q)}(\mathbf{x}), \quad \forall \mathbf{x} \in K,\quad r = 0, \ldots, d+1,
\end{align*}
where $\mathbf{w}_{r}^{(q)}(t)$ is a row vector of unknown modal coefficients corresponding to the basis vector $\boldsymbol{\phi}_K^{(q)}(\mathbf{x})$. 
Thus, the approximate solution $\mathbf{W}_h \in [V_h^m]^{d+2}$ can be expressed in matrix form: 
\begin{align}\label{matrix form}
\mathbf{W}_h(t,\mathbf{x}) = \sum_{q = 0}^m \mathbf{\mathcal{W}}^{(q)}_K(t)\boldsymbol{\phi}_K^{(q)}(\mathbf{x}) = \mathbf{\mathcal{W}}_K(t)\mathbf{\Phi}_K(\mathbf{x}), \quad \forall \mathbf{x} \in K,
\end{align}
where the modal coefficient matrix $\mathbf{\mathcal{W}}_K(t)$ and the basis matrix $\mathbf{\Phi}_K(\mathbf{x})$ are defined as
\begin{align}\label{matrix}
\mathbf{\mathcal{W}}^{(q)}_K(t) = 
\begin{bmatrix}
\mathbf{w}_{0}^{(q)}(t) \\
\mathbf{w}_{1}^{(q)}(t) \\
\vdots \\
\mathbf{w}_{d+1}^{(q)}(t)
\end{bmatrix}_{(d+2) \times N_q}, \quad 
\mathbf{\mathcal{W}}_K(t) = 
\begin{bmatrix}
\mathbf{w}_{0}^{(0)}(t) & \mathbf{w}_{0}^{(1)}(t) & \cdots & \mathbf{w}_{0}^{(m)}(t) \\
\mathbf{w}_{1}^{(0)}(t) & \mathbf{w}_{1}^{(1)}(t) & \cdots & \mathbf{w}_{1}^{(m)}(t) \\
\vdots & \vdots & \vdots & \vdots \\
\mathbf{w}_{d+1}^{(0)}(t) & \mathbf{w}_{d+1}^{(1)}(t) & \cdots & \mathbf{w}_{d+1}^{(m)}(t)
\end{bmatrix}_{(d+2) \times N}, \quad
\mathbf{\Phi}_K(\mathbf{x}) =
\begin{bmatrix}
\boldsymbol{\phi}_K^{(0)}(\mathbf{x}) \\
\boldsymbol{\phi}_K^{(1)}(\mathbf{x}) \\
\vdots \\
\boldsymbol{\phi}_K^{(m)}(\mathbf{x})
\end{bmatrix}_{N \times 1}.
\end{align}
Given the arbitrariness of $v_h(\mathbf{x}) \in V_h^m$, the semi-discrete DG scheme \eqref{DG form1} can be written as
\begin{align}\label{DG form2}
\int_K \frac{\partial \mathbf{W}_h(t,\mathbf{x})}{\partial t} \mathbf{\Phi}_K^{\top}(\mathbf{x}) \, \mathrm{d}\mathbf{x} 
= & \int_K \mathbf{S}(\mathbf{W}_h(t,\mathbf{x})) \mathbf{\Phi}_K^{\top}(\mathbf{x}) \, \mathrm{d}\mathbf{x} 
+ \int_K \mathbf{H}(\mathbf{W}_h(t,\mathbf{x})) \nabla \mathbf{\Phi}_K^{\top}(\mathbf{x}) \, \mathrm{d}\mathbf{x} \notag \\
& - \sum_{\mathcal{E} \in \partial K} \int_{\mathcal{E}} \widehat{\mathbf{H}}^{\text{HLL}}(\mathbf{W}_h^-, \mathbf{W}_h^+; \mathbf{n}_{\mathcal{E}}) \mathbf{\Phi}_K^{\top}(\mathbf{x}) \, \mathrm{d}s.
\end{align}
Substituting \eqref{matrix form} into \eqref{DG form2} and applying suitable quadrature rules to approximate the flux and source terms, we obtain the matrix formulation:
\begin{align}\label{DG form3}
\frac{\mathrm{d} \mathbf{\mathcal{W}}_K(t)}{\mathrm{d} t} = \left( \mathbf{A}_K^S + \mathbf{A}_K^H - \mathbf{A}_K^{\widehat{H}} \right) \mathbf{A}_K^{-1}, \quad \forall K \in \mathcal{T}_h,
\end{align}
where the mass matrix $\mathbf{A}_K \in \mathbb R^{N\times N}$ and the matrices $\mathbf{A}_K^S\in \mathbb R^{(d+2)\times N}$, $\mathbf{A}_K^H\in \mathbb R^{(d+2)\times N}$, and $\mathbf{A}_K^{\widehat{H}}\in \mathbb R^{(d+2)\times N}$ are defined as follows:
\begin{align}
{\bf A}_K & := \int_{K}\mathbf{\Phi}_K(\mathbf{x}) \mathbf{\Phi}_K^{\top}(\mathbf{x})\mathrm{d}\mathbf{x} 
= {\rm diag} \left\{ \int_{K} \left( {\varphi}_K^{(\mathbf{a})}(\mathbf{x})\right)^2 d \mathbf{x} : |\mathbf{a}| \le m   \right \} , \label{mass}\\
{\bf A}_K^S & :=  |K| \sum_{p =1}^{Q_K}{\omega}_{K,p}\mathbf{S}(\mathcal{W}_K(t)\mathbf{\Phi}_K({\mathbf{x}}_{K,p})) \mathbf{\Phi}_K^{\top}({\mathbf{x}}_{K,p}) \approx \int_{K}\mathbf{S}(\mathbf{W}_h(t,\mathbf{x})) \, \mathbf{\Phi}_K^{\top}(\mathbf{x})\,\mathrm{d}\mathbf{x},\notag\\
{\bf A}_K^H & := |K| \sum_{p =1}^{Q_K}{\omega}_{K,p}\mathbf{H}(\mathcal{W}_K(t)\mathbf{\Phi}_K({\mathbf{x}}_{K,p}))\nabla \mathbf{\Phi}_K^{\top}({\mathbf{x}}_{K,p})  \approx \int_{K} \mathbf{H}(\mathbf{W}_h(t,\mathbf{x})) \nabla \mathbf{\Phi}_K^{\top}(\mathbf{x})\mathrm{d}\mathbf{x} ,\notag\\
{\bf A}_K^{\widehat H} & :=  \sum_{\mathcal{E} \in \partial K  \atop  K' \cap K = \mathcal{E}} \sum_{p =1}^{Q_\mathcal{E}}|\mathcal{E}|{\omega}_{\mathcal{E},p}\widehat{\mathbf{H}}^{HLL}(\mathcal{W}_{K}(t)\mathbf{\Phi}_K({\mathbf{x}}_{\mathcal{E},p}),\mathcal{W}_{K'}(t)\mathbf{\Phi}_{K'} ( {\mathbf{x}}_{\mathcal{E},p});\mathbf{n}_{\mathcal{E}})\mathbf{\Phi}_K^{\top}({\mathbf{x}}_{\mathcal{E},p}) \notag\\
&\,\,\approx \sum_{{\mathcal{E}} \in \partial K}\int_{\mathcal{E}} \widehat{\mathbf{H}}^{HLL}(\mathbf{W}_h^{K^-}, \mathbf{W}_h^{K^+};\mathbf{n}_{\mathcal{E}})\mathbf{\Phi}_K^{\top}\mathrm{d}s.\notag
\end{align}
Here, $|\mathcal{E}|$ denotes the $(d-1)$-dimensional measure of the interface $\mathcal{E}$, $\{\mathbf{x}_{K,p}\}_{p=1}^{Q_K}$ and $\{\mathbf{x}_{\mathcal{E},p}\}_{p=1}^{Q_{\mathcal{E}}}$ are the quadrature nodes on $K$ and $\mathcal{E}$, respectively, and $\{\omega_{K,p}\}_{p=1}^{Q_K}$ and $\{\omega_{\mathcal{E},p}\}_{p=1}^{Q_{\mathcal{E}}}$ are the corresponding quadrature weights.

\subsection{Fully-Discrete OEDG Method}\label{subsec3-2}
For simplicity, the time interval $[0,T]$ is divided into non-uniform time steps at mesh points $\{t_n\}_{n=0}^M$ such that: (i) $t_0 = 0$; (ii) $t_{n+1} = t_n + \Delta t_n$; and (iii) $t_M = T$. The time step size $\Delta t_n$ is determined according to a suitable CFL condition to ensure stability.

Let $\Lambda_K = \{ K \} \cup \{ K' : K' \cap K = \mathcal{E} \in \partial K \}$ denote the set consisting of cell $K$ and its adjacent cells that share a common interface with $K$. Due to the compactness property of DG schemes, the right-hand side of \eqref{DG form3} depends only on the information from cell $K$ and its adjacent cells. To make this dependence explicit, we define the operator
$$
\mathcal{L}_K\left(  \left\{ \mathbf{\mathcal{W}}_{K'} \right\}_{K' \in \Lambda_K}  \right) := 
\left( \mathbf{A}_K^S + \mathbf{A}_K^H - \mathbf{A}_K^{\widehat{H}}\right) \mathbf{A}_K^{-1}.
$$
Thus, the semi-discrete system \eqref{DG form2} can be regarded as a system of ordinary differential equations (ODEs):
\begin{align*}
\frac{\mathrm{d}\mathbf{\mathcal{W}}_K}{\mathrm{d}t} = \mathcal{L}_K \left(  \left\{ \mathbf{\mathcal{W}}_{K'} \right\}_{K' \in \Lambda_K}  \right), \quad \forall K \in \mathcal{T}_h.
\end{align*}
This ODE system can be discretized in time using explicit RK methods. When using conventional RKDG methods, numerical solutions of the GRHD system may suffer from spurious oscillations near discontinuities. To address this, following our recent work \cite{PSW2024}, we apply an OE procedure to suppress spurious oscillations. The OEDG method alternates between evolving the semi-discrete DG scheme and solving a damping equation:
\begin{align}\label{split}
\frac{\mathrm{d}\mathbf{\mathcal{W}}_K}{\mathrm{d}t} = \mathcal{L}_K \left(  \left\{ \mathbf{\mathcal{W}}_{K'} \right\}_{K' \in \Lambda_K}  \right), \qquad
\frac{\mathrm{d}\mathbf{\mathcal{W}}_{\sigma,K}}{\mathrm{d}t} = -\mathbf{\mathcal{W}}_{\sigma,K} \mathbf{D} \left(  \left\{ \mathbf{\mathcal{W}}_{K'} \right\}_{K' \in \Lambda_K}  \right),
\end{align}
where the damping operator $\mathbf{D}(\mathbf{\mathcal{W}})$ is a diagonal matrix of the form
\begin{align*}
\mathbf{D}(\mathbf{\mathcal{W}}) = \text{diag} \left\{ d_0(\mathbf{\mathcal{W}}), d_1(\mathbf{\mathcal{W}}) \mathbf{I}_{N_1}, \dots, d_m(\mathbf{\mathcal{W}}) \mathbf{I}_{N_m} \right\},
\end{align*}
with $\mathbf{I}_{N_q}$ being the identity matrix of size $N_q$, and damping coefficients $\{d_j\}_{1 \leq j \leq m}$ to be defined later. 
The second ODE system in \eqref{split} serves as a damping mechanism, evolving the DG solution to an essentially oscillation-free one. Since the damping operator $\mathbf{D} \left(  \left\{ \mathbf{\mathcal{W}}_{K'} \right\}_{K' \in \Lambda_K}  \right)$ is independent of $\mathbf{\mathcal{W}}_{\sigma}$, this ODE system is linear and can be solved exactly using an exponential operator:
\begin{align}\label{OEsolu}
\mathbf{\mathcal{W}}_{\sigma,K}(\tau) = \mathbf{\mathcal{W}}_{\sigma,K}(0) \exp \left( - \tau \mathbf{D}\left(  \left\{ \mathbf{\mathcal{W}}_{K'} \right\}_{K' \in \Lambda_K}  \right) \right), \quad \tau \geq 0.
\end{align}
The OEDG method retains the high-order accuracy of the original DG method, as the damping $\mathbf{\mathcal{W}}_{\sigma,K} \mathbf{D} \left(  \left\{ \mathbf{\mathcal{W}}_{K'} \right\}_{K' \in \Lambda_K}  \right)$ is a high-order term for smooth solutions. 
This will be confirmed by our numerical experiments.

The fully-discrete OEDG method, integrating the OE procedure after each RK stage, is as follows:
\begin{itemize}
    \item Initialize $\mathbf{\mathcal{W}}_{\sigma,K}^{n,0} = \mathbf{\mathcal{W}}_{\sigma,K}^n$.
    \item For $\ell = 0, 1, \dots, k-1$, compute the intermediate modal coefficient matrices:
    \begin{itemize}
        \item {\bf RK update:}
        \begin{align*}
        \mathbf{\mathcal{W}}_K^{n,\ell+1} = \sum_{0 \leq s \leq \ell} \left( c_{\ell s} \mathbf{\mathcal{W}}_{\sigma,K}^{n,s} + \Delta t_n d_{\ell s} \mathcal{L}_K \left(  \left\{ \mathbf{\mathcal{W}}_{\sigma,K'}^{n,s} \right\}_{K' \in \Lambda_K}  \right) \right);
        \end{align*}
        \item {\bf OE procedure:}
        \begin{align}\label{OEpro}
        \mathbf{\mathcal{W}}_{\sigma,K}^{n,\ell+1} = \mathbf{\mathcal{W}}_K^{n,\ell+1} \exp \left( - \Delta t_n \mathbf{D} \left(  \left\{ \mathbf{\mathcal{W}}_{K'}^{n,\ell+1} \right\}_{K' \in \Lambda_K}  \right) \right);
        \end{align}
    \end{itemize}
    \item Set $\mathbf{\mathcal{W}}_{\sigma,K}^{n+1} = \mathbf{\mathcal{W}}_{\sigma,K}^{n,k}$.
\end{itemize}
Here, $c_{\ell s}$ and $d_{\ell s}$ are the RK coefficients, with $\sum_{0 \leq s \leq \ell} c_{\ell s} = 1$. The matrix $\mathbf{\mathcal{W}}_K^n$ represents the modal coefficient matrix of the numerical solution $\mathbf{W}_h^n$ on cell $K$ at time $t_n$. Similarly, $\mathbf{\mathcal{W}}_{\sigma,K}^{n,\ell+1}$ is the modal coefficient matrix after $\ell$ RK stages and OE procedures.

In fact, the OE procedure \eqref{OEpro} is the solution to the damping equation \eqref{split}, with $\tau = \Delta t_n$ and initial data $\mathbf{\mathcal{W}}_{\sigma,K}(0) = \mathbf{\mathcal{W}}^{n,\ell+1}_K$. It can be interpreted as a nonlinear filter. The OEDG method, originally proposed in \cite{PSW2024}, is the first to bridge damping and filtering techniques by decoupling the damping equations from the nonlinear DG schemes. This approach was motivated by the DG-coupled damping technique developed by Lu, Liu, and Shu in \cite{LLS2021,LLS2022}.

In \cite{PSW2024}, the OE procedure was originally derived by solving the following problem, using the original RKDG solution $\mathbf{W}^{n,\ell+1}_h$ after the $\ell$-th RK stage as the initial value:
\begin{align}\label{Ini}
\begin{cases}
\displaystyle\frac{\mathrm{d}}{\mathrm{d}\tau}\int_K\mathbf{W}_{\sigma}v_h\,\mathrm{d}\mathbf{x} + \sum_{r=0}^m\delta^r_K(\mathbf{W}^{n,\ell+1}_h)\int_K(\mathbf{W}_{\sigma} - P_h^{r-1}\mathbf{W}_{\sigma})v_h\,\mathrm{d}\mathbf{x} = 0,\quad \forall v_h \in V_h^m,\\
\mathbf{W}_{\sigma}(0,\mathbf{x}) = \mathbf{W}^{n,\ell+1}_h(\mathbf{x}),
\end{cases}
\end{align}
where the damping coefficients $\delta^r_K(\mathbf{W}^{n,\ell+1}_h)$ will be defined later. Here, $P_h^r$ is the standard $L^2$ projection operator onto $V_h^r$ for $r \geq 0$, and $P_h^{-1}$ is defined as $P_h^0$. Note that both $\mathbf{W}^{n,\ell+1}_h$ and $\mathbf{W}_{\sigma}$ belong to $[V_h^m]^{d+2}$, and they can be expressed on $K$ as
\begin{align}\label{com2}
\mathbf{W}_{\sigma}(\tau,\mathbf{x}) = \sum_{q=0}^m \mathbf{\mathcal{W}}^{(q)}_{\sigma,K}(\tau)\boldsymbol{\phi}_K^{(q)}(\mathbf{x}), \qquad \mathbf{W}^{n,\ell+1}_h(\mathbf{x}) = \mathbf{W}_{\sigma}(0,\mathbf{x}) = \sum_{q=0}^m \mathbf{\mathcal{W}}^{(q)}_{\sigma,K}(0)\boldsymbol{\phi}_K^{(q)}(\mathbf{x}).
\end{align}
By the orthogonality of the basis functions, we have
\begin{align}\label{com3}
\mathbf{W}_{\sigma} - P_h^{r-1}\mathbf{W}_{\sigma} = \sum_{q = \max\{1,r\}}^m \mathbf{\mathcal{W}}^{(q)}_{\sigma,K}(\tau)\boldsymbol{\phi}_K^{(q)}(\mathbf{x}).
\end{align}
Substituting \eqref{com2}--\eqref{com3} into \eqref{Ini}, the solution $\mathbf{W}_{\sigma}$ is obtained as
\begin{align}\label{OE}
\mathbf{W}^{n,\ell+1}_{\sigma} := \mathbf{W}_{\sigma}(\tau,\mathbf{x}) = \mathbf{\mathcal{W}}^{(0)}_{\sigma,K}(0)\boldsymbol{\phi}_K^{(0)}(\mathbf{x}) + \sum_{q=1}^m \exp\left(-\Delta t_n\sum_{r=0}^q\delta_K^r(\mathbf{W}^{n,\ell+1}_h)\right)\mathbf{\mathcal{W}}^{(q)}_{\sigma,K}(0)\boldsymbol{\phi}_K^{(q)}(\mathbf{x}).
\end{align}
The modal coefficient matrix of $\mathbf{W}^{n,\ell+1}_{\sigma}$ is  the matrix $\mathbf{\mathcal{W}}_{\sigma,K}^{n,\ell+1}$ in \eqref{OEpro} modified by the OE procedure. The damping coefficients are defined as
$$
d_0\left(  \left\{ \mathbf{\mathcal{W}}^{n,\ell+1}_{K'} \right\}_{K' \in \Lambda_K} \right) := 0, \qquad d_q\left(  \left\{ \mathbf{\mathcal{W}}^{n,\ell+1}_{K'} \right\}_{K' \in \Lambda_K} \right) := \sum_{r=0}^q\delta_K^r(\mathbf{W}^{n,\ell+1}_h), \quad 1 \leq q \leq m.
$$

The coefficient $\delta^r_K(\mathbf{W}^{n,\ell+1}_h)$, which depends on $\mathbf{W}^{n,\ell+1}_h$, is a crucial factor. It must be small in smooth regions to maintain accuracy, but sufficiently large near discontinuities to effectively suppress oscillations. A highly effective, scale-invariant, and evolution-invariant damping coefficient was proposed in \cite{PSW2024}:
\begin{align*}
\delta_K^r(\mathbf{W}^{n,\ell+1}_h) = \max_{0 \leq i \leq d+1} \left\{ \sum_{\mathcal{E} \in \partial K} \eta_{\mathbf{n}_{\mathcal{E}}}^K \frac{\sigma_{\mathcal{E},K}^r(W_h^{(i)})}{h_{\mathcal{E},K}} \right\},
\end{align*}
where $\eta_{\mathbf{n}_{\mathcal{E}}}^K$ is the spectral radius of the matrix $\sum_{i=1}^d n_i \frac{\partial \mathbf{H}_i}{\partial \mathbf{W}}(\overline{\mathbf{W}}_{K}^{n,\ell+1})$, with $n_i$ and $\mathbf{H}_i$ representing the $i$-th component of the outward unit normal vector $\mathbf{n}_{\mathcal{E}}$ and the flux vector $\mathbf{H}$, respectively. Here, $W_h^{(i)}$ is the $(i+1)$-th component of $\mathbf{W}^{n,\ell+1}_h$, and $\overline{\mathbf{W}}_{K}^{n,\ell+1}$ is the cell average of $\mathbf{W}_h^{n,\ell+1}({\bf x})$ over $K$. The quantity $h_{\mathcal{E},K} = \sup_{\mathbf{x} \in K} \mathrm{dist}(\mathbf{x}, \mathcal{E})$ is the maximum distance between points in $K$. 
The coefficient $\sigma_{\mathcal{E},K}^r(W_h^{(i)})$ is defined by
\begin{align*}
\sigma^r_{\mathcal{E},K}(W_h^{(i)}) =
\begin{cases}
0, & \text{if} \quad W_h^{(i)} \equiv \mathrm{avg}_\Omega(W_h^{(i)}), \\
\displaystyle \frac{(2r+1) h_{\mathcal{E},K}^r}{2(2m-1) r!} \frac{\frac{1}{|{\mathcal{E}}|} \int_{\mathcal{E}} \sqrt{\sum_{|\mathbf{a}|=r} [\![ \partial^{\mathbf{a}}W_h^{(i)} ]\!]^2} \, \mathrm{d}s}{\| W_h^{(i)} - \mathrm{avg}_\Omega(W_h^{(i)}) \|_{L^{\infty}(\Omega)}}, & \text{otherwise},
\end{cases}
\end{align*}
where $\mathrm{avg}_\Omega(v) = \frac{1}{|\Omega|} \int_\Omega v(\mathbf{x}) \, \mathrm{d}\mathbf{x}$ denotes the global average of $v$ over the entire domain $\Omega$, and $[\![\partial^{\mathbf{a}} W_h^{(i)}]\!]$ is the jump of $\partial^{\mathbf{a}} W_h^{(i)}$ across the cell interface $\mathcal{E}$, with 
$\partial^{\mathbf{a}} W_h^{(i)}(\mathbf{x}) = \frac{\partial^{|\mathbf{a}|} W_h^{(i)}(\mathbf{x})}{\partial (x^1)^{a_1} \cdots \partial (x^d)^{a_d}}.$

\begin{remark}
From \eqref{OE} and the fact that $d_0 \equiv 0$, we observe that the OE procedure only affects the higher-order moments of the DG solution. This implies that: (i) the cell average (zero moment) of the numerical solution remains unchanged during the OE procedure, preserving local mass conservation, and (ii) the OE procedure does not interfere with the DG spatial discretization or the RK time discretization, making it easily implementable in existing DG codes as an independent module. The OEDG method also has several desirable properties, including stability with normal time step sizes, scale invariance, evolution invariance, and high simplicity and efficiency, without requiring characteristic decomposition. More details on the OE procedure can be found in \cite{PSW2024}.
\end{remark}


\subsection{Physical-Constraint-Preserving Numerical Approach}\label{subsec3-3}

In this subsection, we design fully PCP, high-order DG schemes via the following three steps:
\begin{itemize}
    \item {\em Provable PCP Property of Cell Averages}: We first employ the GQL approach to rigorously prove that the high-order OEDG schemes, with a suitable HLL flux, satisfy the weak PCP property, i.e., preserve the PCP property of the updated cell averages, if the DG solution at the previous time step satisfies the physical constraints. This ensures that the cell averages of the evolved variables remain within the admissible state set.
    \item {\em PCP Limiter for Point Values}: Based on the weak PCP property, we then design a PCP limiter to enforce the PCP property of the OEDG solution polynomials at certain quadrature points. This limiter guarantees that the solution at these critical points remains in the admissible state set $\mathcal{G}^{(2)}$, thereby ensuring that the conserved variables $\mathbf{U}$ remain in $\mathcal{G}_{\chi}^{(1)}$. 
    \item {\em PCP Recovery of Primitive Variables}: The final step involves developing provably convergent, iterative algorithms for robustly recovering the corresponding primitive variables ${\bf Q}=(\rho, \bm{v}, p)^\top$. These algorithms ensure that the approximate primitive variables always satisfy the physical constraints throughout the iterations.
\end{itemize}
Consequently, the numerical solutions evolved by the OEDG method remain within the admissible state set and always respect the physical constraints throughout all the computational processes.

\subsubsection{Proof of PCP Property for Cell Averages}\label{sec:PCPave}

The admissible state set $\mathcal{G}_{\chi}^{(1)}$ in \eqref{adm state 2} or $\mathcal{G}^{(2)}$ in \eqref{adm state 3} involves 
nonlinear constraints, whose preservation is difficult to analyze in theory.  
To facilitate the proof of the PCP property for cell averages, we first introduce 
the GQL technique \cite{WS2023}, which transforms nonlinear constraints into a series of equivalent linear constraints.

\begin{lemma}[GQL representation]\label{GQL}
The admissible state set $\mathcal{G}_{\chi}^{(1)}$ as defined in \eqref{adm state 2} is equivalent to 
\begin{align}
\mathcal{G}_{\chi}^* = \left\{\mathbf{U} = (D,{\bm m},E)^{\top}\,: \, \mathbf{U}\cdot {\bm \xi}>0 \,\, \forall {\bm \xi} \in \Xi_*  \right\}
\end{align}
with the vector set 
\begin{align}\label{eq:DefXi}
\Xi_* := 
\{ {\bm \xi}_1\} \cup \left\{ 
{\bm \xi}_*= \left(-\sqrt{1 - {\bm v}_* \mathbf{X} {\bm v}_*^{\top}},\,-{\bm v}_*\mathbf{X} ,\,1\right)^{\top}\,:~{\bm v}_*\in \mathbb{R}^d,\,\|{\bm v}_*\|_{\mathbf{X}}< 1
\right\}, 
\end{align}
where ${\bm \xi}_1 = (1,\,0,\,\cdots,\,0)^\top$ is the first column of the identity matrix of size $(d+2)$, 
and 
${\bm v}_*$ is referred to as the extra free auxiliary variable in the GQL framework. Note that ${\bm \xi}$ is independent of $\bf U$, and all the constraints in the equivalent representation $\mathcal{G}_{\chi}^*$ become linear with respect to $\mathbf{U}$. 
\end{lemma}

\begin{proof}
To show the equivalence between $\mathcal{G}_{\chi}^{(1)}$ and $\mathcal{G}_{\chi}^*$, we prove that these two sets are subsets of each other.

{\bf Step 1}: Show that $\mathcal{G}_{\chi}^{(1)} \subseteq \mathcal{G}_{\chi}^*$:

For any $\mathbf{U} = (D, {\bm m}, E)^\top \in \mathcal{G}_{\chi}^{(1)}$, by definition, we have $D > 0$ and $q_{\chi}(\mathbf{U}) > 0$. Using the Cauchy--Schwarz inequality, we obtain
\begin{align*}
\mathbf{U}\cdot {\bm \xi}_* &= E - {\bm v}_*\mathbf{X}{\bm m}^{\top} - D \sqrt{1-{\bm v}_* \mathbf{X} {\bm v}_*^{\top}} \\
&\geq E - \sqrt{{\bm m} \mathbf{X} {\bm m}^{\top} + D^2} \sqrt{{\bm v}_* \mathbf{X} {\bm v}_*^{\top} + (1-{\bm v}_* \mathbf{X} {\bm v}_*^{\top})} \\
&= E - \sqrt{{\bm m} \mathbf{X} {\bm m}^{\top} + D^2} = q_{\chi}(\mathbf{U}) > 0. 
\end{align*}
This together with $\mathbf{U}\cdot {\bm \xi}_1 = D>0$ yields $\mathbf{U}\in \mathcal{G}_{\chi}^*$.

{\bf Step 2}: Show that $\mathcal{G}_{\chi}^* \subseteq \mathcal{G}_{\chi}^{(1)}$:

For any $\mathbf{U}\in \mathcal{G}_{\chi}^*$, choose ${\bm v}_* = \frac{{\bm m}}{\sqrt{{\bm m} \mathbf{X} {\bm m}^{\top} + D^2}}$. This gives
\begin{align*}
0 < \mathbf{U}\cdot {\bm \xi}_* = E - {\bm v}_*\mathbf{X}{\bm m}^{\top} - D \sqrt{1-{\bm v}_* \mathbf{X} {\bm v}_*^\top} = E - \sqrt{{\bm m} \mathbf{X} {\bm m}^{\top} + D^2} = q_{\chi}(\mathbf{U}).
\end{align*}
Thus, $\mathbf{U}\in \mathcal{G}_{\chi}^{(1)}$. This completes the proof.
\end{proof}

\begin{lemma}\label{lem:1}
Given any $\mathbf{U} \in \mathcal{G}_{\chi}^*$, for any unit vector $\mathbf{n}=(n_1,\dots,n_d) \in \mathbb{R}^d$, if $s^+\ge \alpha \lambda^{(d+1)}_{\mathbf{n}}(\mathbf{U})$ and $s^-\le \alpha \lambda^{(0)}_{\mathbf{n}}(\mathbf{U})$, then

\begin{align}
\pm s^{\pm}\mathbf{U} \cdot {\bm \xi}> \pm\alpha \mathbf{G}_{\mathbf{n}}(\mathbf{U})\cdot {\bm \xi}\qquad  \forall {\bm \xi} \in \Xi_*,
\end{align}
where $\alpha>0$ is the lapse function, ${\mathbf{G}_\mathbf{n}} = \sum_{i=1}^n n_i {\bf G}^i$, and 
$\lambda^{(0)}_{\mathbf{n}}$ and $\lambda^{(d+1)}_{\mathbf{n}}$ are the smallest and largest eigenvalues of the Jacobian matrix $\partial{\mathbf{G}_{\bf n}}/\partial \mathbf{U}$. 
\end{lemma}

\begin{proof}
From the definition in \eqref{flux 1}, we can express the flux $\mathbf{G}(\mathbf{U})=( {\bf G}^1,\dots,{\bf G}^d )$ as
\begin{align*}
\mathbf{G}(\mathbf{U}) = \mathbf{U}\tilde{\mathbf{v}} + \begin{pmatrix}
    \mathbf{0}\\
    p\mathbf{I}_d
    \\
    p\mathbf{v}
\end{pmatrix}, 
\end{align*}
where $\tilde{v}^i = v^i - \beta^i/\alpha$, and $\mathbf{I}_d$ is the identity matrix of size $d$. For convenience, we introduce the following notations: $\bm v^{\mathbf{n}} = \sum_i^d n_i v^i$, $\tilde{\bm v}^{\mathbf{n}} = \sum_i^d n_i \tilde{v}^i$, and $c^+ := s^+ - \alpha \tilde{\bm v}^{\mathbf{n}}$. 
For any vector ${\bm \xi}={\bm \xi}_*$ in \eqref{eq:DefXi}, 
we have 
\begin{align} \nonumber
\Big[s^+ \mathbf{U} - \alpha \mathbf{G}_{\mathbf{n}}(\mathbf{U})\Big]\cdot {\bm \xi}_* &= \left[(s^+ - \alpha \tilde{\bm v}^{\mathbf{n}})\mathbf{U} - \alpha(0, p \mathbf{n}, p \bm v^{\mathbf{n}})^{\top}\right]\cdot {\bm \xi}_* \\ \nonumber
&= c^+\mathbf{U}\cdot {\bm \xi}_* - \alpha (0, p \mathbf{n}, p \bm v^{\mathbf{n}})^{\top}\cdot {\bm \xi}_* \\ \nonumber
&= (c^+E - \alpha p \bm v^{\mathbf{n}}) - {\bm v}_*\mathbf{X} (c^+{\bm m} - \alpha p \mathbf{n})^{\top} - c^+D \sqrt{1 - {\bm v}_* \mathbf{X} {\bm v}_*^{\top}} \\
&\geq (c^+E - \alpha p \bm v^{\mathbf{n}}) - \sqrt{(c^+{\bm m} - \alpha p \mathbf{n}) \mathbf{X} (c^+{\bm m} - \alpha p \mathbf{n})^{\top} + (c^+)^2 D^2}. \label{eq:wko333}
\end{align}
We aim to prove  
\begin{align}\label{c-cond1}
(c^+E - \alpha p \bm v^{\mathbf{n}}) - \sqrt{(c^+{\bm m} - \alpha p \mathbf{n}) \mathbf{X} (c^+{\bm m} - \alpha p \mathbf{n})^{\top} + (c^+)^2 D^2} > 0.
\end{align}
Consider the function 
\begin{align*}
f_1(c) &:= (cE - \alpha p \bm v^{\mathbf{n}})^2 - \left[(c{\bm m} - \alpha p \mathbf{n}) \mathbf{X} (c{\bm m} - \alpha p \mathbf{n})^{\top} + c^2 D^2\right]\\
&=  (E^2 - {\bm m} \mathbf{X} {\bm m}^{\top} - D^2)c^2 - 2\alpha(Ep \bm v^{\mathbf{n}} - p {\bm m} \mathbf{X} \mathbf{n}^\top)c + [(\bm v^{\mathbf{n}})^2 - \mathbf{n}\mathbf{X} \mathbf{n}^{\top} ] \alpha^2 p^2.
\end{align*}
Since $\mathbf{U} \in \mathcal{G}_{\chi}^*$, we know $E^2 - {\bm m} \mathbf{X} {\bm m}^{\top} - D^2 > 0$, implying $f(c)$ is a quadratic polynomial with respect to $c$. 
Moreover, 
$$f_1(0)= [(\bm v^{\mathbf{n}})^2 - \mathbf{n}\mathbf{X} \mathbf{n}^{\top} ] \alpha^2 p^2 < 0,$$
because 
$
(\bm v^{\mathbf{n}})^2 = ({\bm v} \mathbf{X} \mathbf{n}^\top)^2 = {\bm v} \mathbf{X} {\bm v}^{\top} \mathbf{n} \mathbf{X} \mathbf{n}^{\top} < \mathbf{n} \mathbf{X} \mathbf{n}^{\top}.
$ 
Hence, $f_1(c) = 0$ has a positive root (denoted by $c_1$) and a negative root.  This implies 
that the target inequality \eqref{c-cond1} is equivalent to 
\begin{align}\label{cond2}
c^+>c_1.
\end{align}
Next, define 
\begin{align*}
c_2 := \alpha(\lambda_{\mathbf{n}}^{(d+1)} - \tilde{\bm v}^{\mathbf{n}})& = \frac{\alpha\left[-c_s^2(1-\|\bm v\|_{\mathbf{X}}^2)\bm v^{\mathbf{n}} + c_s\gamma^{-1}\sqrt{(1 - c_s^2 \bm \|\bm v\|_{\mathbf{X}}^2) \mathbf{n} \mathbf{X} \mathbf{n}^{\top} - (1 - c_s^2)(\bm v^{\mathbf{n}})^2}\right]}{1 - \|\bm v\|_{\mathbf{X}}^2c_s^2}.
\end{align*}
We notice that $c_2$ is the unique positive root of the following quadratic function
\begin{align*}
f_2(c) := \alpha^{-2}(1-\|\bm v\|_{\mathbf{X}}^2c_s^2)\,c^2 + 2\alpha^{-1}\bm v^{\mathbf{n}}c_s^2(1-\|\bm v\|_{\mathbf{X}}^2)\,c + c_s^2(1-\|\bm v\|_{\mathbf{X}}^2)[(\bm v^{\mathbf{n}})^2 - \mathbf{n}\mathbf{X}\mathbf{n}^{\top}].
\end{align*}
After careful investigation, we observe that the function $f_1(c)$ can be reformulated as
\begin{align}\label{eq:key}
f_1(c) = \delta_1f_2(c) + \frac{\delta_2 \alpha^2}{c_s^2(1 - \|\bm v\|_{\mathbf{X}}^2)}\left[\frac{1 - \|\bm v\|_{\mathbf{X}}^2c_s^2}{\alpha^2}c^2 - f_2(c)\right],
\end{align}
with
\begin{align*}
\delta_1 := \frac{\alpha^2 (E^2 - {\bm m} \mathbf{X} {\bm m}^{\top} - D^2)}{1 - \|\bm v\|_{\mathbf{X}}^2c_s^2},\qquad 
\delta_2 := \frac{c_s^2\left[(\rho h -p)^2 - \rho^2 - p^2\|\bm v\|^2_{\mathbf{X}}\right]}{1 - \|\bm v\|_{\mathbf{X}}^2c_s^2}-p^2.
\end{align*}
For the ideal gas, substituting $c_s^2=\frac{\Gamma p}{\rho h}$ and $h = 1 + \frac{\Gamma p}{(\Gamma-1)\rho}$, we deduce that
\begin{align*}
\delta_2 =\frac{(\Gamma+1)\rho p}{\Gamma(\Gamma-1)} + \frac{(2 - \Gamma)p}{(\Gamma-1)^2}>0,
\end{align*}
where we have used $\Gamma \in (1,2]$. Since $f_2(c_2) =0$, it follows from \eqref{eq:key} that 
\begin{align}\label{eq:ccc}
f_1(c_2) = \frac{\delta_2 ( 1 - \|\bm v\|_{\mathbf{X}}c_s^2) }{c_s^2(1 - \|\bm v\|_{\mathbf{X}}^2)}c_2^2 >0. 
\end{align}
Recall that the quadratic function $f_1(c)$ has only one positive root $c_1$, implying $f_1(c)>0$ for $c>c_1$ 
and $f_1(c)<0$ for $0< c \le c_1$. This fact, combined with \eqref{eq:ccc}, yields $c_2 > c_1$.
Therefore, if $s^+\ge \alpha \lambda^{(d+1)}_{\mathbf{n}}(\mathbf{U})$, then 
$$
c^+ = s^+ - \alpha  \tilde{\bm v}^{\mathbf{n}} \ge \alpha(\lambda_{\mathbf{n}}^{(d+1)} - \tilde{\bm v}^{\mathbf{n}}) = c_2 > c_1, 
$$
which satisfy \eqref{cond2}, or equivalently \eqref{c-cond1}. 
By combining \eqref{eq:wko333} with \eqref{c-cond1}, we obtain 
\begin{align}\label{wkl1}
\Big[s^+ \mathbf{U} - \alpha \mathbf{G}_{\mathbf{n}}(\mathbf{U})\Big]\cdot {\bm \xi}_*>0. 
\end{align}
Since $\mathbf{U}\cdot {\bm \xi}_1>0$, it holds that 
\begin{align}\label{wkl2}
\Big[s^+ \mathbf{U} - \alpha \mathbf{G}_{\mathbf{n}}(\mathbf{U})\Big]\cdot \boldsymbol{\xi}_1 = \left[( s^+ - \alpha \tilde{\bm v}^{\mathbf{n}})\mathbf{U} - \alpha(0, p \mathbf{n}, p \bm v^{\mathbf{n}})^{\top}\right]\cdot \boldsymbol{\xi}_1 = ( s^+ - \alpha \tilde{\bm v}^{\mathbf{n}})\mathbf{U}\cdot \boldsymbol{\xi}_1>0.
\end{align}
By combining \eqref{wkl1} with \eqref{wkl2}, we have 
\begin{align*} 
s^+ \mathbf{U} \cdot {\bm \xi}> \alpha \mathbf{G}_{\mathbf{n}}(\mathbf{U})\cdot {\bm \xi}\qquad  \forall {\bm \xi} \in \Xi_*.
\end{align*}
Similarly, for $s^- \le \alpha \lambda_{\mathbf{n}}^{(0)}(\mathbf{U})$, we can prove that 
\begin{align*} 
-s^- \mathbf{U} \cdot {\bm \xi}> -\alpha \mathbf{G}_{\mathbf{n}}(\mathbf{U})\cdot {\bm \xi}\qquad  \forall {\bm \xi} \in \Xi_*.
\end{align*}
The proof is completed. 
\end{proof}

\begin{lemma}\label{Coro_1}
Given any $\mathbf{U} \in \mathcal{G}_{\chi}^*$, for any unit vector $\mathbf{n}=(n_1,\dots,n_d) \in \mathbb{R}^d$, if $s^+\ge \alpha \lambda^{(d+1)}_{\mathbf{n}}(\mathbf{U})$ and $s^-\le \alpha \lambda^{(0)}_{\mathbf{n}}(\mathbf{U})$, then
\begin{align}\label{cor1}
\pm s^{\pm}\mathbf{W} \cdot \widehat{{\bm \xi}}> \pm \mathbf{H}_{\mathbf{n}}(\mathbf{W})\cdot \widehat{{\bm \xi}}\qquad  \forall  \,  \widehat{{\bm \xi}} \in \widehat{\bm \Xi}_*,
\end{align}
where $\mathbf{W} = \sqrt{\chi}\mathbf{L}\mathbf{U}$, 
$\mathbf{H}_{\mathbf{n}} = \sum_{i=1}^d n_i {\bf H}^i$, and 
\begin{align}\label{thea}
\widehat{\bm \Xi}_* := \{\bm \xi_1\}\cup \left\{\widehat{\bm \xi}_* = \left(-\sqrt{1 - \|{\bm \hat{\bm v}}_*\|^2},\,-{\bm \hat{\bm v}}_*,\,1\right)^{\top}\,\,:\, \hat{\bm v}_* \in \mathbb{R}^d,\,\|\hat{\bm v}_*\|^2 < 1\right\}.
\end{align}
\end{lemma}

\begin{proof}
Let $\bm v_* = \hat{\bm v}_*(\mathbf{\Theta^{\top}})^{-1}$. Note that 
\begin{align*}
\bm v_* \mathbf{X} \bm v_*^{\top} = \hat{\bm v}_*(\Theta^{\top})^{-1} \mathbf{X}\Theta^{-1}  \hat{\bm v}_*^{\top} = \| \hat{\bm v}_*\|^2 \le 1,
\end{align*}
and 
\begin{align*}
\bm \xi_* =  \left(-\sqrt{1 - \bm v_* \mathbf{X} \bm v_*^{\top}},\,-\bm v_*\mathbf{X},\,1 \right)^\top =\left(-\sqrt{1 -\| \hat{\bm v}_*\|^2},\,-\hat{\bm v}_*\mathbf{\Theta},\,1 \right)^\top = {\bf L} \widehat{\bm {\xi}}_*.
\end{align*}
It follows, by simple algebraic manipulations, that 
\begin{align}\label{cor1:pro 1}
\mathbf{U}\cdot \bm \xi_* =\sqrt{\chi}^{-1} \mathbf{W}\cdot\widehat{\bm {\xi}}_*,\qquad\qquad
\alpha \mathbf{G}_{\mathbf{n}}(\mathbf{U})\cdot \bm \xi_*  =  \sqrt{\chi}^{-1} {\bf H}_{\mathbf{n}}\cdot\widehat{\bm {\xi}}_*.
\end{align}
Thanks to Lemma \ref{lem:1}, we have 
$
\pm s^{\pm}\mathbf{U}^+ \cdot \bm \xi_*>  \pm \alpha \mathbf{G}_{\mathbf{n}}(\mathbf{U})\cdot \bm \xi_*
$, which 
together with \eqref{cor1:pro 1} implies 
\begin{align}\label{eq:Wieq}
    \pm s^{\pm}\mathbf{W} \cdot \widehat{\bm {\xi}}_*> \pm \mathbf{H}_{\mathbf{n}}(\mathbf{W})\cdot \widehat{\bm {\xi}}_*.
\end{align}
Note that
\begin{align*}
\mathbf{U} \cdot \bm \xi_1 = \sqrt{\chi}^{-1}\mathbf{W} \cdot \bm \xi_1,\qquad
\alpha \mathbf{G}_{\mathbf{n}}(\mathbf{U})\cdot \bm \xi_1 =  \sqrt{\chi}^{-1}\mathbf{H}_{\mathbf{n}}(\mathbf{W})\cdot \bm \xi_1.
\end{align*}
Recalling that Lemma \ref{lem:1} has shown $\pm s^{\pm}\mathbf{U} \cdot \bm \xi_1> \pm\alpha \mathbf{G}_{\mathbf{n}}(\mathbf{U})\cdot \bm \xi_1$, we obtain 
$$\pm s^{\pm}\mathbf{W} \cdot \bm \xi_1> \pm \mathbf{H}_{\mathbf{n}}(\mathbf{W})\cdot \bm \xi_1,$$
which together with \eqref{eq:Wieq} yields \eqref{cor1} and completes the proof. 
\end{proof}

Similar to Lemma \ref{GQL}, the GQL approach \cite{WS2023} can also be used to derive an 
 equivalent yet linear representation of the admissible state set $\mathcal{G}^{(2)}$ defined in \eqref{adm state 3}, given by 
\begin{align*}
\mathcal{G}^* = \left\{\mathbf{W} = (W_0,\,W_1,\,\cdots,\,W_d)^{\top}:\,\mathbf{W}\cdot \widehat{\bm{\xi}}>0 ,\,\, \forall \widehat{\bm{\xi}} \in \widehat{\bm{\Xi}}_*\right\}.
\end{align*}
where $\widehat{\bm{\Xi}}_*$ is defined by \eqref{thea}.
\begin{remark}
For the ADM form and the W-form, based on the relationships between their evolved variables $\mathbf{U}$ and $\mathbf{W}$, along with the GQL representations of the admissible state sets, we establish the following equivalence relations:
\begin{align}
\mathbf{U} \in \mathcal{G}_{\chi}^{(1)} \Longleftrightarrow \mathbf{U} \in \mathcal{G}_{\chi}^* \Longleftrightarrow \mathbf{W} \in \mathcal{G}^{(2)} \Longleftrightarrow \mathbf{W} \in \mathcal{G}^*.
\end{align}
\end{remark}
\begin{lemma}\label{lem3}
For any $\mathbf{W} \in \mathcal{G}^*$, it holds that 
\begin{align}\label{lem3:eq}
\lambda_{\bf W} \mathbf{W}\cdot \widehat{\bm \xi} \ge -{\mathbf{S}(\mathbf{W})}\cdot \widehat{\bm \xi} \qquad  \forall  \,  \widehat{{\bm \xi}} \in \widehat{\bm \Xi}_*. 
\end{align}
Here, $\lambda_{\mathbf{W}} =0$ if $q\left({\mathbf{S}(\mathbf{W})}\right)\ge 0$; otherwise $\lambda_{\mathbf{W}}$ is the unique positive solution to 
\begin{align}\label{lem3:eq2}
q\left(\mathbf{W} +\lambda_{\mathbf{W}}^{-1}{\mathbf{S}(\mathbf{W})}\right) =0.
\end{align}
\end{lemma}
\begin{proof}
Since the first component of ${\mathbf{S}(\mathbf{W})}$ is zero,  we have 
\begin{align*}
\lambda_{\bf W} \mathbf{W}\cdot \bm \xi_1 \ge 0 = -{\mathbf{S}(\mathbf{W})}\cdot \bm \xi_1.
\end{align*}
Next, we focus on the case $\widehat{\bm \xi} = \widehat{\bm \xi}_*$, where $\widehat{\bm \xi}_*$ is defined in \eqref{thea}. 
Similar to Step 1 of the proof of Lemma \ref{GQL}, using the Cauchy--Schwarz inequality gives 
\begin{align}\label{eq:wkl33}
    (\lambda_{\bf W} \mathbf{W} + {\mathbf{S}(\mathbf{W})} ) \cdot \widehat{\bm \xi}_* \ge q (\lambda_{\bf W} \mathbf{W} + {\mathbf{S}(\mathbf{W})} ).
\end{align}
If $q\left({\mathbf{S}(\mathbf{W})}\right)\ge 0$, then taking $\lambda_{\bf W}=0$ in \eqref{eq:wkl33} yields \eqref{lem3:eq}; otherwise equation \eqref{lem3:eq2} has a unique positive root due to the concavity of $q({\bf W})$ with respect to $\bf W$, and using \eqref{eq:wkl33} and $q (\lambda_{\bf W} \mathbf{W} + {\mathbf{S}(\mathbf{W})} ) = \lambda_{\bf W} q\left(\mathbf{W} +\lambda_{\mathbf{W}}^{-1}{\mathbf{S}(\mathbf{W})}\right) =0$ gives \eqref{lem3:eq}. 
\end{proof}


Next, we will use the above lemmas to analyze the PCP property for the updated cell averages of the proposed (OE) DG schemes with the HLL flux.
We focus on the forward Euler method for time discretization, while our analysis and conclusons are also valid for the high-order SSP time discretizations. 

Let $\overline{\mathbf{W}}_K^n$ be  the cell average of $\mathbf{W}_h^n({\bf x})$ on cell $K$ at time $t_n$. Taking the test function $v_h(\mathbf{x}) =1$ and using the forward Euler time discretization, we obtain the evolution equation for the cell averages of our (OE) DG schemes: 
\begin{align}\label{first_order00}
\overline{\mathbf{W}}_K^{n+1} = \overline{\mathbf{W}}_K^n - \frac{\Delta t_n}{|K|}
\sum_{\mathcal{E} \in \partial K \atop K' \cap K = \mathcal{E}} 
\sum_{p = 1}^{Q_\mathcal{E}} |\mathcal{E}| {\omega}_{\mathcal{E},p} 
\widehat{\mathbf{H}}^{HLL}\left(\mathbf{W}_h^n(\mathbf{x}_{\mathcal{E},p})\big|_K, 
\mathbf{W}_h^n(\mathbf{x}_{\mathcal{E},p})\big|_{K'}; \mathbf{n}_{\mathcal{E}}\right)  + \Delta t_n \sum_{p = 1}^{Q_K} {\omega}_{K,p} \mathbf{S}(\mathbf{W}_h^n({\mathbf{x}}_{K,p})).
\end{align}
On each cell $K$, we assume that there exists a positive quadrature with algebraic accuracy $\ge m$ and the quadrature nodes including $\{ \mathbf{x}_{\mathcal{E},p}: 1\le p \le Q_{\mathcal{E}}, \mathcal{E} \in \partial K \}$. This quadrature leads to the following convex decomposition of a cell average on $K$: 
\begin{align}\label{CAD}
\frac1{|K|}\int_K f({\bm x}) d{\bm x} = \sum_{p = 1}^{\widehat Q_K}\widehat{\omega}_{K,p} f(\widehat{\mathbf{x}}_{K,p})+ \sum_{\mathcal{E} \in \partial K}\sum_{p=1}^{Q_{\mathcal{E}}} \widehat{\omega}_{\mathcal{E},p} f(\mathbf{x}_{\mathcal{E},p}) \qquad \forall f \in \mathbb P^m,
\end{align}
where the weights $\widehat{\omega}_{\mathcal{E},p}$ and $\widehat{\omega}_{K,p}$ are all positive and satisfy
\begin{align*}
\sum_{p=1}^{\widehat Q_K}\widehat{\omega}_{K,p}  
 +\sum_{e \in \partial K}\sum_{p=1}^{Q_{\mathcal{E}}}\widehat {\omega}_{\mathcal{E},p}=1.
\end{align*}
Such cell average decomposition \eqref{CAD} has been successfully constructed for Cartesian meshes in any space dimension and triangular meshes; see \cite{ZS2010,CDW2023,CDW2024,ZXS2012,DCW2024}. 
It will play an important role in the theoretical PCP proof. 
Define the point set 
\begin{align}\label{ssk}
\mathbb{S}_K = \left\{\mathbf{x}_{K,p}\right\}_{p=1}^{Q_K} \cup \left\{\widehat{\mathbf{x}}_{K,p}\right\}_{p=1}^{\widehat Q_K}\cup
\left(\bigcup_{\mathcal{E} \in \partial K}\left\{\mathbf{x}_{\mathcal{E},p}\right\}_{p = 1}^{Q_{\mathcal{E}}}\right),
\end{align}
which includes all the quadrature nodes involved in \eqref{first_order00} and \eqref{CAD}. 
Assume that the DG solution $\mathbf{W}_h^n({\bf x})$ satisfies 
\begin{align}\label{PCPpointcond}
\mathbf{W}_h^n({\bf x}) \in {\mathcal G}^* \qquad \forall {\bf x} \in \mathbb{S}_K, \quad \forall K \in {\mathcal T}_h,
\end{align}
which can be enforced by a PCP limiter presented later if the OEDG solution does not automatically meet this condition. 
For convenience, denote 
\begin{align}\label{vex3}
\mathbf{W}_h^n(\widehat{\mathbf{x}}_{K,p}) = \widehat{\mathbf{W}}_{K,p}\in {\mathcal G}^*, \quad
\mathbf{W}_h^n(\mathbf{x}_{K,p}) = \mathbf{W}_{K,p}\in {\mathcal G}^*, \quad
\mathbf{W}_h^n(\mathbf{x}_{\mathcal{E},p})|_{K'} = \mathbf{W}_{\mathcal{E},p}^+\in {\mathcal G}^*, \quad 
\mathbf{W}_h^n(\mathbf{x}_{\mathcal{E},p})|_K = \mathbf{W}_{\mathcal{E},p}^-\in {\mathcal G}^*. 
\end{align}
According to the quadrature for the source term and the cell average decomposition \eqref{CAD}, we can split the cell average into a convex combination of several point values:
\begin{align}\label{vex2}
\overline{\mathbf{W}}_K^n &= \frac{1}{|K|}\int_K \mathbf{W}_h^n(\mathbf{x}) \,\mathrm{d}\mathbf{x} = \sum_{p = 1}^{Q_K}{\omega}_{K,p} \mathbf{W}_{K,p},
\\
\label{vex1}
\overline{\mathbf{W}}_K^n & = \sum_{p = 1}^{\widehat Q_K}\widehat{\omega}_{K,p} \widehat{\mathbf{W}}_{K,p}+ \sum_{\mathcal{E} \in \partial K}\sum_{p=1}^{Q_{\mathcal{E}}} \widehat{\omega}_{\mathcal{E},p} \mathbf{W}_{\mathcal{E},p}^-.
\end{align}

Under the condition \eqref{PCPpointcond}, Lemma \ref{lem3} implies that 
\begin{align}\label{lem3:eq_app}
 {\mathbf{S}( \mathbf{W}_{K,p} )}\cdot \widehat{\bm \xi} \ge -\lambda_{ \mathbf{W}_{K,p} } ~ \mathbf{W}_{K,p}  \cdot \widehat{\bm \xi} \qquad  \forall  \,  \widehat{{\bm \xi}} \in \widehat{\bm \Xi}_*.
\end{align}
Thanks to Lemma \ref{Coro_1},  if the HLL speeds satisfy 
\begin{align}\label{wkl-HLL}
s^+_{\mathcal{E},p} \ge  \alpha \max\left\{\lambda_{\mathbf{n}_{\mathcal{E}}}^{(d+1)}(\mathbf{W}^{-}_{\mathcal{E},p}),\,\lambda_{\mathbf{n}_{\mathcal{E}}}^{(d+1)}(\mathbf{W}^{+}_{\mathcal{E},p}),\,0\right\}, \quad 
s^-_{\mathcal{E},p} \le \alpha \min\left\{\lambda_{\mathbf{n}_{\mathcal{E}}}^{(0)}(\mathbf{W}^{-}_{\mathcal{E},p}),\,\lambda_{\mathbf{n}_{\mathcal{E}}}^{(0)}(\mathbf{W}^{+}_{\mathcal{E},p}),\,0\right\}, \quad \forall p,\mathcal{E},
\end{align}
then
\begin{align*}
\pm s^{\pm}_{\mathcal{E},p}\mathbf{W}_{\mathcal{E},p}^-\cdot \widehat{\bm {\xi}} > \pm \mathbf{H}_{\mathbf{n}_{\mathcal{E}}}(\mathbf{W}_{\mathcal{E},p}^-)\cdot \widehat{\bm {\xi}},\qquad \pm s^{\pm}_{\mathcal{E},p}\mathbf{W}_{\mathcal{E},p}^+ \cdot \widehat{\bm {\xi}} > \pm \mathbf{H}_{\mathbf{n}_{\mathcal{E}}}(\mathbf{W}_{\mathcal{E},p}^+)\cdot \widehat{\bm {\xi}}, \qquad \forall  \,  \widehat{{\bm \xi}} \in \widehat{\bm \Xi}_*.
\end{align*}
We then have the following estimate: 
\begin{align*}
\begin{aligned}
-&\widehat{\mathbf{H}}^{HLL}(\mathbf{W}_{\mathcal{E},p}^-,\mathbf{W}_{\mathcal{E},p}^+;\mathbf{n}_{\mathcal{E}})\cdot \widehat{\bm {\xi}} \\
&= -\frac{s_{\mathcal{E},p}^+\mathbf{H}_{\mathbf{n}_{\mathcal{E}}}(\mathbf{W}_{\mathcal{E},p}^-)-s_{\mathcal{E},p}^-\mathbf{H}_{\mathbf{n}_{\mathcal{E}}}(\mathbf{W}_{\mathcal{E},p}^+)+ s_{\mathcal{E},p}^-s_{\mathcal{E},p}^+\left(\mathbf{W}_{\mathcal{E},p}^+  -\mathbf{W}_{\mathcal{E},p}^-\right)}{s_{\mathcal{E},p}^+ - s_{\mathcal{E},p}^-}\cdot \widehat{\bm {\xi}} \\
&= \left(s_{\mathcal{E},p}^-\mathbf{W}_{\mathcal{E},p}^--\mathbf{H}_{\mathbf{n}_{\mathcal{E}}}(\mathbf{W}_{\mathcal{E},p}^-)\right)\cdot \widehat{\bm {\xi}}  - s_{\mathcal{E},p}^-\frac{\left(s_{\mathcal{E},p}^+\mathbf{W}_{\mathcal{E},p}^+ -\mathbf{H}_{\mathbf{n}_{\mathcal{E}}}(\mathbf{W}_{\mathcal{E},p}^+)\right) +\left(\mathbf{H}_{\mathbf{n}_{\mathcal{E}}}(\mathbf{W}_{\mathcal{E},p}^-) -s_{\mathcal{E},p}^-\mathbf{W}_{\mathcal{E},p}^-\right)}{s_{\mathcal{E},p}^+ - s_{\mathcal{E},p}^-}\cdot \widehat{\bm {\xi}}\\
&\ge 
-\left(s^+_{\mathcal{E},p} -s_{\mathcal{E},p}^-\right)\mathbf{W}_{\mathcal{E},p}^-\cdot \widehat{\bm {\xi}} \qquad \forall  \,  \widehat{{\bm \xi}} \in \widehat{\bm \Xi}_*.
\end{aligned}
\end{align*}
This, together with \eqref{lem3:eq_app} and \eqref{first_order00}, yields 
\begin{align} \nonumber
\overline{\mathbf{W}}_K^{n+1} \cdot \widehat{\bm {\xi}} & = \overline{\mathbf{W}}_K^n \cdot \widehat{\bm {\xi}} - \frac{\Delta t_n}{|K|}  \sum_{\mathcal{E} \in \partial K } \sum_{p =1}^{Q_\mathcal{E}}|\mathcal{E}|{\omega}_{\mathcal{E},p} \widehat{\mathbf{H}}^{HLL}(\mathbf{W}_{\mathcal{E},p}^-,\mathbf{W}_{\mathcal{E},p}^+;\mathbf{n}_{\mathcal{E}})\cdot \widehat{\bm {\xi}}
+ \Delta t_n \sum_{p =1}^{Q_K}{\omega}_{K,p} {\mathbf{S}( \mathbf{W}_{K,p} )}\cdot \widehat{\bm \xi}
\\
& > \overline{\mathbf{W}}_K^n \cdot \widehat{\bm {\xi}} - \frac{\Delta t_n}{|K|}  \sum_{\mathcal{E} \in \partial K } \sum_{p =1}^{Q_\mathcal{E}}|\mathcal{E}|{\omega}_{\mathcal{E},p}  \left(s^+_{\mathcal{E},p} -s_{\mathcal{E},p}^-\right)\mathbf{W}_{\mathcal{E},p}^-\cdot \widehat{\bm {\xi}} 
- \Delta t_n \sum_{p =1}^{Q_K}{\omega}_{K,p} \lambda_{ \mathbf{W}_{K,p} } ~ \mathbf{W}_{K,p}  \cdot \widehat{\bm \xi}.  \label{wkl331}
\end{align}
Based on \eqref{vex1} and \eqref{vex2}, we observe that 
\begin{align} \nonumber
  \sum_{\mathcal{E} \in \partial K } \sum_{p =1}^{Q_\mathcal{E}}|\mathcal{E}|{\omega}_{\mathcal{E},p}  \left(s^+_{\mathcal{E},p} -s_{\mathcal{E},p}^-\right)\mathbf{W}_{\mathcal{E},p}^-\cdot \widehat{\bm {\xi}} 
&= \sum_{\mathcal{E} \in \partial K } \sum_{p =1}^{Q_\mathcal{E}}|\mathcal{E}| \left( 
 \frac{{\omega}_{\mathcal{E},p}}{ \widehat  {\omega}_{\mathcal{E},p}}  \left(s^+_{\mathcal{E},p} -s_{\mathcal{E},p}^-\right)
 \right)   \widehat  {\omega}_{\mathcal{E},p}
 \mathbf{W}_{\mathcal{E},p}^-\cdot \widehat{\bm {\xi}}   
 \\ \nonumber
 & \le  \left( \max_{ \mathcal{E} \in \partial K } \max_{1\le p \le Q_{\mathcal E}} \frac{ |\mathcal{E}| {\omega}_{\mathcal{E},p}}{ \widehat  {\omega}_{\mathcal{E},p}}  \left(s^+_{\mathcal{E},p} -s_{\mathcal{E},p}^-\right)   \right) 
 \sum_{\mathcal{E} \in \partial K } \sum_{p =1}^{Q_\mathcal{E}}  \widehat  {\omega}_{\mathcal{E},p}
 \mathbf{W}_{\mathcal{E},p}^-\cdot \widehat{\bm {\xi}}   
 \\ \label{wkl332}
 & \le   \left( \max_{ \mathcal{E} \in \partial K } \max_{1\le p \le Q_{\mathcal E}} \frac{|K|{\omega}_{\mathcal{E},p}}{ \widehat  {\omega}_{\mathcal{E},p}}  \left(s^+_{\mathcal{E},p} -s_{\mathcal{E},p}^-\right)   \right)  \overline{\mathbf{W}}_K^n \cdot \widehat{\bm {\xi}},
\end{align}
and 
\begin{align}\nonumber
\sum_{p =1}^{Q_K}{\omega}_{K,p} \lambda_{ \mathbf{W}_{K,p} } ~ \mathbf{W}_{K,p}  \cdot \widehat{\bm \xi} & \le \Big(\max_{1\le p \le Q_{K}} \lambda_{ \mathbf{W}_{K,p} } \Big) 
\sum_{p =1}^{Q_K}{\omega}_{K,p} \mathbf{W}_{K,p}  \cdot \widehat{\bm \xi}
\\ \label{wkl333}
& = \Big(\max_{1\le p \le Q_{K}} \lambda_{ \mathbf{W}_{K,p} } \Big)  \overline{\mathbf{W}}_K^n \cdot \widehat{\bm {\xi}}. 
\end{align}
Combining \eqref{wkl332} and \eqref{wkl333} with \eqref{wkl331}, we have 
for any $\widehat{{\bm \xi}} \in \widehat{\bm \Xi}_*$ that 
\begin{align*}
\overline{\mathbf{W}}_K^{n+1} \cdot \widehat{\bm {\xi}} & > 
\overline{\mathbf{W}}_K^n \cdot \widehat{\bm {\xi}} - \frac{\Delta t_n}{|K|}   \left( \max_{ \mathcal{E} \in \partial K } \max_{1\le p \le Q_{\mathcal E}} \frac{|\mathcal{E}|{\omega}_{\mathcal{E},p}}{ \widehat  {\omega}_{\mathcal{E},p}}  \left(s^+_{\mathcal{E},p} -s_{\mathcal{E},p}^-\right)   \right)  \overline{\mathbf{W}}_K^n \cdot \widehat{\bm {\xi}}
- \Delta t_n \Big(\max_{1\le p \le Q_{K}} \lambda_{ \mathbf{W}_{K,p} } \Big)  \overline{\mathbf{W}}_K^n \cdot \widehat{\bm {\xi}}
\\
&=
\left( 1 -  \frac{\Delta t_n}{|K|}   \left( \max_{ \mathcal{E} \in \partial K } \max_{1\le p \le Q_{\mathcal E}} \frac{|\mathcal{E}|{\omega}_{\mathcal{E},p}}{ \widehat  {\omega}_{\mathcal{E},p}}  \left(s^+_{\mathcal{E},p} -s_{\mathcal{E},p}^-\right)   \right) - \Delta t_n \max_{1\le p \le Q_{K}} \lambda_{ \mathbf{W}_{K,p} }    \right) \overline{\mathbf{W}}_K^n \cdot \widehat{\bm {\xi}},  
\end{align*}
which is positive for any $\widehat{{\bm \xi}} \in \widehat{\bm \Xi}_*$ under the following CFL-type condition:
\begin{align}\label{PCP_CFL}
\Delta t_n \left( \max_{1\le p \le Q_{K}} \lambda_{ \mathbf{W}_{K,p} } 
+ 
 \max_{ \mathcal{E} \in \partial K } \max_{1\le p \le Q_{\mathcal E}} \frac{|\mathcal{E}|{\omega}_{\mathcal{E},p}}{ |K| \widehat  {\omega}_{\mathcal{E},p}}  \left(s^+_{\mathcal{E},p} -s_{\mathcal{E},p}^-\right) 
\right) \le 1 \qquad \forall K \in {\mathcal T}_h. 
\end{align}

In conclusion, we have proven the following weak PCP property for the updated cell average. 

\begin{theorem}\label{pcp1}
If the DG solution ${\bf W}_h({\bf x})$ satisfies \eqref{PCPpointcond} and the HLL speeds satisfy \eqref{wkl-HLL} for all $K \in {\mathcal T}_h$, then the 
  updated cell averages $\overline{\bf W}_K^{n+1} $ given by our scheme \eqref{first_order00} belong to the admissible state set ${\mathcal G}^*$ or equivalently ${\mathcal G}^{(2)}$, under the CFL-type condition \eqref{PCP_CFL}. 
 \end{theorem}

The condition \eqref{PCPpointcond} may not be always satisfied by the (OE) DG solutions and should be enforced by a PCP limiter, which will be introduced in Section \ref{sec:PCPlim}. 
Although our above PCP analysis was focused on the forward Euler time discretization, it is also valid for the high-order SSP methods, which are formally the convex combination of forward Euler steps. When a high-order SSP RK method is used, the condition \eqref{PCPpointcond} should be enforced at each RK stage.

\subsubsection{PCP Limiter}\label{sec:PCPlim}

Although the high-order accurate OEDG method can effectively suppress spurious oscillations, it does not inherently guarantee that the numerical solution $\mathbf{W}_h^n$ always satisfies \eqref{PCPpointcond}, particularly in cases involving low density (or pressure) or a large Lorentz factor. Therefore, it is necessary to apply a PCP limiter to enforce the condition \eqref{PCPpointcond}.

Denote the OEDG solution at the $n$-th time level as $\mathbf{W}_{\sigma}^n(\mathbf{x}) = \left(W_{\sigma}^{(0)}(t_n,\mathbf{x}), W_{\sigma}^{(1)}(t_n,\mathbf{x}), \dots, W_{\sigma}^{(d+1)}(t_n,\mathbf{x})\right)^\top$. Let $\overline{\mathbf{W}}_{\sigma,K}^n = (\overline{W}_{\sigma,K}^{(0)}, \overline{W}_{\sigma,K}^{(1)}, \dots, \overline{W}_{\sigma,K}^{(d+1)})^\top \in \mathcal{G}^{(2)}$ represent the cell average of $\mathbf{W}_{\sigma,K}^n(\mathbf{x})$ over cell $K$. 
A PCP limiter is applied after each OE procedure to modify $\mathbf{W}_{\sigma}^n(\mathbf{x})$ to $\widetilde{\mathbf{W}}_{\sigma}^n(\mathbf{x})$. This modification ensures that the updated solution satisfies
\begin{align*}
\widetilde{\mathbf{W}}_{\sigma}^n(\mathbf{x}) \in \mathcal{G}^{(2)} = {\mathcal G}^* \qquad \forall \, \mathbf{x} \in \mathbb{S}_K, \quad \forall K \in {\mathcal T}_h, 
\end{align*}
where $\mathbb{S}_K$ denotes a set of critical points in cell $K$. According to Theorem \ref{pcp1}, under the CFL-type condition \eqref{PCP_CFL}, the updated cell-averaged OEDG numerical solution $\overline{\mathbf{W}}_{\sigma}^{n+1}$ will remain in $\mathcal{G}^{(2)}$. This then ensures the validity of the PCP limiter at the next time step. 

The PCP limiter \cite{W2017} is an extension of the limiters for non-relativistic and special relativistic hydrodynamics \cite{ZS20102,WT2015,QSY2016,WT2017}. Specifically, the PCP limiting procedure on each cell $K$ consists of the following two steps. 

\vspace{2mm}
\textbf{Step 1: Enforcing the positivity of the first component of $\mathbf{W}$.}

Define $W_{\min}^{(0)} = \min\limits_{\mathbf{x} \in \mathbb{S}_K} W_{\sigma}^{(0)}(t_n, \mathbf{x})$. Set $\epsilon_1 = \min\left\{ 10^{-13}, \overline{W}_{\sigma,K}^{(0)} \right\}$ to avoid the effects of round-off errors. If $W_{\min}^{(0)} < \epsilon_1$, modify the first component of the OEDG solution as follows:
\begin{align*}
\widehat{W}_{\sigma}^{(0)}(t_n, \mathbf{x}) = \theta_1 \left(W_{\sigma}^{(0)}(t_n, \mathbf{x}) - \overline{W}_{\sigma,K}^{(0)}\right) + \overline{W}_{\sigma,K}^{(0)},
\end{align*}
where the scaling factor $\theta_1$ is given by
\begin{align*}
\theta_1 = \frac{\overline{W}_{\sigma,K}^{(0)} - \epsilon_1}{\overline{W}_{\sigma,K}^{(0)} - W_{\min}^{(0)}}.
\end{align*}
The limited solution is then denoted by $\widehat{\mathbf{W}}_{\sigma}^n(\mathbf{x}) = \left(\widehat{W}_{\sigma}^{(0)}(t_n, \mathbf{x}), W_{\sigma}^{(1)}(t_n, \mathbf{x}), \dots, W_{\sigma}^{(d+1)}(t_n, \mathbf{x})\right)^\top$.

\vspace{2mm}
\textbf{Step 2: Enforcing the positivity of $q(\mathbf{W})$.}

Define $q_{\min} = \min\limits_{\mathbf{x} \in \mathbb{S}_K} q(\widehat{\mathbf{W}}_{\sigma}^n(\mathbf{x}))$. Set $\epsilon_2 = \min\left\{ 10^{-13}, q(\overline{\mathbf{W}}_{\sigma,K}^n)\right\}$. If $q_{\min} < \epsilon_2$, modify $\widehat{\mathbf{W}}_{\sigma}^n(\mathbf{x})$ as follows:
\begin{align*}
\widetilde{\mathbf{W}}_{\sigma}^n(\mathbf{x}) = \theta_2 \left(\widehat{\mathbf{W}}_{\sigma}^n(\mathbf{x}) - \overline{\mathbf{W}}_{\sigma,K}^n\right) + \overline{\mathbf{W}}_{\sigma,K}^n,
\end{align*}
where the scaling factor $\theta_2$ is defined as
\begin{align*}
\theta_2 = \frac{q(\overline{\mathbf{W}}_{\sigma,K}^n) - \epsilon_2}{q(\overline{\mathbf{W}}_{\sigma,K}^n) - q_{\min}}.
\end{align*}


\subsubsection{PCP Convergent Iteration Algorithms for Recovering Primitive Variables}

As seen in the DG formulation \eqref{DG form2}, it is necessary to evaluate the flux vectors $\mathbf{H}^i(\mathbf{W})$, the source term $\mathbf{S}(\mathbf{W})$, and the HLL wave speeds $s^\pm$ in each cell at every RK stage. These quantities depend on both the spacetime metric and the evolved variables $\mathbf{W}$, as well as the primitive variables $\mathbf{Q} = (\rho, {\bm v}, p)^\top$. However, due to the nonlinear coupling introduced by relativistic effects, the primitive variables $\mathbf{Q}$ cannot be explicitly expressed in terms of the evolved variables $\mathbf{W}$. 
Thus, it is essential to recover the primitive variables $\mathbf{Q}$ from the evolved variables $\mathbf{W}$ during the computation. Typically, this involves solving a nonlinear algebraic equation to obtain an intermediate variable (e.g., pressure $p$), followed by calculating the remaining primitive variables using this intermediate variable and $\mathbf{W}$. This process is known as the {\bf primitive variable recovery problem}, a challenging aspect encountered in almost all numerical schemes for RHD, but not present in the non-relativistic case, where the primitive variables can be directly derived from conserved quantities.

In previous sections \ref{sec:PCPave} and \ref{sec:PCPlim}, we have discussed how to ensure the PCP property for the evolved variables $\mathbf{W}$ for both cell averages and point values. To guarantee that our schemes fully maintain the PCP property throughout the entire computational process, it is crucial to ensure that the recovered primitive variables also satisfy the physical constraints. In this subsection, we introduce efficient and provably PCP-convergent iteration algorithms, which extend those developed for the special relativistic case \cite{CW2022, CQW2024}. The PCP convergent algorithms ensure the approximately recovered primitive variables always meet the physical constraints, provided that the given evolved variables $\mathbf{W} \in \mathcal{G}^{(2)}$.

Given $\mathbf{W} = \sqrt{\chi}(D, {\bm m}\mathbf{\Theta}^\top, E)^\top =: (\widehat{D}, \widehat{{\bm m}}, \widehat{E})^\top$ and the spacetime metric information, the corresponding conserved quantities in the ADM formulation \eqref{eq: form2} are computed as
$$
\mathbf{U} =  (D, {\bm m}, E)^\top = \left(\frac{\widehat{D}}{\sqrt{\chi}}, \frac{1}{\sqrt{\chi}} \widehat{{\bm m}} \big( \mathbf{\Theta}^\top \big)^{-1}, \frac{\widehat{E}}{\sqrt{\chi}} \right)^\top \in {\mathcal G}_{\chi}^{(1)}.
$$
Using $\gamma = \left(1 - \frac{{\bm m}\mathbf{X}{\bm m}^\top}{E + p}\right)^{-\frac{1}{2}}$ and the inverse transformation of \eqref{pri2con}, the pressure $p(\mathbf{U})$ is determined by solving the following nonlinear algebraic equation:
\begin{align}\label{nonlinear eq1}
\phi_{\mathbf{U}}(p) := \frac{p}{\Gamma - 1} - E + \frac{{\bm m}\mathbf{X}{\bm m}^\top}{E + p} + D\sqrt{1 - \frac{{\bm m}\mathbf{X}{\bm m}^\top}{(E + p)^2}} = 0.
\end{align}
Once the approximate $p(\mathbf{U})$ is found, the velocity and rest-mass density are computed as
\begin{align} 
{\bm v}(\mathbf{U}) = \frac{{{\bm m}}}{{E} + p(\mathbf{U})},\qquad \rho(\mathbf{U}) = D \sqrt{1 - \|{\bm v}(\mathbf{U})\|_{\bf X}^2}. \label{nonlinear eq2}
\end{align}
Note that the function $\phi_{\mathbf{U}}(p)$ has the following properties:
\begin{align}\label{properties}
\phi_{\mathbf{U}}(0) < 0,\qquad \lim_{p \to +\infty} \phi_{\mathbf{U}}(p) = +\infty,\qquad \phi_{\mathbf{U}}'(p) > 0 \quad \forall p \in [0, +\infty).
\end{align}
Thus, $\phi_{\mathbf{U}}(p)$ is monotonically increasing with a unique positive root for $p \in [0, +\infty)$. Since obtaining an analytical expression for $p(\mathbf{U})$ from the nonlinear equation \eqref{nonlinear eq1} is generally intractable, iterative numerical methods are used.

For a given $\mathbf{W} \in \mathcal{G}^{(2)}$ or $\mathbf{U} \in \mathcal{G}_{\chi}^{(1)}$, an iterative algorithm is considered PCP if its iterative sequence $\{p^{(n)}\}_{n \geq 1}$ satisfies the following conditions:
\begin{enumerate}
    \item The sequence converges to the physical pressure $p(\mathbf{U})$, i.e., $\lim_{n \to +\infty} p^{(n)} = p(\mathbf{U})$.
    \item The sequence remains non-negative, i.e., $p^{(n)} \geq 0$ for all $n \geq 1$.
\end{enumerate}
The first condition ensures convergence, while the second guarantees stability and that the primitive variables remain PCP during the iterations. Indeed, if the recovered pressure $p(\mathbf{U}) > 0$, then $\|{\bm v}(\mathbf{U})\|_{\mathbf{X}} < 1$ and $\rho(\mathbf{U}) > 0$, meaning that maintaining a positive pressure is sufficient to ensure the PCP property for the recovered primitive variables.

Inspired by our previous work on the special RHD case \cite{CQW2024, CW2022}, we extend several PCP convergent algorithms to the primitive variable recovery process for GRHDs. 

\vspace{3mm}

\noindent\textbf{Algorithm \uppercase\expandafter{\romannumeral1}: Provably PCP, Linearly Convergent Hybrid Algorithm.}

This algorithm combines the bisection method with the fixed-point iteration method to recover primitive variables.

\vspace{2mm}

\textbf{Bisection Method:}  
From \eqref{nonlinear eq1}, we know that
\begin{align*}
p(\mathbf{U}) = (\Gamma - 1)\left(E - \frac{{\bf m} {\bf X} {\bf m}^\top}{E + p(\mathbf{U})} - D \sqrt{1 - \frac{{\bf m} {\bf X} {\bf m}^\top}{(E + p(\mathbf{U}))^2}}\right)
\le (\Gamma - 1)\left(E - D \sqrt{1 - \frac{{\bf m} {\bf X} {\bf m}^\top}{E^2}}\right) =: p_R^{(0)}.
\end{align*}
This shows that $p(\mathbf{U})$ lies within the interval $[0, p_R^{(0)}]$, which can be used as the initial range for the bisection method. Due to the monotonicity of $\phi_{\mathbf{U}}(p)$, the bisection method is convergent and PCP. The error estimate is given by
\begin{align*}
|p^{(n)} - p(\mathbf{U})| \le \frac{p_R^{(0)}}{2^{n+1}},
\end{align*}
where $p^{(n)}$ is the midpoint of the $n$-th iteration interval. This indicates that the method converges linearly with a contraction rate of $0.5$.

\vspace{2mm}

\textbf{Fixed-Point Iteration:} 
By rearranging the nonlinear equation \eqref{nonlinear eq1}, we derive an equivalent form:
\begin{align*}
p = p - (\Gamma - 1)\phi_{\mathbf{U}}(p),
\end{align*}
which suggests the following fixed-point iteration scheme:
\begin{align}\label{fixedpoint}
p^{(n)} = p^{(n-1)} - (\Gamma - 1)\phi_{\mathbf{U}}(p^{(n-1)}), \qquad n = 1, 2, \dots
\end{align}
For any $p > 0$ and $\Gamma \in (1,2]$, we have
\begin{align*}
0 \le 1 - (\Gamma - 1)\phi_{\mathbf{U}}'(p) = (\Gamma - 1)\frac{{\bf m} {\bf X} {\bf m}^\top}{(E + p)^2}\left(1 - \frac{D}{\sqrt{(E + p)^2 - {\bf m} {\bf X} {\bf m}^\top}}\right)
 \le (\Gamma - 1)\frac{{\bf m} {\bf X} {\bf m}^\top}{E^2} =: \delta < 1.
\end{align*}
This shows that the iterative function $p - (\Gamma - 1)\phi_{\mathbf{U}}(p)$ is monotonically increasing for $p \in [0, +\infty)$ and has a Lipschitz constant less than 1. By the Banach fixed-point theorem, the sequence $\{p^{(n)}\}_{n \geq 1}$ converges to $p(\mathbf{U})$ at a rate of $\delta$. In addition, $p^{(n)}$ remains positive as long as $p^{(0)} > 0$.

\vspace{2mm}

\textbf{Hybrid Strategy:} 
Both the bisection and fixed-point iteration methods converge linearly, with rates of $0.5$ and $\delta$, respectively. To improve convergence speed, we can construct a hybrid algorithm by selectively applying these two methods based on their contraction rates. Specifically:
\begin{align*} 
\begin{cases} 
\text{Apply the bisection method with initial interval } [0, p_R^{(0)}], & \text{if} \quad \delta > 0.5, \\
\text{Apply fixed-point iteration \eqref{fixedpoint} with } p^{(0)} = p_R^{(0)}/2, & \text{otherwise}. 
\end{cases}
\end{align*}
The procedure for this hybrid iteration algorithm is summarized in Algorithm \ref{alg1}.

\begin{algorithm}[H]
\caption{Provably PCP, Linearly Convergent Hybrid Algorithm} 
\label{alg1} 
\renewcommand{\algorithmicrequire}{\textbf{Input:}}
\renewcommand{\algorithmicensure}{\textbf{Output:}}
\begin{algorithmic}
\REQUIRE Given $\mathbf{W} \in \mathcal{G}^{(2)}$ or $\mathbf{U} = (D, {\bf m}, E)^\top \in \mathcal{G}_{\chi}^{(1)}$ and the spacetime metric information.  
$\epsilon_{r}$ is the machine epsilon, and $\epsilon_{tol}$ is the tolerance for error. 
\STATE $p_L \leftarrow 0$;
\STATE $p_R \leftarrow (\Gamma-1)\left(E - D \sqrt{1 - \frac{{\bf m} {\bf X} {\bf m}^\top}{E^2}}\right)$;
\STATE $p \leftarrow \frac{p_R}{2}$;
\STATE $\phi_0 \leftarrow \epsilon_{tol} + 1$;
\STATE Set $r = \min\left\{0.5, (\Gamma-1)\frac{{\bf m} {\bf X} {\bf m}^\top}{E^2} + \epsilon_r \right\},\, n = 0,\, N = \frac{\log(\epsilon_r/p)}{\log(r)}$;
\WHILE {$\phi_0 > \epsilon_{tol}$ \& $n \le N$}
\STATE Compute $\phi_{\mathbf{U}} \leftarrow \frac{p}{\Gamma -1} - E + \frac{{\bf m} {\bf X} {\bf m}^\top}{E + p} + D \sqrt{1 - \frac{{\bf m} {\bf X} {\bf m}^\top}{(E + p)^2}}$;
\IF{$r < 0.5$}
\STATE $p \leftarrow p - (\Gamma - 1)\phi_{\mathbf{U}}$;
\ELSE
\IF{$\phi_{\mathbf{U}} < 0$}
\STATE $p_L \leftarrow p$;
\ELSE
\STATE $p_R \leftarrow p$;
\ENDIF
\STATE $p \leftarrow \frac{p_L + p_R}{2}$;
\ENDIF
\STATE $n \leftarrow n + 1$; \quad $\phi_0 \leftarrow |\phi_{\mathbf{U}}|$;
\ENDWHILE
\ENSURE $p$ is an approximation to $p(\mathbf{U})$.
\end{algorithmic}
\end{algorithm}

The above iterative algorithm is linearly convergent. Next, we introduce a PCP convergent Newton method, which converges quadratically. 

\vspace{3mm}

\noindent\textbf{Algorithm \uppercase\expandafter{\romannumeral2}: Provably PCP, Quadratically Convergent Newton Algorithm}

\vspace{2mm}

\textbf{Step 1: Reformulation of the Nonlinear Equation \eqref{nonlinear eq1}.}

By multiplying both sides of equation \eqref{nonlinear eq1} by $(E + p)$, we obtain
\begin{align}\label{nonlinear eq4}
\widehat{\phi}_{\mathbf{U}}(p) := {\bf m} {\bf X} {\bf m}^\top + \left(E + p\right)\left(\frac{p}{\Gamma-1} - E\right) + D\sqrt{\left(E + p\right)^2 - {\bf m} {\bf X} {\bf m}^\top} = 0.
\end{align}

\vspace{2mm}

\textbf{Step 2: Providing a Robust Initial Value.}

The initial value for the Newton method is defined as
\begin{align}\label{eq:robustP0}
p^{(0)} =
\begin{cases}
0, & \text{if} \quad D E + {\bf m} {\bf X} {\bf m}^\top \ge E^2, \\
p_{\star}, & \text{otherwise},
\end{cases}
\end{align}
where
\begin{align}\label{eq:pstar}
p_\star = \frac{(\Gamma - 2)E + \sqrt{(2 - \Gamma)^2 E^2 - 4(\Gamma - 1)\left({\bf m} {\bf X} {\bf m}^\top + D\sqrt{E^2 - {\bf m} {\bf X} {\bf m}^\top} - E^2  \right)}}{2}.
\end{align}

\vspace{2mm}

\textbf{Step 3: Iterative Formula.}

The iterative sequence is given by
\begin{align}\label{eq:Newton}
p^{(n+1)} = p^{(n)} - \frac{\widehat{\phi}_{\mathbf{U}}(p^{(n)})}{\widehat{\phi}_{\mathbf{U}}'(p^{(n)})},
\end{align}
where
\begin{align*}
\widehat{\phi}_{\mathbf{U}}'(p) = \frac{2 p + E(2 - \Gamma)}{\Gamma - 1} + \frac{D\left(E + p\right)}{\sqrt{\left(E + p\right)^2 - {\bf m} {\bf X} {\bf m}^\top}}.
\end{align*}

Similar to the special RHD case \cite{CQW2024}, we can rigorously prove the following properties of the Newton iteration \eqref{eq:Newton} for solving $\widehat{\phi}_{\mathbf{U}}(p)$:
\begin{itemize}
    \item The Newton iteration \eqref{eq:Newton} is always PCP convergent when the initial value is $p^{(0)} = 0$. This ensures that the algorithm preserves the physical constraints from the start with this specific initial condition.
    
    \item If $D E + {\bf m} {\bf X} {\bf m}^\top < E^2$, then the Newton iteration \eqref{eq:Newton} is always PCP convergent for any non-negative initial value $p^{(0)} \geq 0$. In this case, the initial guess $p_\star$, as defined in \eqref{eq:pstar}, is typically closer to the true root $p(\mathbf{U})$ and may serve as a more efficient starting point compared to $p^{(0)} = 0$.
\end{itemize}
Consequently, the Newton iteration \eqref{eq:Newton}, when initialized with the robust starting value \eqref{eq:robustP0}, is provably PCP convergent.
Moreover, if an iteration value $p^{(n_*)} \in [ p(\mathbf{U}), +\infty )$ appears, the subsequent sequence $\{ p^{(n)} \}_{n \geq n_*}$ will converge monotonically to $p(\mathbf{U})$. This monotonicity is crucial for detecting and preventing oscillations in the Newton iteration when the sequence approaches $p(\mathbf{U})$. 
It can also be shown that $p(\mathbf{U})$ is not a repeated root of $\widehat{\phi}_{\mathbf{U}}(p)$. As a result, the Newton iteration \eqref{eq:Newton} converges quadratically. 

For the reader's convenience, the pseudo code for this Newton algorithm is summarized in Algorithm \ref{alg2}. 
Given its robustness and efficiency, this algorithm is highly recommended for practitioners and is employed in our simulations.

\begin{algorithm}[H]
\caption{Provably PCP, Quadratically Convergent Newton Algorithm}
\label{alg2}
\renewcommand{\algorithmicrequire}{\textbf{Input:}}
\renewcommand{\algorithmicensure}{\textbf{Output:}}
\begin{algorithmic}
\REQUIRE Given $\mathbf{W} \in \mathcal{G}^{(2)}$ or $\mathbf{U} = (D, {\bf m}, E)^\top \in \mathcal{G}_{\chi}^{(1)}$ with spacetime metric information. Set $\epsilon_{tol} = 10^{-14}$.
\IF{$DE + {\bf m} {\bf X} {\bf m}^\top \ge E^2$}
\STATE $p \leftarrow 0$.
\ELSE
\STATE $p \leftarrow \frac{(\Gamma - 2)E + \sqrt{(2 - \Gamma)^2 E^2 - 4(\Gamma - 1)\left({\bf m} {\bf X} {\bf m}^\top+ D\sqrt{E^2 - {\bf m} {\bf X} {\bf m}^\top}  - E^2  \right)}}{2}$.
\ENDIF
\STATE Set $z \leftarrow \epsilon_{tol} + 1$, $N_{osc} \leftarrow 0$, $\phi_0 \leftarrow 0$.
\WHILE{$|z| > \epsilon_{tol}$ \& $N_{osc} < 3$}
\STATE Compute $\phi \leftarrow {\bf m} {\bf X} {\bf m}^\top + \left(E + p\right)\left(\frac{p}{\Gamma - 1} - E\right) + D\sqrt{\left(E + p\right)^2 - {\bf m} {\bf X} {\bf m}^\top}$.
\STATE Compute $d\phi \leftarrow \frac{2 p + E(2 - \Gamma)}{\Gamma - 1} + \frac{D\left(E + p\right)}{\sqrt{\left(E + p\right)^2 - {\bf m} {\bf X} {\bf m}^\top}}$.
\STATE Update $z \leftarrow \frac{\phi}{d\phi}$; \quad $p \leftarrow p - z$.
\IF{$\phi_0 \times \phi < 0$}
\STATE $N_{osc} \leftarrow N_{osc} + 1$.
\ENDIF
\STATE Update $\phi_0 \leftarrow \phi$.
\ENDWHILE
\ENSURE $p$ is an approximation to $p(\mathbf{U})$.
\end{algorithmic}
\end{algorithm}

\section{Illustrations and Applications of PCP-OEDG Schemes to RHD under Various Spacetime Metrics}\label{sec4}

In this section, we demonstrate the application of the proposed PCP-OEDG method for the RHD equations under different spacetime metrics and coordinate systems.

\subsection{Application to Special RHD Equations under the Minkowski Metric in Cartesian Coordinates}

This subsection showcases the PCP-OEDG method applied to the special RHD equations in flat spacetime (Minkowski metric). For simplicity, we focus on 1D and 2D cases.

\subsubsection{1D Case}

We consider the 1D special RHD system in the spatial domain $\Omega = [a,b]$:
\begin{align}\label{1DSRHD}
\frac{\partial \mathbf{U}}{\partial t} + \frac{\partial \mathbf{F}(\mathbf{U})}{\partial x} = \mathbf{0},
\end{align}
where $\mathbf{U} = (D, m, E)^{\top}$ is the vector of conserved variables, and $\mathbf{F} = (Dv, m v + p, m)^{\top}$ is the flux. The computational domain $\Omega$ is divided into $N_{x}$ uniform cells:
\begin{align}
a = x_{\frac{1}{2}} < x_{\frac{3}{2}} < \cdots < x_{N_{x} + \frac{1}{2}} = b,
\end{align}
where the $j$-th cell is denoted by $K_j = [x_{j-\frac{1}{2}}, x_{j+\frac{1}{2}}]$. The center of cell $K_j$ is $x_{j} = \frac{x_{j-\frac{1}{2}} + x_{j+\frac{1}{2}}}{2}$, and the cell width is $h_{x} = \frac{b-a}{N_{x}}$. 
We define the Legendre polynomials $\{\widehat{\phi}^{(q)}(\xi): 0 \leq q \leq m\}$ on the reference cell $\widehat{K} = [-1,1]$:
\begin{align*}
\widehat{\phi}^{(0)} = 1, \quad \widehat{\phi}^{(1)} = \xi, \quad \widehat{\phi}^{(2)} = \frac{1}{2}(3\xi^2 - 1), \quad \widehat{\phi}^{(3)} = \frac{1}{2}(5\xi^3 - 3\xi), \quad \dots
\end{align*}
Let $F_{K_j}: \widehat{K} \to K_j$ be the linear mapping defined by $x = F_{K_j}(\xi) := x_{j} + \frac{h_{x}}{2}\xi$ with $\xi \in \widehat{K}$. The basis functions on $K_j$ are then given by ${\phi}^{(q)}_{K_j}(x) = \widehat{\phi}^{(q)}(F_{K_j}^{-1}(x))$. In this 1D case, $N = \text{dim}(\mathbb{P}^m) = m+1$.

The semi-discrete DG scheme for \eqref{1DSRHD} is to find $\mathbf{U}_h \in [V_h^m]^3$ such that
\begin{align*}
\int_{K_j} \frac{\partial \mathbf{U}_h(t,x)}{\partial t} v_h(x) \,\mathrm{d}x = \int_{K_j} \mathbf{F}(\mathbf{U}_h(t,x)) v_h'(x) \,\mathrm{d}x - \widehat{\mathbf{F}}_{j+\frac{1}{2}}v_h(x_{j+\frac{1}{2}}^-) + \widehat{\mathbf{F}}_{j-\frac{1}{2}}v_h(x_{j-\frac{1}{2}}^+), \quad \forall v_h \in V_h^m,
\end{align*}
where the numerical fluxes are given by
\[
\widehat{\mathbf{F}}_{j+\frac{1}{2}} = \widehat{\mathbf{F}}(\mathbf{U}_{j+ \frac{1}{2}}^-, \mathbf{U}_{j+ \frac{1}{2}}^+), \quad \widehat{\mathbf{F}}_{j-\frac{1}{2}} = \widehat{\mathbf{F}}(\mathbf{U}_{j- \frac{1}{2}}^-, \mathbf{U}_{j- \frac{1}{2}}^+),
\]
and $\mathbf{U}_{j+ \frac{1}{2}}^{\pm} = \mathbf{U}_h(t, x_{j+\frac{1}{2}}^{\pm})$ represent the left and right limiting values at $x_{j+\frac{1}{2}}$. For the 1D case, the HLL flux is
\begin{align}\label{HLL flux for RHD}
\widehat{\mathbf{F}}^{\mathrm{HLL}}(\mathbf{U}^-, \mathbf{U}^+) = \frac{s^+ \mathbf{F}(\mathbf{U}^-) - s^- \mathbf{F}(\mathbf{U}^+) + s^+s^- (\mathbf{U}^+ - \mathbf{U}^-)}{s^+ - s^-},
\end{align}
where $s^-, s^+$ are the wave speeds, computed as degenerate versions of \eqref{s-} and \eqref{s+} with $\alpha = 1$ and the eigenvalues $\lambda_{x}^{(0)} = \frac{v-c_s}{1-v c_s}$ and ${\lambda_{x}^{(2)}}=\frac{v+c_s}{1+v c_s}$.

The semi-discrete DG scheme can be written in matrix form as
\begin{align}\label{1DRHDDG}
\frac{\mathrm{d}\mathcal{U}_{K_j}(t)}{\mathrm{d}t} = \left(\mathbf{A}^{F}_{K_j} - \mathbf{A}^{\widehat{F}}_{K_j}\right) \mathbf{A}^{-1}_{K_j}, \qquad  j=1,\,2,\,\cdots,\,N_x,
\end{align}
where $\mathcal{U}_{K_j}(t) \in \mathbb{R}^{3 \times (m+1)}$ contains the modal coefficients of $\mathbf{U}_h$ on $K_j$, and the mass matrix $\mathbf{A}_{K_j}$ is defined by
\[
\mathbf{A}_{K_j} = \text{diag}\left\{ \frac{h_x}{2i+1}\right\}_{0\le i \le m}.
\]
The matrices $\mathbf{A}^{F}_{K_j}$ and $\mathbf{A}^{\widehat{F}}_{K_j}$ are
\begin{align}
\mathbf{A}^{F}_{K_j} &= |K_j| \sum_{p=1}^Q \omega_{p} \mathbf{F}(\mathcal{U}_{K_j}(t)\mathbf{\Phi}_{K_j}(x_{j,p}))(\mathbf{\Phi}'_{K_j})^{\top}(x_{j,p}) \approx \int_{K_j} \mathbf{F}(\mathbf{U}_h(t,x))(\mathbf{\Phi}'_{K_j})^{\top} \,\mathrm{d}x, \label{AFK}\\
\mathbf{A}^{\widehat{F}}_{K_j} &= -\widehat{\mathbf{F}}_{j+\frac{1}{2}}\mathbf{\Phi}_{K_j}(x_{j+\frac{1}{2}}^-) + \widehat{\mathbf{F}}_{j-\frac{1}{2}}\mathbf{\Phi}_{K_j}(x_{j-\frac{1}{2}}^+),
\end{align}
where $\{x_{j,p}\}_{p=1}^Q$ are the Gauss quadrature points and $\{\omega_{p}\}_{p=1}^Q$ are the corresponding weights.

The scheme \eqref{1DRHDDG} on cell $K_j$ can be expressed as an ODE system: 
\[
\frac{\mathrm{d}\mathcal{U}_{K_j}(t)}{\mathrm{d}t} = \mathcal{L}_{K_j}\left(\{\mathcal{U}_{K}\}_{K \in \Lambda_{K_j}}\right), \qquad \Lambda_{K_j} = \left\{K_{j-1}, K_{j}, K_{j+1}\right\},
\]
where $\mathcal{L}_{K_j}\left(\{\mathcal{U}_{K}\}_{K \in \Lambda_{K_j}}\right) = \left(\mathbf{A}^{F}_{K_j} - \mathbf{A}^{\widehat{F}}_{K_j}\right)\mathbf{A}^{-1}_{K_j}$. 
By discretizing this ODE system in time using a RK method and applying the OE procedure after each stage, we obtain the OEDG method. For example, using a third-order explicit RK scheme, the OEDG method reads:
\begin{align}\label{eq:1D-OEDG}
\begin{cases}
\mathcal{U}_{\sigma,{K_j}}^{n,0} = \mathcal{U}_{\sigma,{K_j}}^{n}, \\
\mathcal{U}_{{K_j}}^{n,1} = \mathcal{U}_{\sigma,{K_j}}^{n,0} + \Delta t_n \mathcal{L}_{K_j}\left(\{\mathcal{U}_{\sigma,{K}}^{n,0}\}_{{K} \in \Lambda_{K_j}}\right), &\mathcal{U}_{\sigma,{K_j}}^{n,1} = \mathcal{U}_{{K_j}}^{n,1}\exp\left(-\Delta t_n \mathbf{D}\left(\{\mathcal{U}_{K}^{n,1}\}_{K \in \Lambda_{K_j}}\right)\right), \\
\mathcal{U}_{{K_j}}^{n,2} = \frac{3}{4}\mathcal{U}_{\sigma,K_j}^{n,0} + \frac{1}{4}\left(\mathcal{U}_{\sigma,K_j}^{n,1} + \Delta t_n \mathcal{L}_{K_j}\left(\{\mathcal{U}_{\sigma,K}^{n,1}\}_{K \in \Lambda_{K_j}}\right)\right), &\mathcal{U}_{\sigma,K_j}^{n,2} = \mathcal{U}_{K_j}^{n,2}\exp\left(-\Delta t_n \mathbf{D}\left(\{\mathcal{U}_{K}^{n,2}\}_{K \in \Lambda_{K_j}}\right)\right), \\
\mathcal{U}_{K_j}^{n,3} = \frac{1}{3}\mathcal{U}_{\sigma,K_j}^{n,0} + \frac{2}{3}\left(\mathcal{U}_{\sigma,K_j}^{n,2} + \Delta t_n \mathcal{L}_{K_j}\left(\{\mathcal{U}_{\sigma,K}^{n,2}\}_{K \in \Lambda_{K_j}}\right)\right), &\mathcal{U}_{\sigma,K_j}^{n,3} = \mathcal{U}_{K_j}^{n,3}\exp\left(-\Delta t_n \mathbf{D}\left(\{\mathcal{U}_{K}^{n,3}\}_{K \in \Lambda_{K_j}}\right)\right),
\end{cases}
\end{align}
where
\[
\mathbf{D}\left(\{\mathcal{U}_K^{n,\ell+1}\}_{K \in \Lambda_{K_j}}\right) = \text{diag}\left\{d_0\left(\{\mathcal{U}_K^{n,\ell+1}\}_{K \in \Lambda_{K_j}}\right), \dots, d_m\left(\{\mathcal{U}_K^{n,\ell+1}\}_{K \in \Lambda_{K_j}}\right)\right\}, \quad \ell = 0, 1, 2,
\]
with
\[
d_0\left(\{\mathcal{U}_K^{n,\ell+1}\}_{K \in \Lambda_{K_j}}\right) =0,\quad d_q\left(\{\mathcal{U}_K^{n,\ell+1}\}_{K \in \Lambda_{K_j}}\right) =\sum_{r =0}^q \delta_{K_j}^r(\mathbf{U}^{n,\ell+1}_h), \quad 1 \leq q \leq m.
\]
The damping coefficient is
\[
\delta_{K_j}^r(\mathbf{U}^{n,\ell+1}_h) = \frac{\eta_j }{h_x} \max_{0 \leq i \leq 2} \left\{\sigma_{K_j}^r(U_h^{(i)})\right\},
\]
where $U_h^{(i)}$ is the $(i+1)$-th component of $\mathbf{U}^{n,\ell+1}_h$, $\eta_j$ is the spectral radius of the Jacobian matrix $\frac{\partial \mathbf{F}}{\partial \mathbf{U}}(\overline{\mathbf{U}}^{n,\ell+1}_{j})$, and
\[
\sigma^r_{K_j}(U_h^{(i)}) =
\begin{cases}
0, & \text{if} \quad U_h^{(i)} \equiv \text{avg}_\Omega(U_h^{(i)}), \\
\displaystyle
\frac{(2r+1)h_x^r}{2(2m-1)r!} \frac{\left|[\![\partial^r_x U_h^{(i)}]\!]_{j-\frac{1}{2}}\right| + \left|[\![\partial^r_x U_h^{(i)}]\!]_{j+\frac{1}{2}}\right|}{\|U_h^{(i)} - \text{avg}_\Omega(U_h^{(i)})\|_{L^{\infty}(\Omega)}}, & \text{otherwise}.
\end{cases}
\]

For special RHD, the GQL representation of the admissible state set is
\[
\mathcal{G}^* = \left\{\mathbf{U} = (D, m, E)^{\top}: \mathbf{U} \cdot \widehat{\bm{\xi}} > 0 \,\, \forall \, \widehat{\bm{\xi}} \in \widehat{\bm{\Xi}}_*\right\},
\]
with the vector set $\widehat{\bm{\Xi}}_*$ defined in \eqref{thea}. Let $\left\{x_{j,p}^{\text{GL}}\right\}_{p =1}^L$ represent the $L$-point Gauss-Lobatto quadrature nodes on $K_j$, where $L = \lceil \frac{m+3}{2} \rceil$. The corresponding weights are denoted by $\{\omega_p^{\text{GL}}\}_{p=1}^L$. The cell average admits the convex decomposition
\begin{align}\label{con avg}
\overline{\mathbf{U}}_j^n = \sum_{p=2}^{L-1}\omega_p^{\text{GL}} \mathbf{U}_h^n(x_{j,p}^{\text{GL}}) + \omega_1^{\text{GL}} \left(\mathbf{U}_{j-\frac{1}{2}}^+ + \mathbf{U}_{j+\frac{1}{2}}^-\right).
\end{align}
Thus, the point set in \eqref{ssk} becomes
\[
\mathbb{S}_{j} = \left\{x_{j,p}^{\text{GL}}\right\}_{p=1}^L \cup \left\{x_{j,p}\right\}_{p=1}^Q,
\]
which includes all the quadrature points involved in \eqref{AFK} and \eqref{con avg}.

As a direct consequence of Theorem \ref{pcp1}, we obtain the following result on the PCP property of our OEDG schemes for 1D special RHD.

\begin{corollary}
Let  $s_{j+\frac{1}{2}}^+ = \max \left\{\lambda_x^{(2)} (\mathbf{U}_{j+\frac12}^-), \lambda_x^{(2)} (\mathbf{U}_{j+\frac12}^+), 0 \right\}$ and $s_{j+\frac{1}{2}}^- = \min \left\{\lambda_x^{(0)} (\mathbf{U}_{j+\frac12}^-), \lambda_x^{(0)} (\mathbf{U}_{j+\frac12}^+), 0 \right\}$ be the HLL wave speeds at $x = x_{j+\frac12}$. 
If the OEDG solution (with a PCP limiter) satisfies
\[
\mathbf{U}_\sigma^n(x) \in \mathcal{G}^* \quad \forall x \in \mathbb{S}_{j}, \, \forall j,
\]
then the updated cell averages $\overline{\mathbf{U}}_j^{n+1}$, as computed by the OEDG scheme with the forward Euler time discretization, belong to the admissible state set $\mathcal{G}^*$, or equivalently $\mathcal{G}^{(2)}$, under the CFL-type condition
\[
\frac{\Delta t_n}{h_x} \max_j \left\{s_{j+\frac{1}{2}}^+ - s_{j+\frac{1}{2}}^-\right\} < \omega_1^{\tt GL}.
\]
\end{corollary}

\subsubsection{2D Case}\label{sec:srhd2D}

We consider the 2D special RHD system
\begin{align}\label{2DSRHD}
\frac{\partial \mathbf{U}}{\partial t} + \frac{\partial \mathbf{F}_1(\mathbf{U})}{\partial x} + \frac{\partial \mathbf{F}_2(\mathbf{U})}{\partial y} = \mathbf{0},
\end{align}
where the vector of conserved variables is $\mathbf{U} = (D, m_1, m_2, E)^{\top}$, and the fluxes in the $x$- and $y$-directions are
\begin{align}\label{2D flux}
\mathbf{F}_1 = (Dv_1, m_1v_1 + p, m_2v_1, m_1)^{\top}, \qquad 
\mathbf{F}_2 = (Dv_2, m_1v_2 + p, m_2v_2 + p, m_2)^{\top}.
\end{align} 

The computational domain $\Omega$ is divided into $N_{x}\times N_{y}$ uniform rectangular cells $K_{ij} = [x_{i-\frac{1}{2}}, x_{i+\frac{1}{2}}] \times [y_{j-\frac{1}{2}}, y_{j+\frac{1}{2}}]$, for $1 \le i \le N_{x}$ and $1 \le j \le N_{y}$, where $h_{x}$ and $h_{y} $ represent the spatial step sizes in the $x$ and $y$ directions, respectively. The center of cell $K_{ij}$ is denoted by $(x_{i}, y_{j})$, with $x_{i} = \frac12 ({x_{i-\frac{1}{2}} + x_{i+\frac{1}{2}}})$ and $y_{j} = \frac12 ({y_{j-\frac{1}{2}} + y_{j+\frac{1}{2}}})$. 
The basis function vectors $\{\hat{\boldsymbol{\phi}}^{(q)}(\boldsymbol{\xi}): 0 \le q \le m\}$ on the reference cell $\widehat{K} = [-1,1]^2$ are given by 
\begin{align*}
\hat{\boldsymbol{\phi}}^{(0)} = (1)^{\top}, \qquad 
\hat{\boldsymbol{\phi}}^{(1)} = \left(\xi_x, \,\xi_y\right)^{\top}, \qquad 
\hat{\boldsymbol{\phi}}^{(2)} = \left(\xi_x^2 - \frac{1}{3}, \,\xi_x \xi_y, \,\xi_y^2 - \frac{1}{3}\right)^{\top}, \quad \dots
\end{align*}
Let $F_{K_{ij}}: \widehat{K} \to K_{ij}$ be the linear mapping defined by
\[
x = x_{i} + \frac{h_{x}}{2}\xi_x, \quad y = y_{j} + \frac{h_{y}}{2}\xi_y, \quad \xi_x, \xi_y \in [-1,1].
\]
The basis functions $\{\boldsymbol{\phi}^{(q)}_{K_{ij}}(\mathbf{x}): 0 \le q \le m\}$ of $\mathbb{P}^m(K_{ij})$ are obtained via this mapping. In the 2D case, the dimension of the polynomial space $\mathbb{P}^m$ is $N = \text{dim}(\mathbb{P}^m) = \frac12 (m+1)(m+2)$.

The semi-discrete DG method for \eqref{2DSRHD} is to find $\mathbf{U}_h \in [V_h^m]^4$ such that
\begin{align*}
\int_{K_{ij}} \frac{\partial \mathbf{U}_h(t, \mathbf{x})}{\partial t} v_h(\mathbf{x}) \,\mathrm{d}\mathbf{x} = \int_{K_{ij}} \mathbf{F}(\mathbf{U}_h(t, \mathbf{x})) \cdot \nabla v_h(\mathbf{x}) \,\mathrm{d}\mathbf{x} - \sum_{\mathcal{E} \in \partial K_{ij}} \int_{\mathcal{E}} \widehat{\mathbf{F}}(\mathbf{U}_h(t, \mathbf{x}); \mathbf{n}_{\mathcal{E}}) v_h(\mathbf{x}) \,\mathrm{d}s, \quad \forall v_h \in V_h^m,
\end{align*}
where $\mathcal{E}$ denotes the edges of the cell $K_{ij}$ and $\mathbf{n}_{\mathcal{E}}$ is the outward unit normal vector on $\mathcal{E}$. The numerical flux $\widehat{\mathbf{F}}$ is taken as the HLL flux. 
This scheme can be expressed in matrix form in terms of the modal coefficients $\mathcal{U}_{K_{ij}}(t) \in \mathbb{R}^{4 \times N}$. Discretizing this system further in time using a RK method and applying the OE procedure after each RK stage results in a fully-discrete OEDG scheme. 

Similar to the 1D case, the 2D OE procedure at the $\ell$-th RK stage is implemented as follows:
\begin{align*}
\mathcal{U}_{\sigma,K_{ij}}^{n,\ell+1} = \mathcal{U}_{K_{ij}}^{n,\ell+1} \exp\left(-\Delta t_n \mathbf{D}\Big(\{\mathcal{U}_{K}^{n,\ell+1}\}_{K \in \Lambda_{K_{ij}}}\Big)\right),
\end{align*}
where the damping operator $\mathbf{D}\Big(\{\mathcal{U}_{K}^{n,\ell+1}\}_{K \in \Lambda_{K_{ij}}}\Big)$ is defined as
\[
\mathbf{D}(\mathbf{\mathcal{U}}) = \text{diag}\left( d_0\Big(\{\mathcal{U}_{K}^{n,\ell+1}\}_{K \in \Lambda_{K_{ij}}}\Big), d_1\Big(\{\mathcal{U}_{K}^{n,\ell+1}\}_{K \in \Lambda_{K_{ij}}}\Big) \mathbf{I}_{N_1}, \dots, d_m\Big(\{\mathcal{U}_{K}^{n,\ell+1}\}_{K \in \Lambda_{K_{ij}}}\Big) \mathbf{I}_{N_m} \right),
\]
with $\mathbf{I}_{N_q}$ being the identity matrix of size $N_q$, where $N_q$ represents the number of basis functions of $q$ degree $\boldsymbol{\phi}^{(q)}_{K_{ij}}(\mathbf{x})$. The components of the diagonal matrix $\mathbf{D}$ are defined as
\[
d_0(\{\mathcal{U}_K^{n,\ell+1}\}_{K \in \Lambda_{K_{ij}}}) = 0, \quad d_q(\{\mathcal{U}_K^{n,\ell+1}\}_{K \in \Lambda_{K_{ij}}}) = \sum_{r=0}^q \delta_{K_{ij}}^r(\mathbf{U}_h^{n,\ell+1}), \quad 1 \le q \le m,
\]
where the damping coefficient $\delta_{K_{ij}}^r(\mathbf{U}_h)$ is defined by
\[
\delta_{K_{ij}}^r(\mathbf{U}_h) = \max_{0 \le s \le 3} \left(\frac{\eta_{ij}^{x} \left(\sigma_{i + \frac{1}{2}, j}^r(U_h^{(s)}) + \sigma_{i - \frac{1}{2}, j}^r(U_h^{(s)})\right)}{h_x} + \frac{\eta_{ij}^{y} \left(\sigma_{i, j+\frac{1}{2}}^r(U_h^{(s)}) + \sigma_{i, j-\frac{1}{2}}^r(U_h^{(s)})\right)}{h_y}\right).
\]
Here, $\eta_{ij}^{x}$ and $\eta_{ij}^{y}$ represent the spectral radii of the Jacobian matrices in the $x$ and $y$ directions, respectively. The terms $\sigma_{i + \frac{1}{2}, j}^r(U_h^{(s)})$ and $\sigma_{i, j+ \frac{1}{2}}^r(U_h^{(s)})$ are given by
\[
\sigma^r_{i + \frac{1}{2}, j}(U_h^{(s)}) =
\begin{cases}
0, & \text{if} \quad U_h^{(s)} \equiv \mathrm{avg}_\Omega(U_h^{(s)}), \\
\displaystyle\frac{(2r+1)h_x^r}{2(2m-1)r!}  \frac{\frac{1}{h_y}\int_{y_{j-\frac{1}{2}}}^{y_{j+\frac{1}{2}}} \sqrt{\sum_{|\mathbf{a}| = r}[\![\partial^{\mathbf{a}}U_h^{(s)}]\!]_{i + \frac{1}{2}, j}^2}\,\mathrm{d}y}{\|U_h^{(s)} - \mathrm{avg}_\Omega(U_h^{(s)})\|_{L^{\infty}(\Omega)}}, & \text{otherwise},
\end{cases}
\]
\[
\sigma^r_{i, j + \frac{1}{2}}(U_h^{(s)}) =
\begin{cases}
0, & \text{if} \quad U_h^{(s)} \equiv \mathrm{avg}_\Omega(U_h^{(s)}), \\
\displaystyle \frac{(2r+1)h_y^r}{2(2m-1)r!} \frac{\frac{1}{h_x}\int_{x_{i-\frac{1}{2}}}^{x_{i+\frac{1}{2}}} \sqrt{\sum_{|\mathbf{a}| = r} [\![\partial^{\mathbf{a}}U_h^{(s)}]\!]_{i, j + \frac{1}{2}}^2} \,\mathrm{d}x}{\|U_h^{(s)} - \mathrm{avg}_\Omega(U_h^{(s)})\|_{L^{\infty}(\Omega)}}, & \text{otherwise}.
\end{cases}
\]

With the forward Euler time discretization, the evolution equation for the cell average on a rectangular cell $K_{ij}$ can be written as:
\begin{align}\label{avg 3}
\overline{\mathbf{U}}_{{ij}}^{n+1} 
= \overline{\mathbf{U}}_{{ij}}^n &- \frac{\Delta t_n}{h_x}\sum_{p=1}^{Q}\omega_p^G\left[\widehat{\mathbf{F}}_1\left(\mathbf{U}_{i + \frac{1}{2},p}^-, \mathbf{U}_{i + \frac{1}{2},p}^+\right) 
 - \widehat{\mathbf{F}}_1\left(\mathbf{U}_{i - \frac{1}{2},p}^-, \mathbf{U}_{i - \frac{1}{2},p}^+\right)\right] \notag \\
 &- \frac{\Delta t_n}{h_y}\sum_{p=1}^{Q}\omega_p^G\left[\widehat{\mathbf{F}}_2\left(\mathbf{U}_{p,j + \frac{1}{2}}^-, \mathbf{U}_{p,j + \frac{1}{2}}^+\right) 
 - \widehat{\mathbf{F}}_2\left(\mathbf{U}_{p,j - \frac{1}{2}}^-, \mathbf{U}_{p,j - \frac{1}{2}}^+\right)\right],
\end{align}
where $\widehat{\mathbf{F}}_1$ and $\widehat{\mathbf{F}}_2$ are the HLL numerical fluxes in the $x$ and $y$ directions, respectively. 
After using the PCP limiter, the OEDG solution is guaranteed to satisfy the physical constraints at the cell interface quadrature points:
\begin{align*}
\mathbf{U}_{i \pm \frac{1}{2},p}^- = \mathbf{U}_h^n(x_{i \pm \frac{1}{2}}^-,y_{j,p}^G) \in \mathcal{G}^*,\quad 
\mathbf{U}_{i \pm \frac{1}{2},p}^+ = \mathbf{U}_h^n(x_{i \pm \frac{1}{2}}^+,y_{j,p}^G)\in \mathcal{G}^*,\\
\mathbf{U}_{p,j \pm \frac{1}{2}}^- = \mathbf{U}_h^n(x_{i,p}^G,y_{j \pm \frac{1}{2}}^-)\in \mathcal{G}^*,\quad
\mathbf{U}_{p,j \pm \frac{1}{2}}^+ = \mathbf{U}_h^n(x_{i,p}^G,y_{j \pm \frac{1}{2}}^+)\in \mathcal{G}^*,
\end{align*}
where the superscripts $\pm$ on $x$, $y$, and $\mathbf{U}$ indicate the right/left-hand side limits; 
$\{x_{i,p}^G\}_{p=1}^Q$ and $\{y_{j,p}^G\}_{p=1}^Q$ denote the $Q$-point Gauss quadrature nodes in the intervals $[x_{i-\frac{1}{2}},\,x_{i+\frac{1}{2}}]$ and $[y_{j-\frac{1}{2}},\,y_{j+\frac{1}{2}}]$, respectively. The corresponding weights are denoted by $\{\omega_p^G\}_{p=1}^Q$, satisfying $\sum_{p=1}^Q \omega_p^G = 1$. We choose $Q = m+1$ for the standard Gauss quadrature. 

As demonstrated in Section \ref{sec:PCPave}, the convex decomposition of the cell averages, as given by \eqref{CAD}, is essential in the analysis and design of  PCP schemes. The feasible convex decomposition is not unique, and different decompositions lead to different point sets $\mathbb{S}_{K_{ij}}$ in the PCP limiter, resulting in varying theoretical PCP CFL conditions and impacting computational cost and efficiency. 
On rectangular meshes, a classical decomposition was proposed in \cite{ZS2010}, based on the tensor product of Gauss quadrature points and Gauss--Lobatto quadrature points. However, this decomposition is not optimal \cite{CDW2023}. Recently, the optimal decomposition that achieves the largest PCP CFL number was systematically studied in \cite{CDW2024}.

For $\mathbb{P}^2$ and $\mathbb{P}^3$ spaces, the optimal convex decomposition of cell averages is given as follows (for higher-degree polynomial spaces, see \cite{CDW2024}):
\begin{align}\label{avg 5}
\overline{\mathbf{U}}_{{ij}}^{n} &= \frac{\mu_1}{2}\sum_{p=1}^Q \omega_p^G\left[\mathbf{U}_{i - \frac{1}{2},p}^+ + \mathbf{U}_{i + \frac{1}{2},p}^-\right]  + \frac{\mu_2}{2} \sum_{p=1}^Q \omega_p^G\left[\mathbf{U}_{p,j - \frac{1}{2}}^+ + \mathbf{U}_{p,j + \frac{1}{2}}^-\right] + \omega \sum_r \mathbf{U}_h^n(\tilde{x}_r,\tilde{y}_r),
\end{align}
where the internal nodes are given by
\[
{\mathbb{S}}^{\tt CDW}_{ {ij}} = \left\{ (\tilde{x}_r,\tilde{y}_r) \right\} =
\begin{cases}
\left(x_i, y_j \pm \frac{h_y}{2\sqrt{3}} \sqrt{\frac{\zeta_* - \zeta_2}{\zeta_*}}\right), \quad \text{if} \quad \zeta_1 \ge \zeta_2, \\
\left(x_i \pm \frac{h_x}{2\sqrt{3}} \sqrt{\frac{\zeta_* - \zeta_1}{\zeta_*}}, y_j\right), \quad \text{if} \quad \zeta_1 < \zeta_2, 
\end{cases}
\]
with
\[
\zeta_1 = \frac{a_x}{h_x}, \quad \zeta_2 = \frac{a_y}{h_y}, \quad \zeta_* = \max\{\zeta_1, \zeta_2\}, \quad \psi = \zeta_1 + \zeta_2 + 2\zeta_*, \quad \mu_1 = \frac{\zeta_1}{\psi}, \quad \mu_2 = \frac{\zeta_2}{\psi}, \quad \omega = \frac{\zeta_*}{\psi}.
\]
The point set \eqref{ssk}, involved in the PCP limiter, becomes
\begin{align}\label{ssk2}
\mathbb{S}_{{ij}} = 
{\mathbb{S}}^{\tt CDW}_{{ij}} \cup  
\left\{ (x_{i \pm \frac{1}{2}}, y_{j,p}^G) \right\}_{p=1}^Q \cup 
\left\{ (x_{i,p}^G, y_{j \pm \frac{1}{2}}) \right\} 
\cup \left(\left\{x_{i,p}\right\}_{p = 1}^Q \otimes \left\{y_{j,p}\right\}_{p = 1}^Q\right).
\end{align}

For $1 \le p \le Q$, the HLL wave speeds are given by
\begin{align*}
s_{i \pm \frac{1}{2}, p}^+ &= \max\left\{\lambda_{x}^{(4)}(\mathbf{U}^{-}_{i \pm \frac{1}{2}, p}), \lambda_{x}^{(4)}(\mathbf{U}^{+}_{i \pm \frac{1}{2}, p}), 0\right\}, \quad 
s_{i \pm \frac{1}{2}, p}^- = \min\left\{\lambda_x^{(0)}(\mathbf{U}^{-}_{i \pm \frac{1}{2}, p}), \lambda_x^{(0)}(\mathbf{U}^{+}_{i \pm \frac{1}{2}, p}), 0\right\}, \\
s_{p, j \pm \frac{1}{2}}^+ &= \max\left\{\lambda_y^{(4)}(\mathbf{U}^{-}_{p, j \pm \frac{1}{2}}), \lambda_y^{(4)}(\mathbf{U}^{+}_{p, j \pm \frac{1}{2}}), 0\right\}, \quad 
s_{p, j \pm \frac{1}{2}}^- = \min\left\{\lambda_y^{(0)}(\mathbf{U}^{-}_{p, j \pm \frac{1}{2}}), \lambda_y^{(0)}(\mathbf{U}^{+}_{p, j \pm \frac{1}{2}}), 0\right\}.
\end{align*}
Following Theorem \ref{pcp1}, we obtain the following weak PCP property for the updated cell averages.

\begin{corollary}
If the OEDG solution (with a PCP limiter) satisfies
\[
\mathbf{U}_\sigma^n(x, y) \in \mathcal{G}^* \quad \forall (x, y) \in \mathbb{S}_{ij}, \, \forall i,j,
\]
then the updated cell averages $\overline{\mathbf{U}}_{ij}^{n+1}$, computed by the 2D $\mathbb{P}^2$- or $\mathbb{P}^3$-based OEDG scheme for \eqref{2DSRHD} with the forward Euler time discretization , belong to the admissible state set $\mathcal{G}^*$ (or equivalently $\mathcal{G}^{(2)}$), under the CFL-type condition:
\[
\Delta t_n \left(\frac{a_x}{h_x} + \frac{a_y}{h_y} + 2 \max\left\{\frac{a_x}{h_x}, \frac{a_y}{h_y}\right\}\right) \leq \frac12,
\]
where
\[
a_x := \max_{i,p} \left\{ s_{i+\frac{1}{2},p}^+ - s_{i+\frac{1}{2},p}^- \right\}, \quad a_y := \max_{j,p} \left\{ s_{p,j+\frac{1}{2}}^+ - s_{p,j+\frac{1}{2}}^- \right\}.
\]
\end{corollary}

\subsection{Application to Axisymmetric RHD Equations in Cylindrical Coordinates}\label{sec:AxisRHD}

The line element for the Minkowski metric in cylindrical coordinates is given by:
\[
\mathrm{d}s^2 = - \mathrm{d}t^2 + \mathrm{d}r^2 + r^2 \mathrm{d}\phi^2 + \mathrm{d}z^2.
\]
In this subsection, we discuss the application of the OEDG method to the axisymmetric RHD equations (with no dependence on the azimuthal angle $\phi$) in cylindrical coordinates $(r,z)$:
\begin{align}\label{Axi RHD}
\frac{\partial \mathbf{U}}{\partial t} + \frac{\partial \mathbf{F}_1}{\partial r} + \frac{\partial \mathbf{F}_2}{\partial z} = {\bf S}(\mathbf{U}, r),
\end{align}
where the source term is given by
\[
{\bf S}(\mathbf{U}, r) = -\frac{1}{r}(Dv_1, m_1v_1, m_2v_2, m_1)^{\top},
\]
and the definitions of the conservative vector $\mathbf{U}$ and fluxes $\mathbf{F}_i$ for $i = 1, 2$ are the same as in \eqref{2D flux}. The fluid variables retain their definitions from the 2D special RHD system, except that the coordinates $x$ and $y$ now correspond to the radial and axial directions in cylindrical coordinates $(r, z)$. 

Assume the computational domain in the $(r, z)$ plane is divided into uniform rectangular meshes. The OE and PCP techniques for the axisymmetric RHD equations \eqref{Axi RHD} are analogous to those discussed in Section \ref{sec:srhd2D}. To ensure the weak PCP property $\overline{\mathbf{U}}_{{ij}}^{n+1} \in \mathcal{G}^*$, we consider the evolution equations for the cell averages of the OEDG solutions to \eqref{Axi RHD}, which read
\begin{align}\label{avg 4}
\overline{\mathbf{U}}_{{ij}}^{n+1} 
= \overline{\mathbf{U}}_{{ij}}^n &- \frac{\Delta t_n}{h_r}\sum_{p =1}^{Q}\omega_p^G\left[\widehat{\mathbf{F}}_1\left(\mathbf{U}_{i + \frac{1}{2},p}^-, \mathbf{U}_{i + \frac{1}{2},p}^+\right) 
 - \widehat{\mathbf{F}}_1\left(\mathbf{U}_{i - \frac{1}{2},p}^-, \mathbf{U}_{i - \frac{1}{2},p}^+\right)\right] \notag \\
 &- \frac{\Delta t_n}{h_z}\sum_{p=1}^{Q}\omega_p^G\left[\widehat{\mathbf{F}}_2\left(\mathbf{U}_{p,j + \frac{1}{2}}^-, \mathbf{U}_{p,j + \frac{1}{2}}^+\right) 
 - \widehat{\mathbf{F}}_2\left(\mathbf{U}_{p,j - \frac{1}{2}}^-, \mathbf{U}_{p,j - \frac{1}{2}}^+\right)\right] \notag \\
 &+ \Delta t_n \sum_{p=1}^Q \sum_{q=1}^Q \omega_p^G  \omega_q^G {\bf S}\left(\mathbf{U}_h^n(r_{i,p}^G,z_{q,j}^G), r_{i,p}^G\right),
\end{align}
where $\widehat{\mathbf{F}}_1$ and $\widehat{\mathbf{F}}_2$ are the numerical fluxes in the radial and axial directions, respectively. The optimal convex decomposition of the cell average $\overline{\mathbf{U}}_{ij}^n$ and the point set $\mathbb{S}_{{ij}}$ for the PCP limiter are analogous to \eqref{avg 5} and \eqref{ssk2}, respectively.

\begin{corollary}\label{eq:2Daxis}
Consider the 2D $\mathbb{P}^2$- or $\mathbb{P}^3$-based OEDG scheme. 
If the numerical solution  (with a PCP limiter) satisfies
\[
\mathbf{U}_\sigma^n(r, z) \in \mathcal{G}^* \quad \forall (r, z) \in \mathbb{S}_{ij}, \, \forall i,j,
\]
then the updated cell averages $\overline{\mathbf{U}}_{ij}^{n+1}$, given by \eqref{avg 4}, belong to the admissible state set $\mathcal{G}^*$ (or equivalently $\mathcal{G}^{(2)}$), under the CFL-type condition:
\[
\Delta t_n \left(\lambda_{\mathbf{U}}^* + 2 \frac{a_r}{h_r} + 2\frac{a_z}{h_z} + 4 \max\left\{\frac{a_r}{h_r}, \frac{a_z}{h_z}\right\}\right) \leq 1,
\]
where
\begin{align*}
&{\lambda_{\mathbf{U}}^*: = \max_{i,j,p,q}\lambda_{\mathbf{U}_{ij,pq}},
\qquad \lambda_{\mathbf{U}_{ij,pq}} = \min_{{i,j,p,q} \in \mathcal{P}_v}\left\{\frac{r_{i,p}^Gq(\mathbf{U}_{ij,pq})}{\left(p(\mathbf{U}_{ij,pq}) + q(\mathbf{U}_{ij,pq})\right)\left|v_1(\mathbf{U}_{ij,pq})\right|}\right\},
\qquad \mathbf{U}_{ij,pq} = \mathbf{U}_h^n(r_{i,p}^G,z_{q,j}^G),}\\
&\mathcal{P}_v = \left\{(i,j,p,q)\,:\, v_1(U_{ij,pq})>0\right\}, \qquad 
a_r := \max_{i,p} \left\{ s_{i+\frac{1}{2},p}^+ - s_{i+\frac{1}{2},p}^- \right\}, \qquad a_z := \max_{j,p} \left\{ s_{p,j+\frac{1}{2}}^+ - s_{p,j+\frac{1}{2}}^- \right\}.
\end{align*}
\end{corollary}
 
\subsection{Application to GRHD Equations in Kerr--Schild Coordinates}

In this subsection, we consider the application of the PCP-OEDG method to GRHD simulations near a rotating Kerr black hole. Using the standard Boyer--Lindquist coordinates $(t, r, \theta, \phi)$, we transform to the Kerr--Schild (KS) coordinates $(\tilde{t}, r, \theta, \tilde{\phi})$ via the following coordinate transformations:
\begin{align*}
\mathrm{d}\tilde{\phi} &= \mathrm{d}\phi + \frac{a}{r^2 - 2Mr + a^2}\,\mathrm{d}r, \quad 
\mathrm{d}\tilde{t} = \mathrm{d}t + \left(\frac{1+Y}{1 + Y - Z}\right)\mathrm{d}r,
\end{align*}
where $M$ is the mass of the black hole and $a$ is its angular momentum per unit mass. In the KS coordinate system, the Kerr line element is expressed as \cite{FIP1999}:
\begin{align}\label{line element}
\mathrm{d}s^2 &= -\left(1-\frac{2Mr}{\varrho^2}\right)\mathrm{d}\tilde{t}^2 - \frac{4Mar\sin^2\theta}{\varrho^2}\mathrm{d}\tilde{t}\,\mathrm{d}\tilde{\phi} + \frac{4Mr}{\varrho^2}\,\mathrm{d}\tilde{t}\,\mathrm{d}r \notag \\
&\quad + \left(1+\frac{2Mr}{\varrho^2}\right)\mathrm{d}r^2 - 2a\left(1 + \frac{2Mr}{\varrho^2}\right)\sin^2\theta\,\mathrm{d}r\,\mathrm{d}\tilde{\phi} \notag \\
&\quad + \varrho^2\,\mathrm{d}\theta^2 + \sin^2\theta\left[\varrho^2 + a^2\left(1+\frac{2Mr}{\varrho^2}\right)\sin^2\theta\right]\,\mathrm{d}\tilde{\phi}^2,
\end{align}
with
\begin{align*}
\varrho^2 = r^2 + a^2\cos^2\theta, \quad Y = \frac{a^2\sin^2\theta}{\varrho^2}, \quad Z = \frac{2Mr}{\varrho^2}.
\end{align*}
The shift vector $\beta^i$ and the lapse function $\alpha$ are given by
\begin{align*}
\beta^i = (Z, 0, -a\sin^2\theta\,Z), \quad \alpha^2 = \frac{1}{1+Z}.
\end{align*}
The symmetric spatial 3-metric in the (3+1) ADM formulation is expressed as
\begin{align*}
\mathbf{X} = (\chi^{ij})_{1 \leq i,j \leq 3}=
\begin{bmatrix}
Z+1 & 0 & -a\sin^2\theta(Z+1) \\
0 & \varrho^2 & 0 \\
-a\sin^2\theta(Z+1) & 0 & \varrho^2 \sin^2\theta(1 + Y(1+Z))
\end{bmatrix},
\end{align*}
whose Cholesky decomposition is $\mathbf{X} = \mathbf{\Theta}^\top \mathbf{\Theta}$ with
\begin{align*}
\mathbf{\Theta} =
\begin{bmatrix}
\sqrt{Z+1} & 0 & \frac{(-a\sin^2\theta(Z+1))}{\sqrt{Z+1}} \\
0 & \varrho & 0 \\
0 & 0 &  \mathbf{\Theta}_{33}
\end{bmatrix},
\end{align*}
where $\mathbf{\Theta}_{33} = \sqrt{\varrho^2 \sin^2\theta(1 + Y(1+Z)) - \frac{[-a\sin^2\theta(Z+1)]^2}{Z+1}}.$

We restrict the hydrodynamic equations to the equatorial plane of the Kerr spacetime, $\theta = \frac{\pi}{2}$. In this case, the spacetime 3-metric $g_{\mu\nu}$ and its corresponding contravariant metric $g^{\mu\nu}$, in units where the black hole mass $M=1$, can be derived from \eqref{line element} and the transformation rule for contravariant tensors:
\begin{align}\label{metric}
g_{\mu\nu} =
\begin{bmatrix}
\frac{2-r}{r} & \frac{2}{r}  & -\frac{2a}{r} \\
\frac{2}{r} & \frac{2+r}{r}  & \frac{-a(r+2)}{r} \\
-\frac{2a}{r} & \frac{-a(r+2)}{r}  & \frac{r^3 + ra^2 +2a^2}{r}
\end{bmatrix}, \quad
g^{\mu\nu} =
\begin{bmatrix}
- \frac{2+r}{r} & \frac{2}{r}  & 0 \\
\frac{2}{r} & \frac{a^2 + r^2 - 2r}{r^2}  & \frac{a}{r^2} \\
0 & \frac{a}{r^2}  & \frac{1}{r^2}
\end{bmatrix}.
\end{align}
This type of metric is regular at the event horizon, facilitating the numerical computation of matter flows near the black hole. The spatial metric and its Cholesky decomposition simplify to
\begin{align} \label{smetric}
\mathbf{X} = 
\begin{bmatrix}
 \frac{2+r}{r}  & \frac{-a(r+2)}{r} \\
\frac{-a(r+2)}{r}  & \frac{r^3 + ra^2 + 2a^2}{r}
\end{bmatrix}, \qquad
\mathbf{\Theta} = 
\begin{bmatrix}
 \sqrt{\frac{2+r}{r}}  & -a{\sqrt{\frac{2+r}{r}}} \\
 0  & r
\end{bmatrix}.
\end{align}
From \eqref{metric} and \eqref{smetric}, we have
\begin{align}
\chi = \det(\chi_{ij}) = r(r+2), \qquad g = \det(g_{\mu\nu}) = -r^2.
\end{align}
In this setting, the (3+1) ADM form of the GRHD equations, as in \eqref{eq: form2}, becomes
\begin{align}\label{eq:bh}
\frac{\partial \mathbf{U}}{\partial t} + \frac{\partial (\alpha\mathbf{G}^r(\mathbf{U}))}{\partial r} + \frac{\partial (\alpha\mathbf{G}^{\tilde{\phi}}(\mathbf{U}))}{\partial \tilde{\phi}} = \mathbf{R}(\mathbf{U}),
\end{align}
where $\alpha = \sqrt{\frac{r}{r+2}}$ is the lapse function, and
\begin{align*}
\mathbf{U} &= (D, m_r, m_{\tilde{\phi}}, E)^{\top}, \qquad \mathbf{R}(\mathbf{U}) = (R_1, R_2, R_3, R_4)^{\top}, \\
\mathbf{G}^r &= (D\tilde{\bm{v}}^r, m_r\tilde{\bm{v}}^r + p, m_{\tilde{\phi}}\tilde{\bm{v}}^r, E\tilde{\bm{v}}^r + p\bm{v}^r)^{\top}, \\
\mathbf{G}^{\tilde{\phi}} &= (D\tilde{\bm{v}}^{\tilde{\phi}}, m_r\tilde{\bm{v}}^{\tilde{\phi}}, m_{\tilde{\phi}}\tilde{\bm{v}}^{\tilde{\phi}} + p, (E+p) \tilde{\bm{v}}^{\tilde{\phi}})^{\top}.
\end{align*}
Detailed expressions of the source term $\mathbf{R}(\mathbf{U})=(R_1,R_2,R_3,R_4)^\top$ 
are provided in \cite{FIP1999} and omitted here for brevity.

The conservative vector $\mathbf{U}$ should remain in the physically admissible state set
\begin{align*}
\mathcal{G}_{\chi}^{(1)} = \left\{\mathbf{U} = (D, m_r, m_{\tilde{\phi}}, E)^{\top} : D > 0, ~~ q_{\chi}(\mathbf{U}) > 0\right\},
\end{align*}
where 
\[
q_{\chi}(\mathbf{U}) = E - \sqrt{D^2 + r^2 m_{\tilde{\phi}}^2 + \frac{(r+2)(m_r - am_{\tilde{\phi}})^2}{r}}.
\]
Our PCP-OEDG schemes are based on the W-form of \eqref{eq:bh}, which can be written as
\begin{align}\label{W-form}
\frac{\partial \mathbf{W}}{\partial t} + \frac{\partial \mathbf{H}^{r}}{\partial r} + \frac{\partial \mathbf{H}^{\tilde{\phi}}}{\partial \tilde{\phi}} = \mathbf{S},
\end{align}
with
\begin{align*}
\mathbf{W} &= \left(\sqrt{r(r+2)}D,\,(r+2)(m_r - am_{\tilde{\phi}}),\,r\sqrt{r(r+2)}m_{\tilde{\phi}},\,\sqrt{r(r+2)}E\right)^{\top}, \\
\mathbf{H}^r &= \left(\alpha rD\tilde{\bm{v}}^r,\,r(m_r\tilde{\bm{v}}^r + p - am_{\tilde{\phi}}\tilde{\bm{v}}^r),\,\alpha r^2 m_{\tilde{\phi}}\tilde{\bm{v}}^r,\,\alpha r\left[E\tilde{\bm{v}}^r + p\bm{v}^r\right]\right)^{\top}, \\
\mathbf{H}^{\tilde{\phi}} &= \left(\alpha r D\tilde{\bm{v}}^{\tilde{\phi}},\,r \left[m_r\tilde{\bm{v}}^{\tilde{\phi}} - a(m_{\tilde{\phi}}\tilde{\bm{v}}^{\tilde{\phi}} + p)\right],\,\alpha r^2(m_{\tilde{\phi}}\tilde{\bm{v}}^{\tilde{\phi}} + p),\,\alpha r(E+p)\tilde{\bm{v}}^{\tilde{\phi}}\right)^{\top}, \\
\mathbf{S} &= \left(rR_1,\,\sqrt{1/\chi}(am_{\tilde{\phi}}\tilde{\bm{v}}^r - m_r\tilde{\bm{v}}^r - p) + \sqrt{\chi}(R_2 - aR_3),\,r^2R_3 + rm_{\tilde{\phi}}\tilde{\bm{v}}^r ,\,rR_4\right)^{\top}.
\end{align*}
For the W-form \eqref{W-form}, the admissible state set for the evolved variables $\mathbf{W}$ is spacetime-independent and given by
\begin{align*}
\mathcal{G}^{(2)} = \left\{\mathbf{W} = (\widehat{D}, \widehat{m}_r, \widehat{m}_{\tilde{\phi}}, \widehat{E})^{\top} : \widehat{D} > 0, ~~ q(\mathbf{W}) = \widehat{E} - \sqrt{\widehat{D}^2 + \widehat{m}_r^2 + \widehat{m}_{\tilde{\phi}}^2} > 0\right\}.
\end{align*}

Similar to Corollary \ref{eq:2Daxis}, we have the following theoretical result. 

\begin{corollary}
	Assume the computational domain in the $(r, \tilde{\phi})$ plane is divided into uniform rectangular meshes. If the OEDG solution (with a PCP limiter) satisfies
	\[
	\mathbf{W}_\sigma^n(r, \tilde{\phi}) \in \mathcal{G}^* \quad \forall (r, \tilde{\phi}) \in \mathbb{S}_{ij}, \, \forall i,j,
	\]
	then the updated cell averages $\overline{\mathbf{W}}_{ij}^{n+1}$, computed by the 2D $\mathbb{P}^2$- or $\mathbb{P}^3$-based OEDG scheme for the W-form \eqref{W-form} of the GRHD equations in Kerr--Schild coordinates with the forward Euler time discretization, belong to the admissible state set $\mathcal{G}^*$ (or equivalently $\mathcal{G}^{(2)}$), under the CFL-type condition:
	\[
	\Delta t_n \left(\lambda_{\max} + 2 \frac{a_r}{h_r} + 2 \frac{a_{\tilde{\phi}}}{h_{\tilde{\phi}}} + 4 \max\left\{\frac{a_r}{h_r}, \frac{a_{\tilde{\phi}}}{h_{\tilde{\phi}}}\right\}\right) \leq 1,
	\]
	where $\lambda_{\max}$ is the maximum value of $\{\lambda_{\mathbf{W}_{ijpq}}\}$ over the Gauss points $(r_{i,p}^G, \tilde{\phi}_{j,q}^G)$ for all $i,j,p,q$, and $\lambda_{\mathbf{W}}$ is defined as in Lemma \ref{lem3}. 
\end{corollary}

\section{Numerical Examples}\label{sec5}

This section presents several examples of RHD problems to demonstrate the accuracy, effectiveness, and robustness of the proposed high-order PCP-OEDG schemes on uniform Cartesian meshes. For accuracy tests involving smooth problems, we couple the $\mathbb{P}^m$-based, $(m+1)$-th order OEDG method with a $(m+1)$-th order explicit RK method for temporal discretization. For problems with discontinuities, the classical third-order RK temporal discretization is employed. To ensure stability, we set the CFL numbers to $C_{cfl} = 0.3, 0.16, 0.1$ for the $\mathbb{P}^1$-, $\mathbb{P}^2$-, and $\mathbb{P}^3$-based OEDG schemes, respectively. 
All numerical experiments are implemented in C++ using double precision.


\subsection{1D Special RHD Examples}


In this subsection, we present several benchmark tests, including a smooth problem, a shock heating problem, a blast wave interaction, two Riemann problems, and a low-density perturbation problem, to validate the accuracy and effectiveness of the PCP-OEDG schemes in flat spacetime.

\begin{example}[1D Smooth Problem]\label{1D smooth exam1}
This first example is designed to verify the accuracy of the 1D PCP-OEDG schemes in the domain $\Omega = [0,1]$, under conditions of low density and pressure. The exact solutions are given by
\begin{align*}
\mathbf{Q}(x,t) = (\rho(x,t), v(x,t), p(x,t))^{\top} 
= \left(1 + 0.9999\sin(2\pi (x - 0.99t)), 0.99, 0.001\right)^{\top}.
\end{align*}
This example models a RHD sine wave propagating periodically, with the adiabatic index set to $5/3$. Table \ref{tb:1D exam1} presents the $L^1$, $L^2$, and $L^{\infty}$ errors, along with the corresponding convergence rates at time $t = 1$ for $\mathbb{P}^m$-based OEDG schemes with $m = 1, 2, 3$. 
The results show that the errors decay at rates exceeding the optimal $(m+1)$-th order convergence. This is due to the high-order damping effects, which dominate the numerical errors on coarse meshes. Additionally, the PCP limiting and OE procedures do not degrade the high-order accuracy of the schemes.
\end{example}

\begin{table}[!thb]
\center
\caption{Example \ref{1D smooth exam1}: Errors and convergence rates for $\mathbb{P}^m$-based OEDG method with $N$ uniform cells.}
\begin{tabular}[c]{c|c|c|c|c|c|c|c}
\toprule
$m$ &  $N$ &  $|| u - u_h||_{L^1(\Omega)}$ &rate & $|| u - u_h||_{L^2(\Omega)}$ &rate & $|| u - u_h||_{L^{\infty}(\Omega)}$ &rate\\
\hline
\multirow{6}{*}{1}
&64&5.4734e-03&	2.741&	6.5562e-03&	2.718&	1.4939e-02&	2.595\\ 
&128&9.1279e-04&	2.584&	1.1428e-03&	2.520&	3.0542e-03&	2.290\\ 
&256&1.8557e-04&	2.298&	2.2268e-04&	2.360&	3.8916e-04&	2.972\\ 
&512&4.5284e-05&	2.035&	5.1548e-05&	2.111&	8.1249e-05&	2.260\\ 
&1024&1.1297e-05&	2.003&	1.2628e-05&	2.029&	1.8809e-05&	2.111\\ 
&2048&2.8224e-06&	2.001&	3.1400e-06&	2.008&	4.5381e-06&	2.051\\ 
\hline
\multirow{5}{*}{2}
&64&3.7294e-05&	4.322&	4.6521e-05&	4.379&	2.0785e-04&	4.299\\ 
&128&2.4704e-06&	3.916&	3.0860e-06&	3.914&	5.6182e-06&	5.209\\ 
&256&2.6568e-07&	3.217&	3.1813e-07&	3.278&	5.4627e-07&	3.362\\ 
&512&3.2043e-08&	3.052&	3.6819e-08&	3.111&	6.0373e-08&	3.178\\ 
&1024&3.9489e-09&	3.021&	4.4633e-09&	3.044&	7.2314e-09&	3.062\\ 
\hline
\multirow{4}{*}{3}
&96&4.8395e-08&	4.924&	5.6041e-08&	4.893&	1.0730e-07&	4.747\\ 
&144&6.7437e-09&	4.861&	8.0129e-09&	4.797&	1.6226e-08&	4.659\\ 
&216&1.0012e-09&	4.704&	1.2177e-09&	4.647&	2.5526e-09&	4.562\\ 
&324&1.7013e-10&	4.371&	2.0030e-10&	4.452&	4.4682e-10&	4.298\\ 
\bottomrule
\end{tabular}\label{tb:1D exam1}
\end{table}


\begin{example}[Shock Heating Problem]\label{1D Rie exam9}
In this example, we consider the shock heating problem. The initial conditions are set as
\begin{align*}
\mathbf{Q}(x,0) =
\left(1,\, 1-10^{-10},\, \frac{10^{-4}}{3}\right)^{\top},
\end{align*}
with a reflecting boundary condition at the right endpoint of the domain and an inflow boundary condition at the left endpoint. As the initial gas moves toward the reflecting boundary, it is compressed and heated due to the conversion of kinetic energy into internal energy. This process forms a strong reflected shock wave that propagates to the left with a speed given by $v_s = \frac{(\Gamma - 1)W_0|v_0|}{W_0 + 1}$, where $\Gamma = 4/3$, $v_0 = 1 - 10^{-10}$, and $W_0 = (1 - v_0^2)^{-1/2} \approx 70710.675$. Behind the reflected shock wave, the gas comes to rest and attains a specific internal energy of $W_0 - 1$.

\begin{figure}[!thb]
\centering
\subfigure{\includegraphics[width=0.48\textwidth]{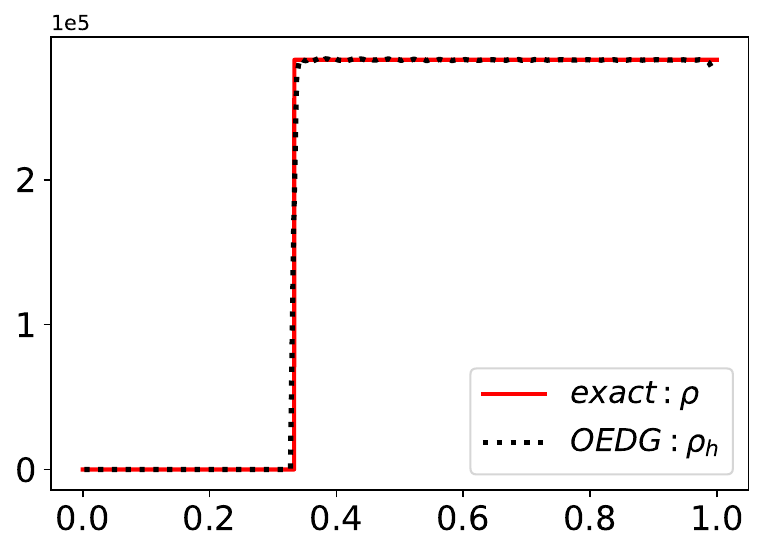}}
\subfigure{\includegraphics[width=0.48\textwidth]{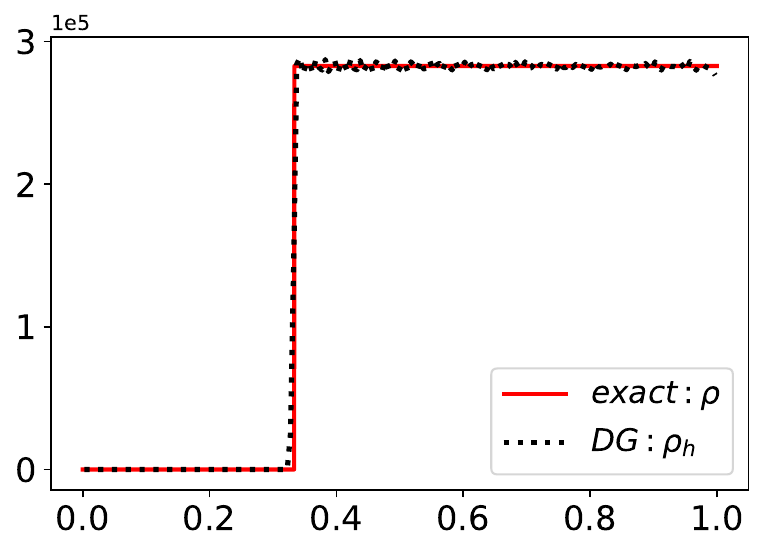}}
\subfigure{\includegraphics[width=0.48\textwidth]{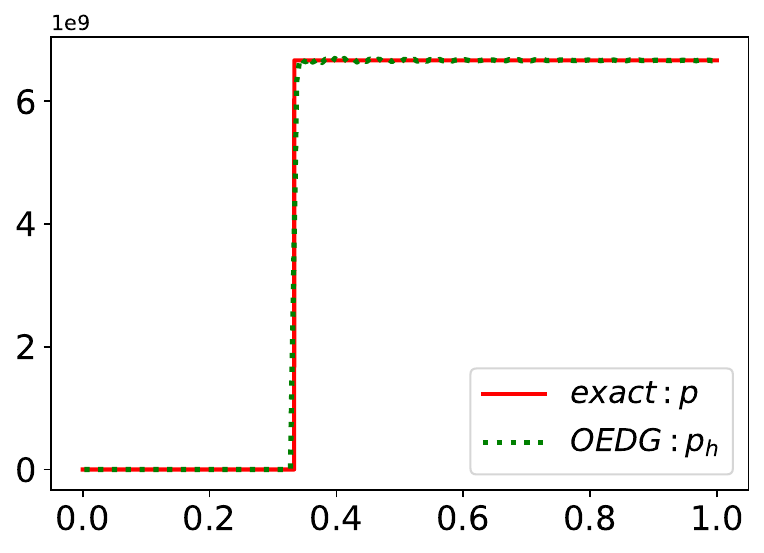}}
\subfigure{\includegraphics[width=0.48\textwidth]{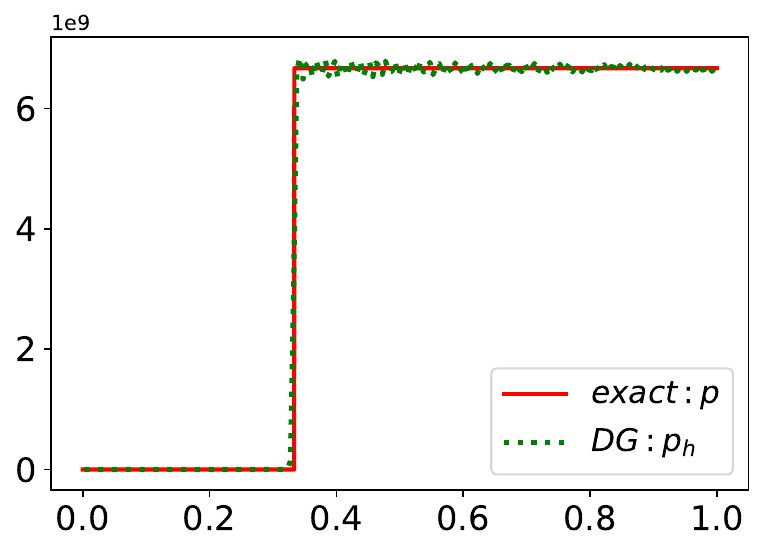}}
\caption{Example \ref{1D Rie exam9}: Numerical solutions at $t = 2$ obtained using the third-order PCP DG method with OE procedure (left) and without OE procedure (right) on a mesh of 200 uniform cells. Top: $\rho$; bottom: $p$.}\label{fig:DG 1D Rie exam9}
\end{figure}

Figure \ref{fig:DG 1D Rie exam9} shows the numerical solutions at $t = 2$ obtained using the third-order OEDG method (top row) and the third-order DG method without the OE procedure (bottom row) on a mesh of 200 uniform cells. For comparison, the exact solution is also provided. Both schemes sharply capture the shocks in this ultra-relativistic problem. 
The PCP limiter is critical for successfully simulating this challenging test. Without it, both the OEDG and conventional DG methods would break down. The OEDG method effectively eliminates spurious oscillations, though the well-known wall-heating phenomenon near the reflecting boundary at $x = 1$ is still observed, as also reported in \cite{WT2015, WT2017}. In contrast, the DG method without the OE procedure exhibits significant numerical oscillations, even when the PCP limiter is applied.

\end{example}

\begin{example}[Blast Wave Interaction]\label{1D Rie exam8}
The initial data for this problem consist of three constant states:
\begin{align*}
\mathbf{Q}(x,0) =
\begin{cases}
(1, 0, 1000)^{\top}, & \quad 0 \le x \le 0.1, \\
(1, 0, 0.01)^{\top}, & \quad 0.1 < x \le 0.9, \\
(1, 0, 100)^{\top}, & \quad 0.9 < x \le 1,
\end{cases}
\end{align*}
with outflow boundary conditions applied at both ends of the domain $[0,1]$. This problem simulates the development and collision of two strong relativistic blast waves, governed by an adiabatic index $\Gamma = 1.4$. The interaction produces complex wave structures, including new contact discontinuities, in a narrow region.

\begin{figure}[!thb]
\centering
\subfigure{\includegraphics[width=0.48\textwidth]{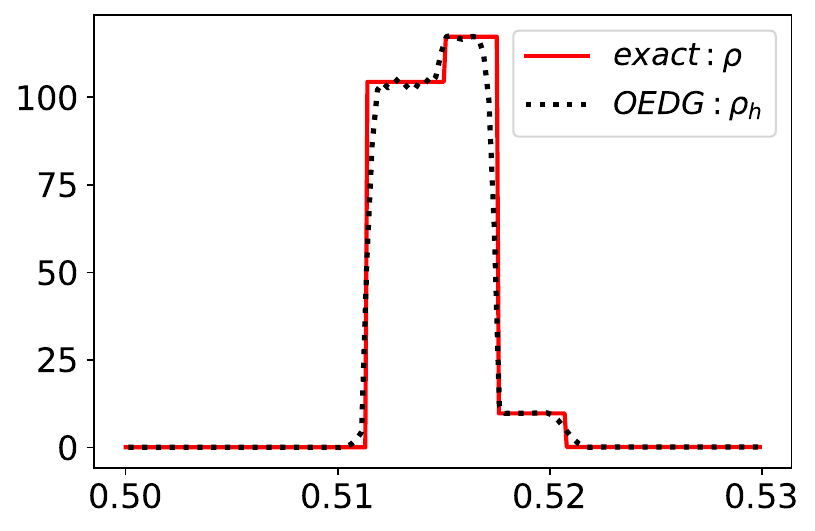}}
\subfigure{\includegraphics[width=0.48\textwidth]{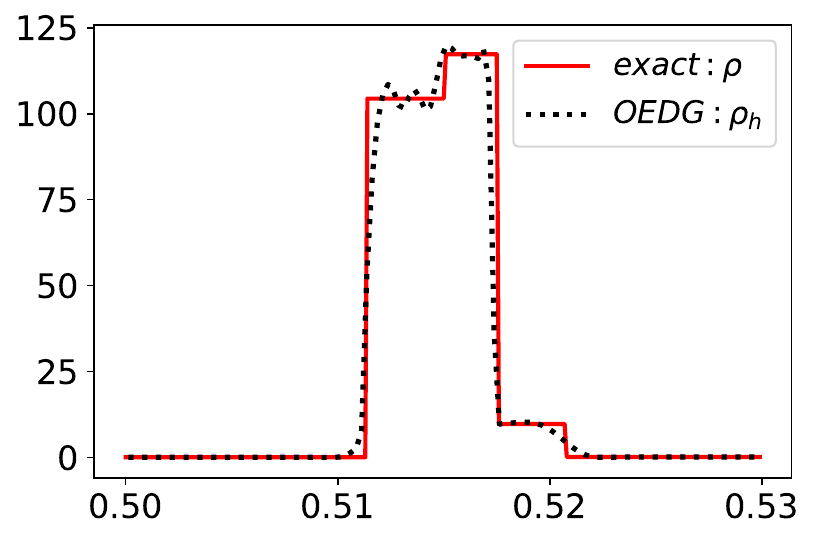}}
\subfigure  {\includegraphics[width=0.48\textwidth]{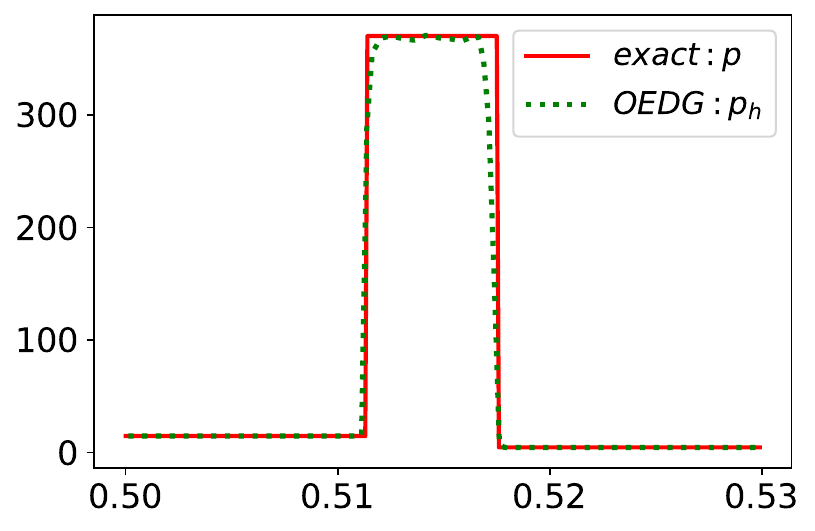}}
\subfigure {\includegraphics[width=0.48\textwidth]{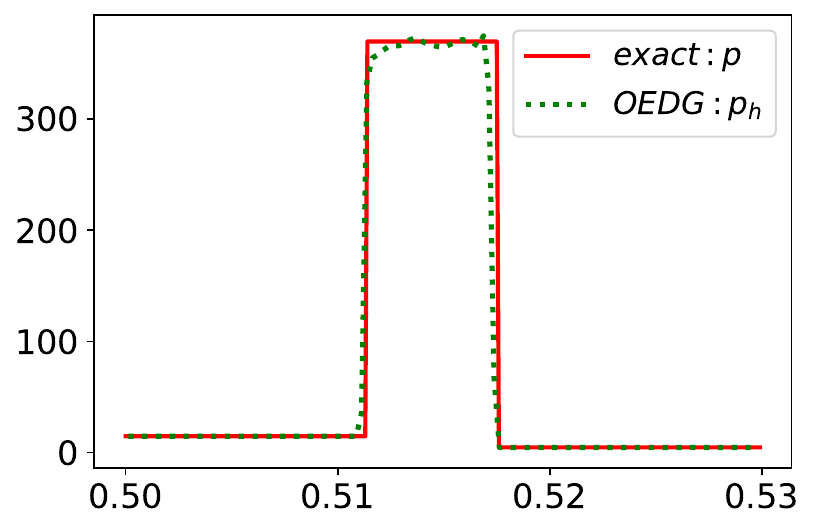}}
\caption{Example \ref{1D Rie exam8}: 
Close-up in $[0.5, 0.53]$ of the numerical solutions at $t = 0.43$, obtained using the third-order PCP DG method with OE procedure (left) and without OE procedure (right) on a mesh of 4000 uniform cells. Top: $\rho$; bottom: $p$.} 
\label{fig:DG 1D Rie exam8}
\end{figure}

Figure \ref{fig:DG 1D Rie exam8} presents the numerical solutions at $t = 0.43$, obtained using the third-order OEDG method and the DG method without the OE procedure on a mesh of 4000 uniform cells. The results show that the OEDG scheme effectively resolves the discontinuities and accurately captures the complex relativistic wave configurations in the collision region. In contrast, the DG method suffers from persistent spurious oscillations, highlighting the importance of the OE procedure in ensuring stability.
\end{example}

\begin{example}[1D Riemann Problem \uppercase\expandafter{\romannumeral1}]\label{1D Rie exam6}
The initial data for this Riemann problem are given by
\begin{align*}
\mathbf{Q}(x,0) =
\begin{cases}
(10, 0, 10^3)^{\top}, & \quad x < 0.5, \\
(1, 0, 0.01)^{\top}, & \quad x > 0.5,
\end{cases}
\end{align*}
with the domain $\Omega = [0,1]$ and outflow boundary conditions. The adiabatic index is set to $\Gamma = 5/3$.

\begin{figure}[!thb]
\centering
\subfigure[$\mathbb{P}^1$, $\rho$]{\includegraphics[width=0.32\textwidth]{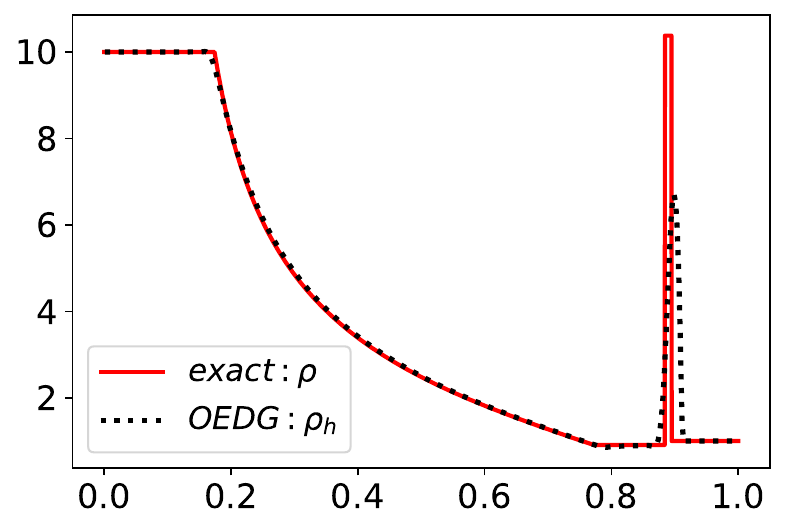}}
\subfigure[$\mathbb{P}^2$, $\rho$]{\includegraphics[width=0.32\textwidth]{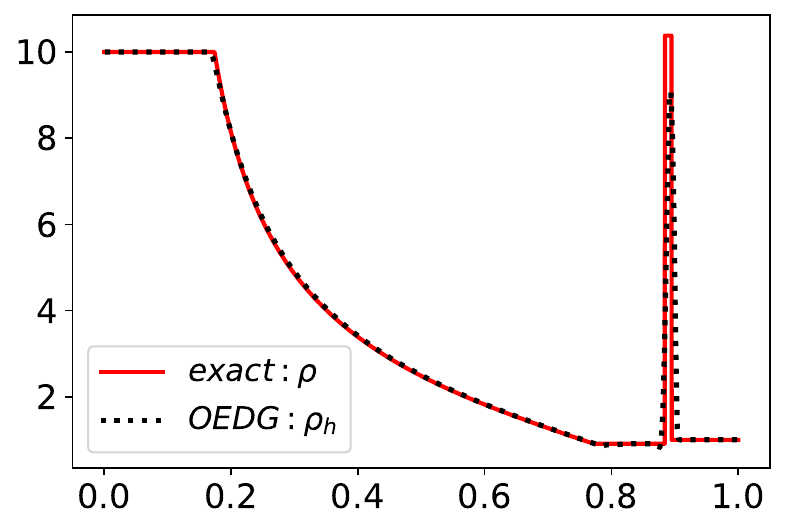}}
\subfigure[$\mathbb{P}^3$, $\rho$]{\includegraphics[width=0.32\textwidth]{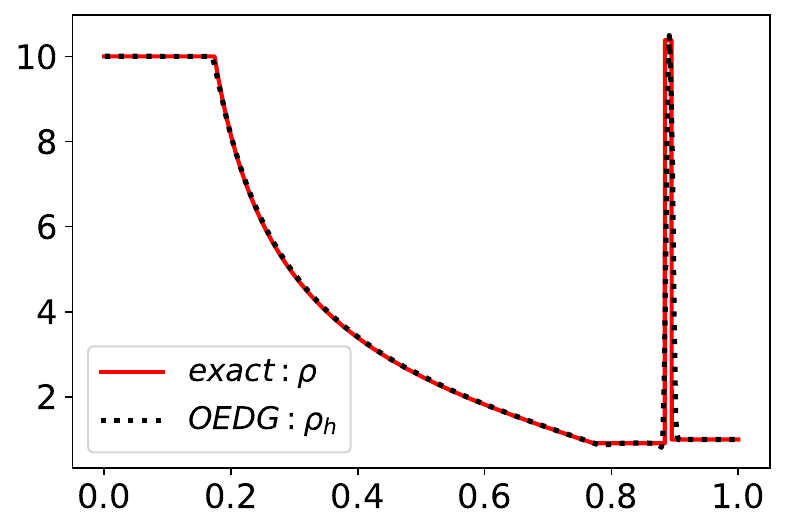}}
\subfigure[$\mathbb{P}^1$, $v$]{\includegraphics[width=0.32\textwidth]{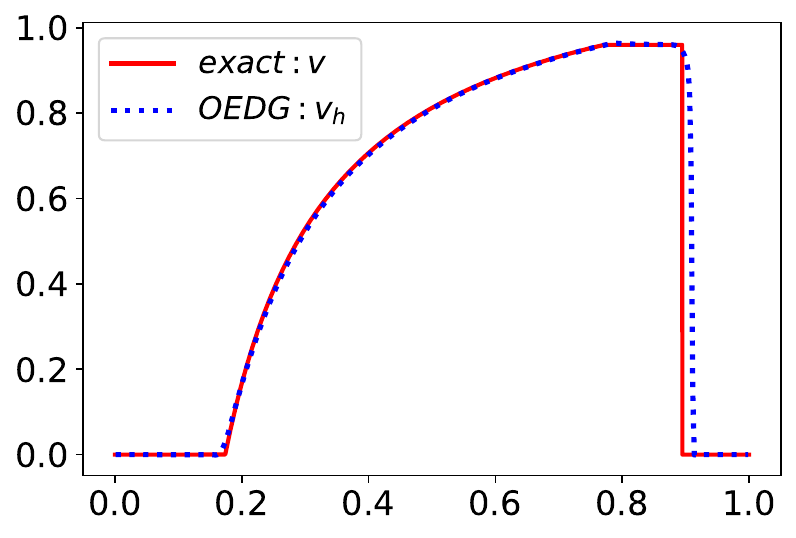}}
\subfigure[$\mathbb{P}^2$, $v$]{\includegraphics[width=0.32\textwidth]{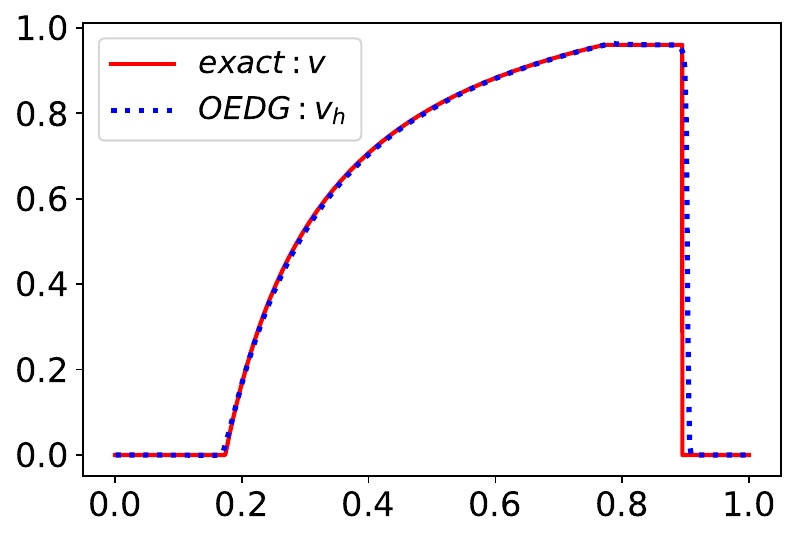}}
\subfigure[$\mathbb{P}^3$, $v$]{\includegraphics[width=0.32\textwidth]{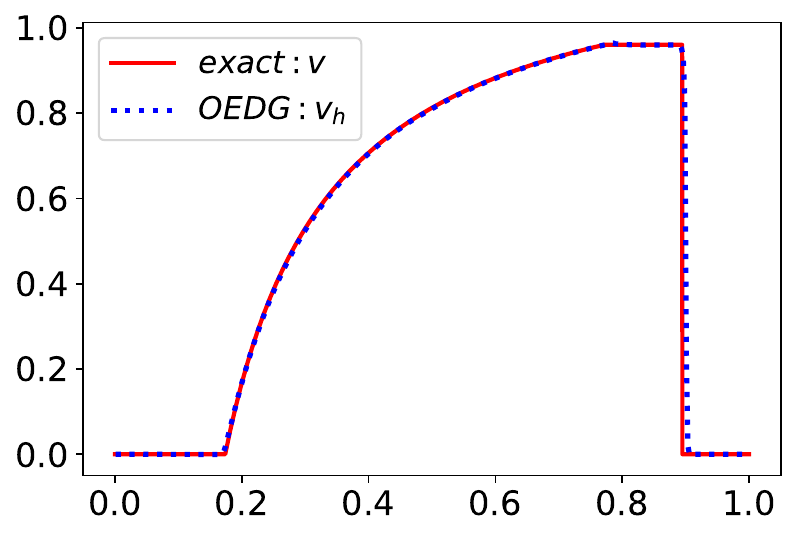}}
\subfigure[$\mathbb{P}^1$, $p$]{\includegraphics[width=0.32\textwidth]{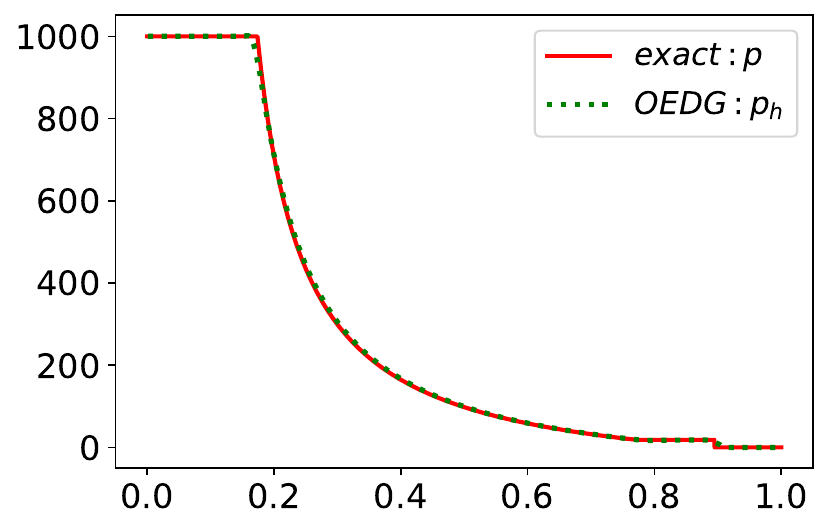}}
\subfigure[$\mathbb{P}^2$, $p$]{\includegraphics[width=0.32\textwidth]{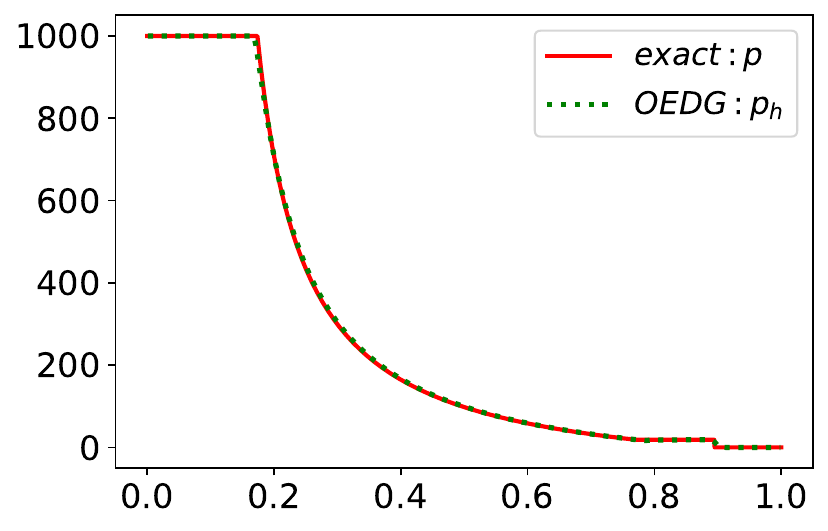}}
\subfigure[$\mathbb{P}^3$, $p$]{\includegraphics[width=0.32\textwidth]{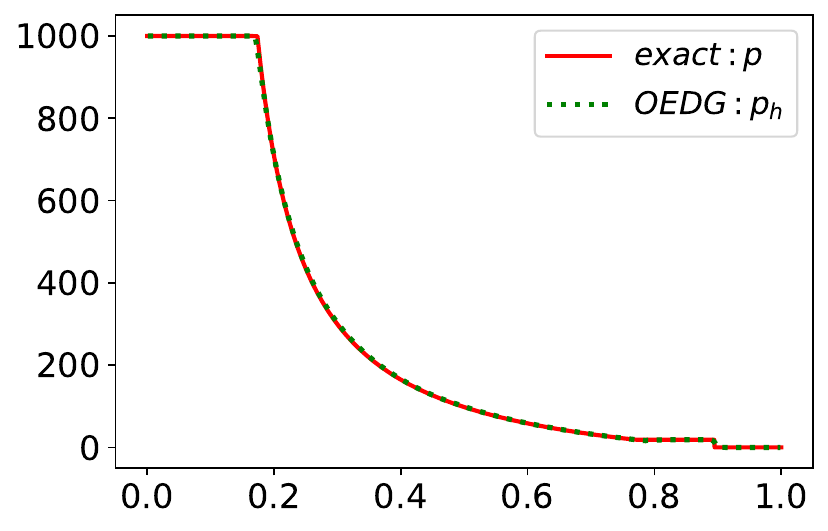}}
\caption{Example \ref{1D Rie exam6}: Numerical results for $\rho,\,v,\,p$ obtained using the $\mathbb{P}^m$-based PCP-OEDG method with 400 uniform cells at $t = 0.4$.}\label{fig:1D Rie exam6}
\end{figure}

In this problem, a highly curved profile for the rarefaction fan is observed due to relativistic effects. The region between the contact discontinuity and the right-moving shock wave is very narrow, making it challenging to resolve these waves accurately. Figure \ref{fig:1D Rie exam6} presents the numerical results at $t = 0.4$, obtained using the proposed PCP-OEDG method on a uniform mesh of 400 cells. As expected, the higher the order of the PCP-OEDG method, the better the resolution of the shock wave and contact discontinuity. Notably, no spurious oscillations are observed in the OEDG solutions.

\end{example}

\begin{example}[1D Riemann Problem \uppercase\expandafter{\romannumeral2}]\label{1D Rie exam7}
The initial data for this Riemann problem are given by
\begin{align*}
\mathbf{Q}(x,0) =
\begin{cases}
(1, 0, 10^4)^{\top}, & \quad x < 0.5, \\
(1, 0, 10^{-8})^{\top}, & \quad x > 0.5,
\end{cases}
\end{align*}
with the adiabatic index $\Gamma = 5/3$ in the domain $[0,1]$. The initial discontinuity evolves into a strong left-moving rarefaction wave, a rapidly right-moving contact discontinuity, and a rapidly right-moving shock wave.

\begin{figure}[!thb]
\centering
\subfigure[$\mathbb{P}^1$, $\rho$]{\includegraphics[width=0.32\textwidth]{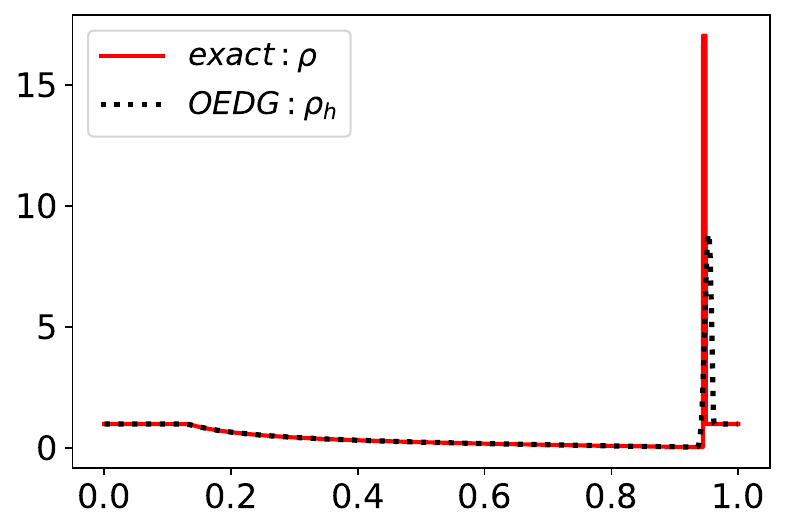}}
\subfigure[$\mathbb{P}^2$, $\rho$]{\includegraphics[width=0.32\textwidth]{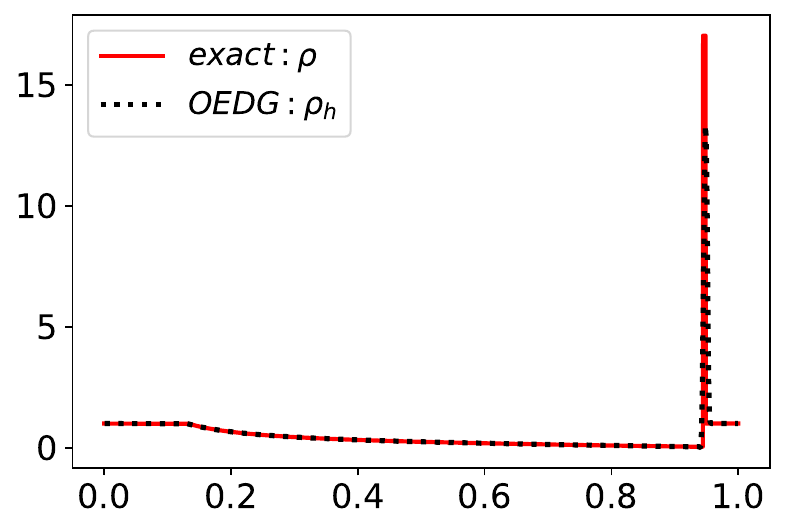}}
\subfigure[$\mathbb{P}^3$, $\rho$]{\includegraphics[width=0.32\textwidth]{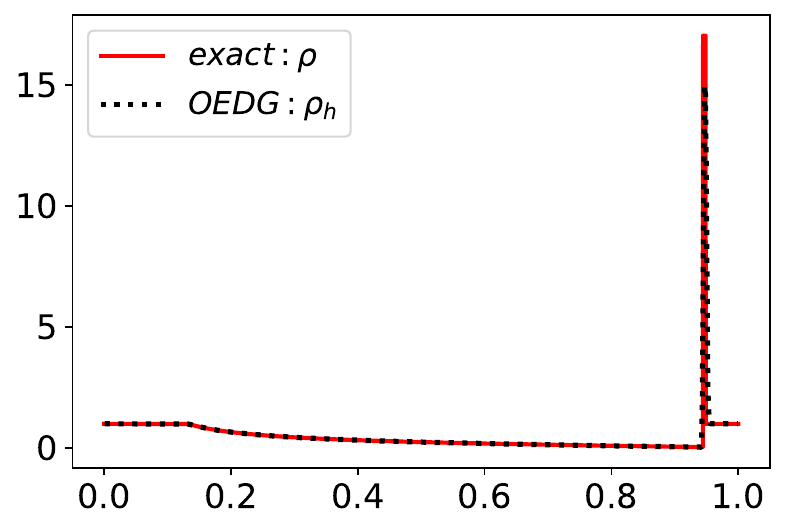}}
\subfigure[$\mathbb{P}^1$, $v$]{\includegraphics[width=0.32\textwidth]{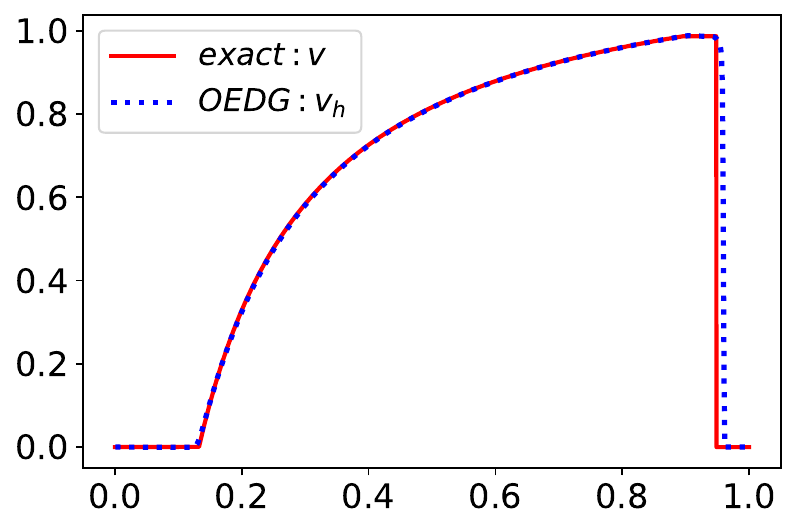}}
\subfigure[$\mathbb{P}^2$, $v$]{\includegraphics[width=0.32\textwidth]{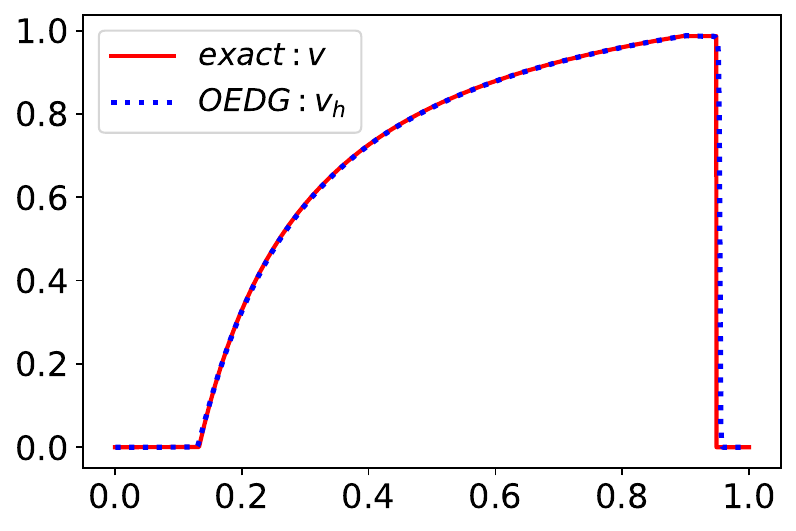}}
\subfigure[$\mathbb{P}^3$, $v$]{\includegraphics[width=0.32\textwidth]{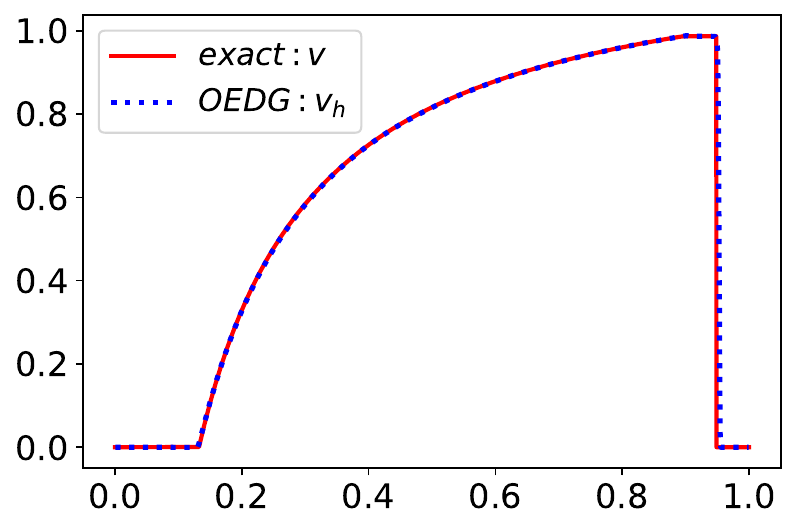}}
\subfigure[$\mathbb{P}^1$, $p$]{\includegraphics[width=0.32\textwidth]{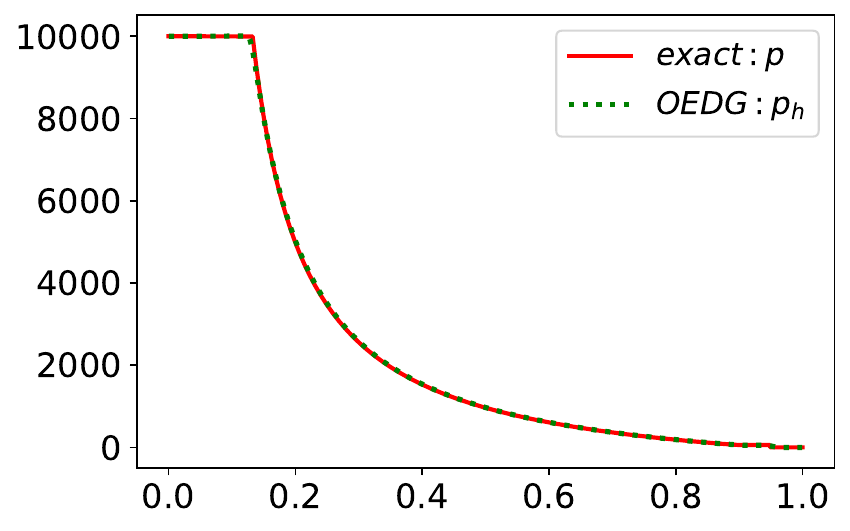}}
\subfigure[$\mathbb{P}^2$, $p$]{\includegraphics[width=0.32\textwidth]{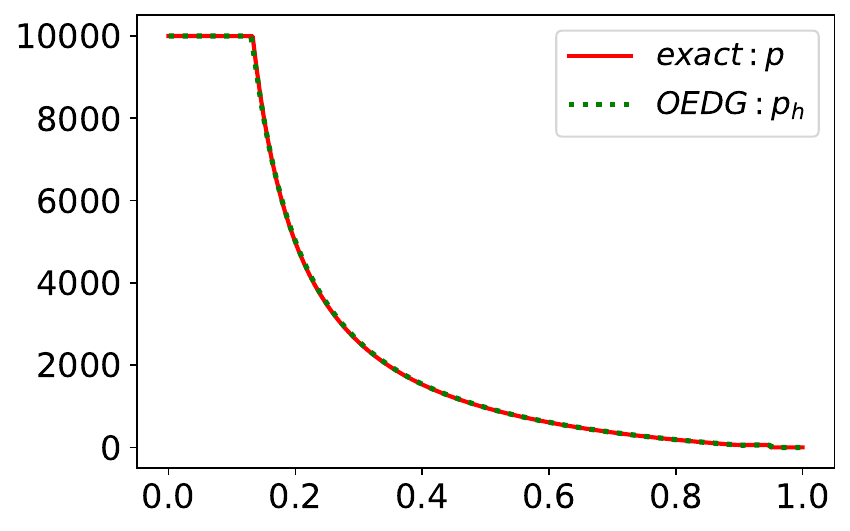}}
\subfigure[$\mathbb{P}^3$, $p$]{\includegraphics[width=0.32\textwidth]{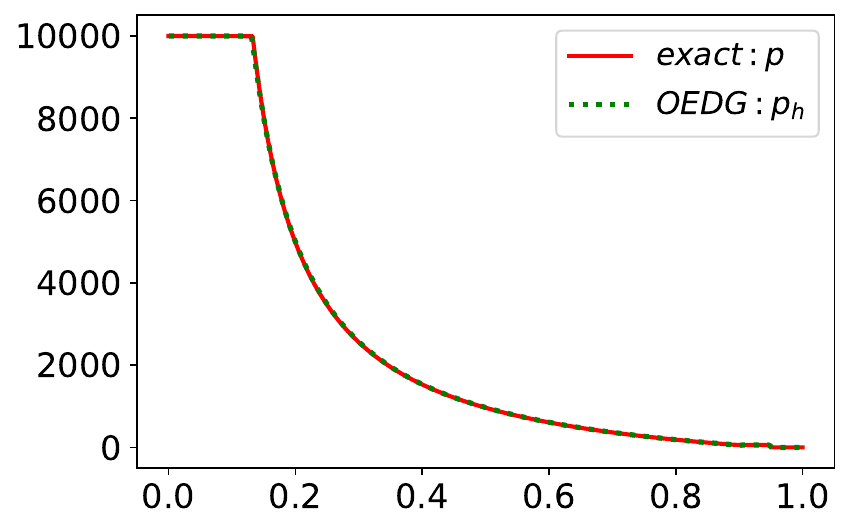}}
\caption{Example \ref{1D Rie exam7}: Numerical results for $\rho,\,v,\,p$ obtained using the $\mathbb{P}^m$-based PCP-OEDG method with 800 uniform cells at $t = 0.45$.}\label{fig:1D Rie exam7}
\end{figure}

The flow pattern in this example is similar to that of Example \ref{1D Rie exam6}, but this problem is more computationally challenging due to the ultra-relativistic effects. The PCP limiter is crucial for this example; without it, nonphysical states exceeding the admissible set would be produced, leading to simulation failure. Both the contact discontinuity and the shock wave move at speeds close to the speed of light, and the region between them has a width of approximately $0.00424$ at the end of the simulation. Figure \ref{fig:1D Rie exam7} shows the numerical results at $t = 0.45$ using 800 uniform cells. The results demonstrate that the high-order PCP-OEDG method provides excellent resolution and accurately captures the wave configuration without any nonphysical oscillations.
\end{example}

\begin{example}[Density Perturbation Problem]\label{1D Rie exam3}
This example validates the ability of the OEDG schemes to resolve high-frequency waves. The initial data are given by
\begin{align*}
\mathbf{Q}(x,0) =
\begin{cases}
(1, 0, 50)^{\top}, & \quad x < 0.5, \\
(2 + 0.3\sin(50x), 0, 10)^{\top}, & \quad x > 0.5,
\end{cases}
\end{align*}
with an adiabatic index of $\Gamma = 5/3$. The computational domain is $\Omega = [0,1]$, with outflow boundary conditions.

\begin{figure}[!thb]
\centering
\subfigure[$\mathbb{P}^1$]{\includegraphics[width=0.32\textwidth]{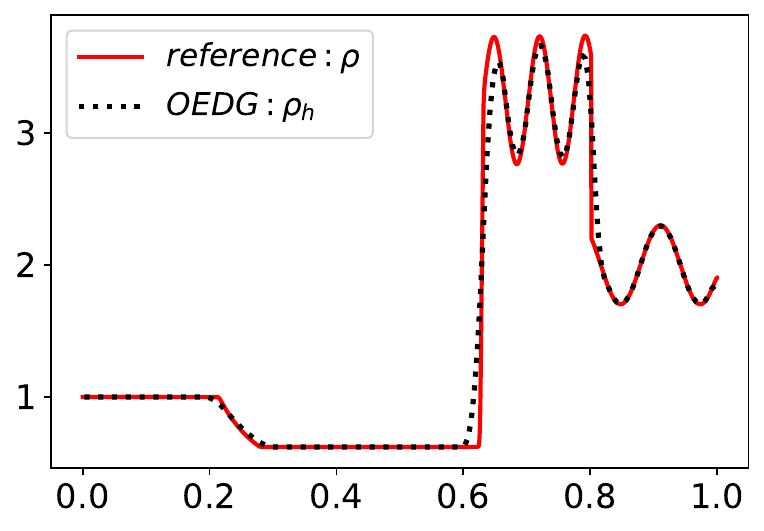}}
\subfigure[$\mathbb{P}^2$]{\includegraphics[width=0.32\textwidth]{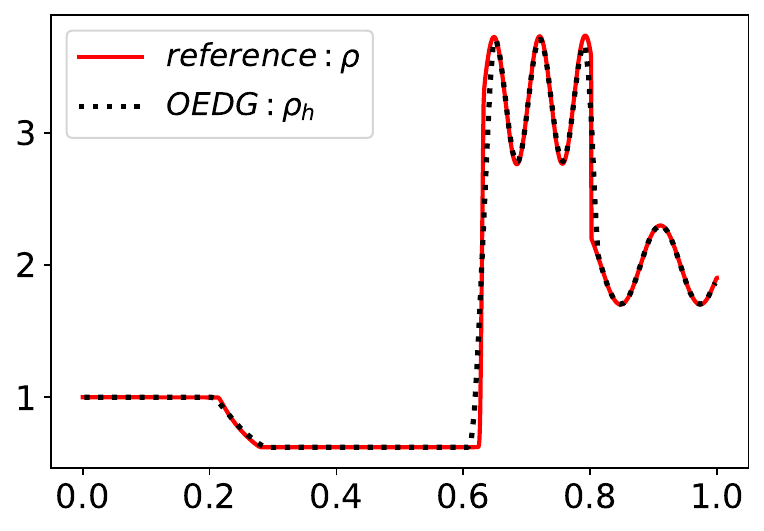}}
\subfigure[$\mathbb{P}^3$]{\includegraphics[width=0.32\textwidth]{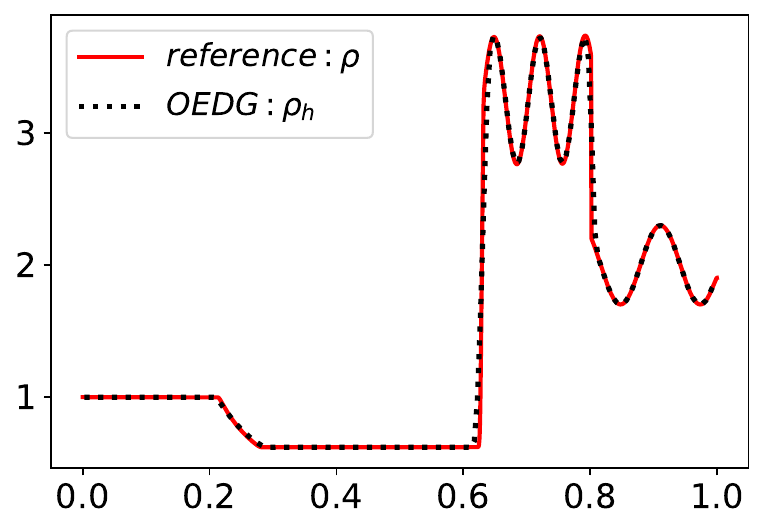}}
\caption{Example \ref{1D Rie exam3}: Numerical results for $\rho$ obtained using the $\mathbb{P}^m$-based OEDG method with 200 uniform cells at $t = 0.35$.}\label{fig:1D Rie exam3}
\end{figure}

Figure \ref{fig:1D Rie exam3} shows the numerical results at $t = 0.35$, using the $\mathbb{P}^m$-based OEDG method with $m = 1, 2, 3$, on a mesh of 200 uniform cells. 
The reference solution is produced by the first-order local Lax--Friedrichs scheme with $100,000$ uniform grids. 
The results indicate that the OEDG methods effectively resolve the small perturbation waves. Higher-order polynomial approximations ($\mathbb{P}^2$ and $\mathbb{P}^3$) demonstrate better resolution of high-frequency waves compared to the lower-order approximation ($\mathbb{P}^1$).

\end{example}

\subsection{2D Special RHD Examples}

In this subsection, we present several benchmark tests, including a smooth problem, two 2D Riemann problems, the double Mach reflection problem, and the shock-bubble interaction problem, to validate the accuracy and effectiveness of the 2D PCP-OEDG schemes in flat spacetime.

\begin{example}[2D Smooth Problem]\label{2D smooth exam1}
The accuracy of the 2D PCP-OEDG schemes is evaluated using a smooth problem in the domain $\Omega = [0,1]^2$, where the wave propagates at a $45^\circ$ angle to the $x$-axis. The exact solution is given by:
\begin{align*}
\mathbf{Q}(x,t) = (\rho, v_1, v_2, p)^{\top}(x,t) = \left(1 + 0.9999\sin(2\pi(x + y - 0.99\sqrt{2}t)), \frac{0.99}{\sqrt{2}}, \frac{0.99}{\sqrt{2}}, 0.01\right)^{\top}. 
\end{align*}
The computational domain $\Omega$ is divided into $N_x \times N_y$ uniform cells, with periodic boundary conditions applied. The adiabatic index is set to $5/3$. Table \ref{tb:2D exam1} lists the $L^1$, $L^2$, and $L^{\infty}$ errors and the corresponding convergence orders obtained at $t = 0.1$ using the $\mathbb{P}^m$-based OEDG method with $m = 1, 2, 3$, respectively. The convergence rate of the density $\rho$ approaches the theoretical $(m+1)$-th order as the mesh is refined.
\end{example}

\begin{table}[!thb]
\center
\caption{Example \ref{1D smooth exam1}: Errors and convergence rates for $\mathbb{P}^m$-based OEDG method with $N_x \times N_y$ uniform cells.}
\begin{tabular}[c]{c|c|c|c|c|c|c|c}
\toprule
$m$ &  $N_x \times N_y$ &  $|| u - u_h||_{L^1(\Omega)}$ &rate & $|| u - u_h||_{L^2(\Omega)}$ &rate & $|| u - u_h||_{L^{\infty}(\Omega)}$ &rate\\
\hline
\multirow{5}{*}{1}
&32  	$\times$	32 &	3.614E-02&	2.557&	4.356E-02&	2.524&	8.597E-02&	2.385\\ 
&64	    $\times$	64 &	6.436E-03&	2.489&	7.758E-03&	2.489&	2.092E-02&	2.039\\ 
&128	$\times$	128&	1.283E-03&	2.326&	1.550E-03&	2.323&	4.900E-03&	2.094\\ 
&256	$\times$	256&	2.780E-04&	2.207&	3.439E-04&	2.172&	1.121E-03&	2.128\\ 
&512	$\times$	512&	6.495E-05&	2.098&	8.233E-05&	2.063&	1.776E-04&	2.658\\
\hline
\multirow{5}{*}{2}
&32 	$\times$	32&	    2.193E-03&	5.431&	2.648E-03&	5.443&	6.188E-03&	5.641\\ 
&64 	$\times$	64&	    1.306E-04&	4.070&	1.669E-04&	3.988&	5.073E-04&	3.609\\ 
&128	$\times$	128&	1.098E-05&	3.572&	1.537E-05&	3.441&	5.605E-05&	3.178\\ 
&256	$\times$	256&	1.218E-06&	3.172&	1.695E-06&	3.181&	6.530E-06&	3.102\\ 
&512	$\times$	512&	1.481E-07&	3.040&	2.026E-07&	3.065&	7.852E-07&	3.056\\ 
\hline
\multirow{5}{*}{3}
&32 	$\times$	32&	    4.073E-04&	7.426&	5.153E-04&	7.354&	1.608E-03&	7.207\\ 
&64 	$\times$	64&	    1.487E-05&	4.775&	1.814E-05&	4.828&	4.984E-05&	5.012\\ 
&128	$\times$	128&	7.232E-07&	4.362&	8.883E-07&	4.352&	2.263E-06&	4.461\\ 
&256	$\times$	256&	4.044E-08&	4.161&	5.133E-08&	4.113&	1.204E-07&	4.232\\ 
&512	$\times$	512&	2.429E-09&	4.057&	3.141E-09&	4.030&	7.007E-09&	4.103\\ 
\bottomrule
\end{tabular}\label{tb:2D exam1}
\end{table}


\begin{example}[2D Riemann Problems]\label{2D:Rie exams}
In this example, we simulate three 2D Riemann problems, each with initial data consisting of four different constant states in the unit square $\Omega = [0,1]^2$. The domain $\Omega$ is divided into $400 \times 400$ uniform cells with outflow boundary conditions. To demonstrate the importance and effectiveness of the OE procedure, we compare the results obtained using both the proposed OEDG method and the conventional DG method without the OE procedure.

The initial states for the first 2D Riemann problem \cite{HT2012} are
\begin{align*}
\mathbf{Q}(x,y,0) =
\begin{cases}
(0.5, 0.5, -0.5, 5)^{\top}, & \quad x > 0.5, \, y > 0.5, \\
(1, 0.5, 0.5, 5)^{\top}, & \quad x < 0.5, \, y > 0.5, \\
(3, -0.5, 0.5, 5)^{\top}, & \quad x < 0.5, \, y < 0.5, \\
(1.5, -0.5, -0.5, 5)^{\top}, & \quad x > 0.5, \, y < 0.5,
\end{cases}
\end{align*}
with an adiabatic index $\Gamma = 5/3$. This problem models the interaction of four contact discontinuities (vortex sheets) in an ideal relativistic fluid. The second and third 2D Riemann problems \cite{WT2015} involve highly ultra-relativistic conditions with large velocities near the speed of light, making the proposed PCP technique essential for stable simulations.
Their initial conditions are
\[
\mathbf{Q}(x,y,0) =
\begin{cases}
(0.1, 0, 0, 0.01)^{\top}, & \quad x > 0.5, \, y > 0.5, \\
(0.1, 0.99, 0, 1)^{\top}, & \quad x < 0.5, \, y > 0.5, \\
(0.5, 0, 0, 1)^{\top}, & \quad x < 0.5, \, y < 0.5, \\
(0.1, 0, 0.99, 1)^{\top}, & \quad x > 0.5, \, y < 0.5,
\end{cases}
\]
and
\[
\mathbf{Q}(x,y,0) =
\begin{cases}
(0.1, 0, 0, 20)^{\top}, & \quad x > 0.5, \, y > 0.5, \\
(0.00414329639576, 0.9946418833556542, 0, 0.05)^{\top}, & \quad x < 0.5, \, y > 0.5, \\
(0.01, 0, 0, 0.05)^{\top}, & \quad x < 0.5, \, y < 0.5, \\
(0.00414329639576, 0, 0.9946418833556542, 0.05)^{\top}, & \quad x > 0.5, \, y < 0.5,
\end{cases}
\]
respectively.

\begin{figure}[!thb]
\centering
\subfigure[PCP-OEDG]{\includegraphics[width=0.48\textwidth]{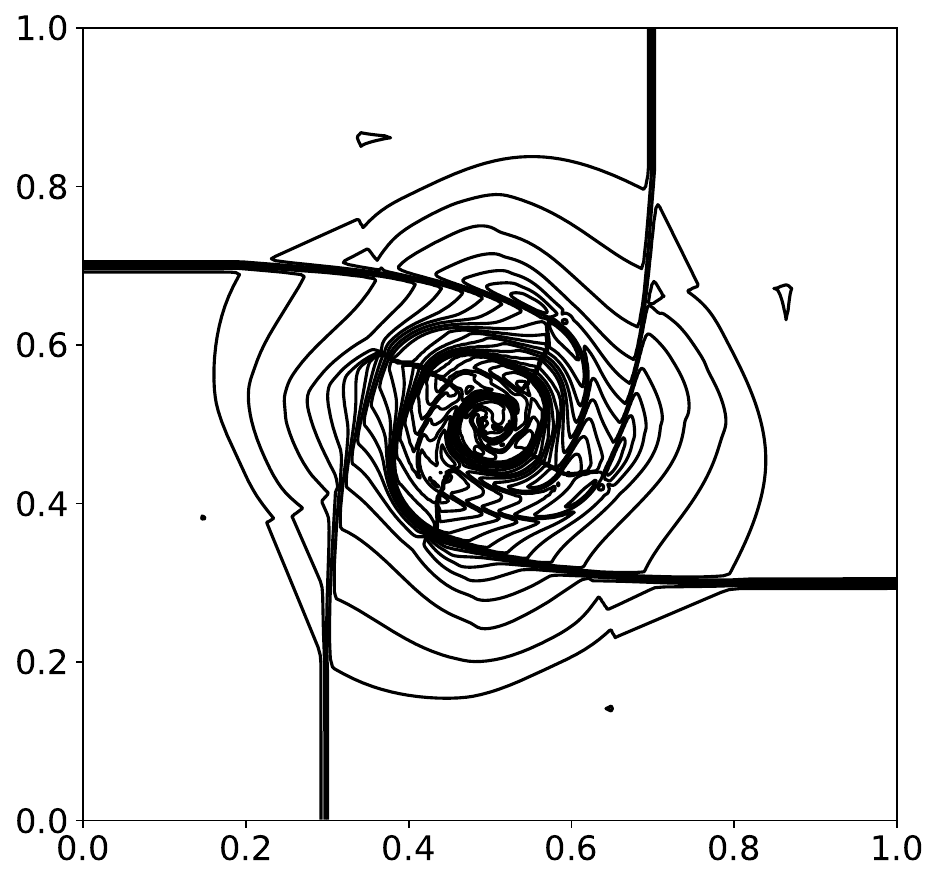}}
\subfigure[PCP-DG without OE]{\includegraphics[width=0.48\textwidth]{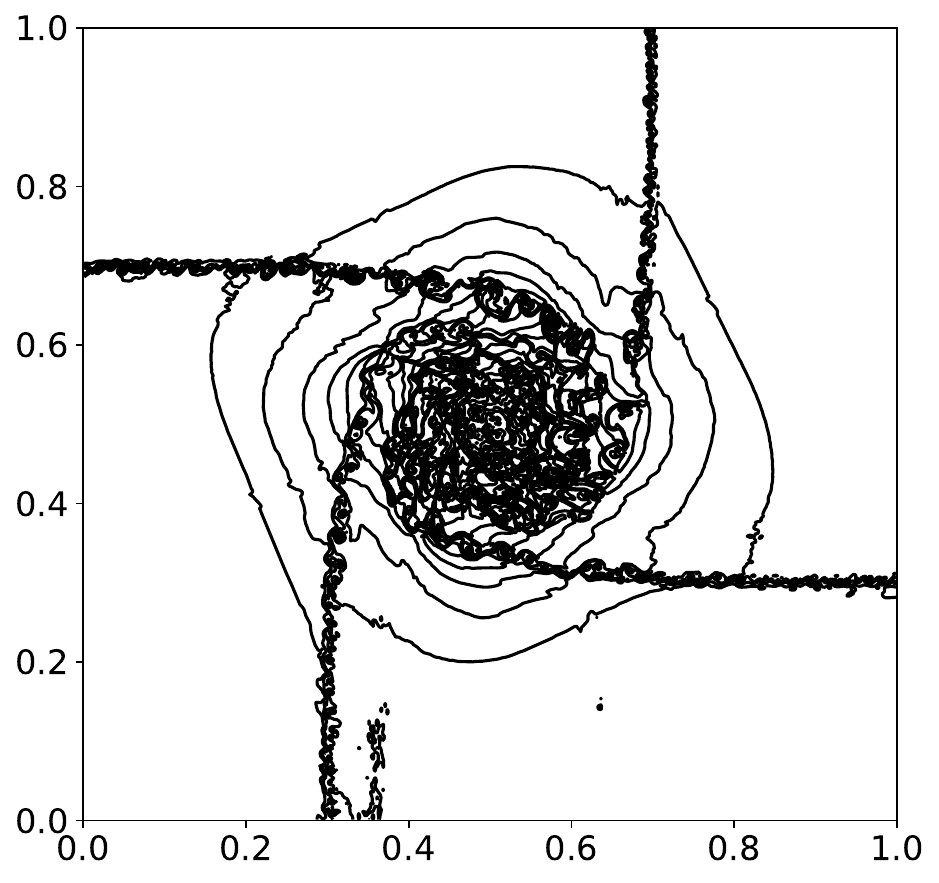}}
\caption{First 2D Riemann problem in Example \ref{2D:Rie exams}: Contours of $\log \rho$ at $t = 0.4$ for the OEDG schemes and the DG schemes without the OE procedure on a mesh of $400 \times 400$ uniform cells.}\label{fig:2D Rie exam3}
\end{figure}

Figure \ref{fig:2D Rie exam3} shows the contours of the density logarithm at $t = 0.4$ for the first 2D Riemann problem using both the third-order OEDG schemes and conventional DG schemes without the OE procedure. As time progresses, the four vortex sheets interact to form a spiral with low density around the center. The difference between the OEDG method and the DG method is evident: the DG method without the OE procedure exhibits significant spurious oscillations, whereas the OEDG solutions are free from such nonphysical oscillations, confirming the effectiveness of the OE process. 
The results for the second and third 2D Riemann problems are shown in Figures \ref{fig:2D:Rie exam6} and \ref{fig:2D:Rie exam5}. Again, the OE procedure is crucial for eliminating spurious oscillations in the conventional DG method, while the proposed PCP-OEDG method demonstrates robustness for these challenging test cases. Without the PCP limiter, both the OEDG and conventional DG methods break down due to nonphysical states that exceed the admissible set.

\begin{figure}[!thb]
\centering
\subfigure[PCP-OEDG]{\includegraphics[width=0.48\textwidth]{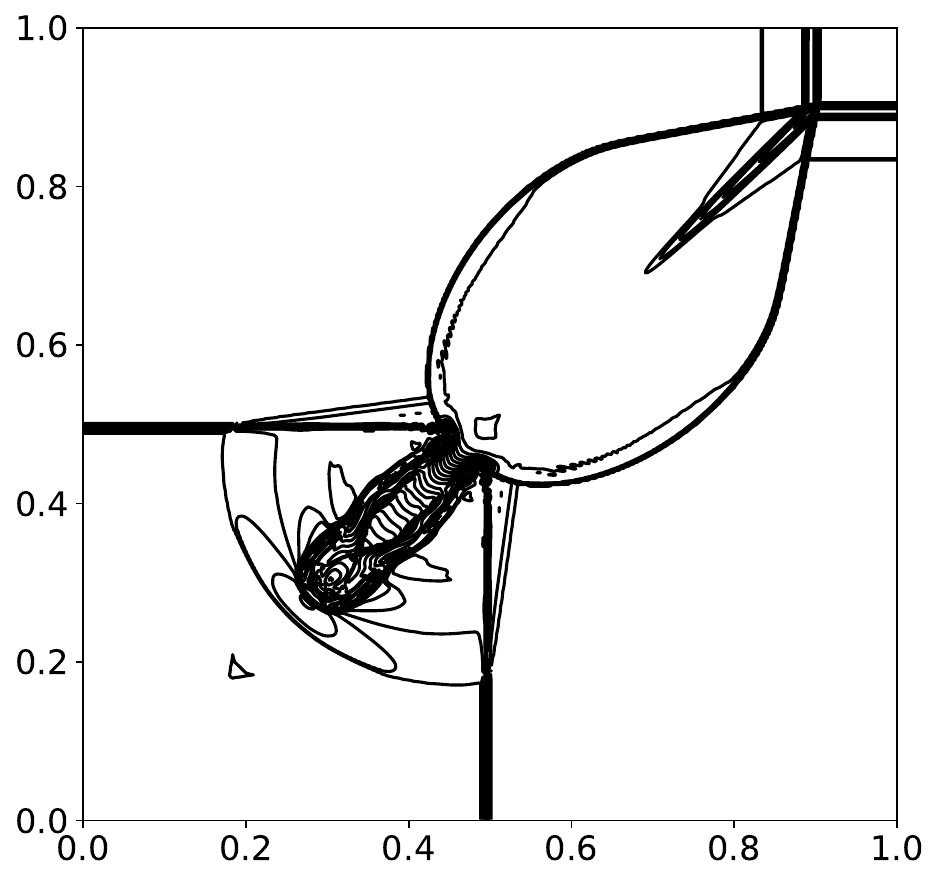}}
\subfigure[PCP-DG without OE]{\includegraphics[width=0.48\textwidth]{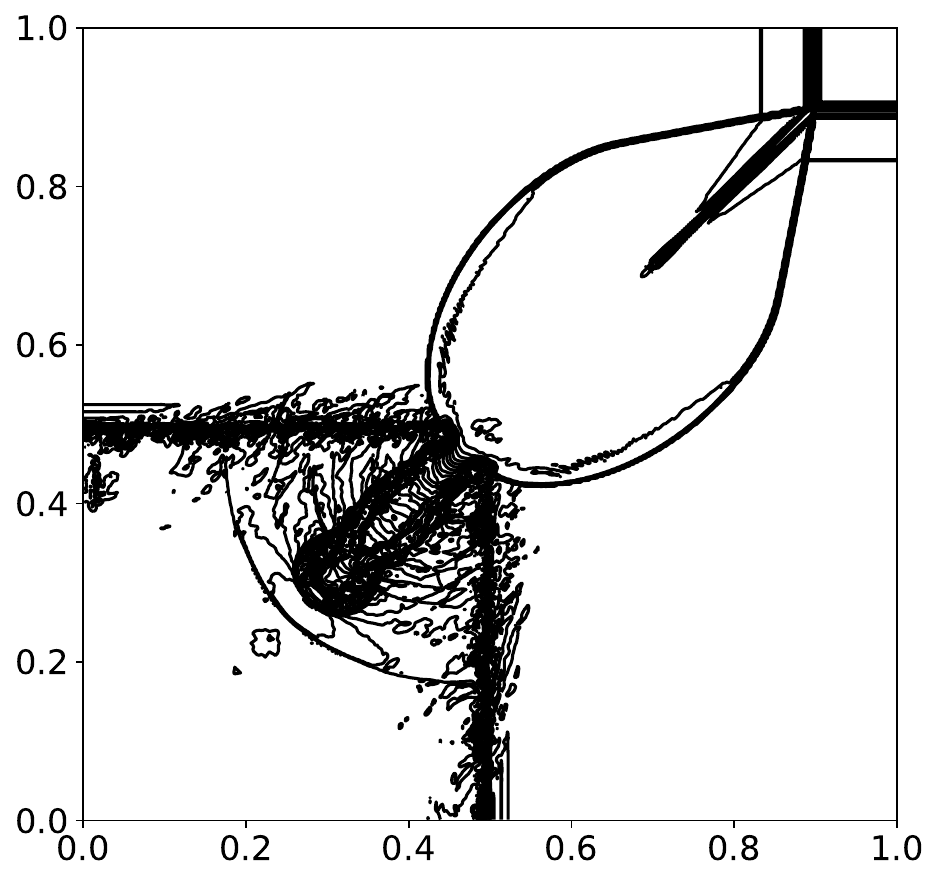}}
\caption{Second 2D Riemann problem in Example \ref{2D:Rie exams}: Contours of $\log \rho$ at $t = 0.4$ for the third-order PCP DG schemes with and without the OE procedure on a mesh of $400 \times 400$ uniform cells.}\label{fig:2D:Rie exam6}
\end{figure}

\begin{figure}[!thb]
\centering
\subfigure[PCP-OEDG]{\includegraphics[width=0.48\textwidth]{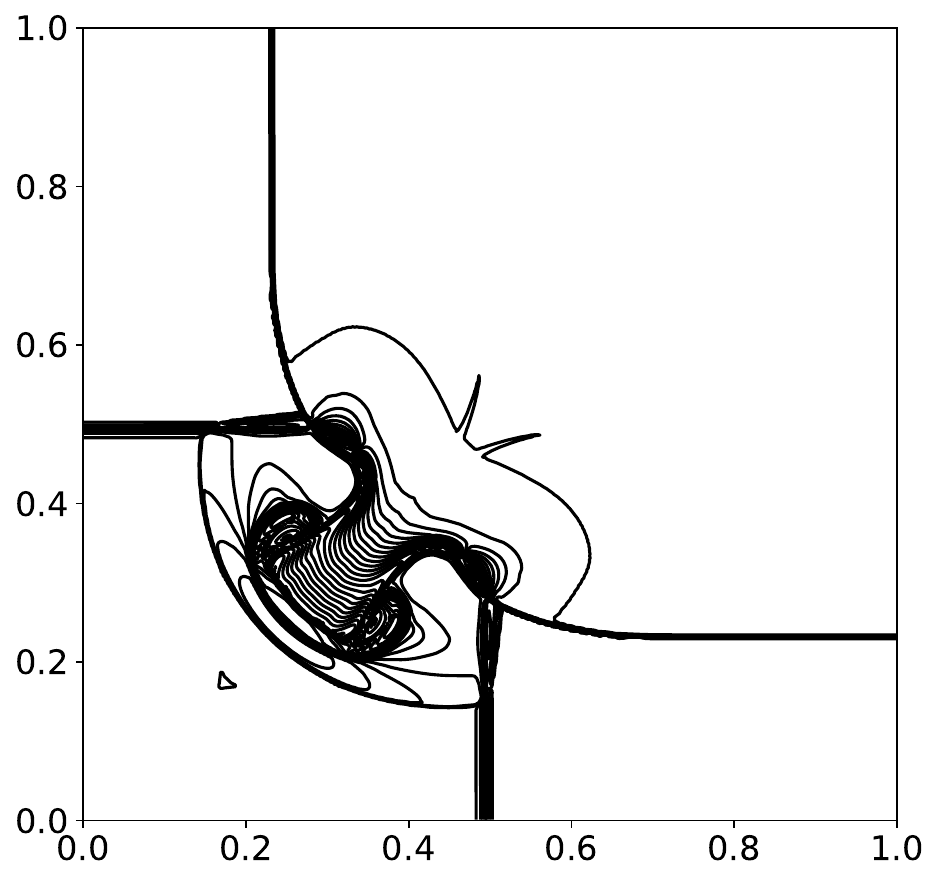}}
\subfigure[PCP-DG without OE]{\includegraphics[width=0.48\textwidth]{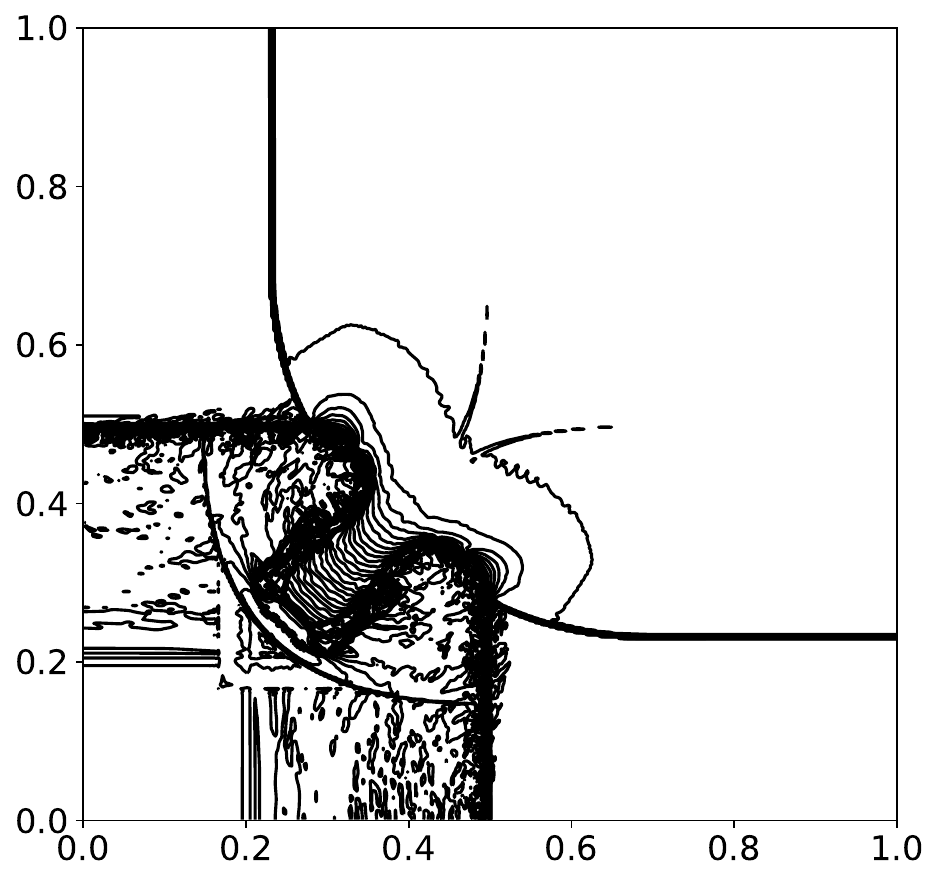}}
\caption{Third 2D Riemann problem in Example \ref{2D:Rie exams}: Contours of $\log \rho$ at $t = 0.4$ for the third-order PCP DG schemes with and without the OE procedure on a mesh of $400 \times 400$ uniform cells.}\label{fig:2D:Rie exam5}
\end{figure}

\end{example}

\begin{example}[Double Mach Reflection]\label{2D:Rie exam6}
The double Mach reflection problem is a standard benchmark for assessing high-resolution shock-capturing methods (see, e.g., \cite{HT2012, ZT2013}). Initially, a right-moving oblique shock wave, traveling at a velocity of $v_s = 0.4984$, is positioned at $(x, y) = (1/6, 0)$ and intersects the wall at a $60^\circ$ angle relative to the $x$-axis. The position of the shock front at time $t$ is described by
\[
h(x, t) = \sqrt{3}\left(x - \frac{1}{6}\right) - 2v_s t.
\]
The primitive variables for the states on either side of the shock are
\[
\mathbf{Q}_L = (8.564, 0.4247 \sin 60^\circ, -0.4247 \cos 60^\circ, 0.3808)^{\top}, \quad \mathbf{Q}_R = (1.4, 0, 0, 0.0025)^{\top}.
\]
The computational domain $\Omega = [0, 4] \times [0, 1]$ is discretized into $960 \times 240$ uniform cells. The adiabatic index is set to $\Gamma = 1.4$. The boundary conditions are specified as follows. 
(i) {\it Left boundary}: The post-shock state $\mathbf{Q}_L$ is imposed.
(ii) {\it Right boundary}: The pre-shock state $\mathbf{Q}_R$ is applied.
(iii) {\it Bottom boundary}: For $x \leq 1/6$, the left-side shock state is enforced, while for $x > 1/6$, a reflective boundary condition is used.
(iv) {\it Top boundary}: For $x > x_s$, the right-side shock state is applied, and for $x < x_s$, the left-side shock state is imposed, where $x_s$ is determined by solving $h(x, t) = 1$.

\begin{figure}[!thb]
\centering
\subfigure[$\mathbb{P}^1$]{\includegraphics[width=0.32\textwidth]{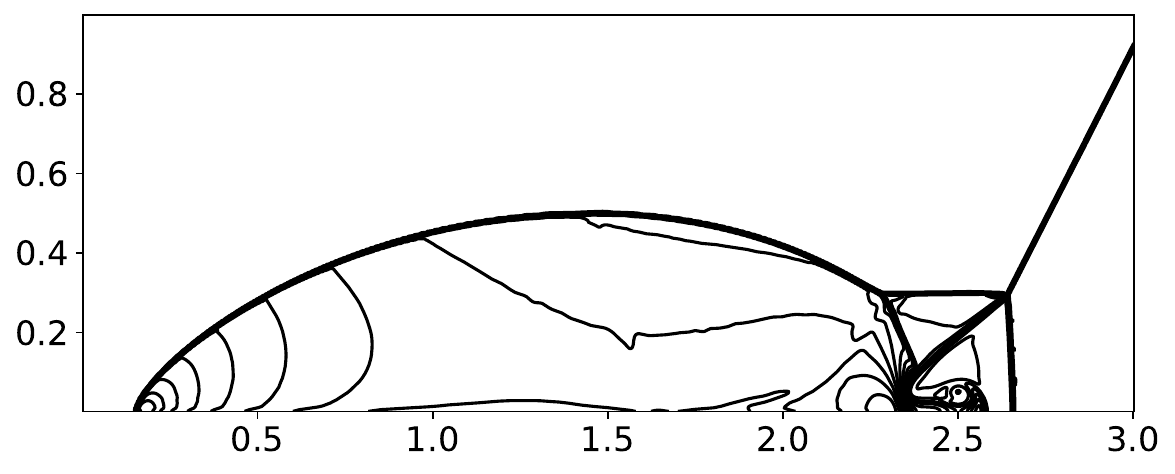}}
\subfigure[$\mathbb{P}^2$]{\includegraphics[width=0.32\textwidth]{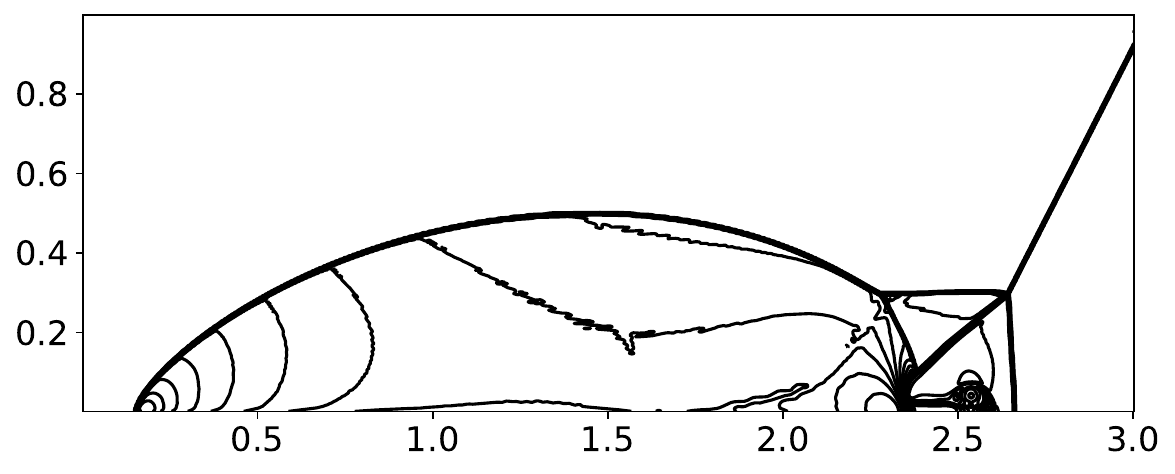}}
\subfigure[$\mathbb{P}^3$]{\includegraphics[width=0.32\textwidth]{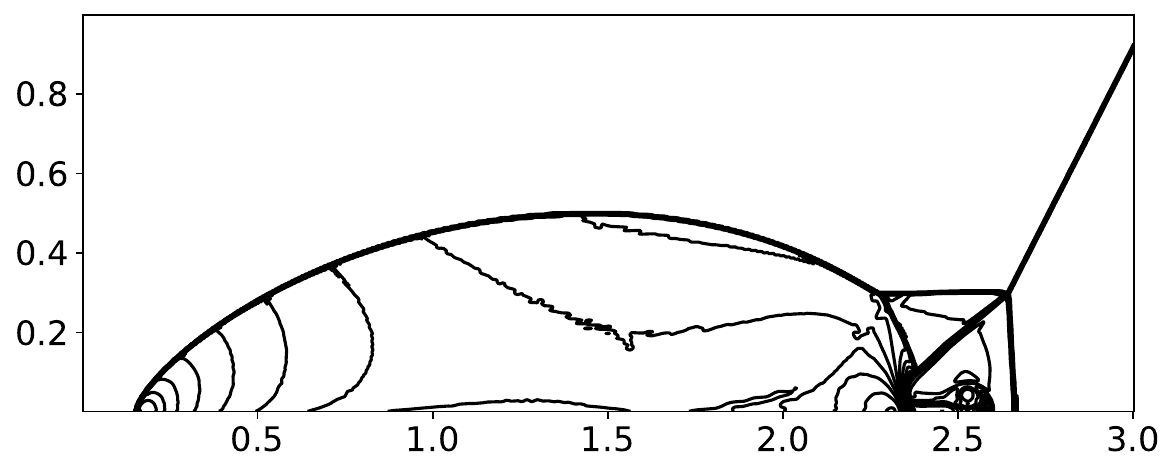}}
\caption{Example \ref{2D:Rie exam6}: Density logarithm contours at $t = 4$ using the $\mathbb{P}^m$-based OEDG method with $960 \times 240$ grid cells.}\label{fig:dm}
\end{figure}

Figure \ref{fig:dm} shows the density logarithm contours at $t = 4$ within the sub-domain $\Omega = [0, 3] \times [0, 1]$. The $\mathbb{P}^m$-based OEDG method for $m = 1, 2, 3$ accurately captures the complex structures around the double Mach reflection region, including the shock waves and contact discontinuities.
\end{example}

\begin{example}[Shock-Bubble Interaction]\label{2D:Rie exam7}
This example \cite{HT2012} extends the classical shock-bubble interaction problem, commonly studied in non-relativistic hydrodynamics, to demonstrate the shock-capturing capabilities of the proposed OEDG schemes. The computational domain is $\Omega = [0, 325] \times [-45, 45]$. Reflective boundary conditions are enforced at $y = \pm 45$, with fluid states specified on either side of the shock wave. Initially, a left-moving relativistic shock wave is positioned at $x = 265$, with the left and right states defined as:
\begin{align*}
\mathbf{Q}(x,y,0) =
\begin{cases}
(1, 0, 0, 0.05)^{\top}, & \quad x < 265, \\
(1.865225080631180, -0.196781107378299, 0, 0.15)^{\top}, & \quad x > 265.
\end{cases}
\end{align*}
Ahead of the shock wave, a cylindrical bubble is centered at $(215, 0)$ with a radius of $25$. The initial state of the fluid inside the bubble is given by:
\begin{align*}
\mathbf{Q}(x,y,0) = (0.1358, 0, 0, 0.05)^{\top}, \qquad \sqrt{(x-215)^2 + y^2} \leq 25,
\end{align*}
representing a lighter fluid compared to the surrounding medium.

\begin{figure}[!thb]
\centering
\subfigure[$\mathbb{P}^1,\,t = 90$]{\includegraphics[width=0.34\textwidth]{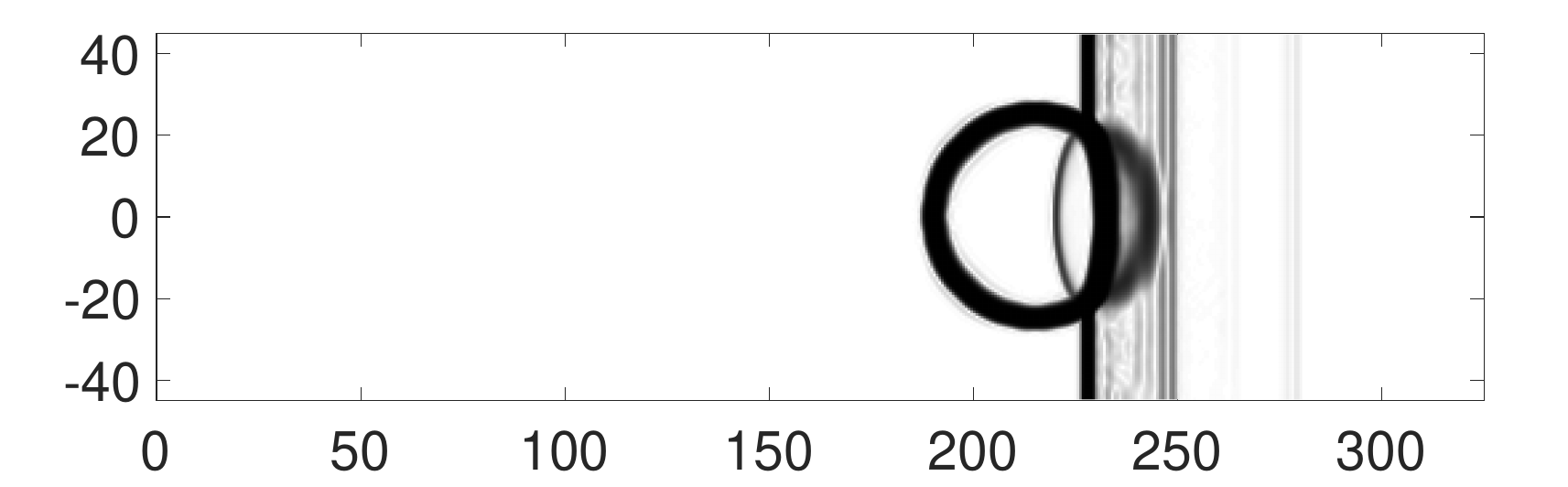}}\hspace{-3mm}
\subfigure[$\mathbb{P}^2,\,t = 90$]{\includegraphics[width=0.34\textwidth]{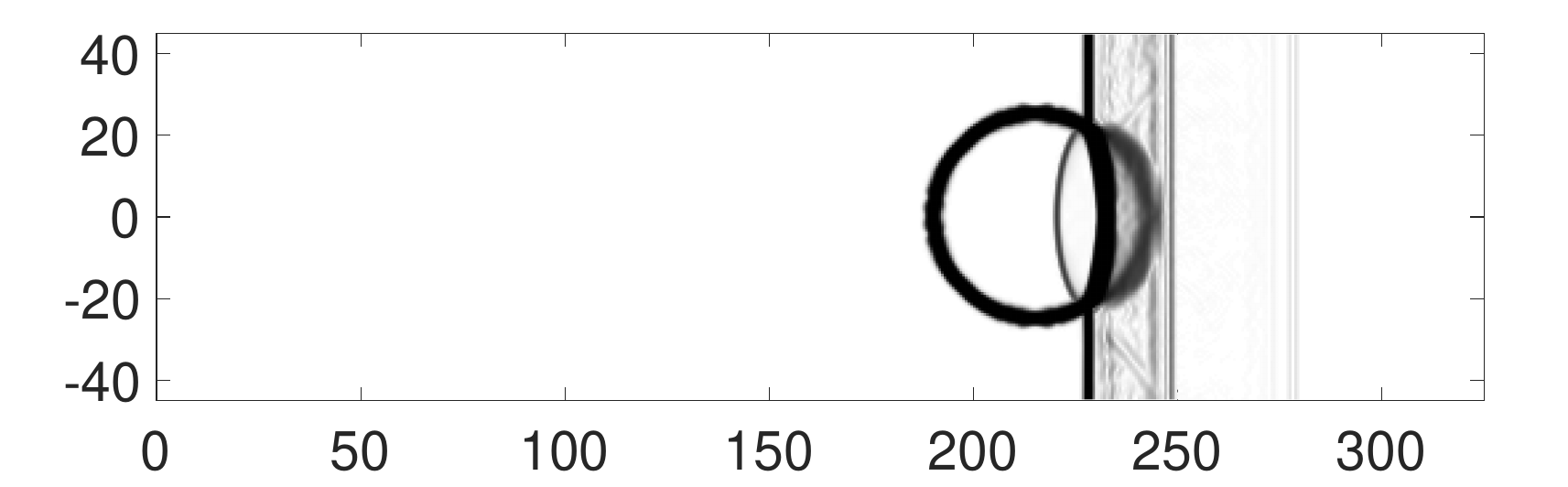}}\hspace{-3mm}
\subfigure[$\mathbb{P}^3,\,t = 90$]{\includegraphics[width=0.34\textwidth]{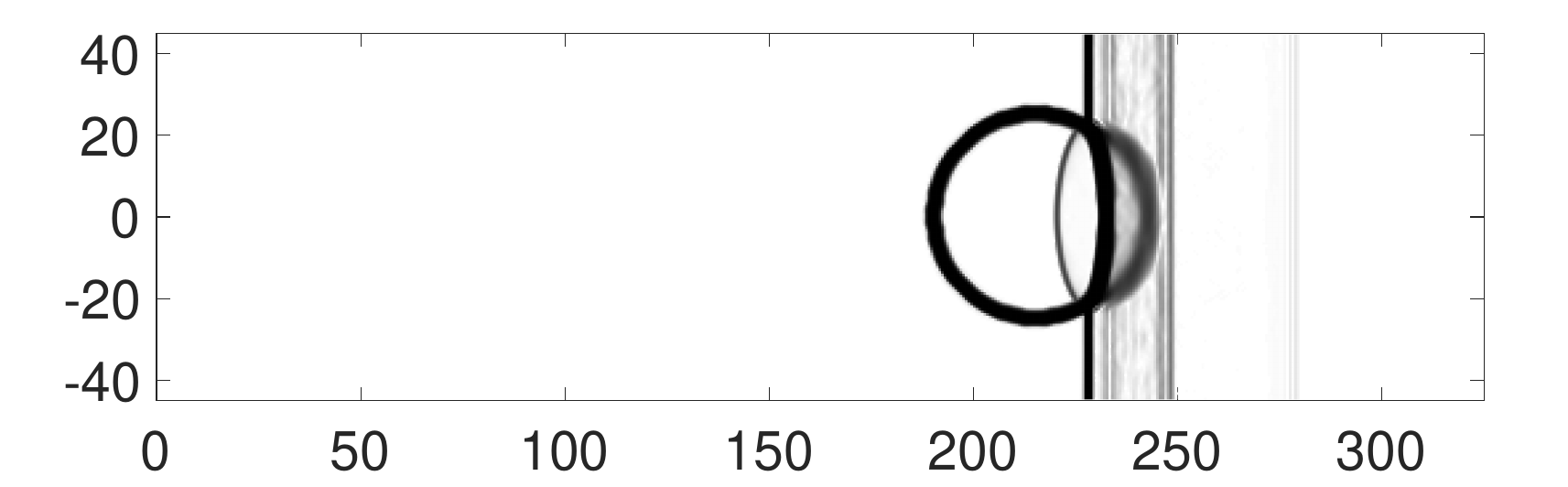}}\\
\subfigure[$\mathbb{P}^1,\,t = 270$]{\includegraphics[width=0.34\textwidth]{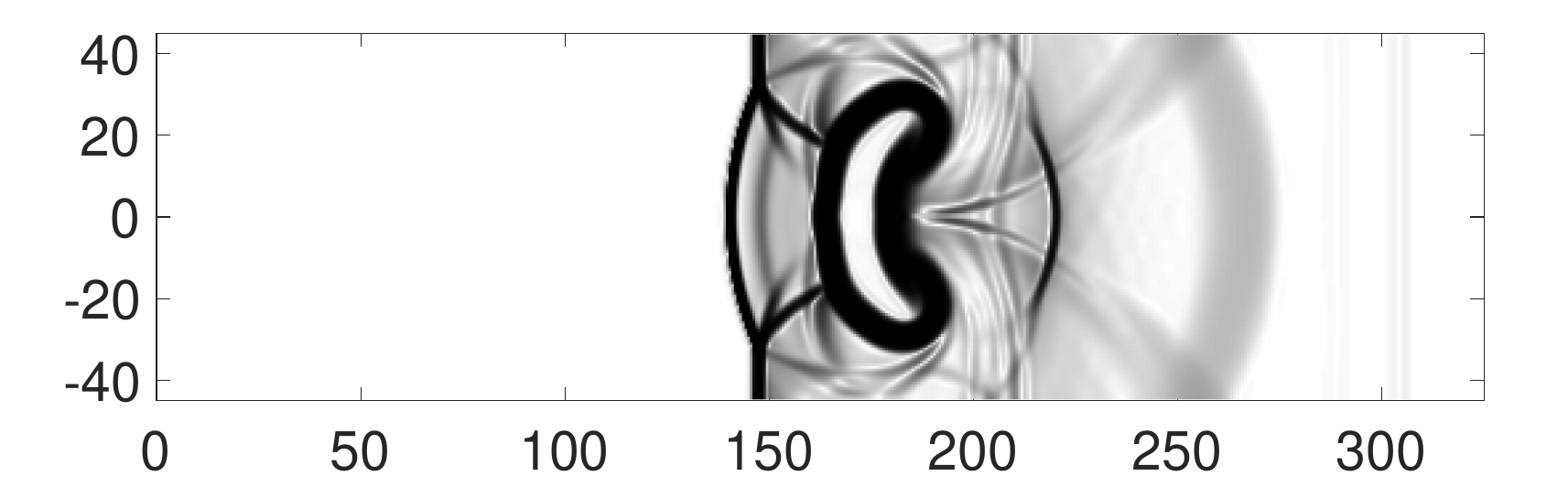}}\hspace{-3mm}
\subfigure[$\mathbb{P}^2,\,t = 270$]{\includegraphics[width=0.34\textwidth]{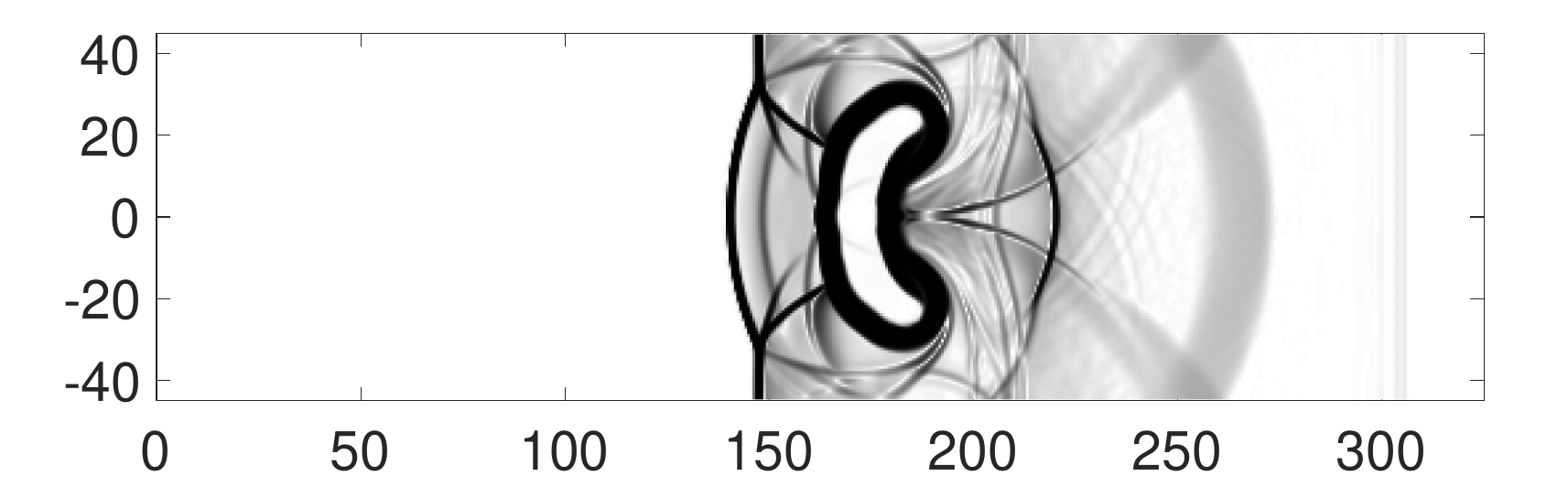}}\hspace{-3mm}
\subfigure[$\mathbb{P}^3,\,t = 270$]{\includegraphics[width=0.34\textwidth]{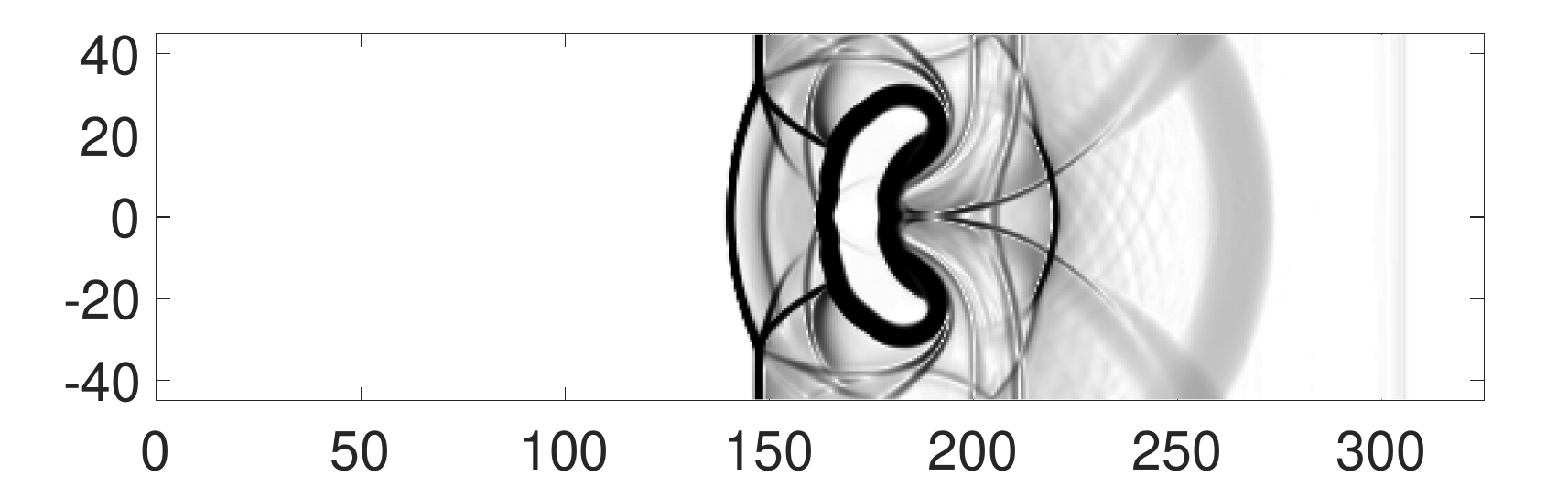}}\\
\subfigure[$\mathbb{P}^1,\,t = 450$]{\includegraphics[width=0.34\textwidth]{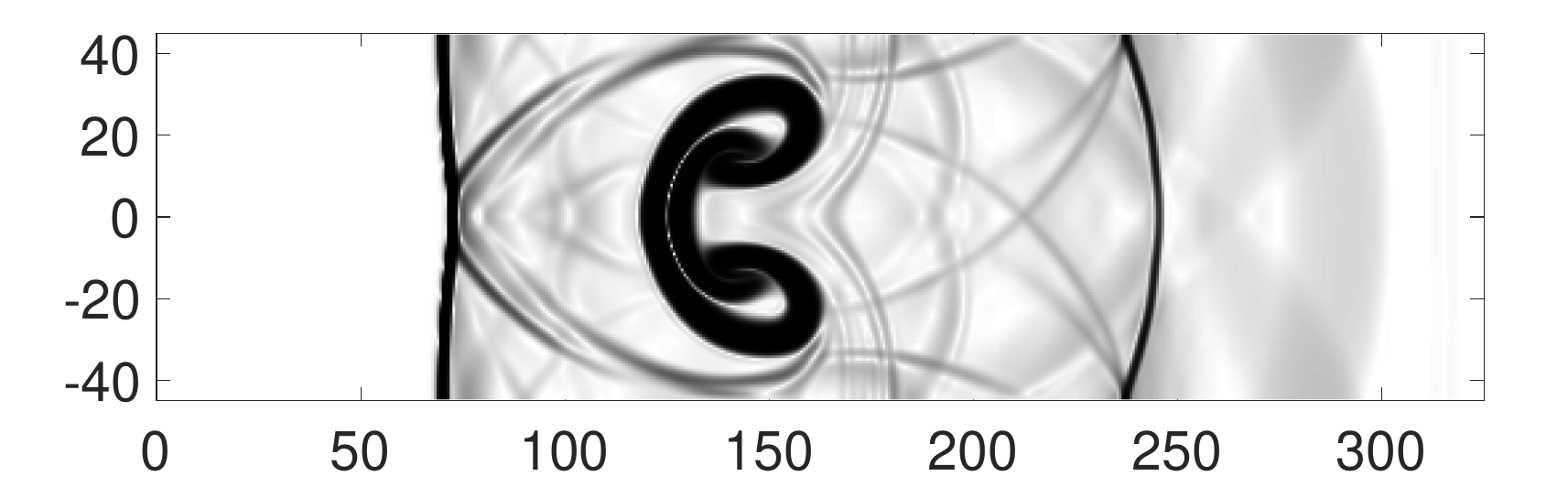}}\hspace{-3mm}
\subfigure[$\mathbb{P}^2,\,t = 450$]{\includegraphics[width=0.34\textwidth]{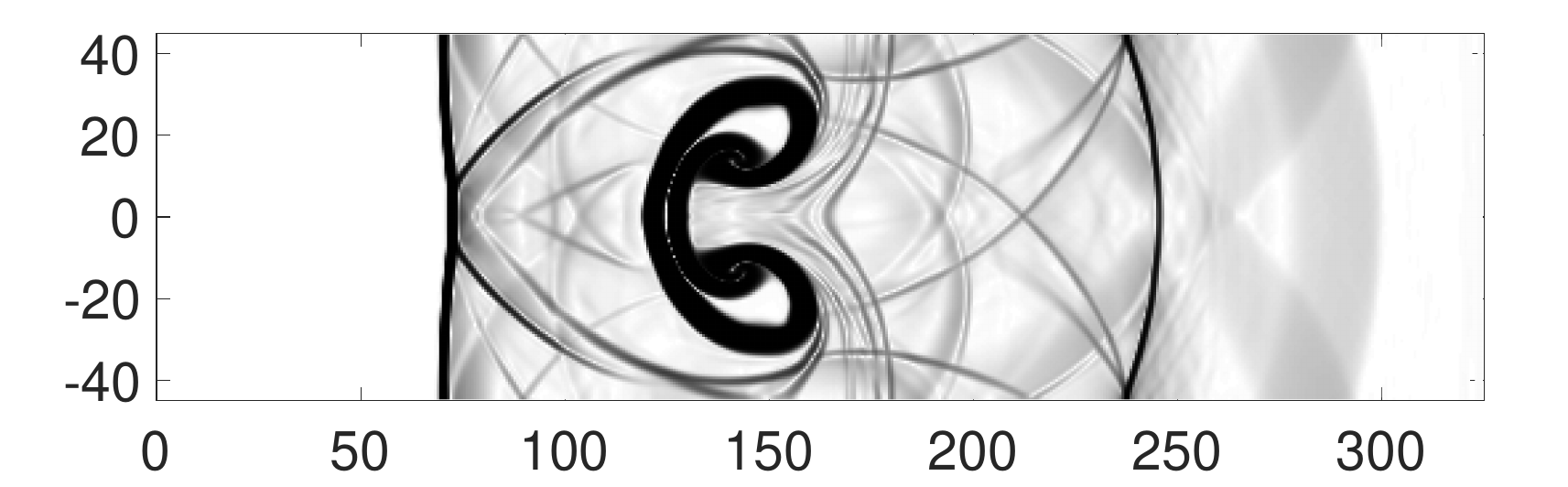}}\hspace{-3mm}
\subfigure[$\mathbb{P}^3,\,t = 450$]{\includegraphics[width=0.34\textwidth]{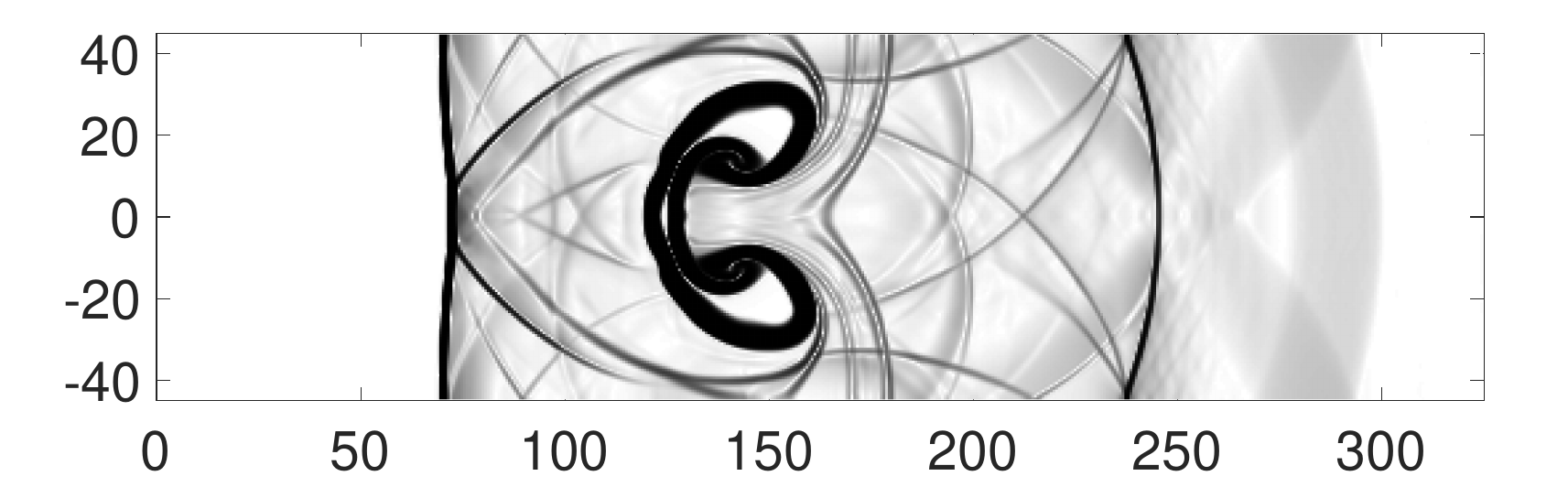}}
\caption{Example \ref{2D:Rie exam7}: Schlieren images of $\log \rho$ at different times for the $\mathbb{P}^m$-based OEDG method with $500 \times 140$ uniform cells.}\label{fig:2D:Rie exam7}
\end{figure}

Figure \ref{fig:2D:Rie exam7} presents schlieren images of the rest-mass density at $t = 90$, $270$, and $450$, computed using the OEDG schemes on a $500 \times 140$ uniform mesh. The results demonstrate the OEDG method's ability to capture the shock-wave dynamics and small-scale wave structures, such as the curling of the bubble interface during the interaction between the shock and the bubble. As expected, higher-order OEDG methods provide better resolution of shock waves and contact discontinuities, effectively eliminating spurious oscillations.
\end{example}

\subsection{Axisymmetric Relativistic Jets in Cylindrical Coordinates}

In this subsection, we simulate two high-speed relativistic jet flows relevant to extragalactic radio sources associated with active galactic nuclei. These simulations feature ultra-relativistic regions, strong shock waves, shear flows, and interface instabilities, posing significant challenges for numerical codes. Without the proposed PCP technique, both the DG codes with and without the OE procedure would fail, producing nonphysical states that violate the admissible set. 
The relativistic jets are modeled by solving the axisymmetric RHD equations in cylindrical coordinates $(r,z)$, as described in Section \ref{sec:AxisRHD}.

\begin{example}[Axisymmetric Relativistic Jet \uppercase\expandafter{\romannumeral1}]\label{2D:Rie exam9}
The first test case is based on the pressure-matched hot A1 model \cite{MMFIM1997}. The initial conditions for the relativistic jet are given by
\[
\mathbf{Q}(r, z, 0) = (1, 0, 0, 0.40611878453038897)^{\top},
\]
with an adiabatic index of $\Gamma = 4/3$. The computational domain, in cylindrical coordinates $(r, z)$, is $[0, 7] \times [0, 50]$, discretized into $210 \times 1500$ uniform cells. Symmetry boundary conditions are enforced at $r = 0$, a fixed inflow beam condition is applied at the nozzle $\{z = 0, r \leq 1\}$, and outflow boundary conditions are set on the remaining boundaries.  
Initially, a light jet beam with $\rho^b = 0.01$ is injected parallel to the $z$-direction through the bottom boundary inlet $(r \leq 1, z = 0)$. The beam propagates at a speed of $v^b = 0.99$, with relativistic effects from the high internal beam energy comparable to those from the near-light-speed velocity. The classical beam Mach number is $M_b = v^b/c_s = 1.72$, and the relativistic Mach number $M_r := M^b \gamma^b / \gamma_s$ is approximately $9.97$, where $\gamma^b = 1/\sqrt{1 - (v^b)^2}$ and $\gamma_s = 1/\sqrt{1 - c_s^2}$ are the Lorentz factors for the jet speed and local sound speed, respectively.

\begin{figure}[!thbp]
\centering
\subfigure[$\mathbb{P}^1,\,t = 10$]{\includegraphics[width=0.135\textwidth]{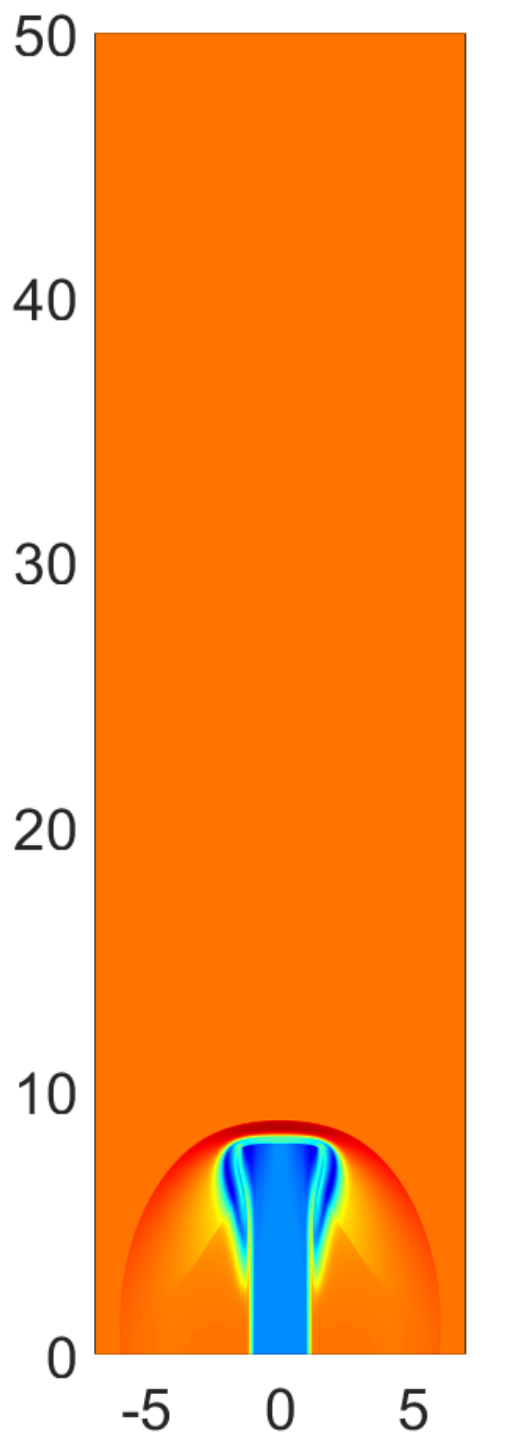}}\hspace{2mm}
\subfigure[$\mathbb{P}^1,\,t = 20$]{\includegraphics[width=0.135\textwidth]{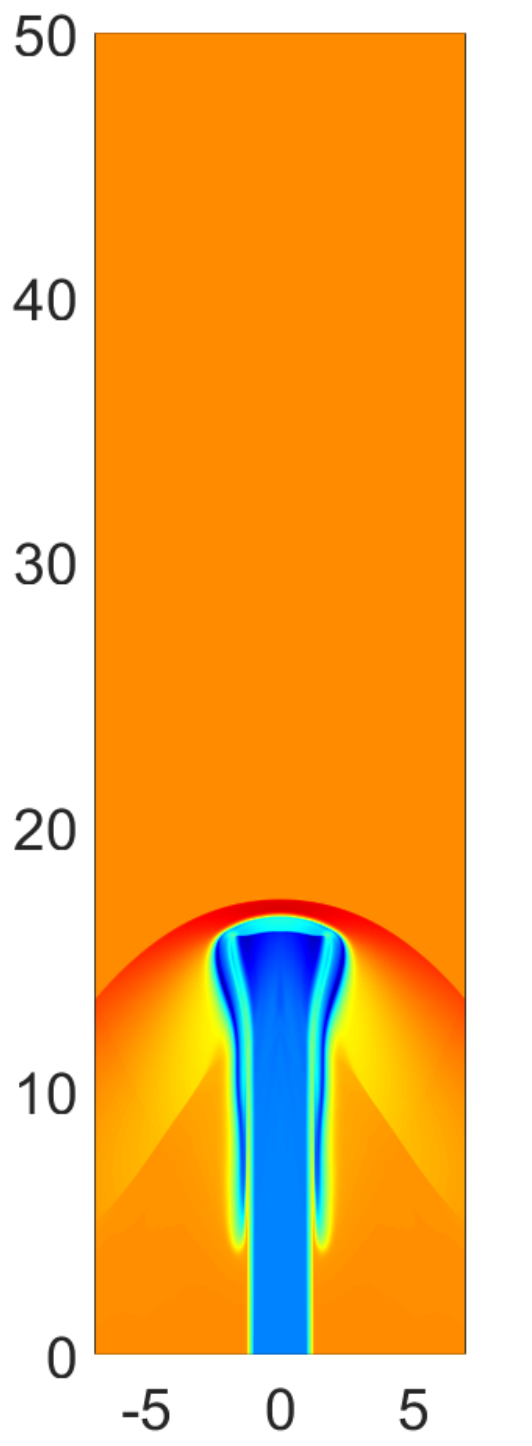}}\hspace{2mm}
\subfigure[$\mathbb{P}^1,\,t = 30$]{\includegraphics[width=0.135\textwidth]{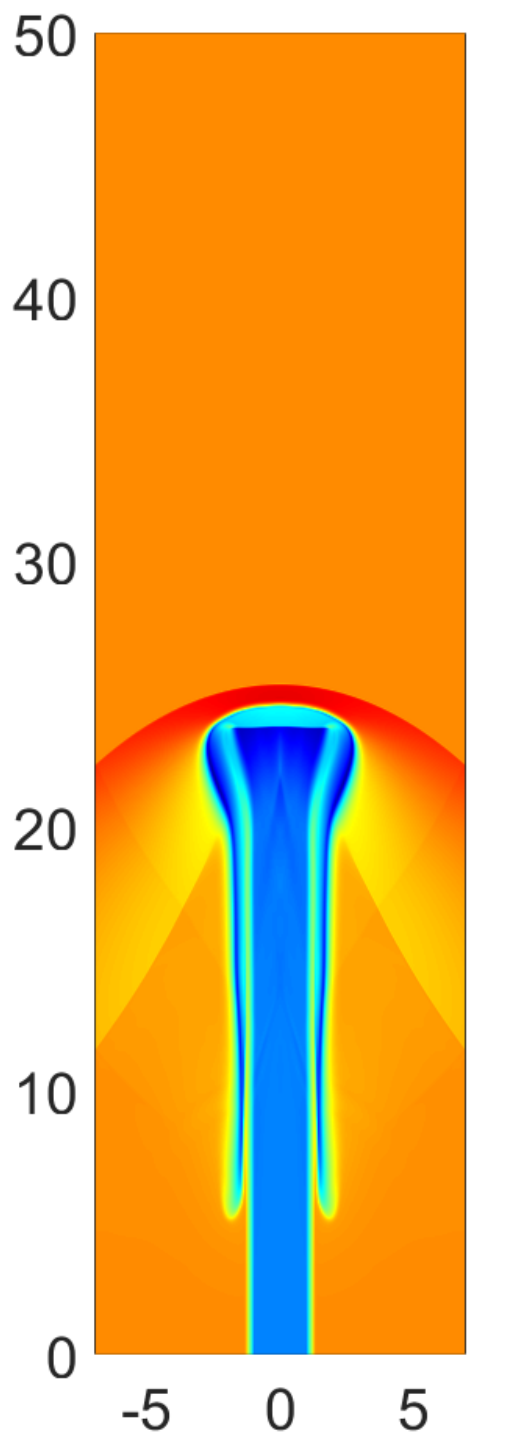}}\hspace{2mm}
\subfigure[$\mathbb{P}^1,\,t = 40$]{\includegraphics[width=0.135\textwidth]{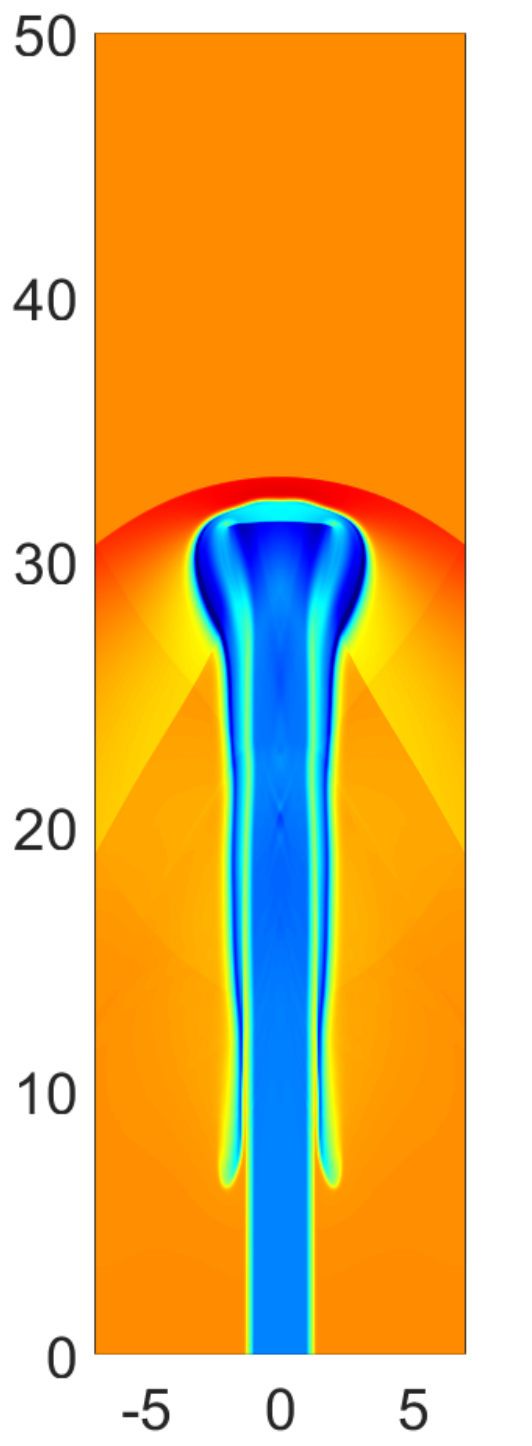}}\hspace{2mm}
\subfigure[$\mathbb{P}^1,\,t = 50$]{\includegraphics[width=0.135\textwidth]{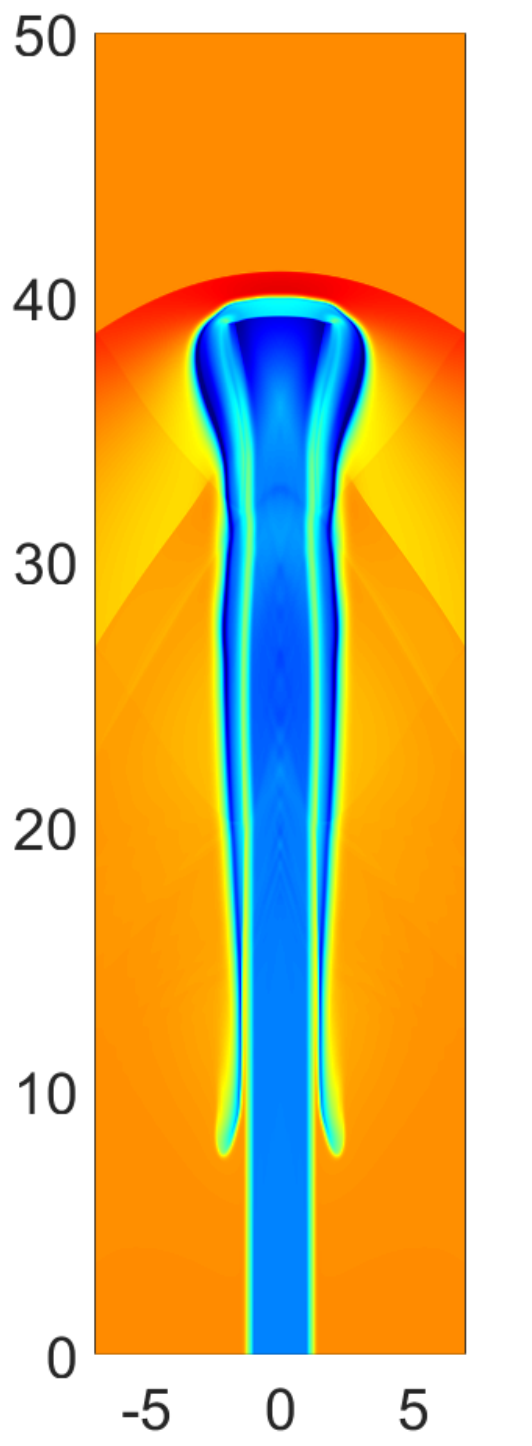}}\hspace{2mm}
\subfigure[$\mathbb{P}^1,\,t = 60$]{\includegraphics[width=0.135\textwidth]{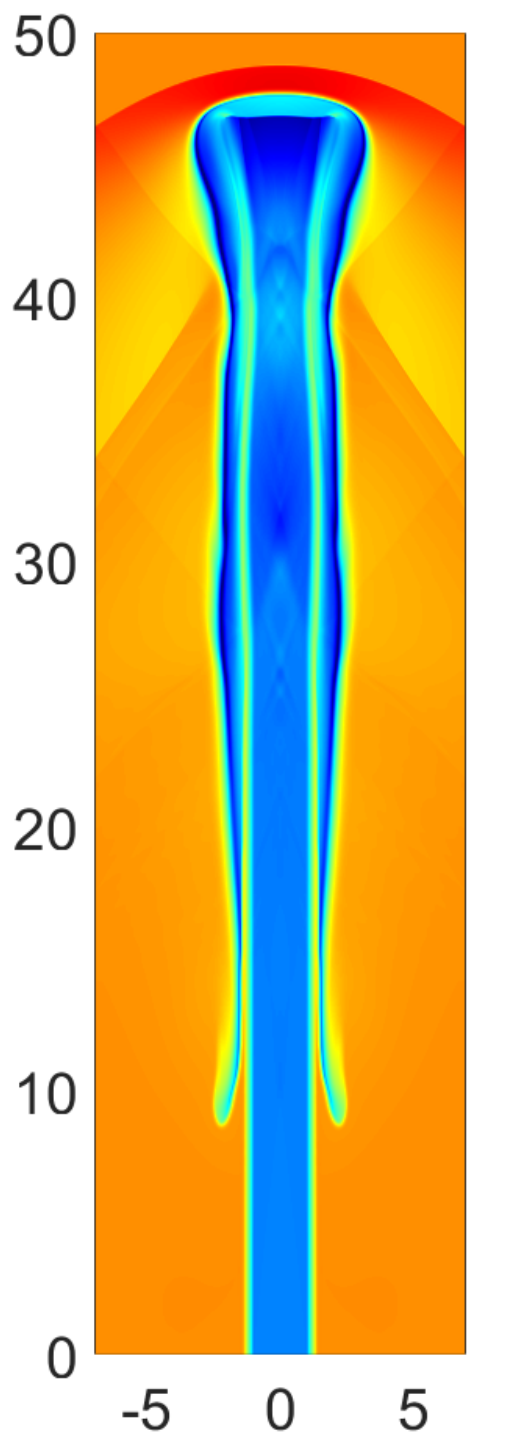}}\\
\subfigure[$\mathbb{P}^2,\,t = 10$]{\includegraphics[width=0.135\textwidth]{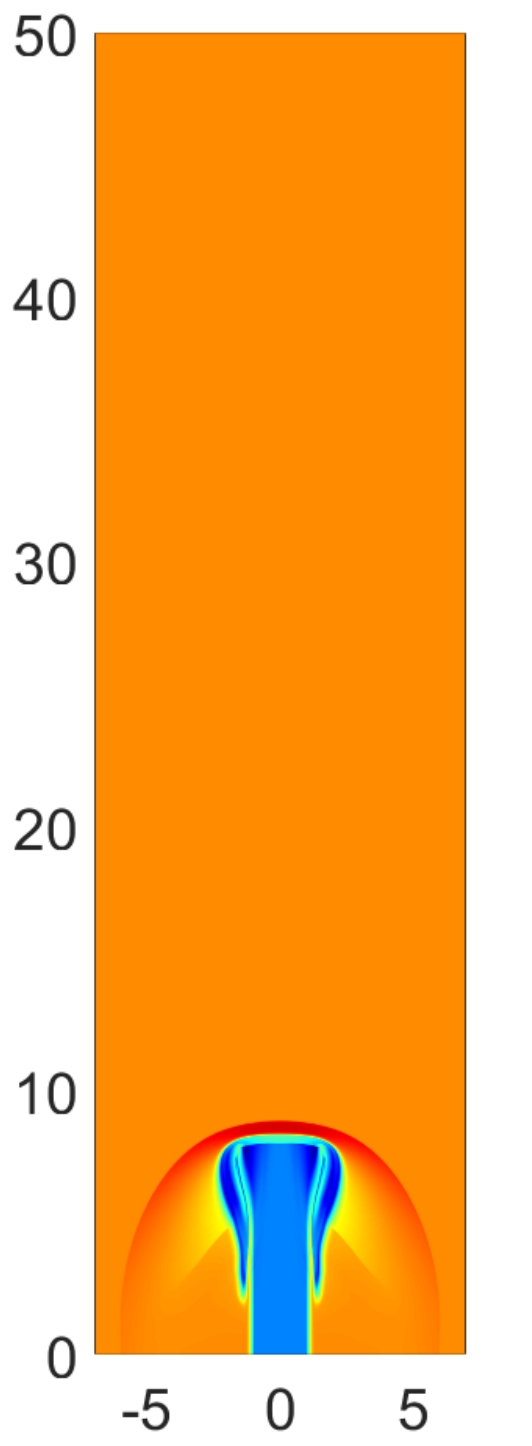}}\hspace{2mm}
\subfigure[$\mathbb{P}^2,\,t = 20$]{\includegraphics[width=0.135\textwidth]{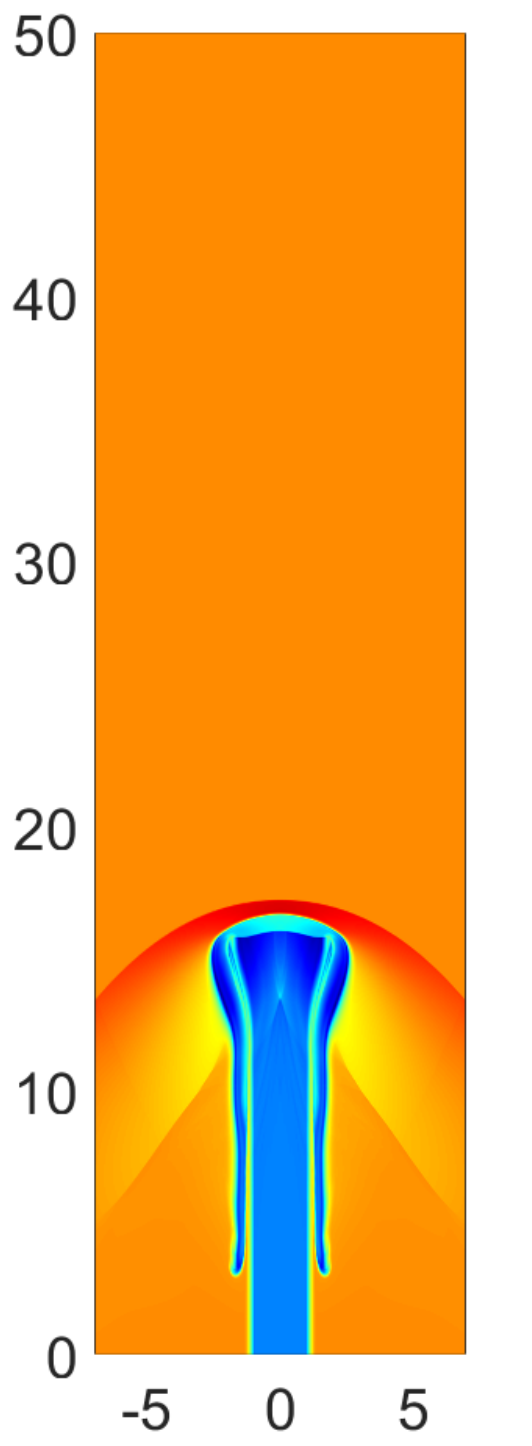}}\hspace{2mm}
\subfigure[$\mathbb{P}^2,\,t = 30$]{\includegraphics[width=0.135\textwidth]{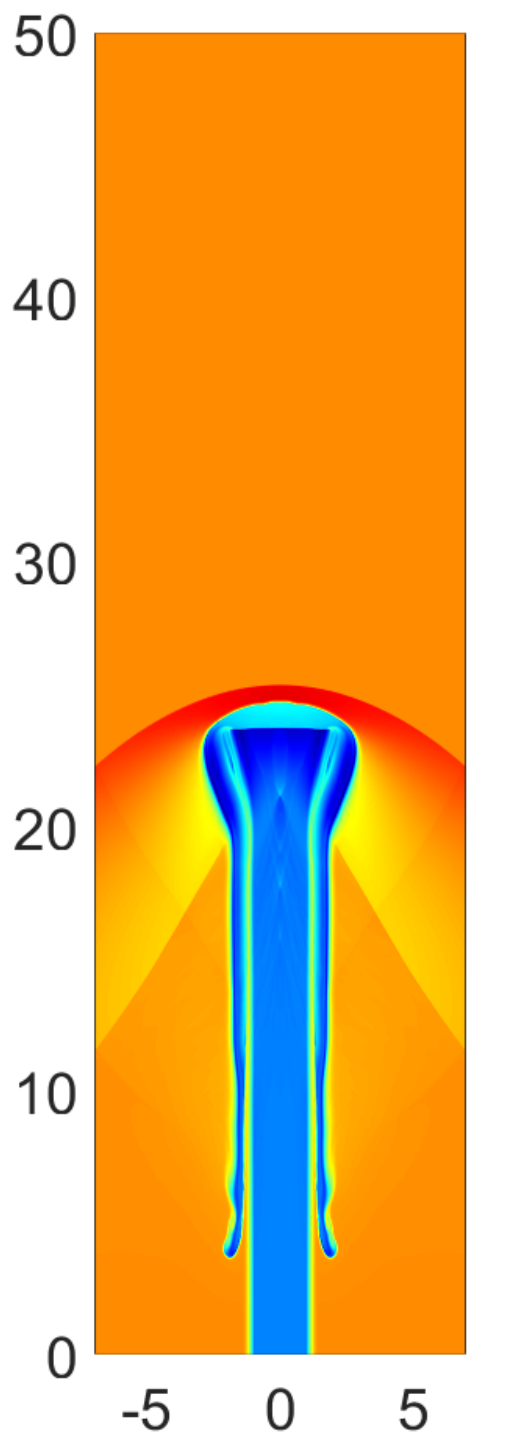}}\hspace{2mm}
\subfigure[$\mathbb{P}^2,\,t = 40$]{\includegraphics[width=0.135\textwidth]{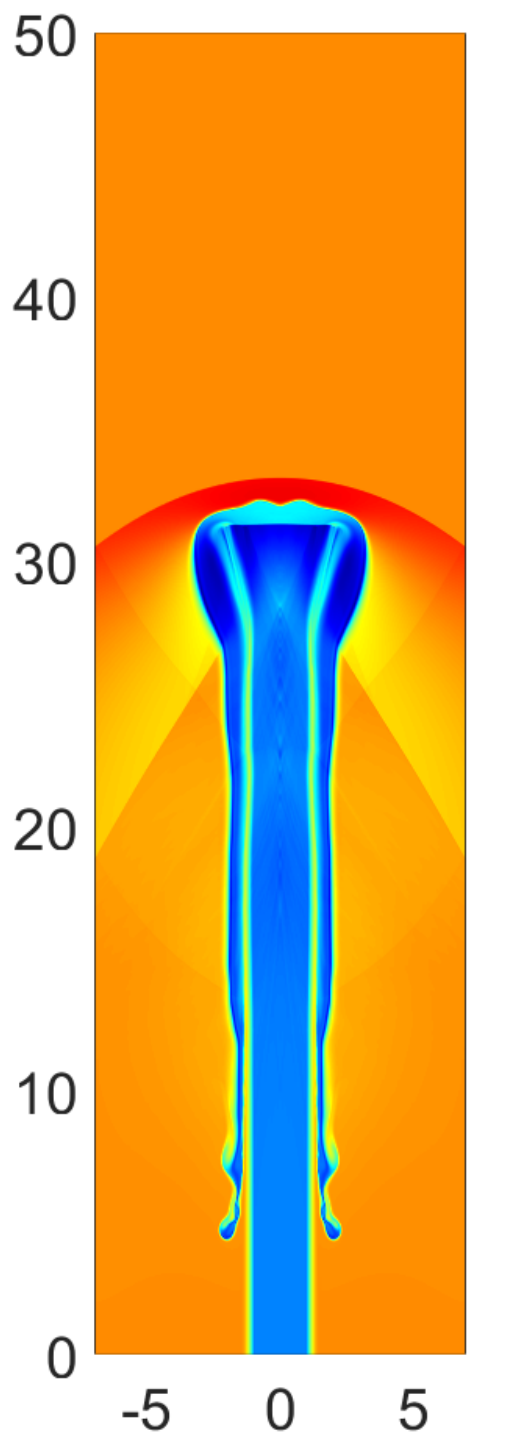}}\hspace{2mm}
\subfigure[$\mathbb{P}^2,\,t = 50$]{\includegraphics[width=0.135\textwidth]{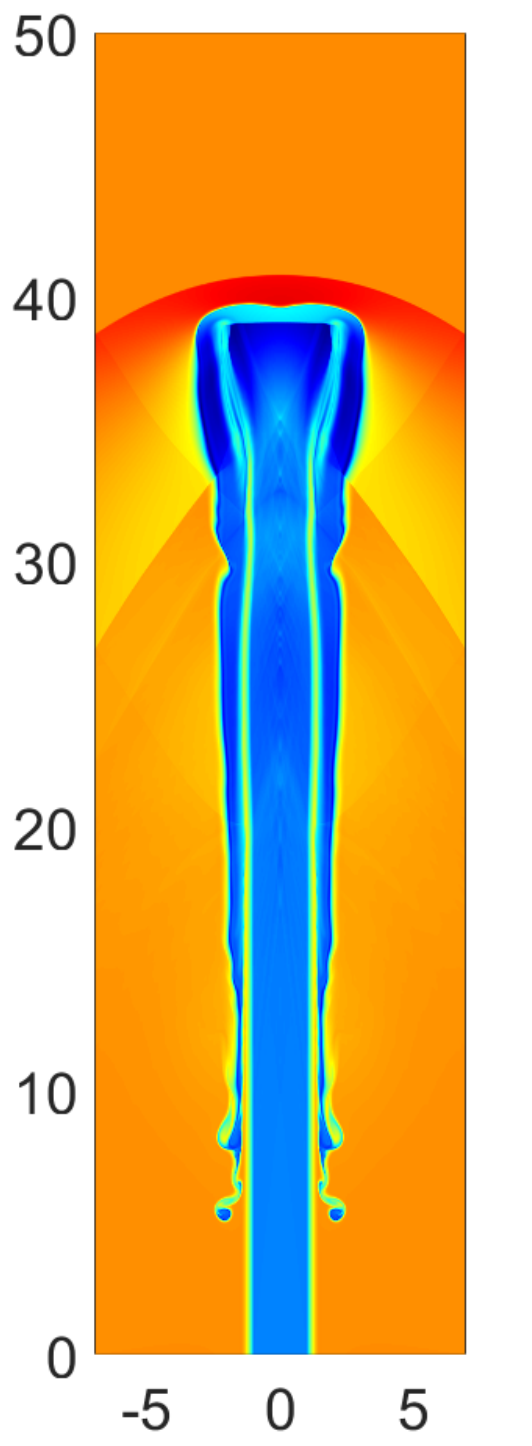}}\hspace{2mm}
\subfigure[$\mathbb{P}^2,\,t = 60$]{\includegraphics[width=0.135\textwidth]{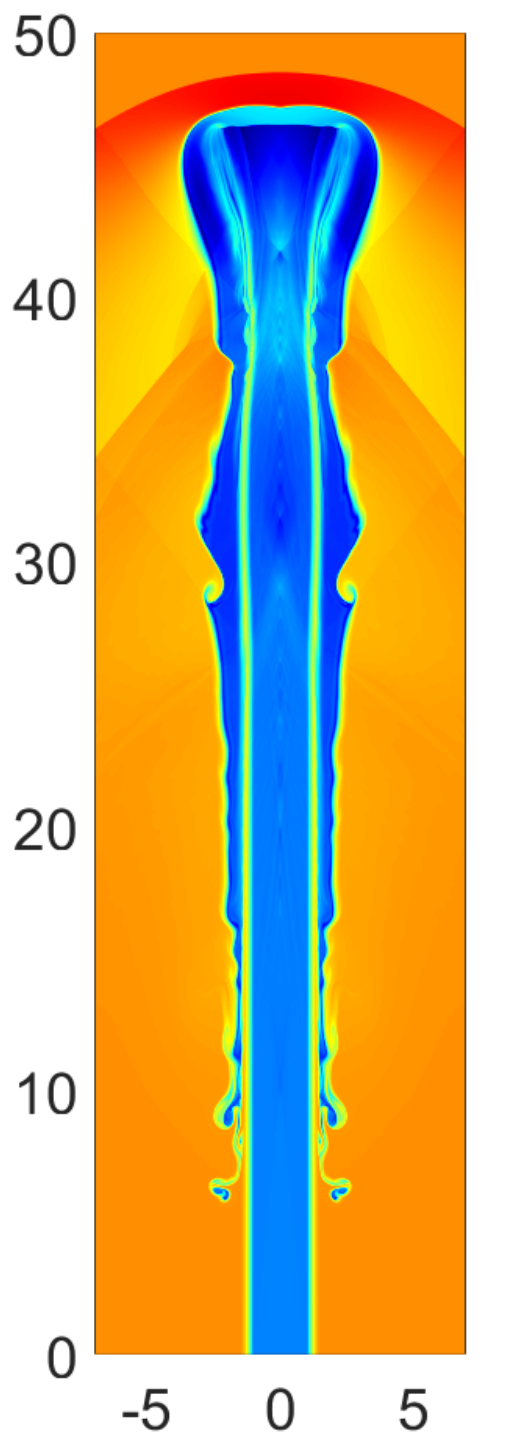}}\\
\subfigure[$\mathbb{P}^3,\,t = 10$]{\includegraphics[width=0.135\textwidth]{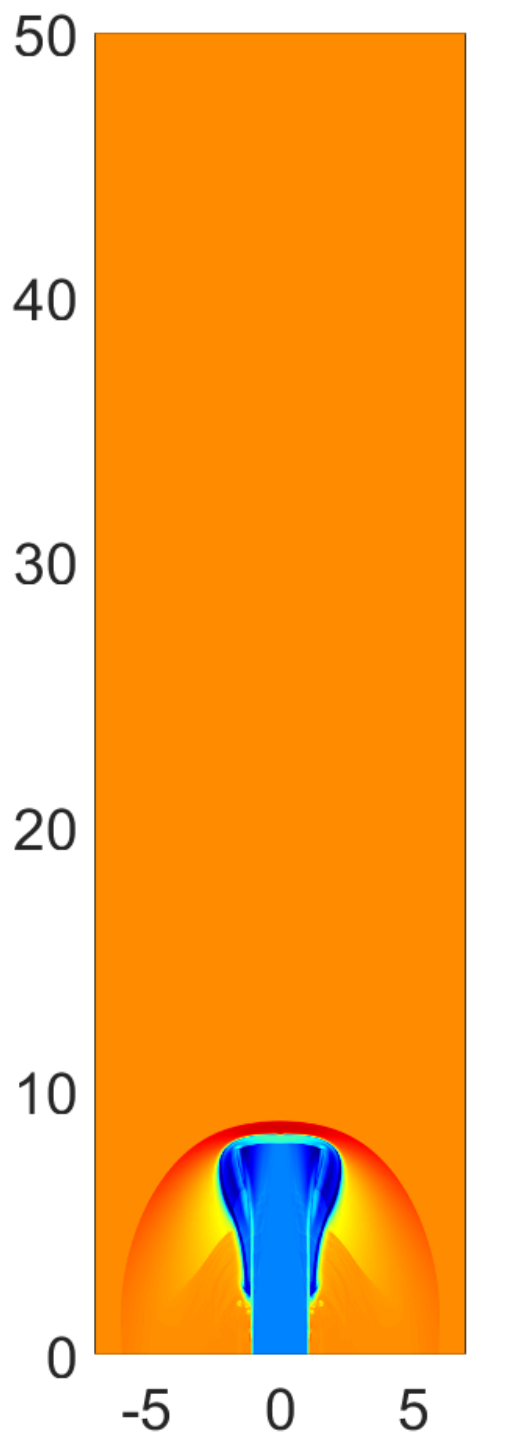}}\hspace{2mm}
\subfigure[$\mathbb{P}^3,\,t = 20$]{\includegraphics[width=0.135\textwidth]{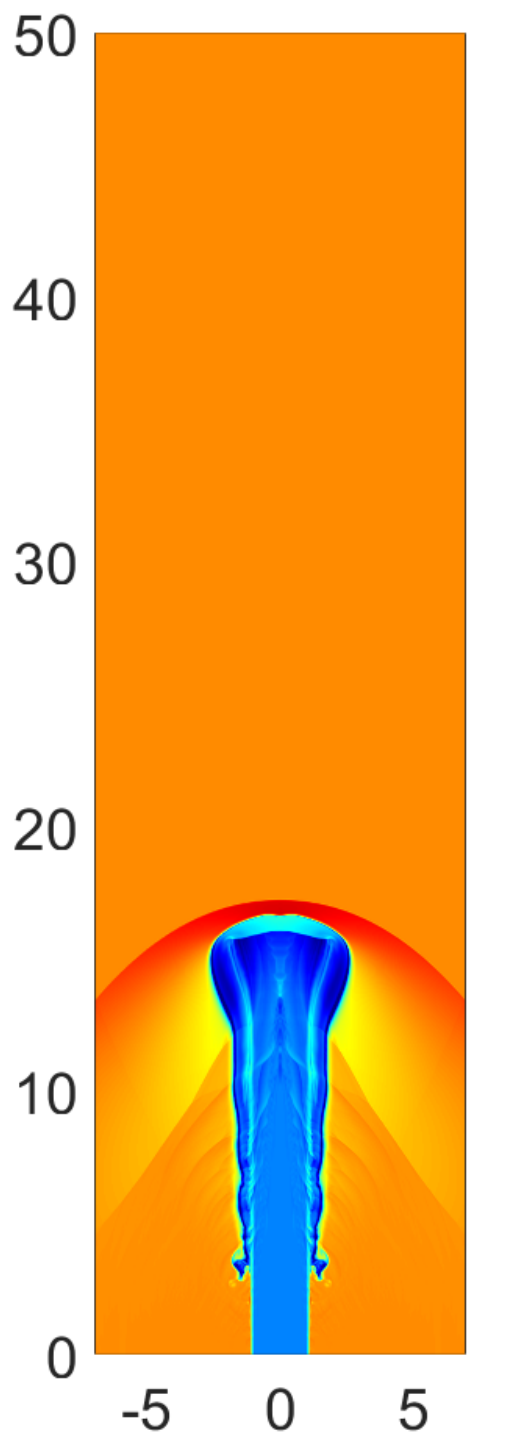}}\hspace{2mm}
\subfigure[$\mathbb{P}^3,\,t = 30$]{\includegraphics[width=0.135\textwidth]{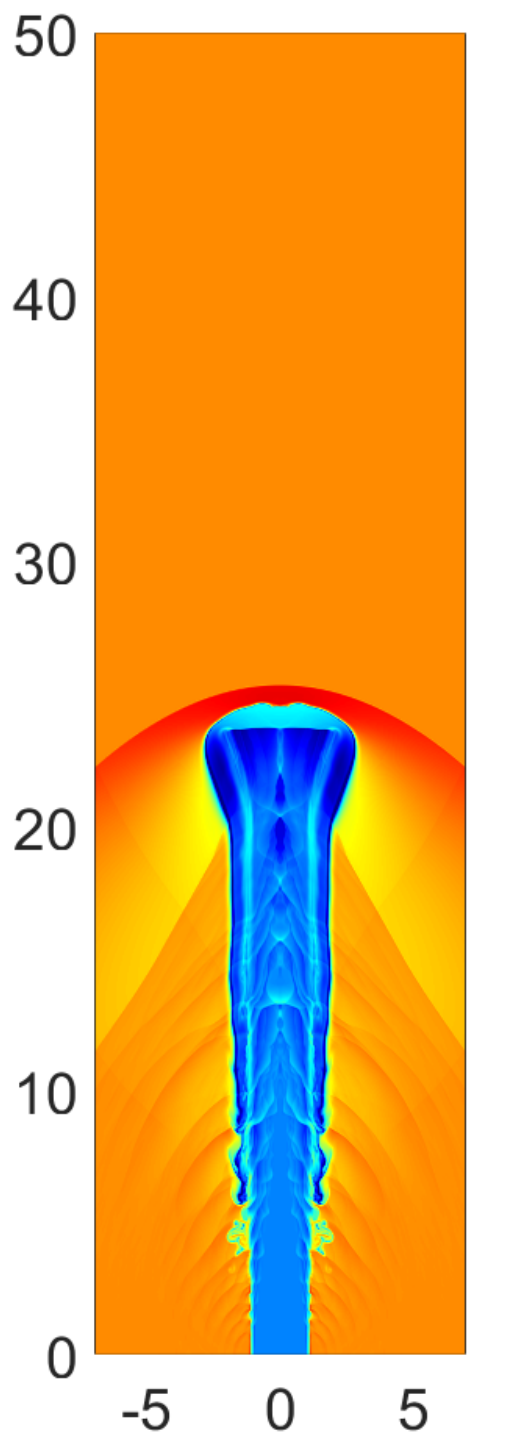}}\hspace{2mm}
\subfigure[$\mathbb{P}^3,\,t = 40$]{\includegraphics[width=0.135\textwidth]{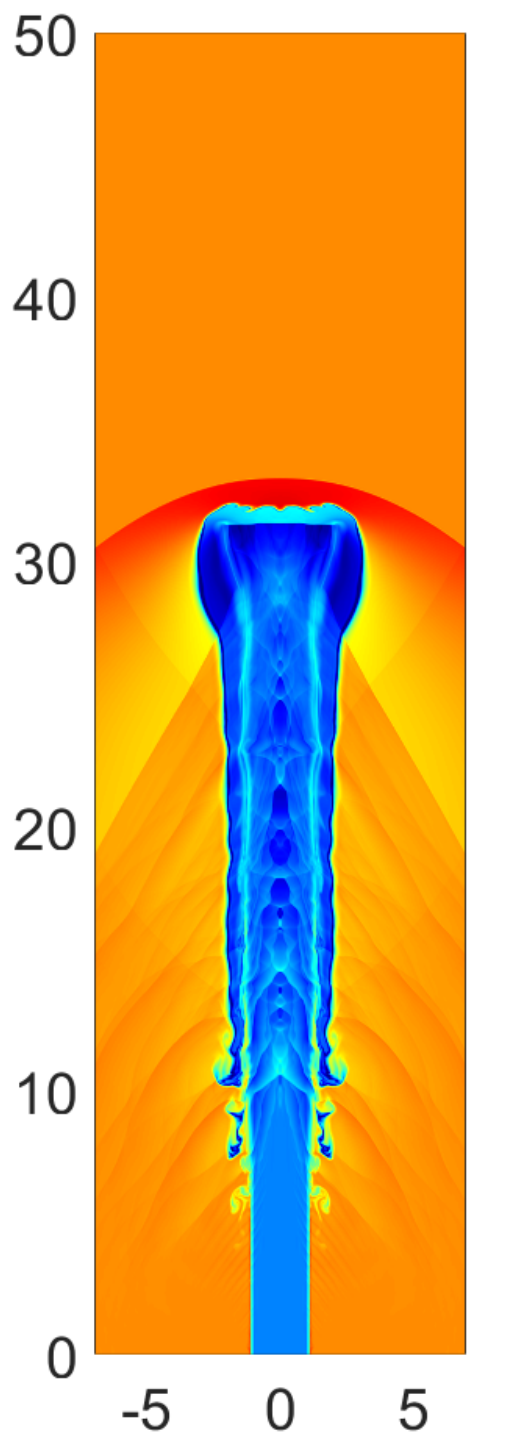}}\hspace{2mm}
\subfigure[$\mathbb{P}^3,\,t = 50$]{\includegraphics[width=0.135\textwidth]{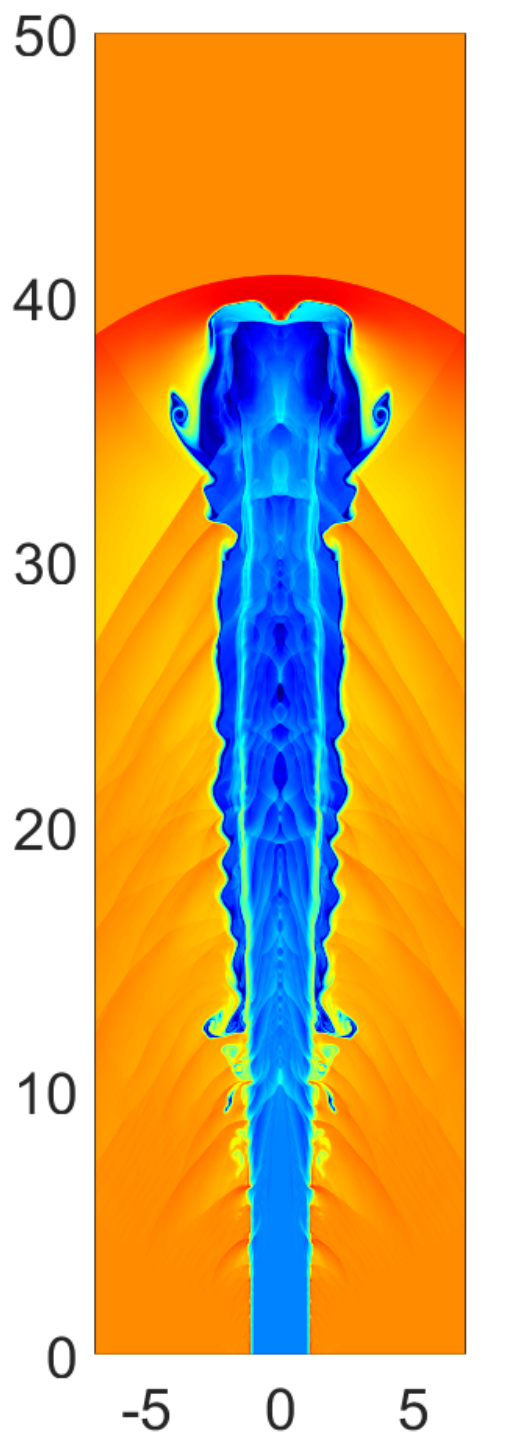}}\hspace{2mm}
\subfigure[$\mathbb{P}^3,\,t = 60$]{\includegraphics[width=0.135\textwidth]{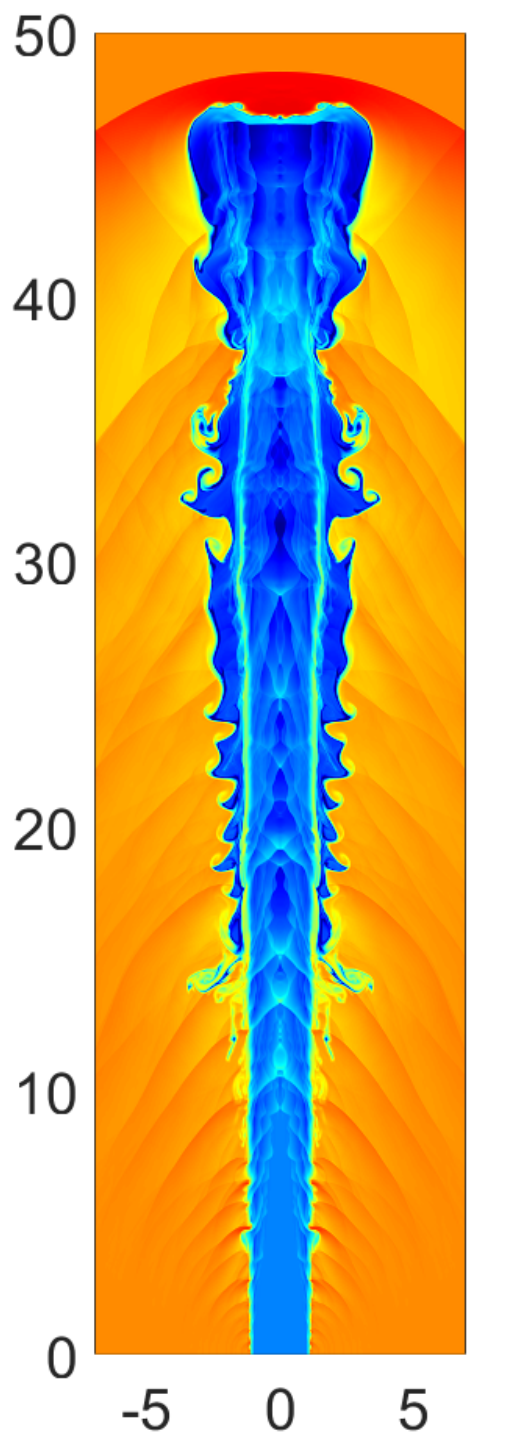}}
\caption{Example \ref{2D:Rie exam9}: Schlieren images of $\log \rho$ at different times for the $\mathbb{P}^m$-based PCP-OEDG method with $210 \times 1500$ uniform cells.}\label{fig:2D:Rie exam9}
\end{figure}

Figure \ref{fig:2D:Rie exam9} shows schlieren images of the rest-mass density logarithm $\log \rho$ in the domain $[-7, 7] \times [0, 50]$ at times $t = 10, 20, 30, 40, 50,$ and $60$, computed using the proposed PCP-OEDG schemes. The results demonstrate the accurate capture of the time evolution of a light, relativistic jet with high internal energy. The Mach shock at the jet's head is consistently resolved throughout the simulation. 
The PCP-OEDG schemes clearly resolve the beam/cocoon interface, and higher-order methods better capture the development of Kelvin--Helmholtz-type instabilities at this interface. These results confirm the robustness and shock-capturing capabilities of the PCP-OEDG schemes in handling complex relativistic jet dynamics.
\end{example}

\begin{figure}[!thbp]
\centering
\subfigure[$\mathbb{P}^1,\,t =60$]{\includegraphics[width=0.28\textwidth]{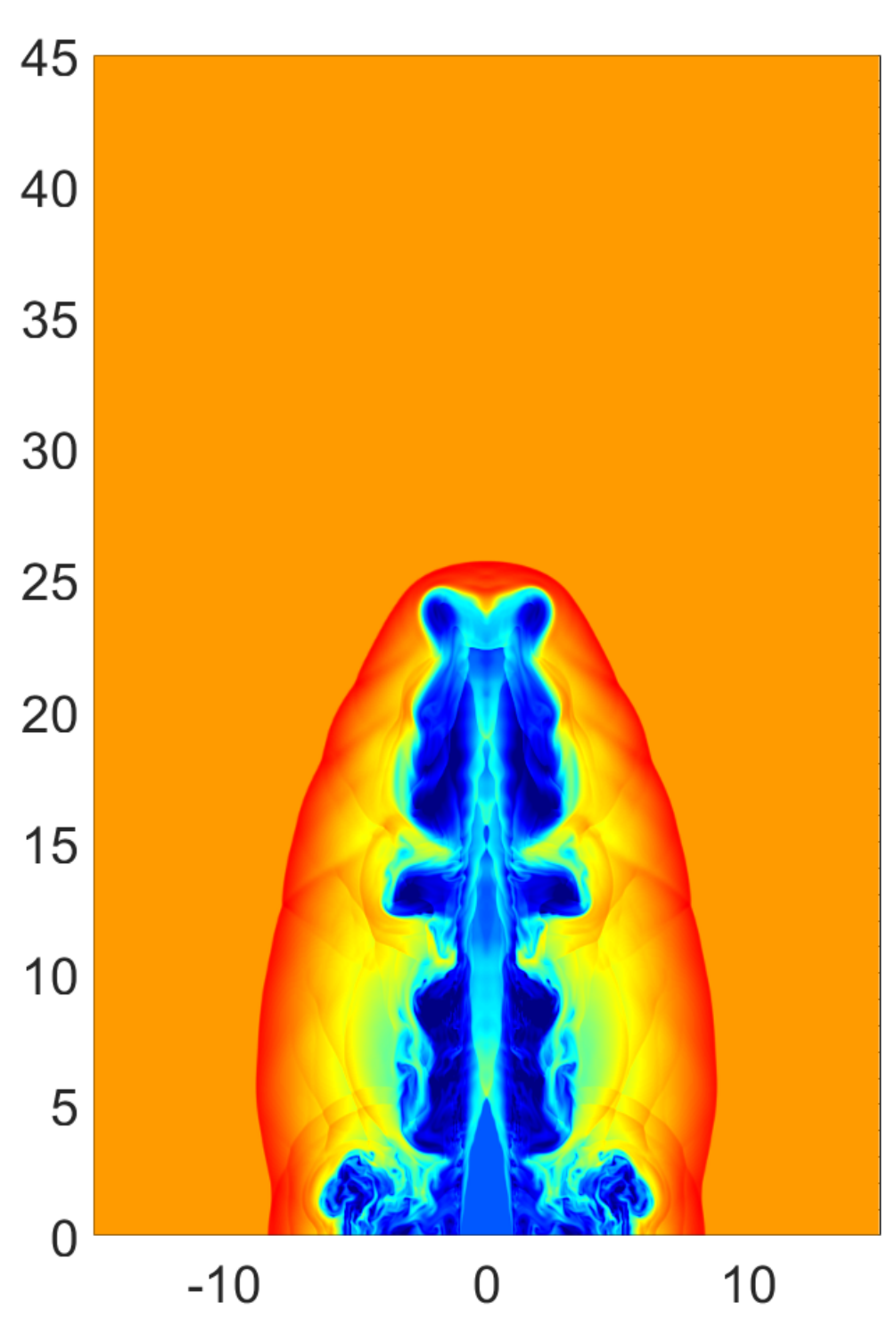}}
\subfigure[$\mathbb{P}^1,\,t =80$]{\includegraphics[width=0.28\textwidth]{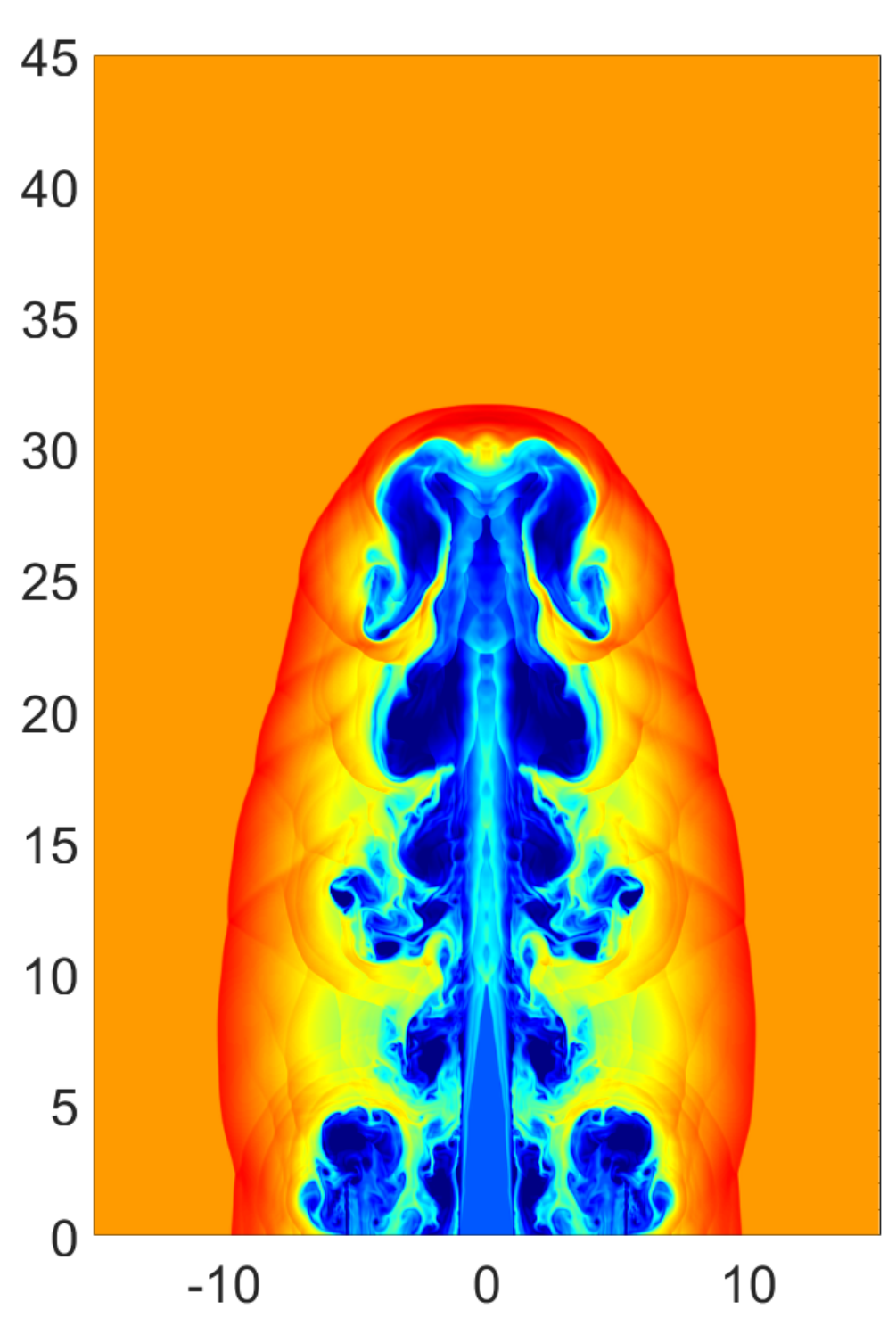}}
\subfigure[$\mathbb{P}^1,\,t =100$]{\includegraphics[width=0.28\textwidth]{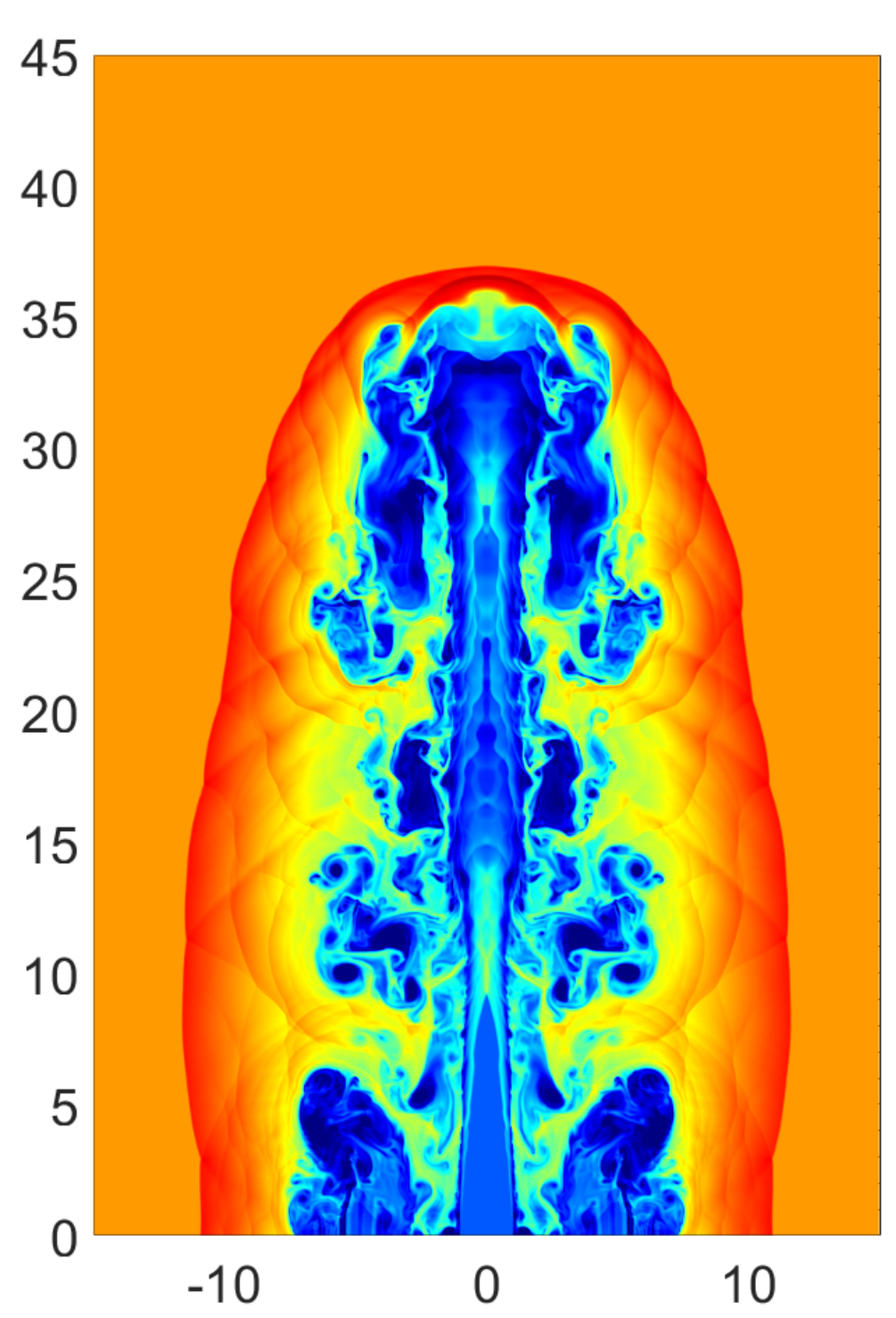}}\vspace{-4mm}\\
\subfigure[$\mathbb{P}^2,\,t =60$]{\includegraphics[width=0.28\textwidth]{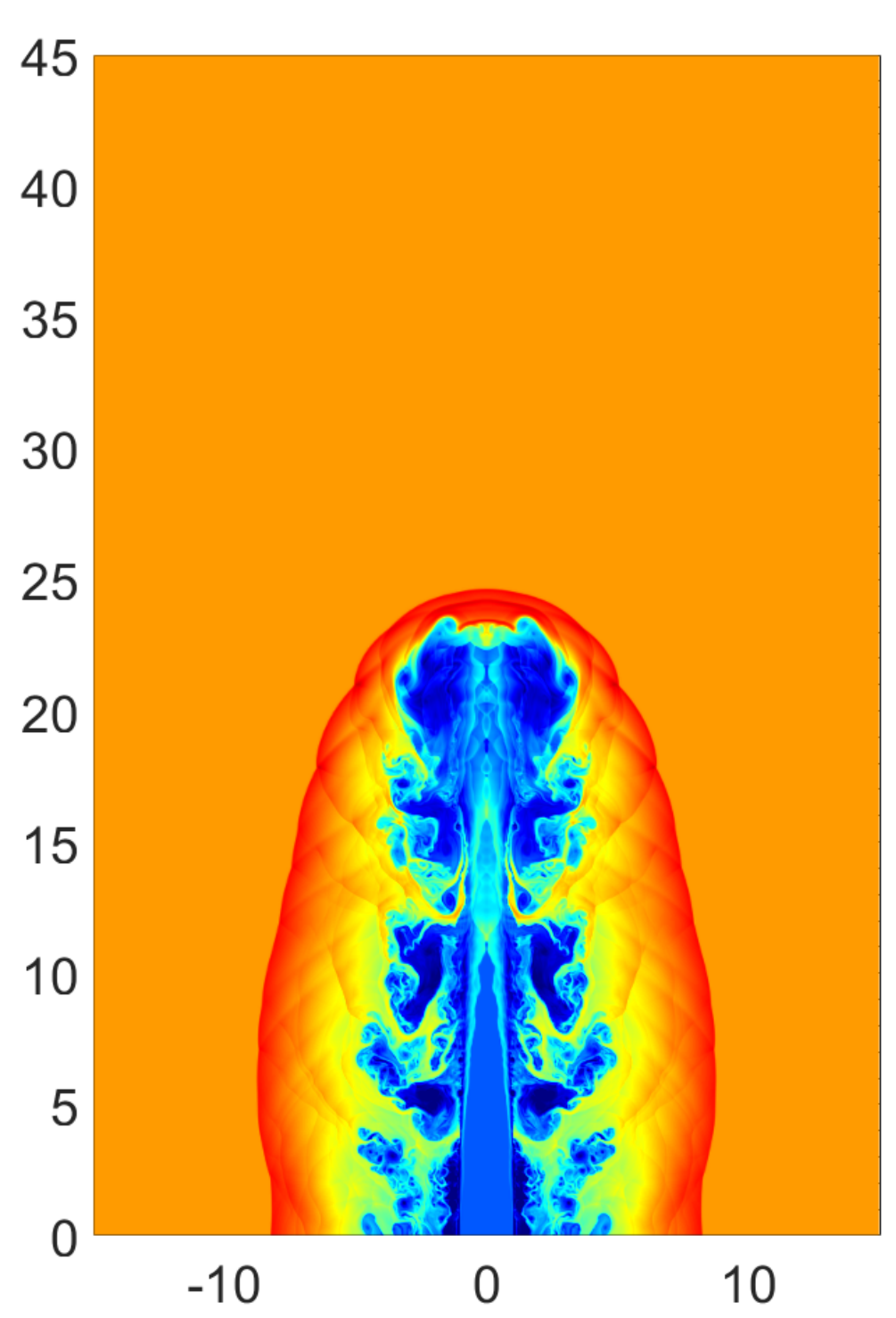}}
\subfigure[$\mathbb{P}^2,\,t =80$]{\includegraphics[width=0.28\textwidth]{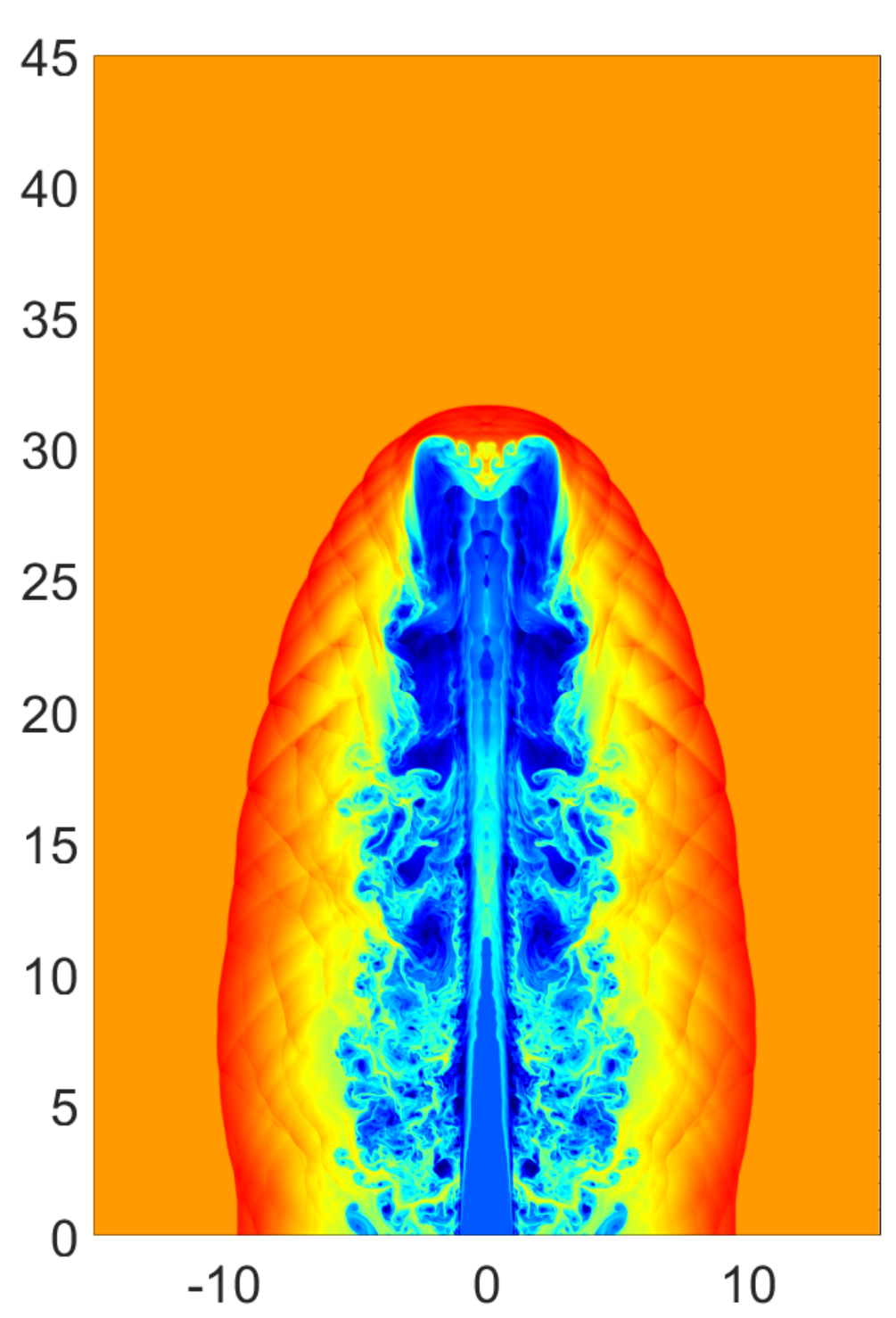}}
\subfigure[$\mathbb{P}^2,\,t =100$]{\includegraphics[width=0.28\textwidth]{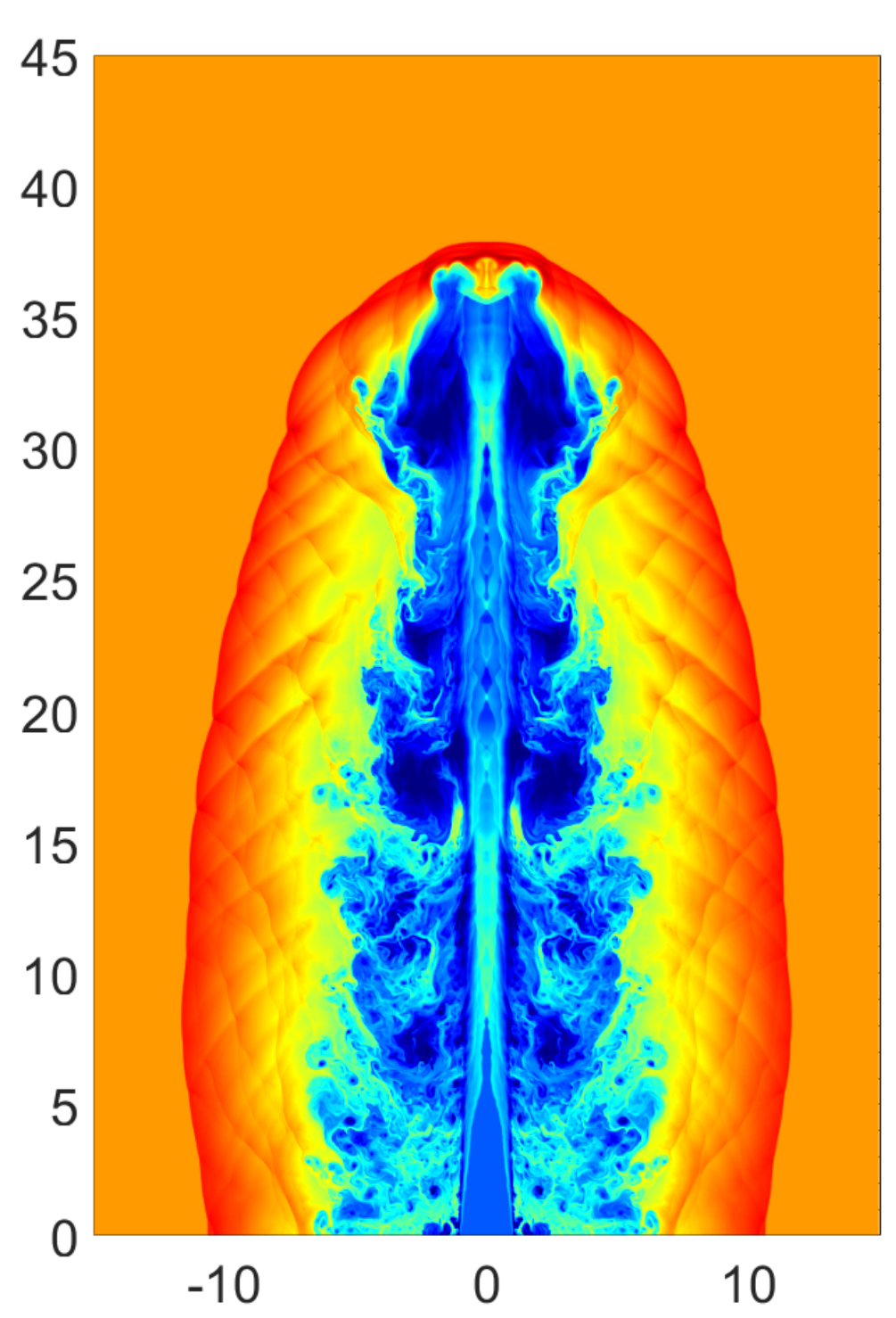}}\vspace{-4mm}\\
\subfigure[$\mathbb{P}^3,\,t =60$]{\includegraphics[width=0.28\textwidth]{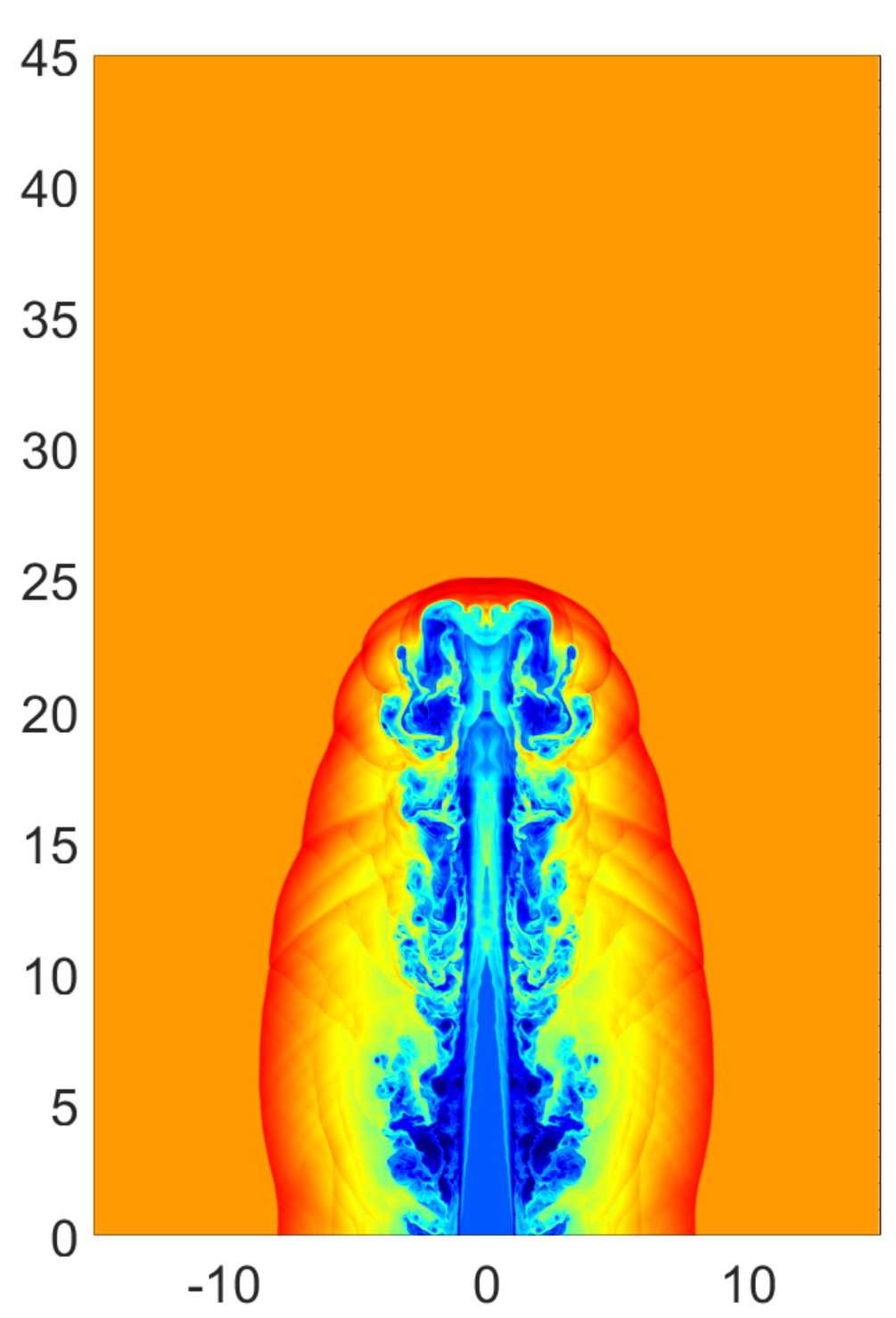}}
\subfigure[$\mathbb{P}^3,\,t =80$]{\includegraphics[width=0.28\textwidth]{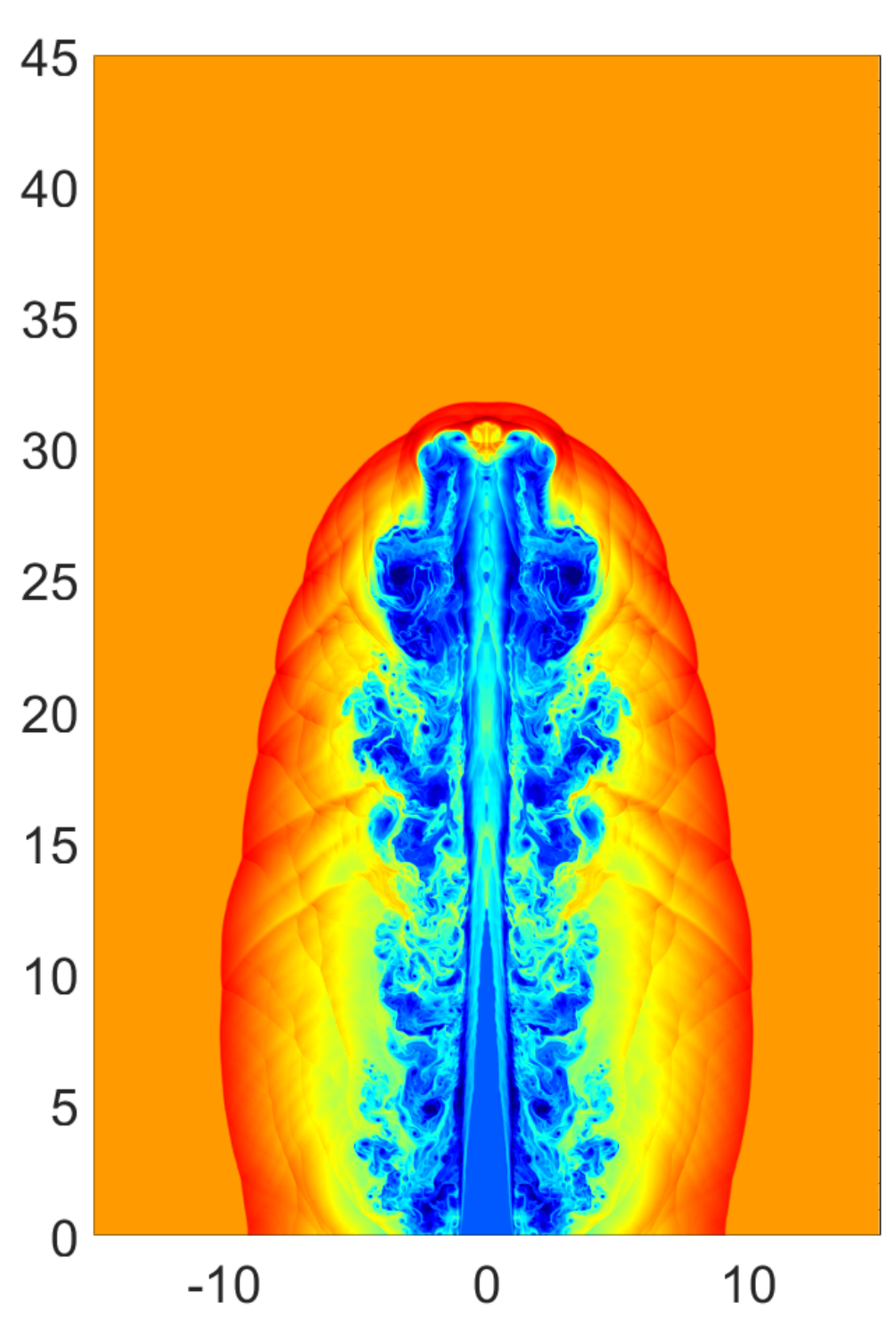}}
\subfigure[$\mathbb{P}^3,\,t =100$]{\includegraphics[width=0.28\textwidth]{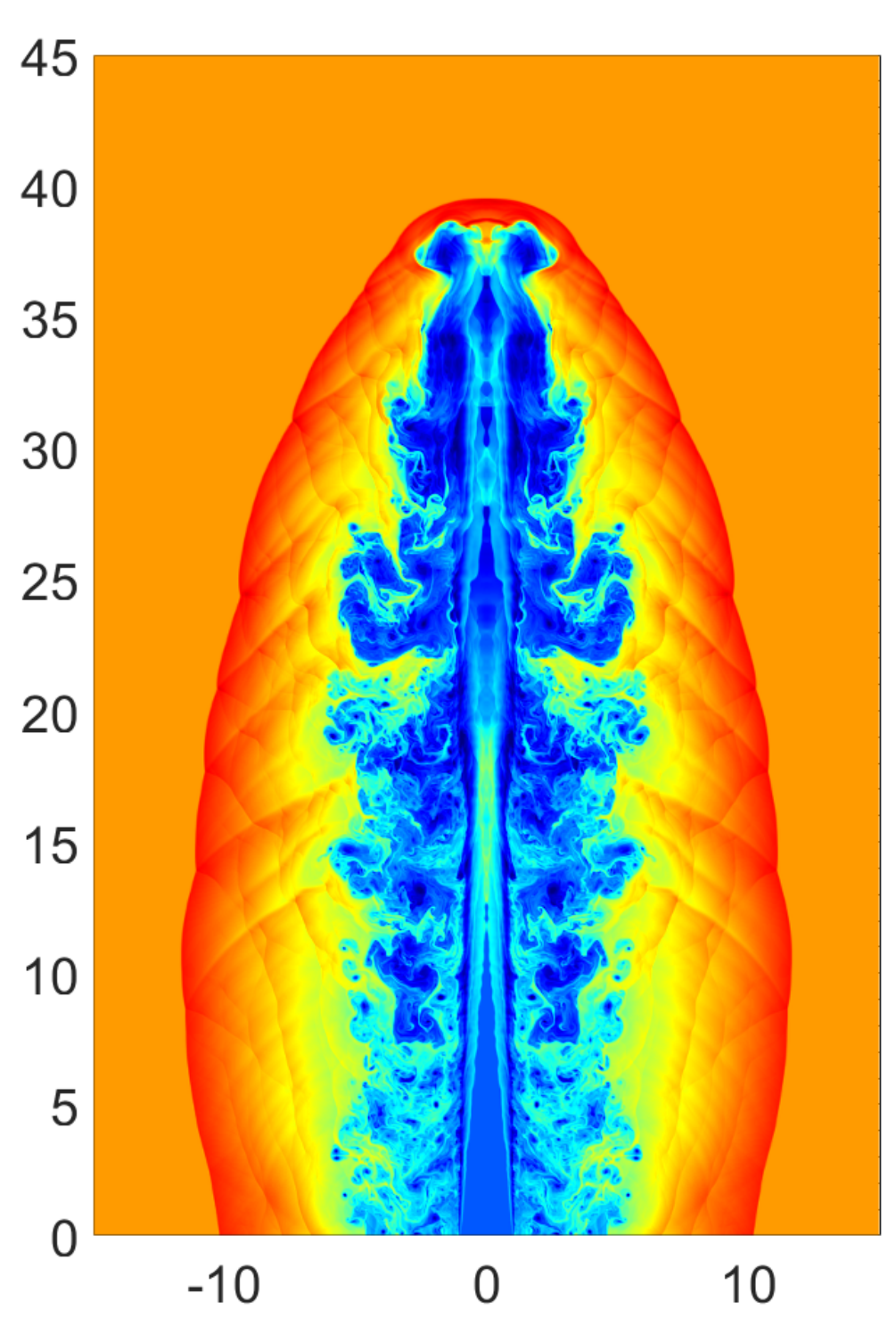}}
\vspace{-2mm}
\caption{Example \ref{2D:Rie exam8}: Schlieren images of $\log \rho$ at different times for the $\mathbb{P}^m$-based PCP-OEDG method with $360 \times 1080$ uniform cells.}\label{fig:2D:Rie exam8}
\end{figure}

\begin{example}[Axisymmetric Relativistic Jet \uppercase\expandafter{\romannumeral2}]\label{2D:Rie exam8}
The second relativistic jet simulation involves the pressure-matched, highly supersonic C2 jet model  \cite{MMFIM1997}, which has been extensively studied in the literature \cite{ZM2006, QSY2016, WT2015,CW2022}. The initial state of the jet is given by
\begin{align*}
\mathbf{Q}(r,z,0) = (1, 0, 0, 1.70305 \times 10^{-4})^{\top}.
\end{align*}
The computational domain is a 2D cylindrical box $[0,15] \times [0,45]$ in the $(r,z)$ plane, discretized into $360 \times 1080$ uniform cells. The initial relativistic jet shares the same rest-mass density and velocity as the hot A1 model but features a higher beam Mach number of $M_b = 6$, which corresponds to a relativistic Mach number of approximately $M_r = 41.95$. This model has an adiabatic index of $5/3$. Highly supersonic jets, also referred to as cold jets, are dominated by relativistic effects due to their high beam speeds, leading to significant differences between cold and hot relativistic jets.

Figure \ref{fig:2D:Rie exam8} presents schlieren images of the rest-mass density logarithm $\log \rho$ at $t = 60, 80$, and $100$ within the symmetrical domain $[-15, 15] \times [0, 45]$. As expected, a bow shock forms at the jet head, and Kelvin--Helmholtz instabilities develop, leading to the formation of large vortices that create turbulent structures. The morphology and dynamics of the relativistic jets depicted in Figure \ref{fig:2D:Rie exam8} align closely with those obtained using adaptive mesh refinement RHD codes in \cite{ZM2006} and high-order PCP finite difference/volume WENO methods in \cite{WT2015,CW2022}.
\end{example}

\subsection{Black Hole Accretion}

In this subsection, we apply the proposed PCP-OEDG schemes to simulate the evolution of matter accreting supersonically onto a rotating black hole in the general relativistic regime, as described in \cite{FIP1999}. We follow the setups in \cite{FIP1999}, using the Kerr metric in Kerr–Schild coordinates, which avoids coordinate singularities at the black hole horizon. The PCP-OEDG schemes demonstrate robust shock-capturing capabilities for general relativistic fluid flows. In particular, we will use these schemes to explore how the morphology of the accretion flow depends on various parameters, including the black hole’s angular momentum, asymptotic flow velocity, Mach number, and the adiabatic index of the fluid. 
These simulations will highlight the efficacy of PCP-OEDG methods in accurately capturing the complex dynamics of black hole accretion flows, which include general relativistic shock waves and other challenging phenomena characteristic of extreme astrophysical environments.

Table \ref{tb:2D exam10} summarizes the specific parameter values used in our simulations. Here, $\Gamma$ represents the adiabatic index, $v_{\infty}$ is the asymptotic flow velocity, and $r_{\min}$ and $r_{\max}$ are the minimum and maximum radial values of the computational domain, respectively. The asymptotic Mach number, $M_{\infty}$, is determined by $v_{\infty}$ and the sound speed $c_s$ using the relation $M_{\infty} = v_{\infty}/c_s$. The black hole's angular momentum parameter $a$ is positive, meaning the black hole rotates counter-clockwise. The location of the event horizon is given by $r^{+} = 1 + \sqrt{1 - a^2}$. All distance measurements are expressed in terms of the black hole’s mass.

The initial conditions for the accretion flow in Kerr--Schild coordinates $(r, \tilde{\phi})$ are defined as follows:
\begin{align*}
\begin{cases}
\rho(r, \tilde{\phi}) &= 1.0,\\
v^r(r, \tilde{\phi}) &= f_1(r)v_{\infty}\cos\tilde{\phi} + f_2(r)v_{\infty}\sin\tilde{\phi},\\
v^{\tilde{\phi}}(r, \tilde{\phi}) &= f_3(r)v_{\infty}\cos\tilde{\phi} - f_4(r)v_{\infty}\sin\tilde{\phi},\\
p(r, \tilde{\phi}) &= \frac{\rho c_s^2(\Gamma-1)}{\Gamma(\Gamma - 1 - c_s^2)},
\end{cases}
\end{align*}
where the functions $f_1$, $f_2$, $f_3$, and $f_4$ are given by
\begin{align*}
f_1(r) &= \frac{1}{\sqrt{g_{rr}}}, \qquad\qquad f_3(r) = -\frac{2g_{r\tilde{\phi}}}{\sqrt{g_{rr}}g_{\tilde{\phi}\tilde{\phi}}},\\
f_4(r) &= \frac{f_1g_{rr} + f_3g_{r\tilde{\phi}}}{\sqrt{(g_{rr}g_{\tilde{\phi}\tilde{\phi}} - g_{r\tilde{\phi}}^2)(f_1^2g_{rr} + f_3^2g_{\tilde{\phi}\tilde{\phi}}+2f_1f_3g_{r\tilde{\phi}})}},\\
f_2(r) &= \frac{f_2f_4g_{\tilde{\phi}\tilde{\phi}} + f_1f_3g_{r\tilde{\phi}}}{f_1g_{rr} + f_3g_{r\tilde{\phi}}}.
\end{align*}
Here, $\{g_{ij}: i = r, \tilde{\phi}; j = r, \tilde{\phi}\}$ are the components of the Kerr metric tensor as defined in equation \eqref{metric}, with $v_{\infty}$ and $\Gamma$ provided in Table \ref{tb:2D exam10}.

For boundary conditions, outflow conditions are imposed at the inner boundary, while asymptotic initial values are applied at the outer boundary. The simulations are performed on a uniform grid with $600$ radial zones and $250$ angular zones, with the simulations running up to time $t = 500$ for all cases. The results are presented in Cartesian coordinates, where $x = r \cos \phi$ and $y = r \sin \phi$, transformed from Boyer--Lindquist coordinates.

\begin{table}[!thb]
\center
\caption{Parameters in different test cases for Kerr black hole accretion.}
\begin{tabular}[c]{c|c|c|c|c|c|c|c|c|c|c}
\toprule
Cases & Case 1 & Case 2 &Case 3 &Case 4 & Case 5  & Case 6  & Case 7  & Case 8 & Case 9 & Case 10\\
\hline
$\Gamma$     &5/3            &5/3            &5/3            &$\mathbf{5/3}$ &$\mathbf{4/3}$ &$\mathbf{2.0}$ & 5/3           & 5/3            & 5/3           & 5/3\\
$M_{\infty}$ &5              &5              &5              &$\mathbf{5}$   &5              &5              &5              &5               &$\mathbf{20}$  &$\mathbf{50}$\\
$v_{\infty}$ &0.5            &0.5            &0.5            &$\mathbf{0.5}$ &0.5            &0.5            &$\mathbf{0.9}$ &$\mathbf{0.99}$ & 0.5           & 0.5\\
$r_{\min}$    &1.8            &1.8            &1.4            &1.0            &1.0            &1.0            &1.0            &1.0             &	1.0           &	1.0\\
$r_{\max}$    &50.9           &50.9           &50.9           &50.9           &50.9           &50.9           &50.9           &50.9            &	50.9          &	50.9\\
a            &$\mathbf{0}$   &$\mathbf{0.5}$ &$\mathbf{0.9}$ &$\mathbf{0.99}$&0.99           &0.99           &0.99           &0.99            &	0.99          &	0.99\\
\bottomrule
\end{tabular}\label{tb:2D exam10}
\end{table}

\begin{example}[Effects of Angular Momentum $a$ for Kerr Black Hole Accretion]\label{2D:Rie exam10a}

To illustrate the dependence of accretion patterns on the black hole's angular momentum, we fix three parameters: $\Gamma$, $M_{\infty}$, and $v_{\infty}$, as shown in the first four cases listed in Table \ref{tb:2D exam10}. The angular momentum parameter $a$ varies as $0,\, 0.5,\, 0.9,\, 0.99$ for cases 1 to 4, respectively. For $a = 0$, the black hole is non-rotating, corresponding to the Schwarzschild case where the spin effect is absent.

\begin{figure}[!thb]
\centering
\subfigure[Case 1, $a = 0.0$]{\includegraphics[width=0.255\textwidth]{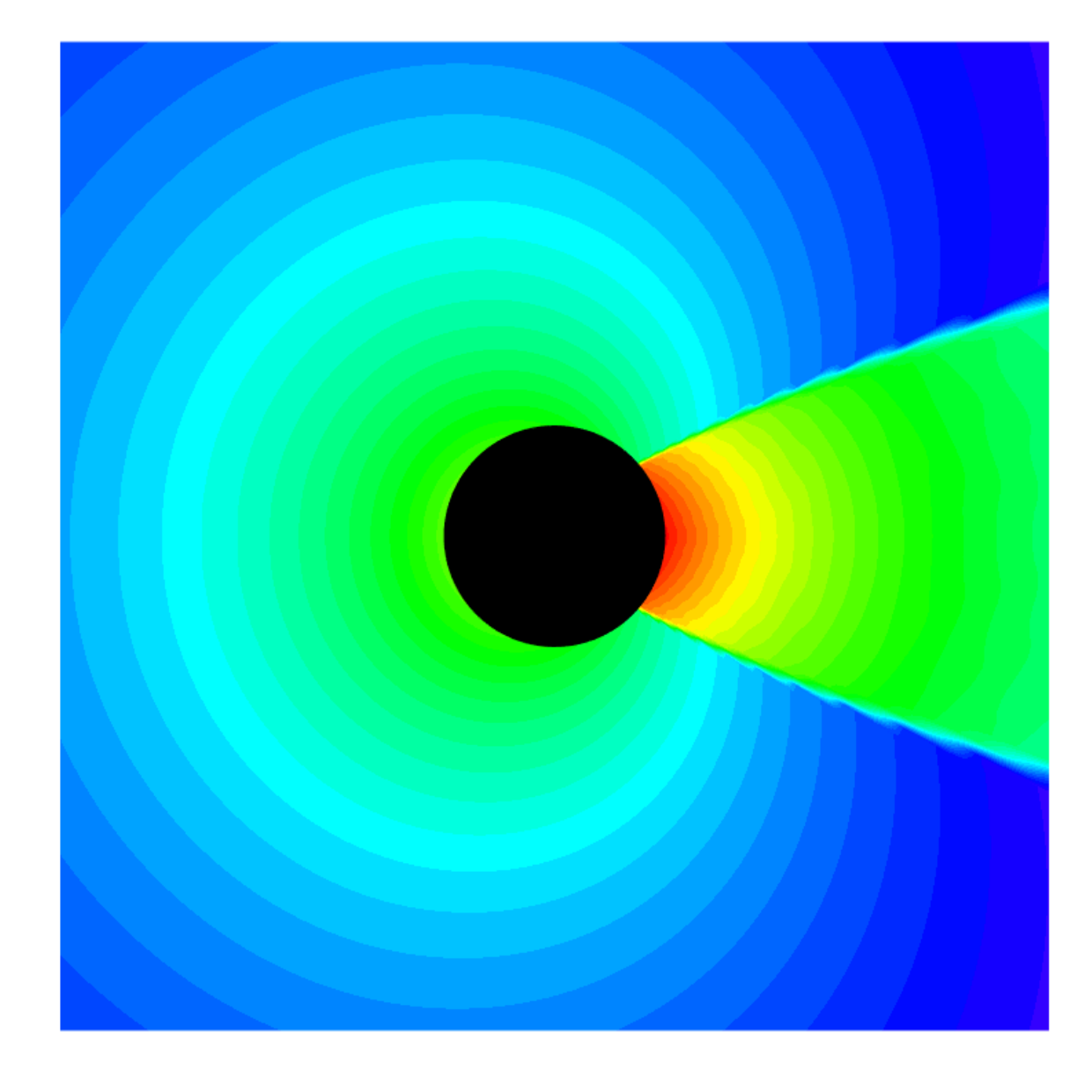}}\hspace{-2mm}
\subfigure[Case 2, $a = 0.5$]{\includegraphics[width=0.255\textwidth]{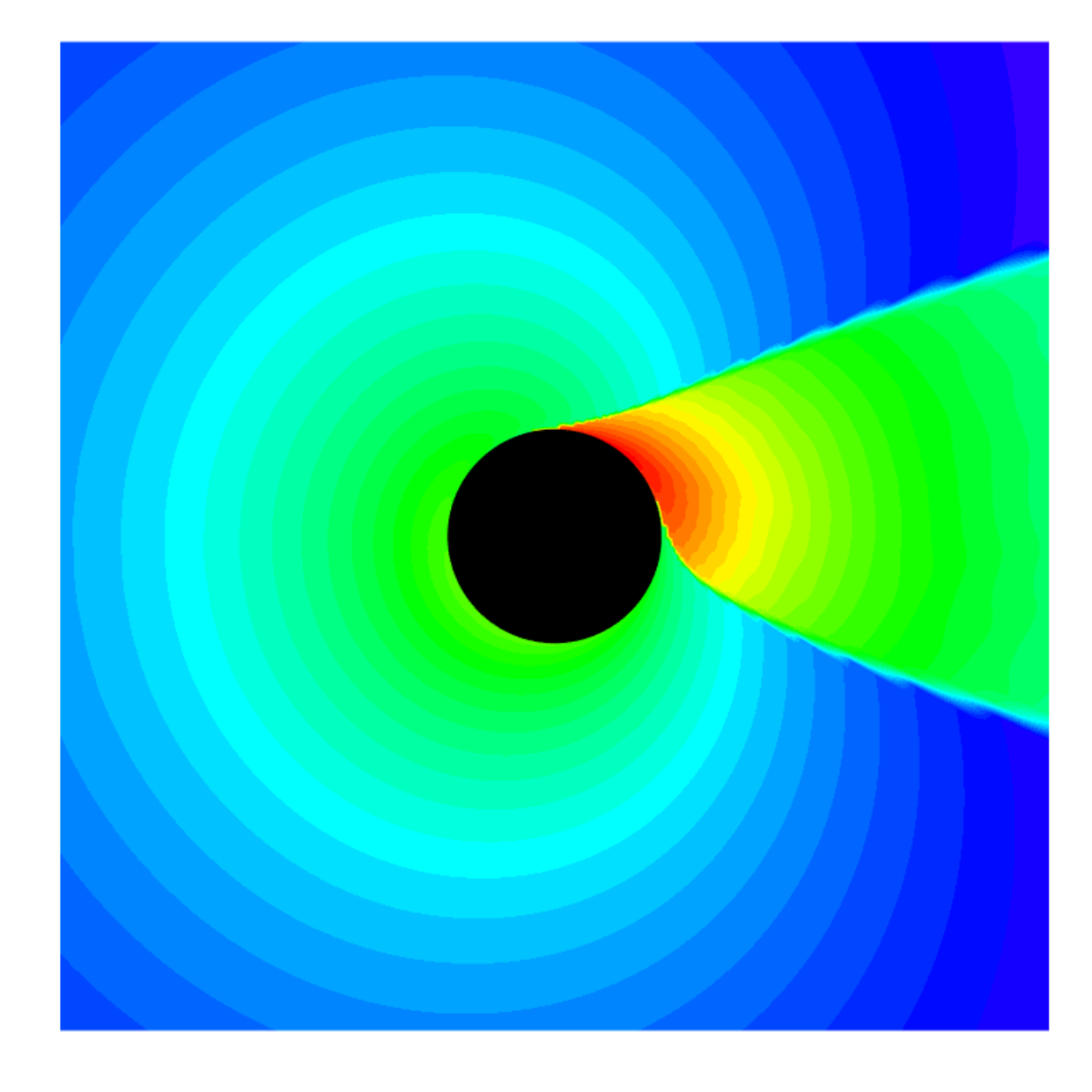}}\hspace{-2mm}
\subfigure[Case 3, $a = 0.9$]{\includegraphics[width=0.255\textwidth]{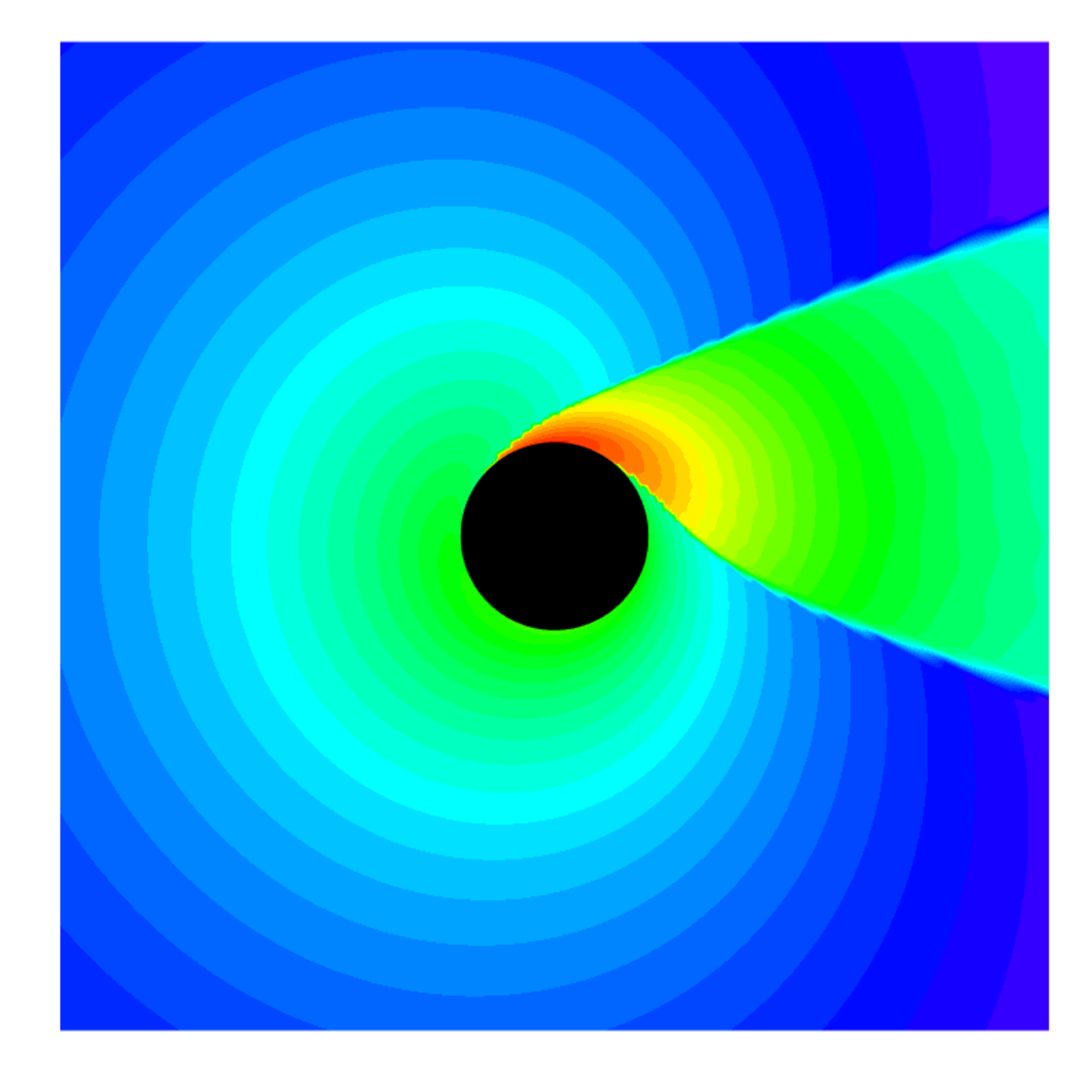}}\hspace{-2mm}
\subfigure[Case 4, $a = 0.99$]{\includegraphics[width=0.255\textwidth]{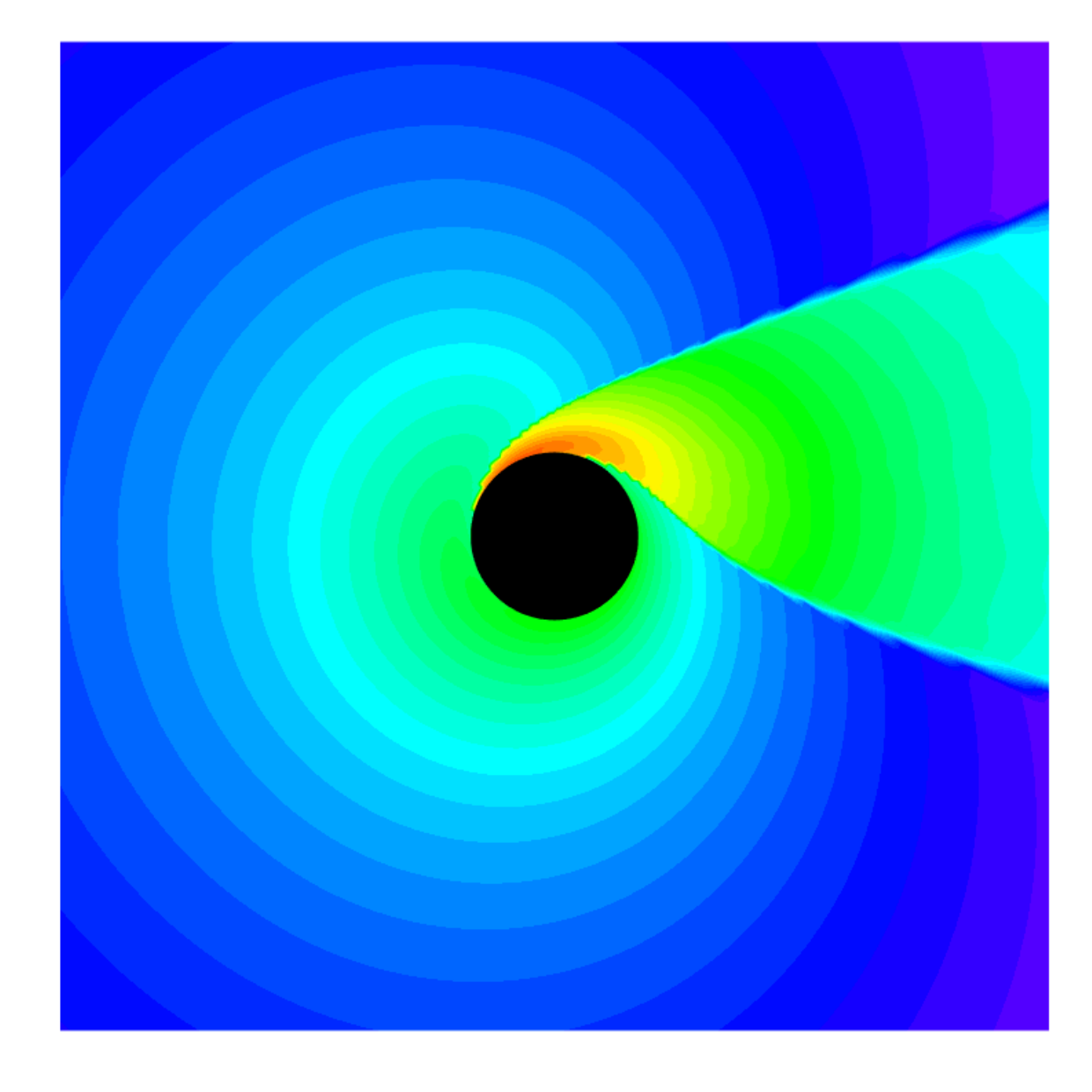}}
\caption{Example \ref{2D:Rie exam10a}: Relativistic accretion onto Kerr black hole with different angular momentums.}\label{fig:Bh p1-1 case1-4}
\end{figure}

\begin{figure}[!thb]
\centering
\subfigure[Case 1, $a = 0.0$]{\includegraphics[width=0.255\textwidth]{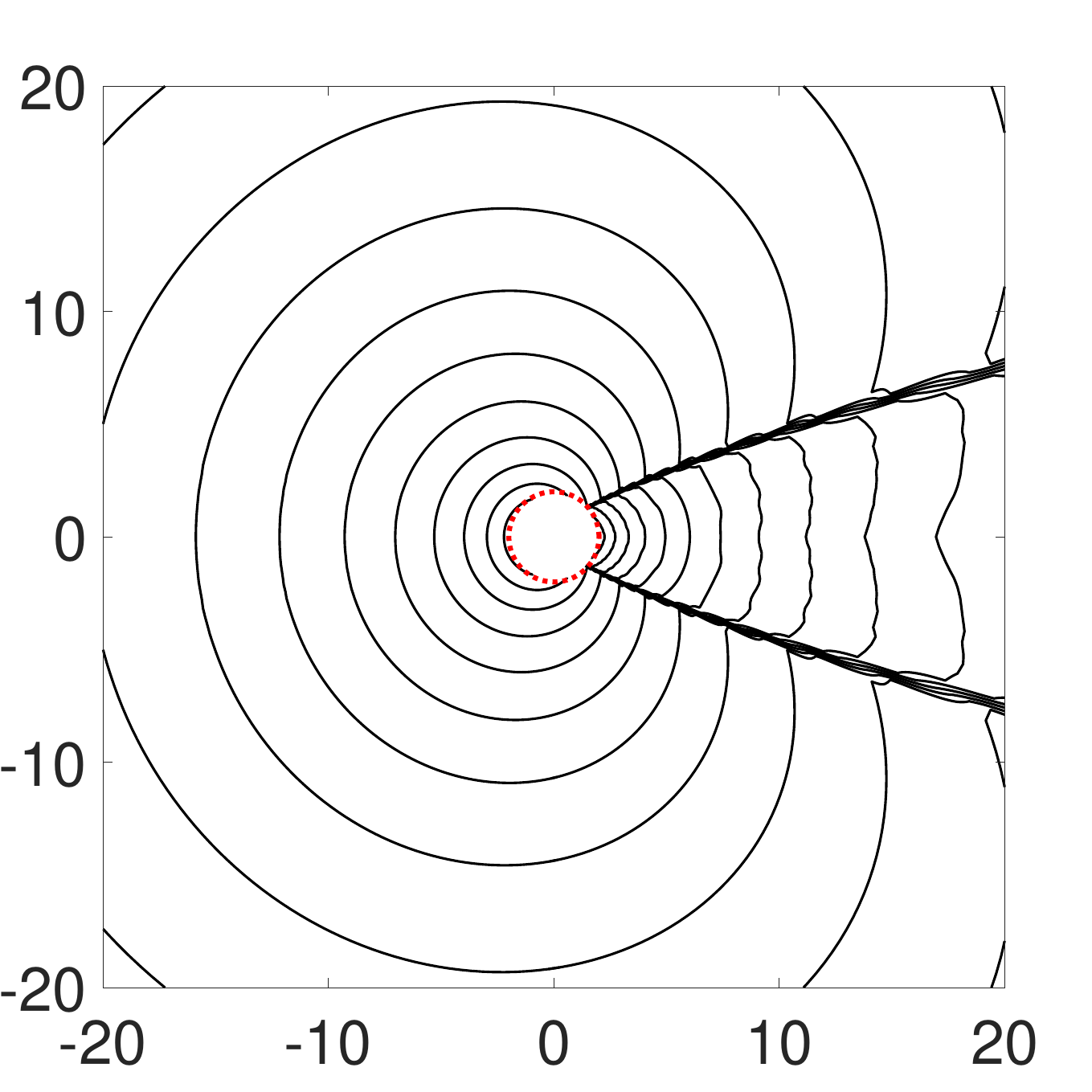}}\hspace{-2mm}
\subfigure[Case 2, $a = 0.5$]{\includegraphics[width=0.255\textwidth]{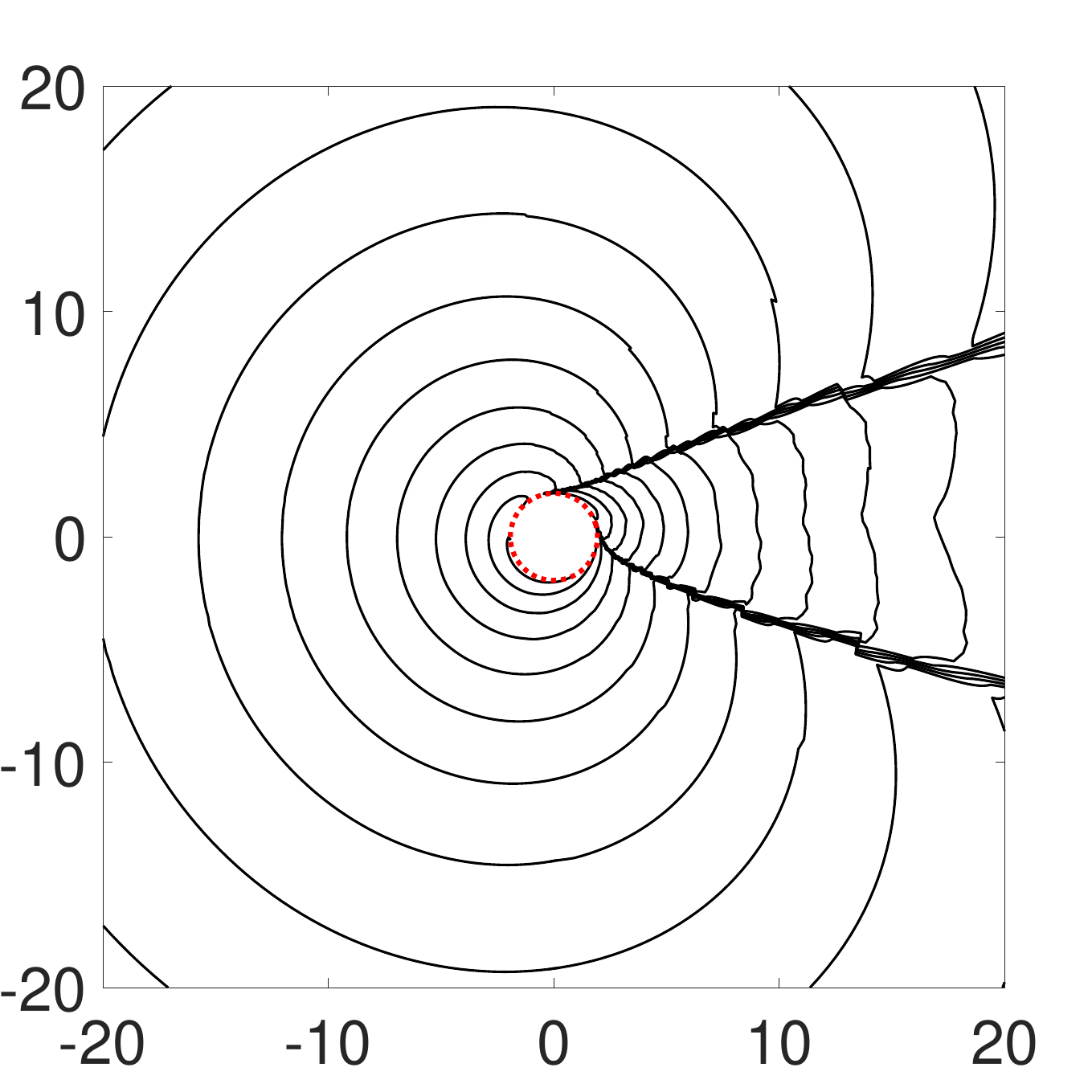}}\hspace{-2mm}
\subfigure[Case 3, $a = 0.9$]{\includegraphics[width=0.255\textwidth]{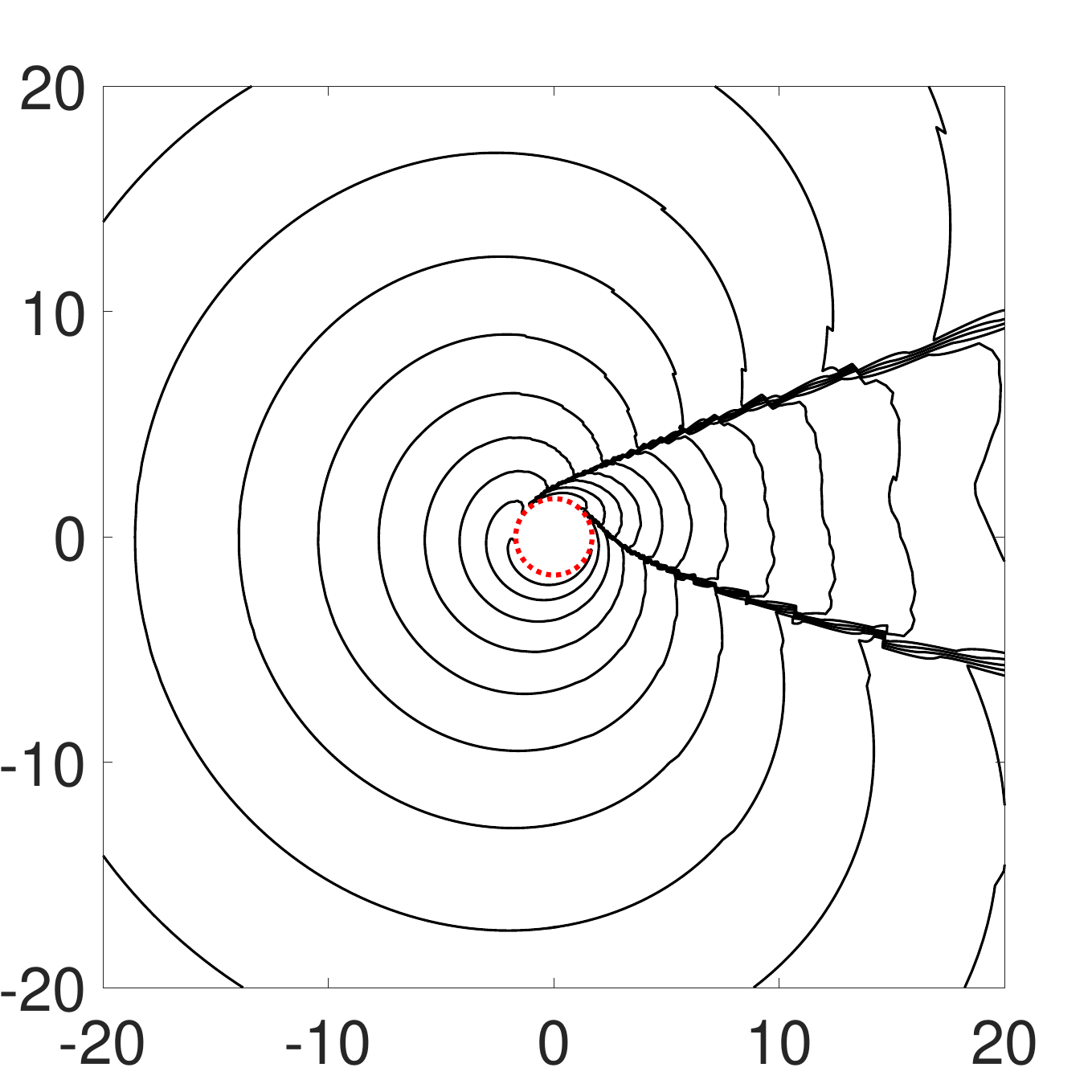}}\hspace{-2mm}
\subfigure[Case 4, $a = 0.99$]{\includegraphics[width=0.255\textwidth]{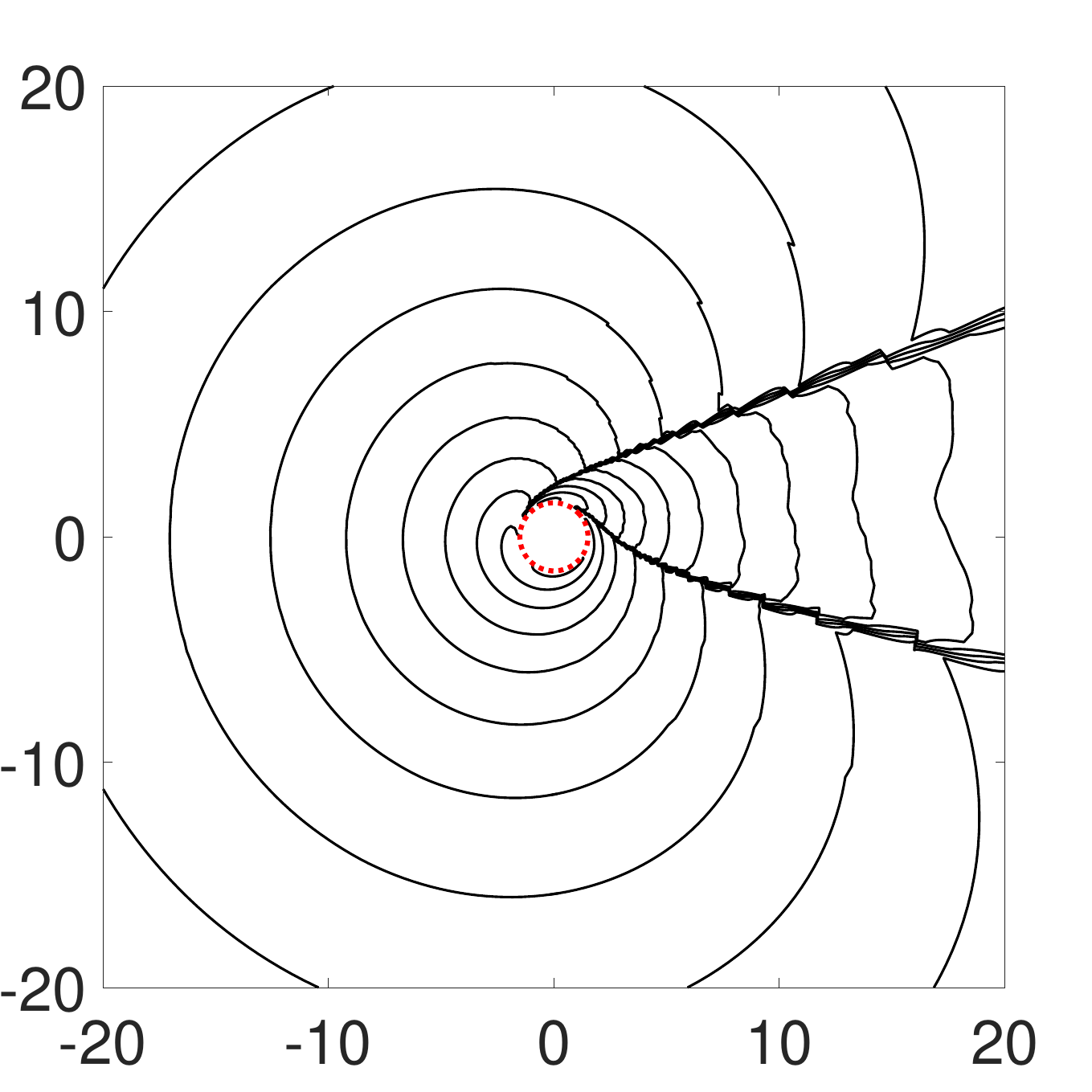}}
\caption{Example \ref{2D:Rie exam10a}: Contours of $\log \rho$ at different angular momentums for the second-order PCP-OEDG method in domain $[-20,20]^2$.}\label{fig:Bh p1-2 case1-4}
\end{figure}

\begin{figure}[!thb]
\centering
\subfigure[Case 1, $a = 0.0$]{\includegraphics[width=0.255\textwidth]{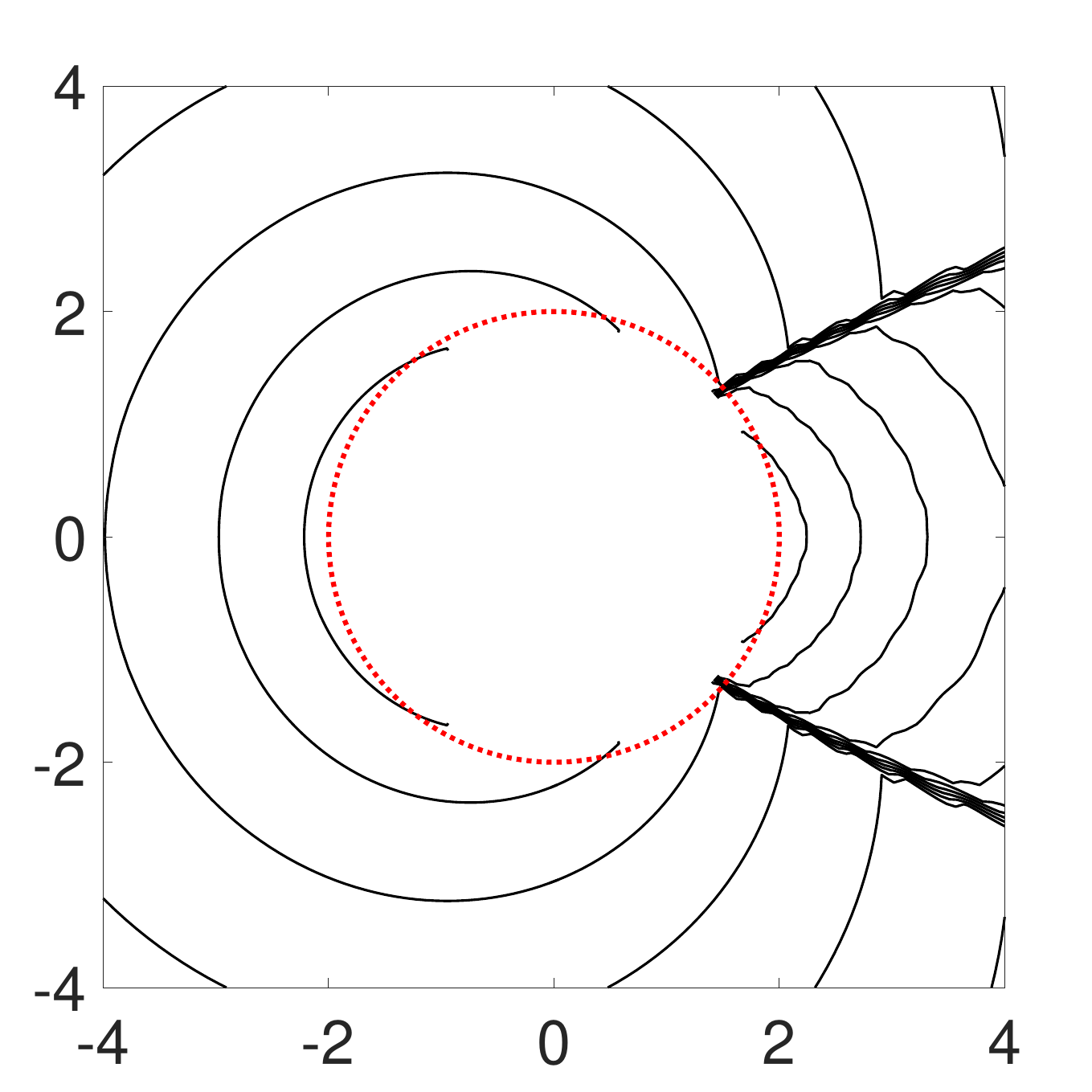}}\hspace{-2mm}
\subfigure[Case 2, $a = 0.5$]{\includegraphics[width=0.255\textwidth]{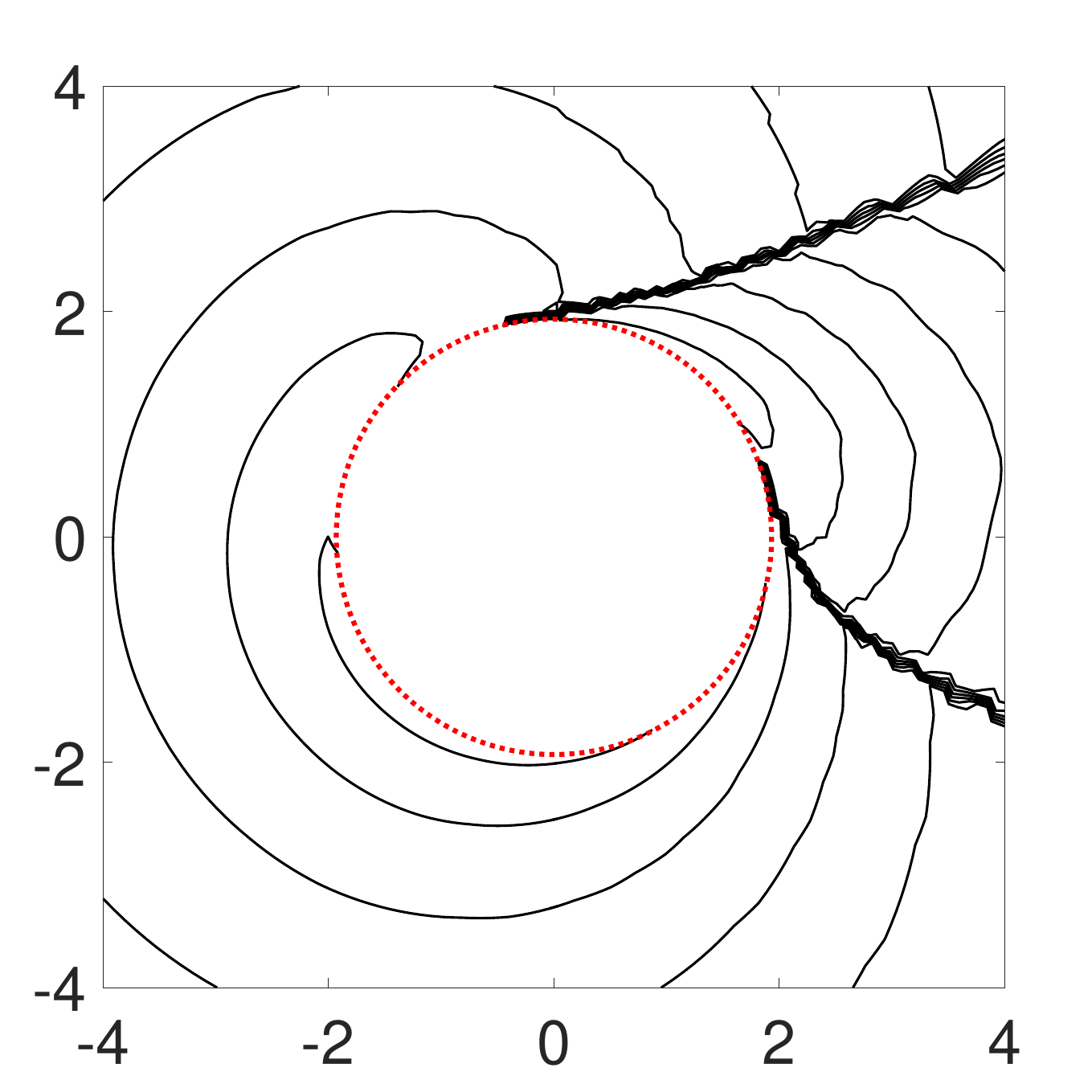}}\hspace{-2mm}
\subfigure[Case 3, $a = 0.9$]{\includegraphics[width=0.255\textwidth]{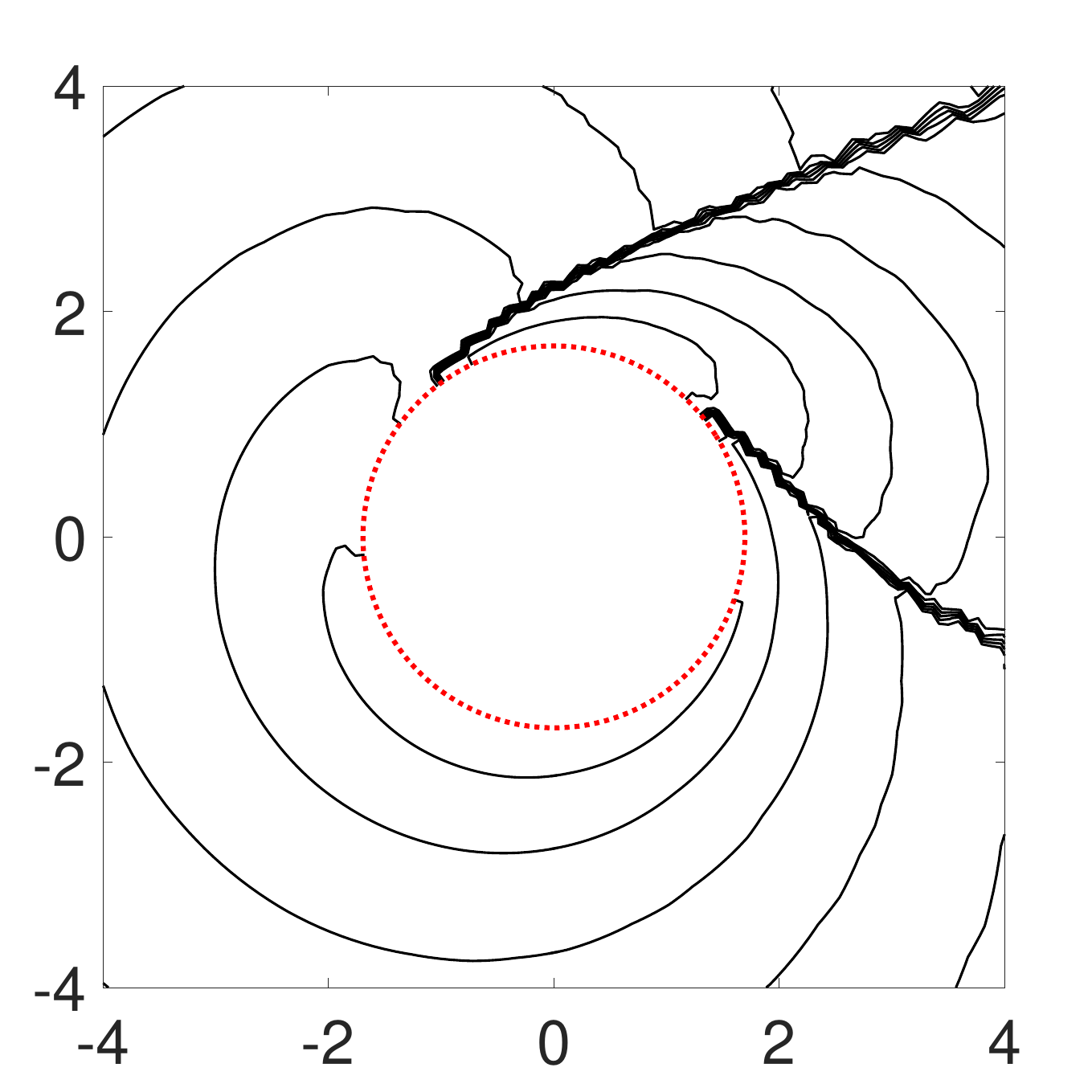}}\hspace{-2mm}
\subfigure[Case 4, $a = 0.99$]{\includegraphics[width=0.255\textwidth]{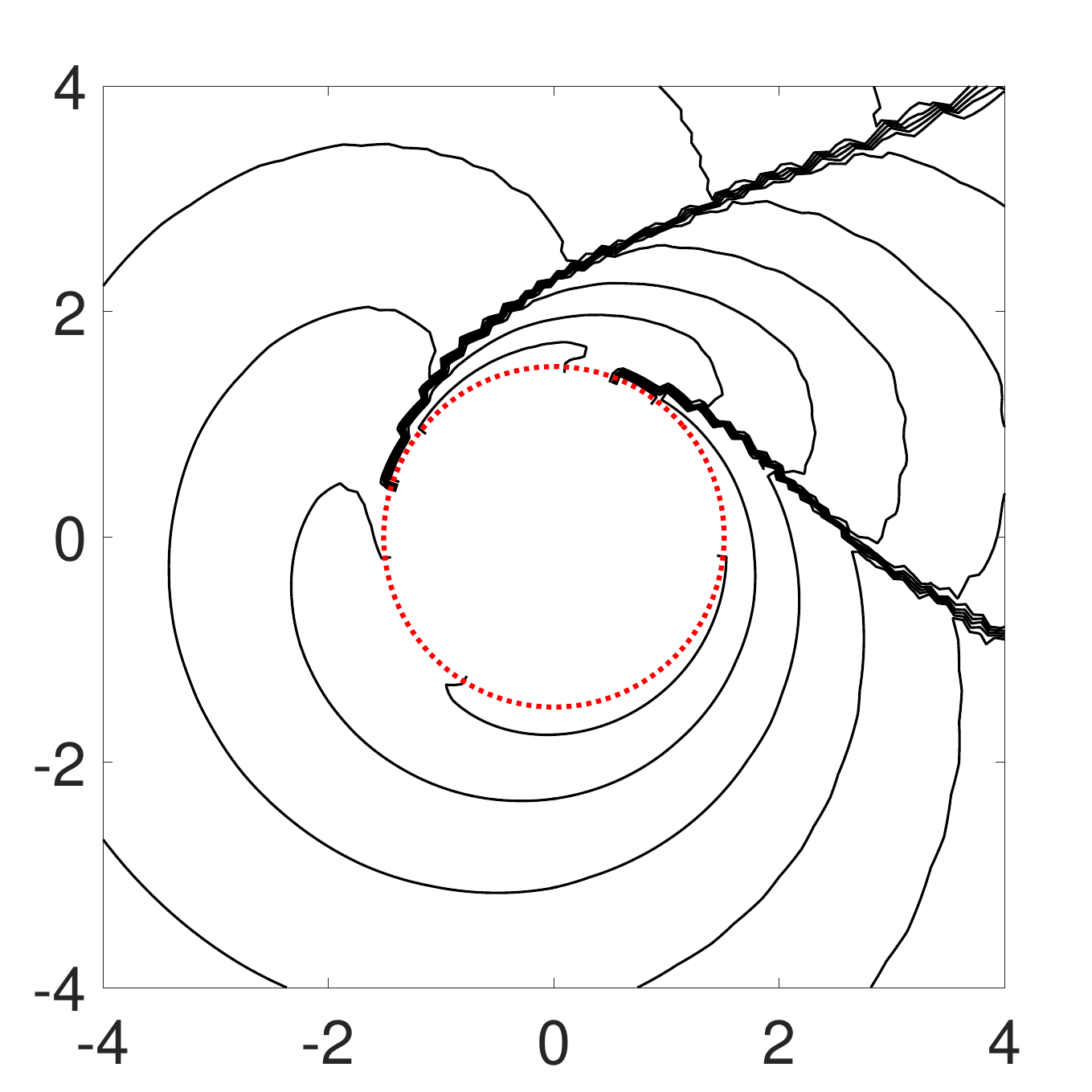}}
\centering
\subfigure[Case 1, $a = 0.0$]{\includegraphics[width=0.255\textwidth]{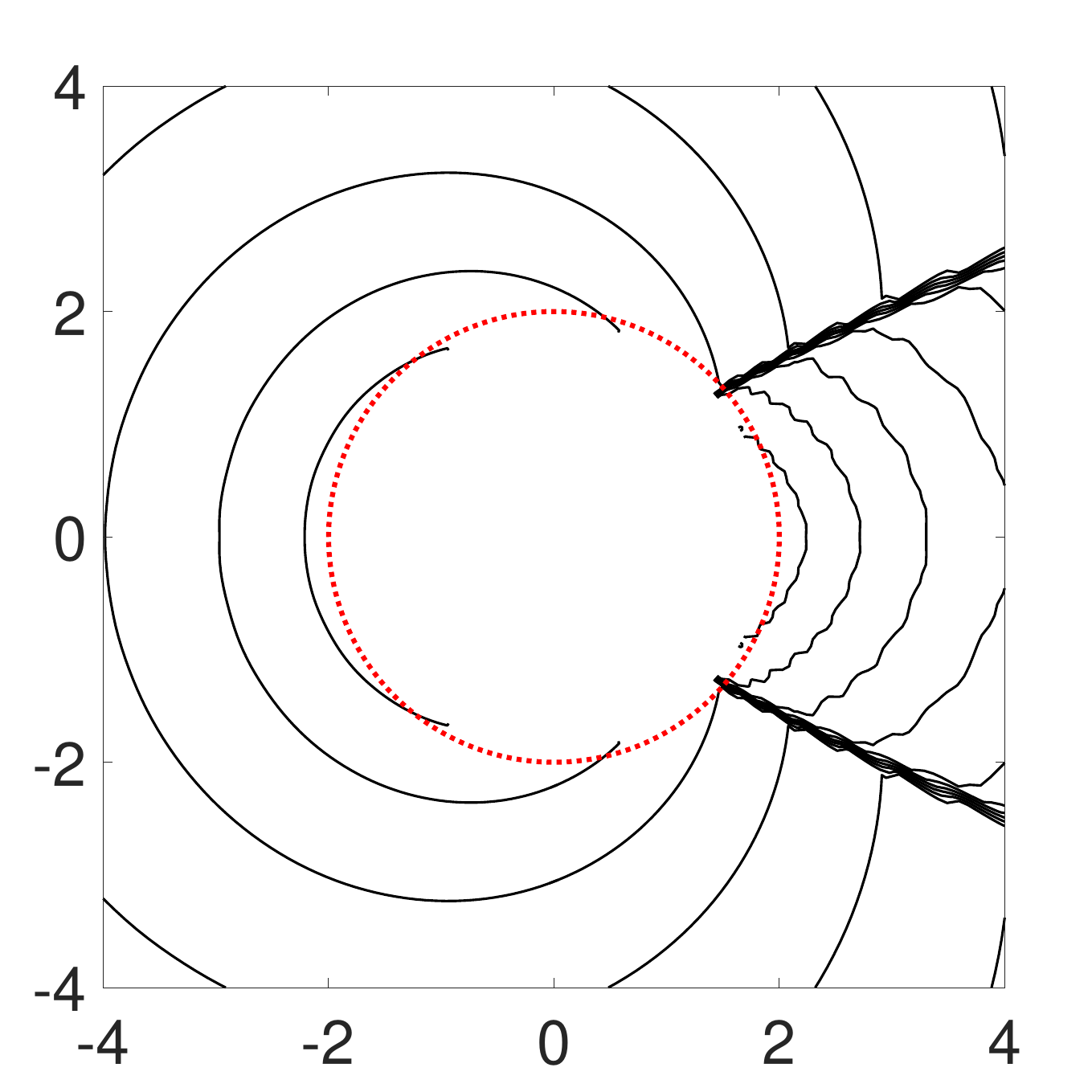}}\hspace{-2mm}
\subfigure[Case 2, $a = 0.5$]{\includegraphics[width=0.255\textwidth]{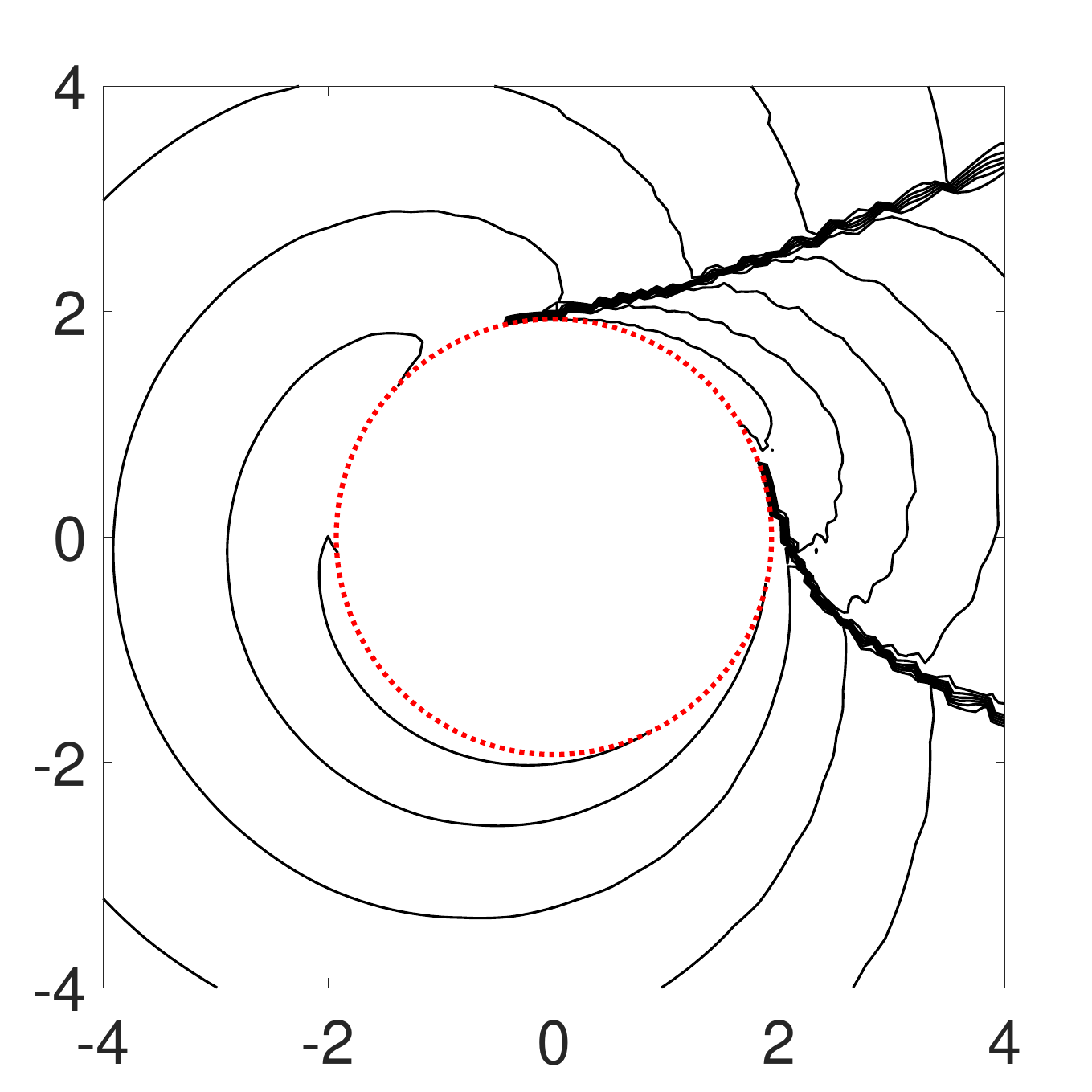}}\hspace{-2mm}
\subfigure[Case 3, $a = 0.9$]{\includegraphics[width=0.255\textwidth]{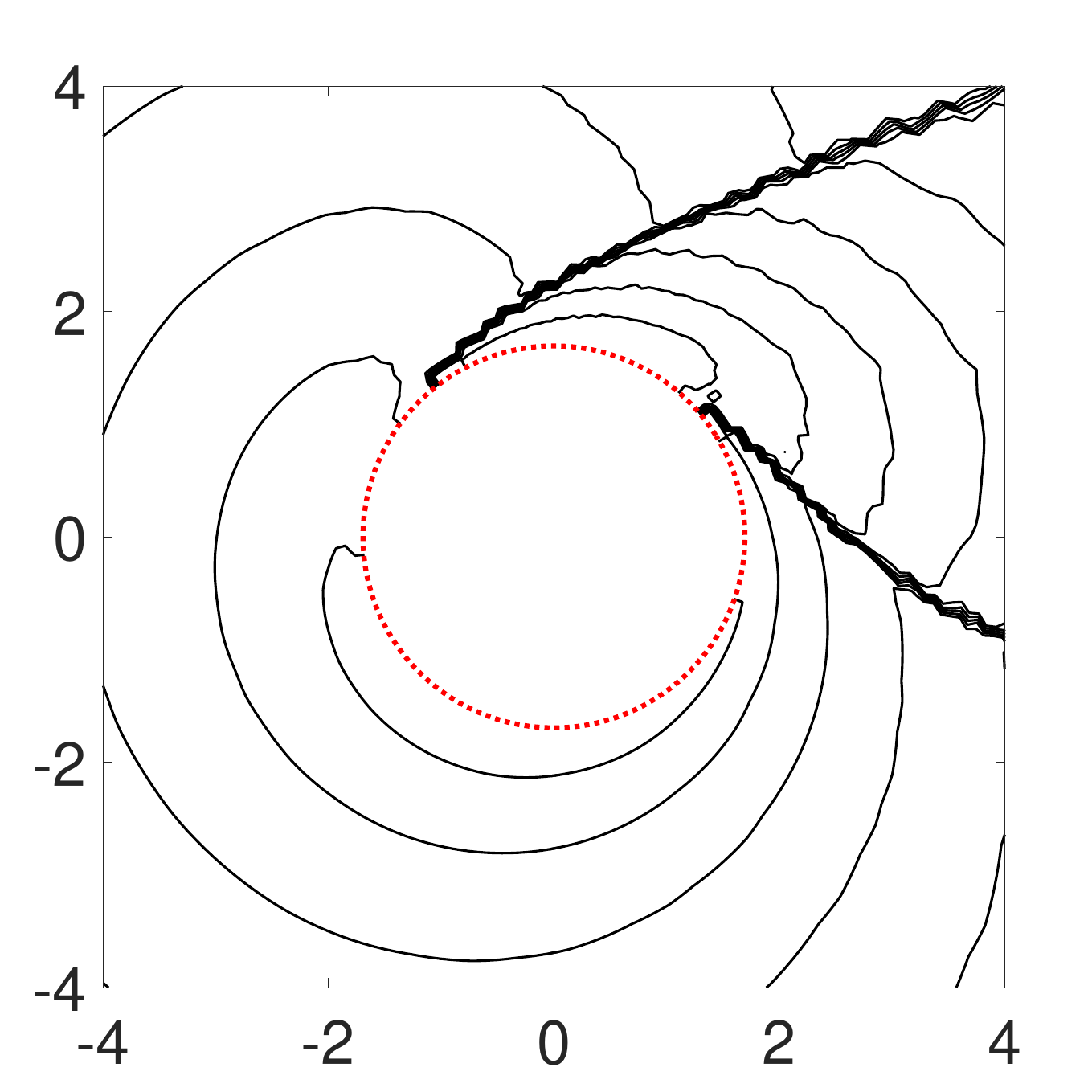}}\hspace{-2mm}
\subfigure[Case 4, $a = 0.99$]{\includegraphics[width=0.255\textwidth]{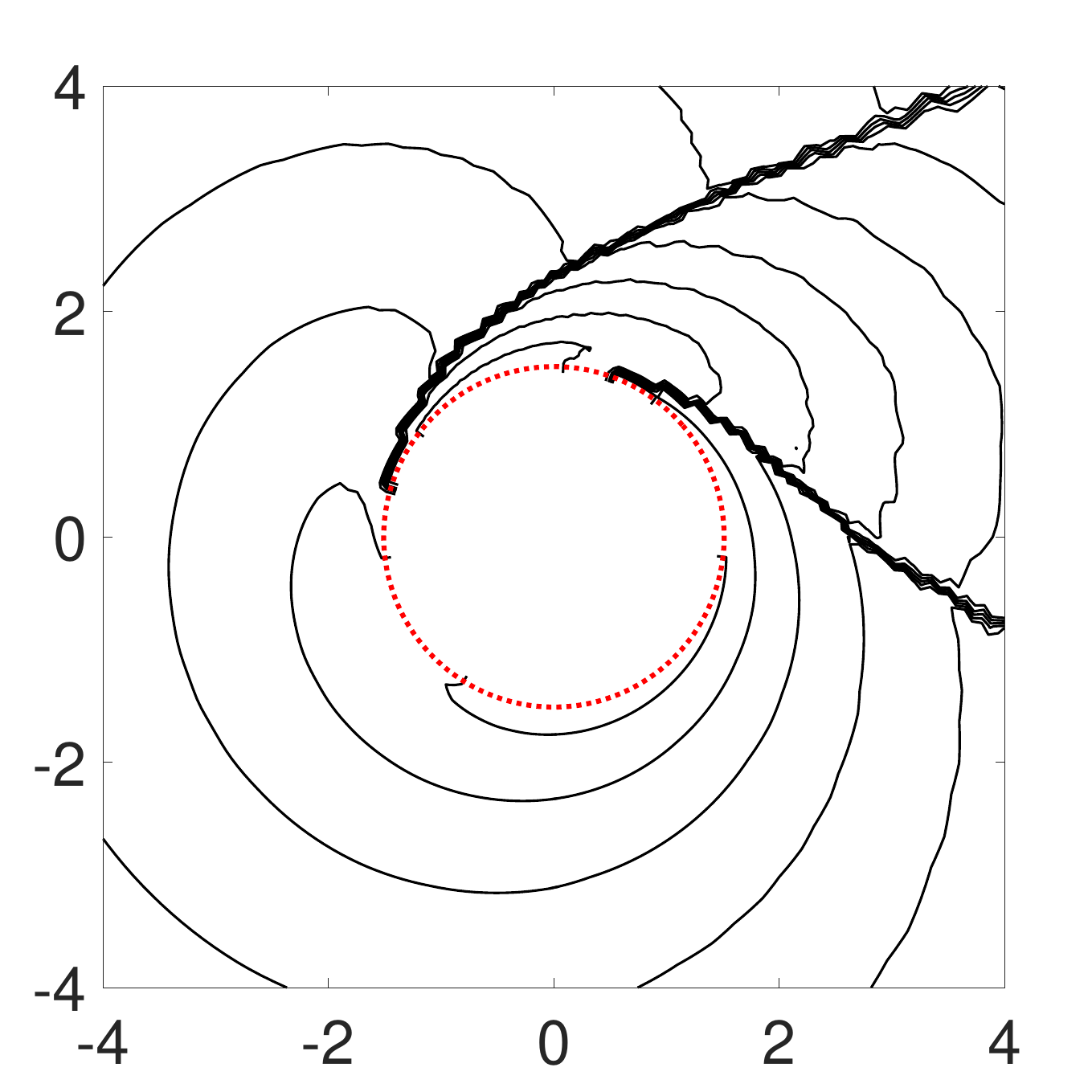}}
\centering
\subfigure[Case 1, $a = 0.0$]{\includegraphics[width=0.255\textwidth]{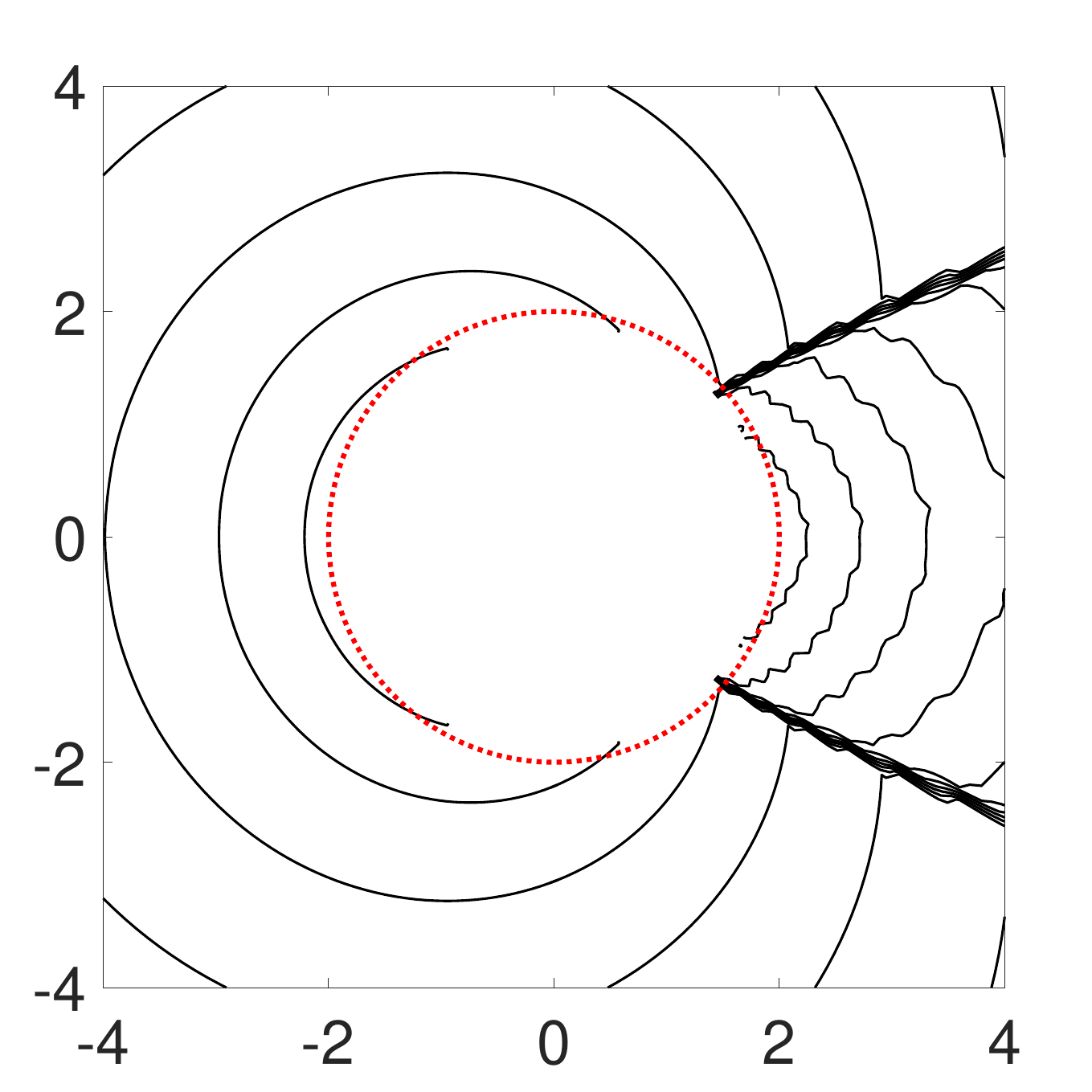}}\hspace{-2mm}
\subfigure[Case 2, $a = 0.5$]{\includegraphics[width=0.255\textwidth]{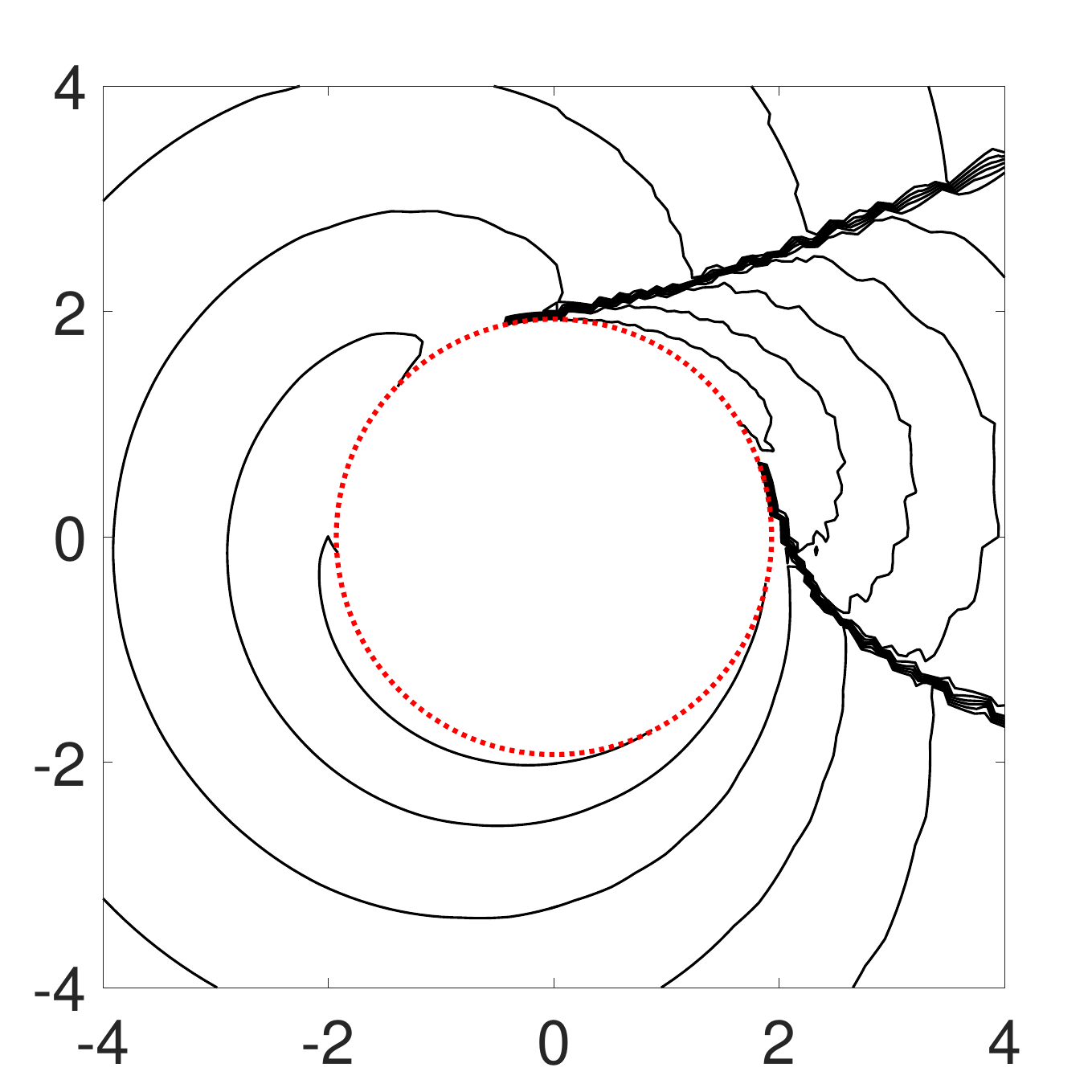}}\hspace{-2mm}
\subfigure[Case 3, $a = 0.9$]{\includegraphics[width=0.255\textwidth]{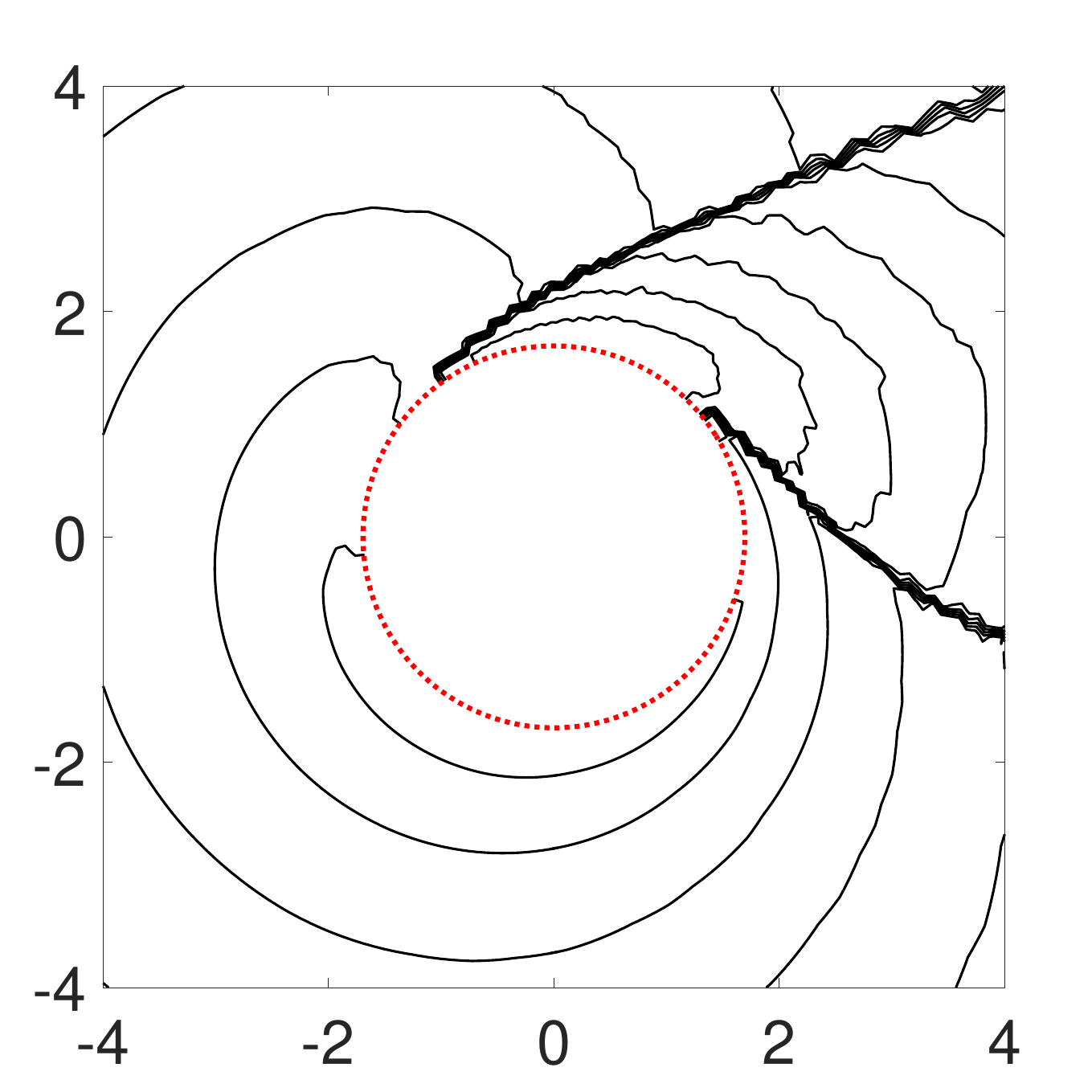}}\hspace{-2mm}
\subfigure[Case 4, $a = 0.99$]{\includegraphics[width=0.255\textwidth]{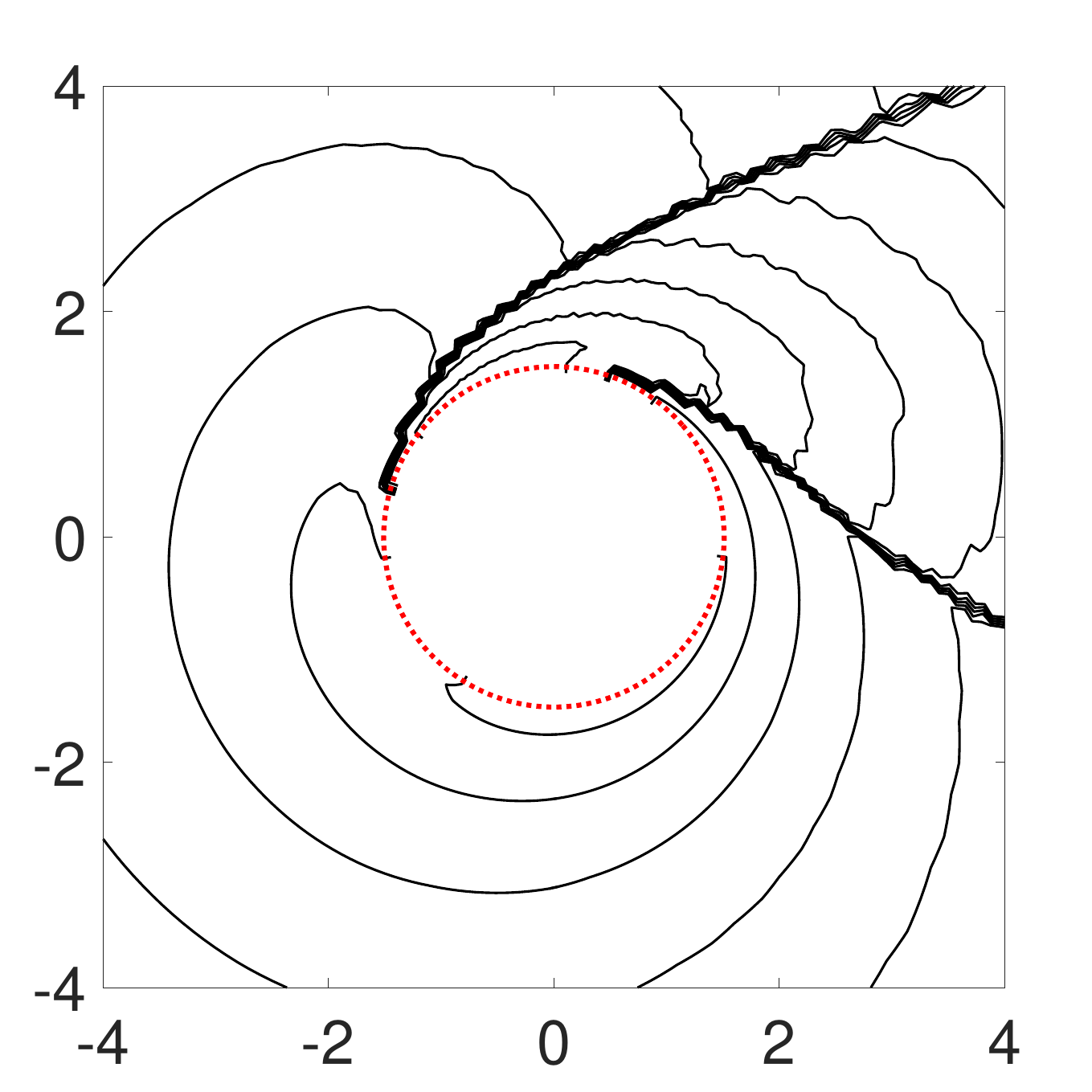}}
\caption{Example \ref{2D:Rie exam10a}: Contours of $\log \rho$ in $[-4,4]^2$. From top to bottom: $\mathbb{P}^1$-, $\mathbb{P}^2$-, and $\mathbb{P}^3$-based PCP-OEDG schemes.}\label{fig:Bh p3 case1-4}
\end{figure}

Figure \ref{fig:Bh p1-1 case1-4} shows the accretion patterns for the second-order PCP-OEDG method at the final stationary time. Red-yellow colors indicate high-density regions, while blue represents low-density zones. The black area in the center denotes the location of the event horizon $r^+$. Figure \ref{fig:Bh p1-2 case1-4} displays the contour lines of the logarithm of the rest-mass density, scaled by its asymptotic value, up to an outer domain of $r = 20$. A close-up view of the domain extending from $-4$ to $4$ is shown in Figure \ref{fig:Bh p3 case1-4}. In both figures, the dotted red line marks the location of the event horizon.

The influence of the rotating black hole on relativistic accretion is clearly evident. The matter density is highest near the event horizon and decreases with distance from the center. The flow pattern includes a stable tail shock wave. The black hole’s angular momentum drags the accreting flow, causing the shock to curve around the event horizon, with the degree of wrapping increasing as the angular momentum parameter $a$ grows. The numerical solution smoothly transitions across the event horizon, with the matter fields remaining regular at this boundary. The rotational effects are most prominent in the inner regions near the black hole, while the overall flow pattern in the outer regions remains largely unchanged for different values of $a$. The flow patterns and shocks are well captured by the PCP-OEDG schemes without any spurious oscillations, in good agreement with those reported in \cite{FIP1999}.

Similar phenomena are observed in the numerical results of the higher-order PCP-OEDG schemes, showing no nonphysical oscillations in the numerical solutions. The corresponding contour lines of the logarithm of the rest-mass density for these cases are presented in Figures \ref{fig:Bh p3 case1-4}. These findings suggest that accurately modeling the behavior of a rotating black hole necessitates a detailed representation of its innermost regions. As a result, in the subsequent test cases, we only present the close-up views of areas within a distance of $4$ units from the black hole's center.

\end{example}

\begin{figure}[!thb]
\centering
\subfigure[Case 5, $\Gamma = 4/3$]{\includegraphics[width=0.255\textwidth]{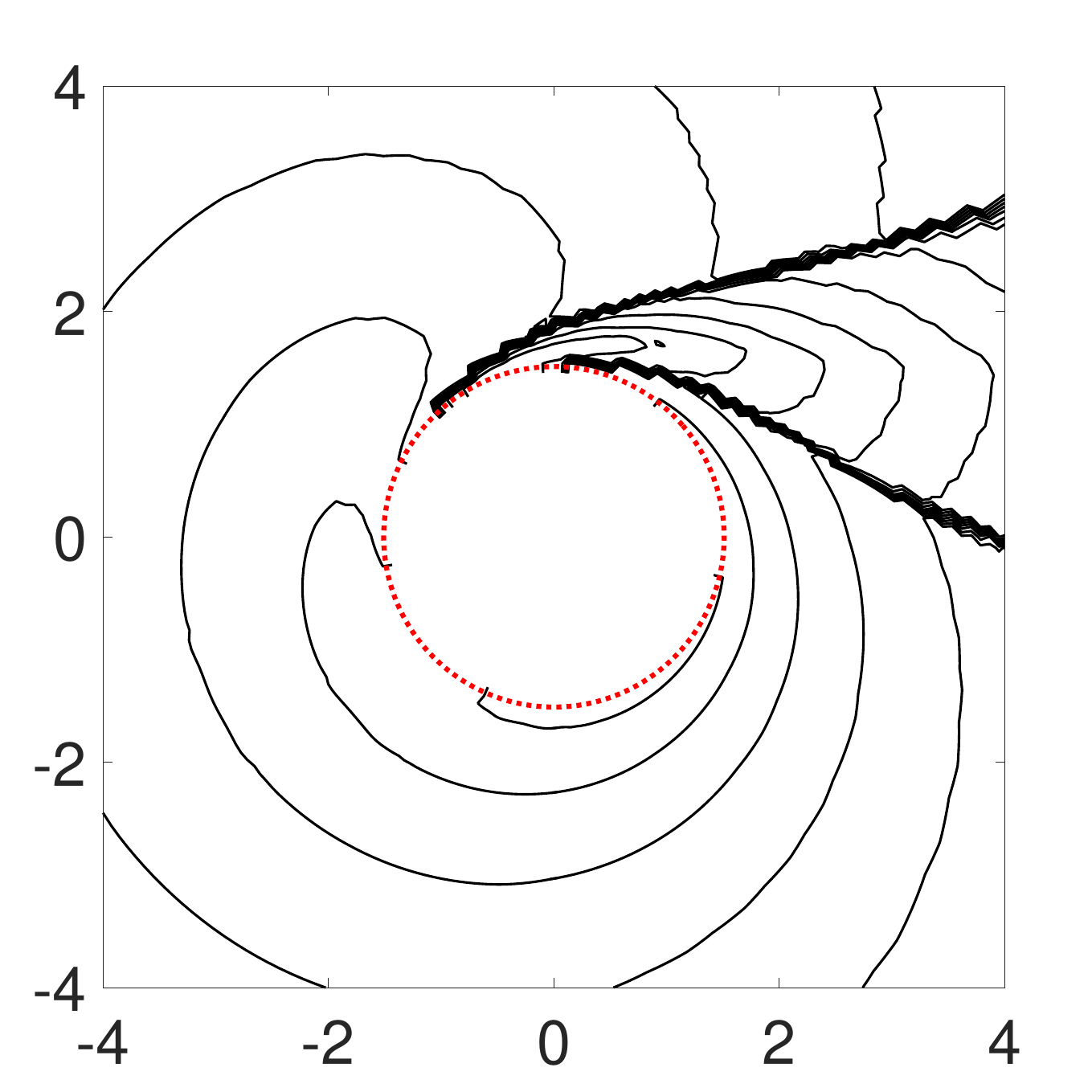}}\hspace{5mm}
\subfigure[Case 4, $\Gamma = 5/3$]{\includegraphics[width=0.255\textwidth]{figs_2D/case5_10_p1.pdf}}\hspace{5mm}
\subfigure[Case 6, $\Gamma = 2$]{\includegraphics[width=0.255\textwidth]{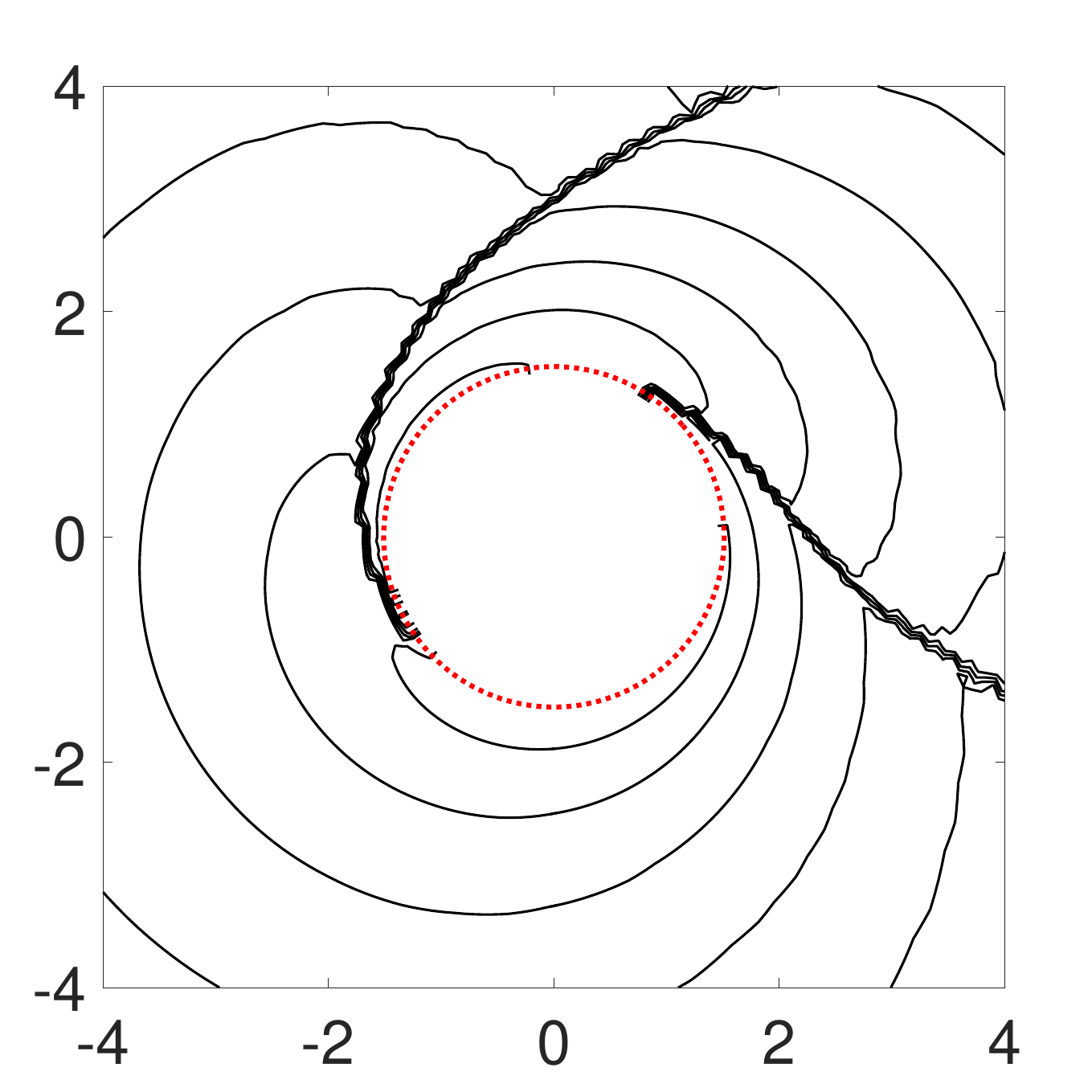}}
\centering
\subfigure[Case 5, $\Gamma = 4/3$]{\includegraphics[width=0.255\textwidth]{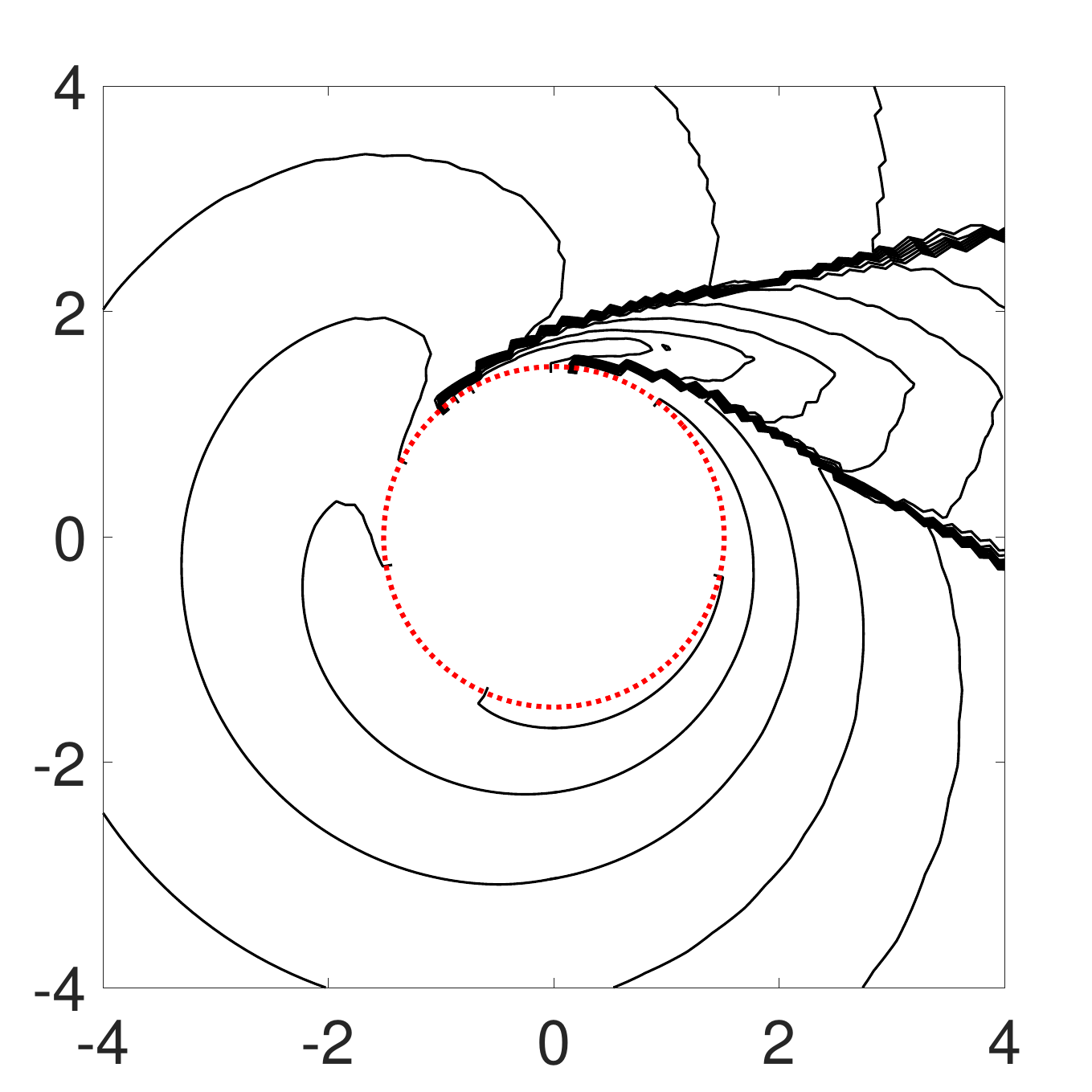}}\hspace{5mm}
\subfigure[Case 4, $\Gamma = 5/3$]{\includegraphics[width=0.255\textwidth]{figs_2D/case5_10_p2.pdf}}\hspace{5mm}
\subfigure[Case 6, $\Gamma = 2$]{\includegraphics[width=0.255\textwidth]{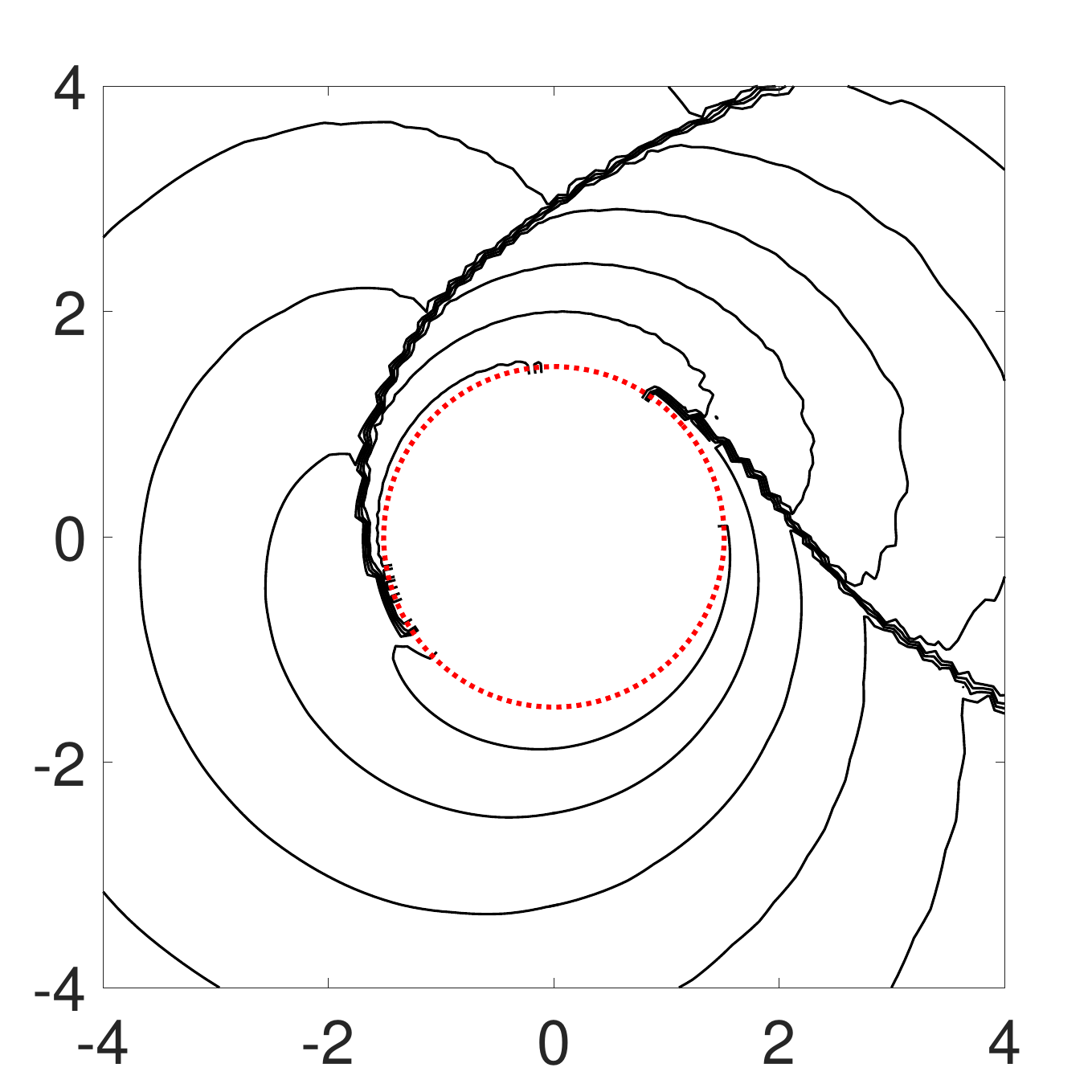}}
\centering
\subfigure[Case 5, $\Gamma = 4/3$]{\includegraphics[width=0.255\textwidth]{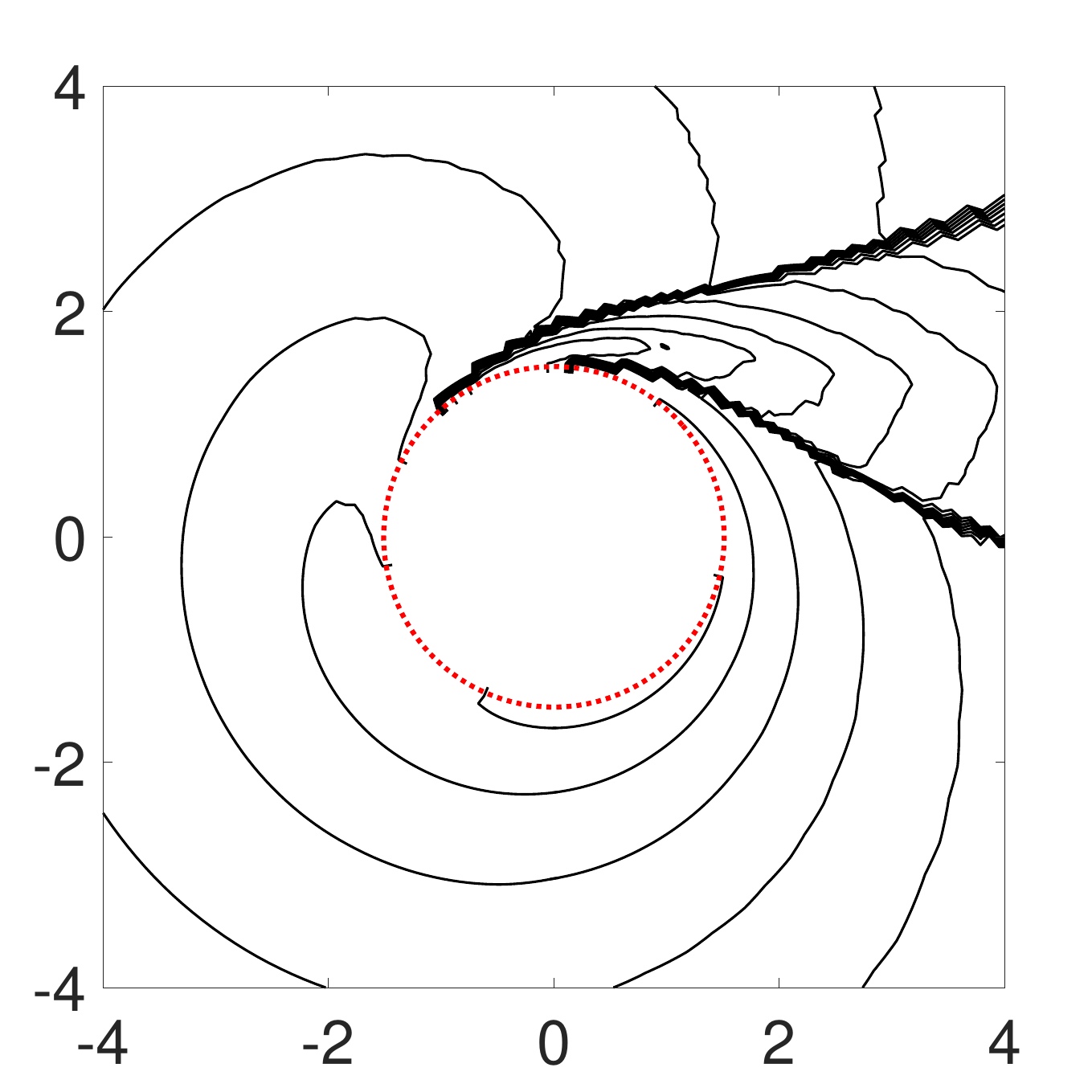}}\hspace{5mm}
\subfigure[Case 4, $\Gamma = 5/3$]{\includegraphics[width=0.255\textwidth]{figs_2D/case5_10_p3.pdf}}\hspace{5mm}
\subfigure[Case 6, $\Gamma = 2$]{\includegraphics[width=0.255\textwidth]{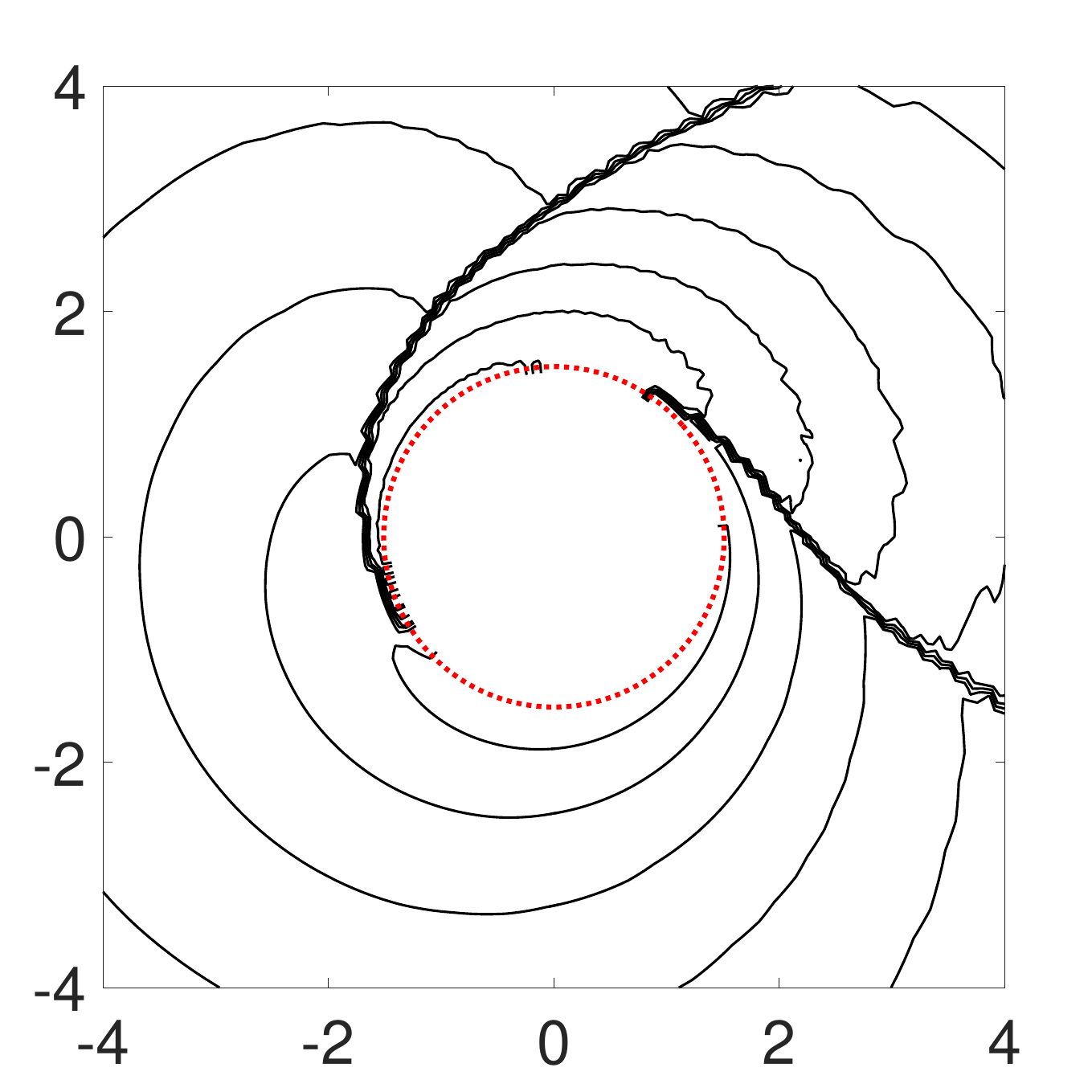}}
\caption{Example \ref{2D:Rie exam10b}: Close-up views of contours of $\log \rho$ in $[-4,4]^2$. From top to bottom: $\mathbb{P}^1$-, $\mathbb{P}^2$-, and $\mathbb{P}^3$-based PCP-OEDG schemes.}\label{fig:Bh p1 case5-8}
\end{figure}
\begin{example}[Effects of Adiabatic Index $\Gamma$ for Kerr Black Hole Accretion]\label{2D:Rie exam10b}

Cases 4 through 6 in Table \ref{tb:2D exam10} describe different thermodynamic flow configurations. The contours of the density logarithm for these cases are plotted in Figure \ref{fig:Bh p1 case5-8}, obtained by the proposed PCP-OEDG schemes. These plots clearly show the dependence of accretion patterns on the adiabatic index $\Gamma$ of the fluid around a rapidly rotating black hole.

The results indicate that as the value of $\Gamma$ increases, the shock wave opening angle becomes larger around the accretor. Additionally, the shock opening angles become larger for higher values of $\Gamma$ due to the increased pressure within the shock ``cone" as $\Gamma$ increases. This indicates that the adiabatic index significantly influences both the structure of the shock and the overall morphology of the accretion flow.

All flow structures are correctly resolved by the PCP-OEDG schemes, and the results obtained using different orders of the schemes are consistent with each other. The numerical results also match those presented in \cite{FIP1999}, demonstrating the accuracy and robustness of the proposed schemes.

\end{example}

\begin{example}[Effects of the Flow Velocity $v_{\infty}$ for Kerr Black Hole Accretion]\label{2D:Rie exam10c}

The PCP-OEDG schemes are employed to further investigate the relationship between the accretion patterns of a rotating black hole and the fluid velocity. To this end, we examine cases with increasing velocities: $v_{\infty} = 0.5,\, 0.9,\, 0.99$, corresponding to Cases 4, 7, and 8 in Table \ref{tb:2D exam10}, respectively. The numerical results for the density logarithm obtained using the $\mathbb{P}^m$-based PCP-OEDG scheme for $m=1,2,3$ at $t = 500$ over the domain $[-4,4]^2$ are presented in Figure \ref{fig:Bh p1 case7-8}.

\begin{figure}[!thb]
\centering
\subfigure[Case 4, $v_{\infty} = 0.5$]{\includegraphics[width=0.255\textwidth]{figs_2D/case5_10_p1.pdf}}\hspace{5mm}
\subfigure[Case 7, $v_{\infty} = 0.9$]{\includegraphics[width=0.255\textwidth]{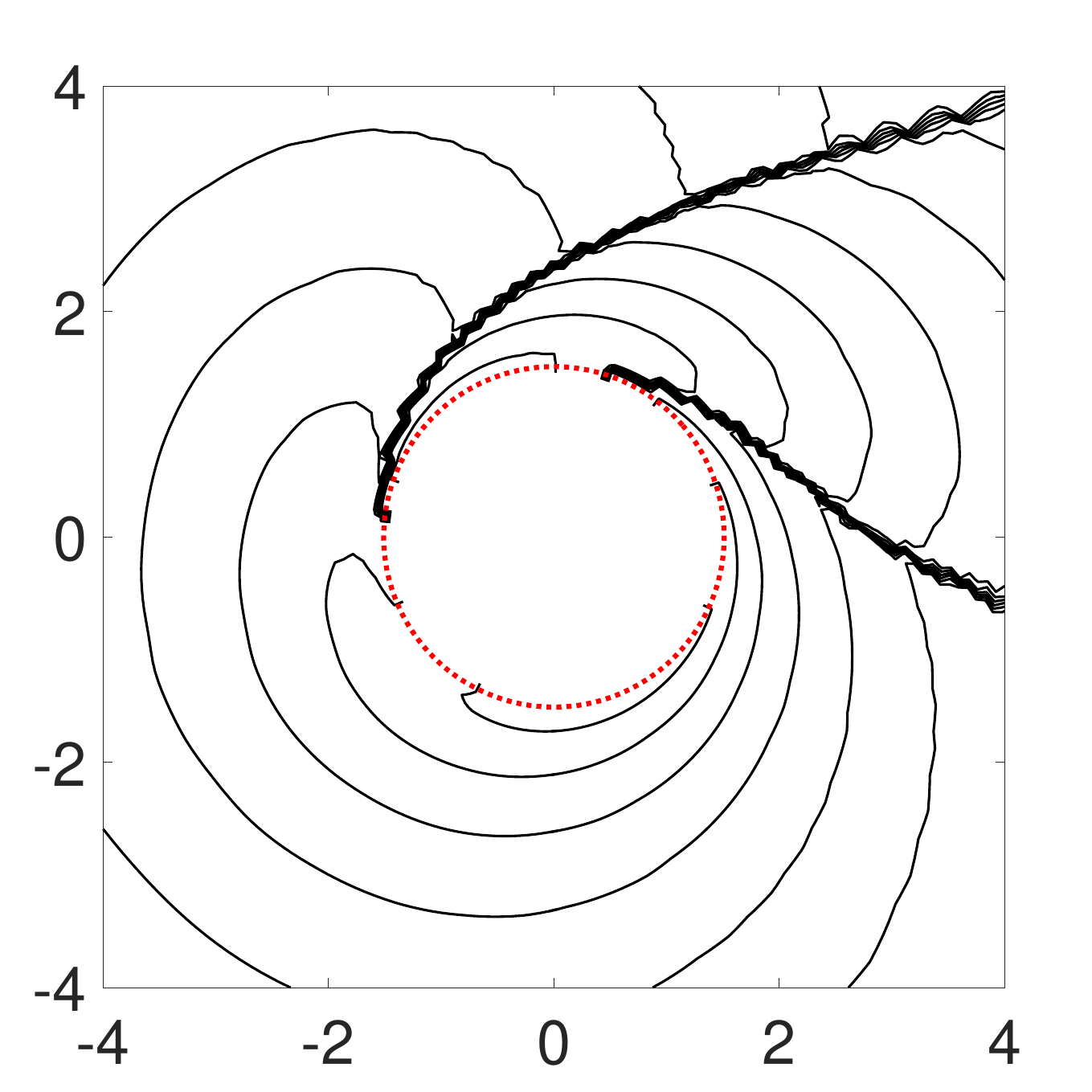}}\hspace{5mm}
\subfigure[Case 8, $v_{\infty} = 0.99$]{\includegraphics[width=0.255\textwidth]{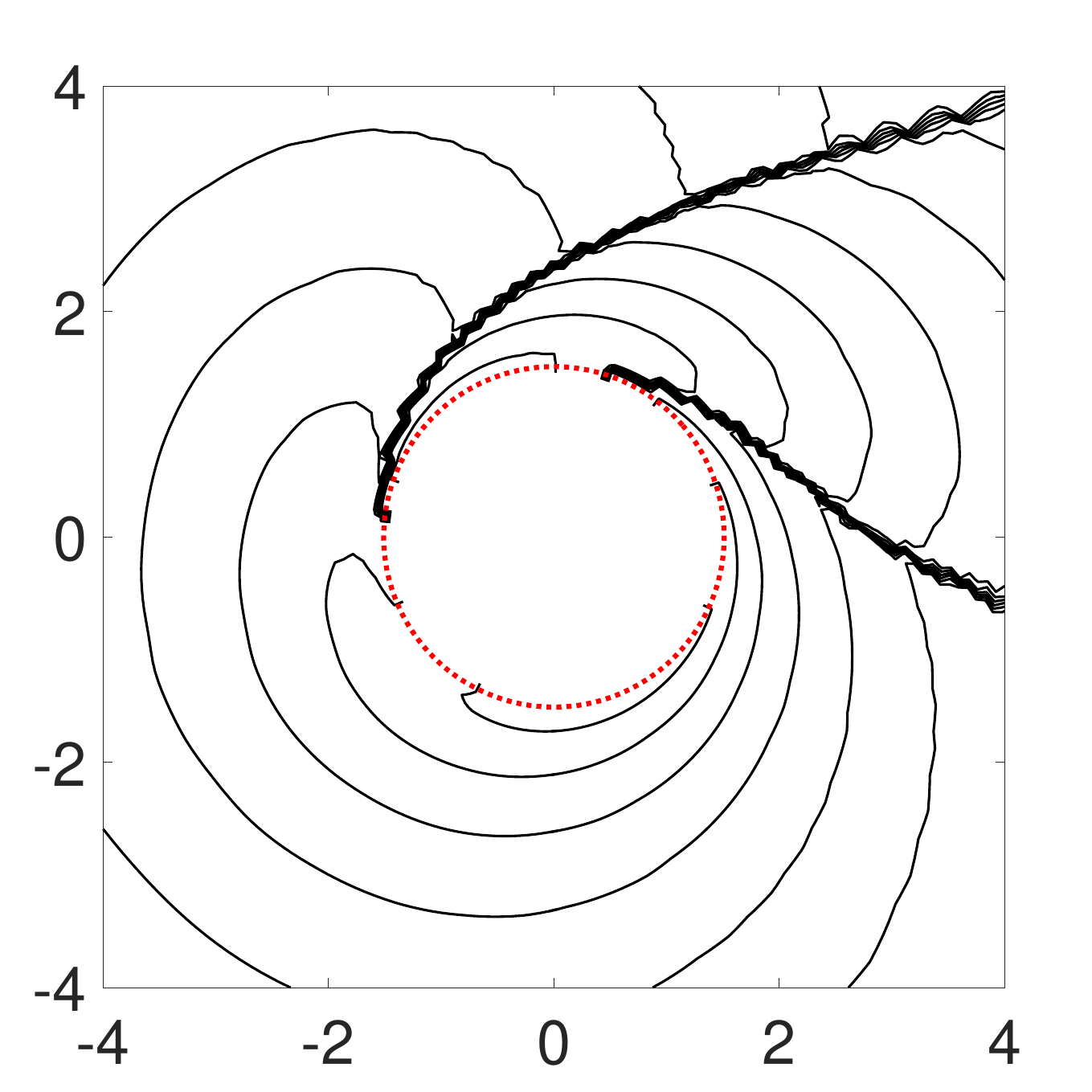}}
\centering
\subfigure[Case 4, $v_{\infty} = 0.5$]{\includegraphics[width=0.255\textwidth]{figs_2D/case5_10_p2.pdf}}\hspace{5mm}
\subfigure[Case 7, $v_{\infty} = 0.9$]{\includegraphics[width=0.255\textwidth]{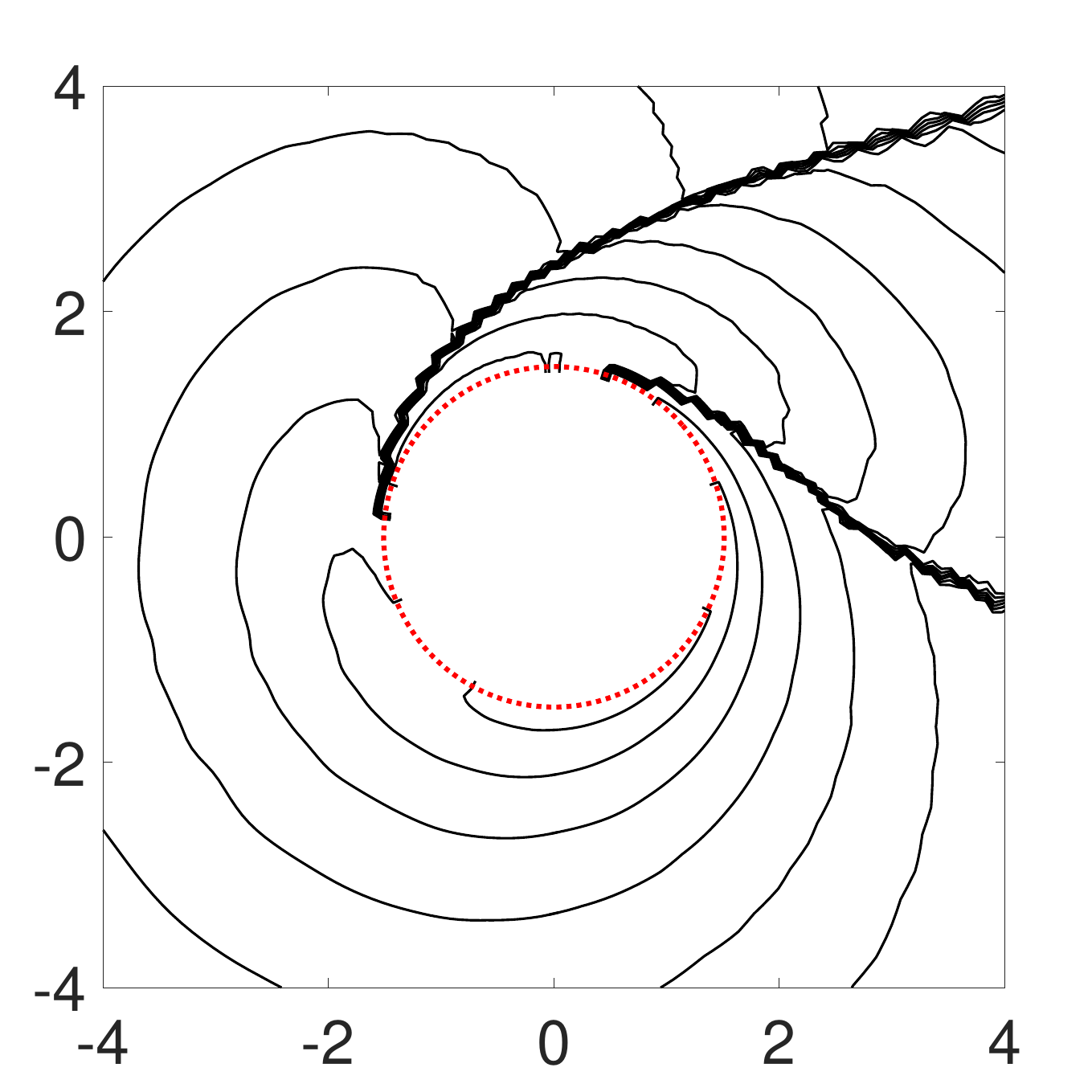}}\hspace{5mm}
\subfigure[Case 8, $v_{\infty} = 0.99$]{\includegraphics[width=0.255\textwidth]{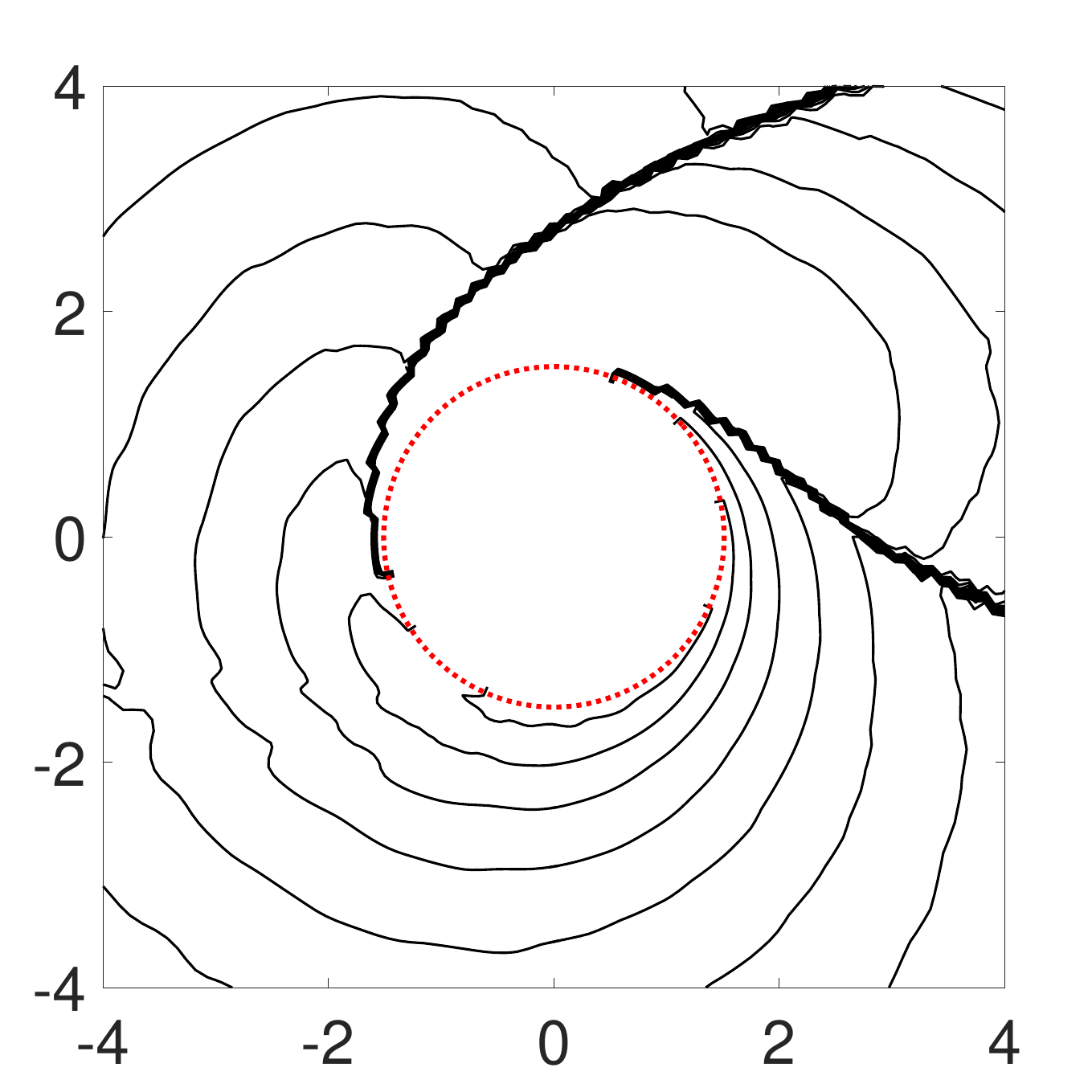}}
\centering
\subfigure[Case 4, $v_{\infty} = 0.5$]{\includegraphics[width=0.255\textwidth]{figs_2D/case5_10_p3.pdf}}\hspace{5mm}
\subfigure[Case 7, $v_{\infty} = 0.9$]{\includegraphics[width=0.255\textwidth]{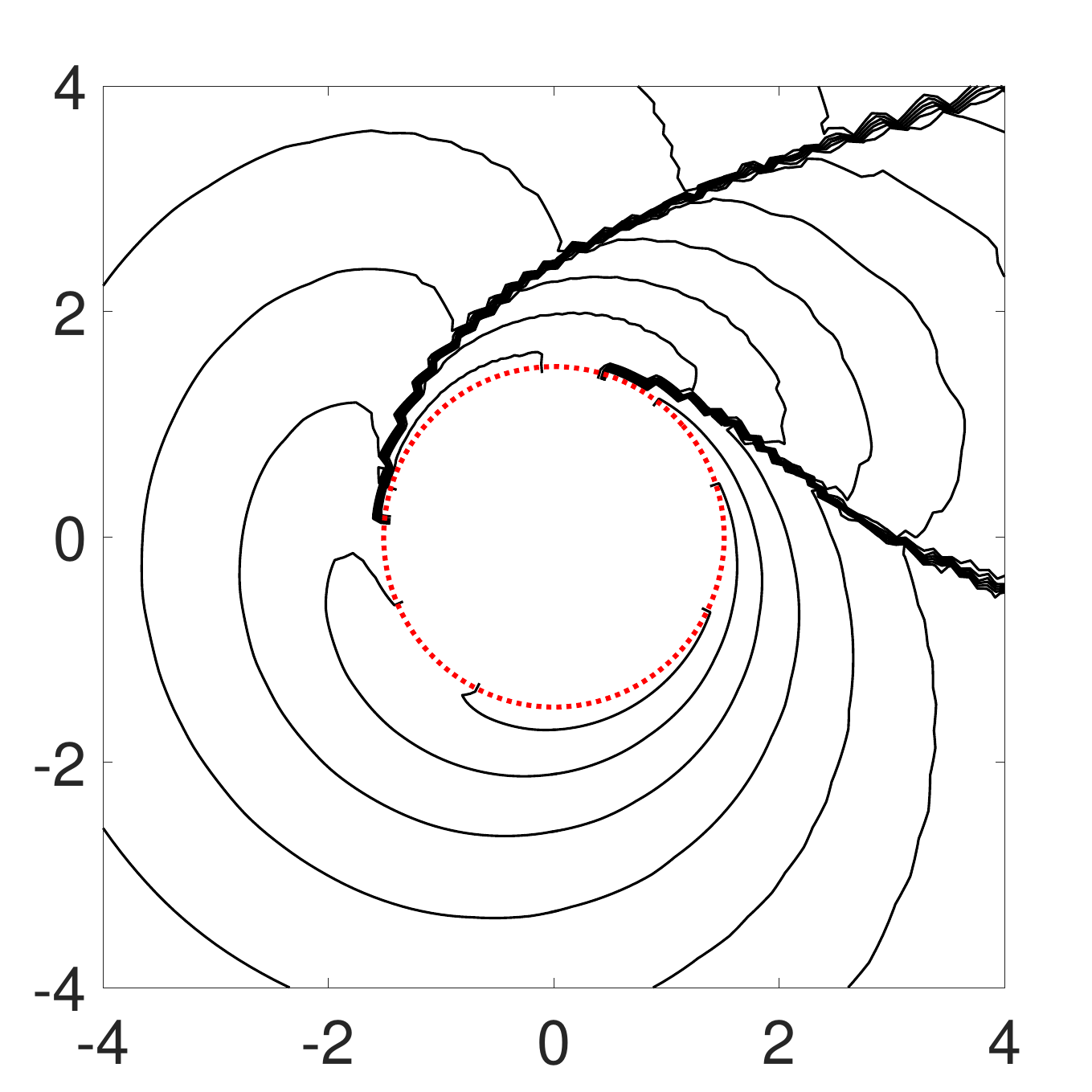}}\hspace{5mm}
\subfigure[Case 8, $v_{\infty} = 0.99$]{\includegraphics[width=0.255\textwidth]{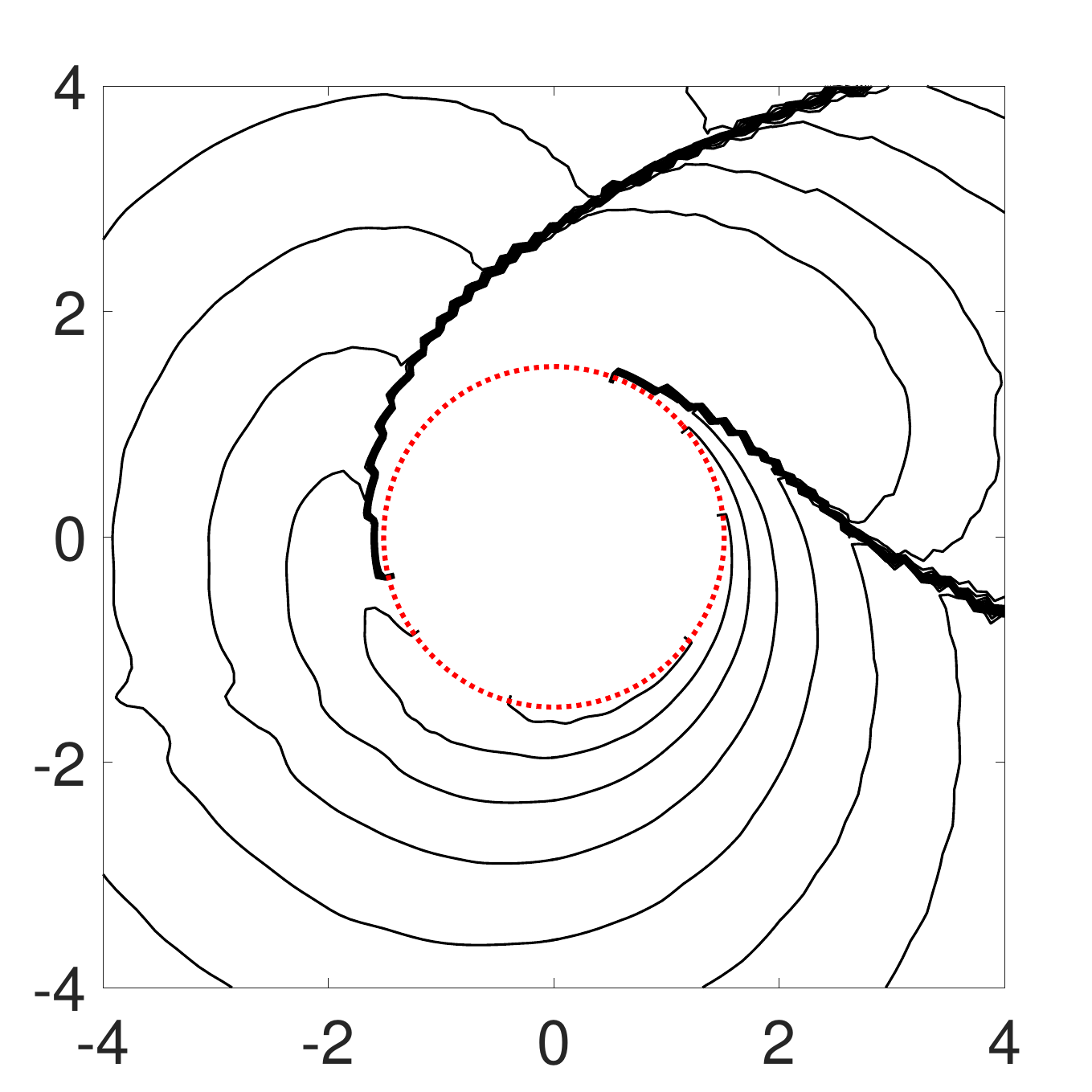}}
\caption{Example \ref{2D:Rie exam10c}: Close-up views of contours of $\log \rho$ in $[-4,4]^2$. From top to bottom: $\mathbb{P}^1$-, $\mathbb{P}^2$-, and $\mathbb{P}^3$-based PCP-OEDG schemes.}\label{fig:Bh p1 case7-8}
\end{figure}

From these results, we observe that while the shock wave wraps around the accretor, higher fluid velocities (larger Lorentz factors) have a relatively minor impact on the overall physical characteristics of black hole accretion. This suggests that the shock structure and accretion flow morphology are predominantly influenced by other parameters, such as the angular momentum and the adiabatic index, rather than the asymptotic flow velocity.

\end{example}

\begin{example}[Effects of the Mach Number $M_{\infty}$ for Kerr Black Hole Accretion]\label{2D:Rie exam10d}

In this final test, we use the PCP-OEDG method to explore the impact of different Mach numbers on the flow morphology around the black hole. Figure \ref{fig:Bh p1 case9-10} shows the contours of the rest-mass density logarithm under identical settings, with the exception of varying Mach numbers, simulated by the PCP-OEDG schemes.

\begin{figure}[!thb]
\centering
\subfigure[Case 4, $M_{\infty} = 5$]{\includegraphics[width=0.255\textwidth]{figs_2D/case5_10_p1.pdf}}\hspace{5mm}
\subfigure[Case 9, $M_{\infty} = 20$]{\includegraphics[width=0.255\textwidth]{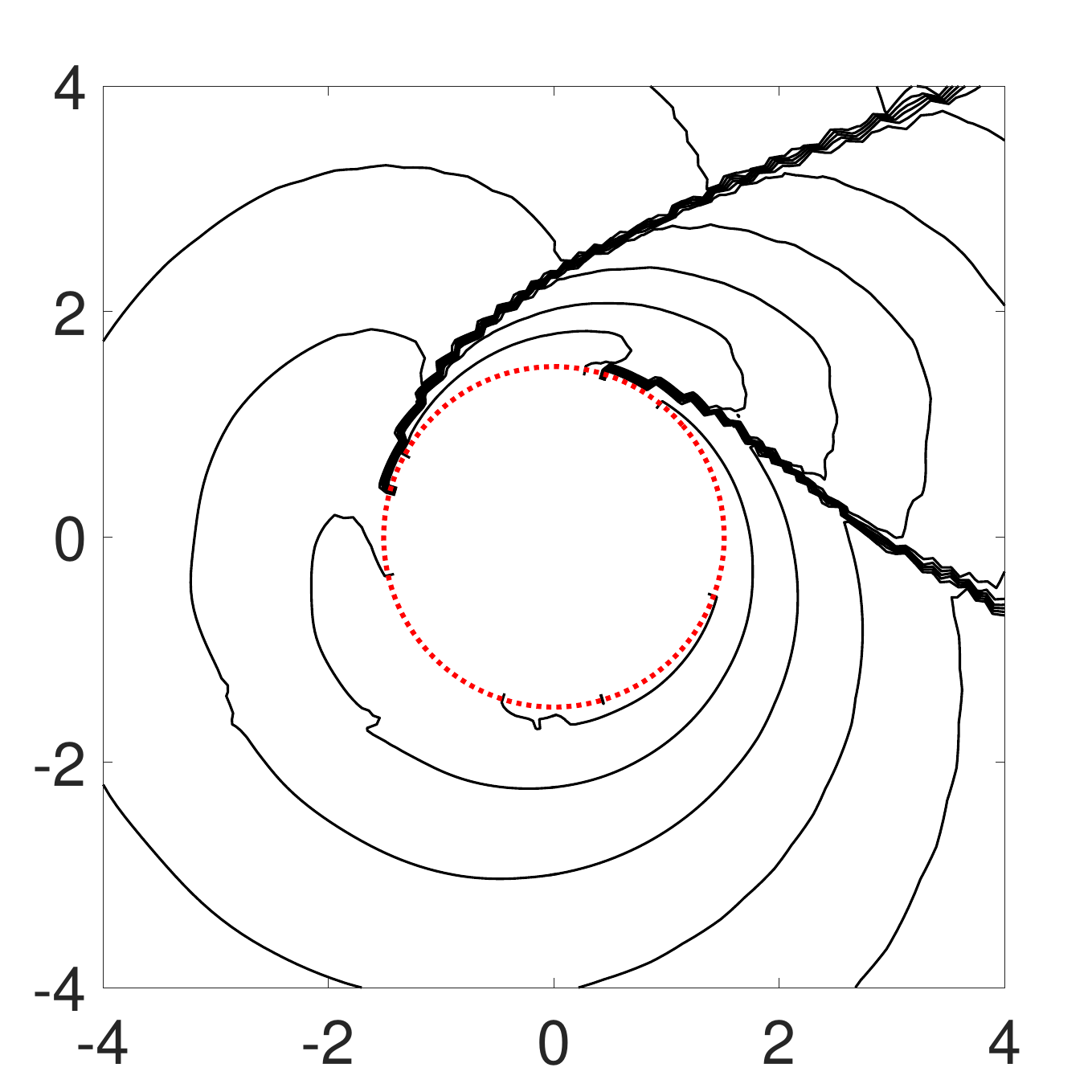}}\hspace{5mm}
\subfigure[Case 10, $M_{\infty} = 50$]{\includegraphics[width=0.255\textwidth]{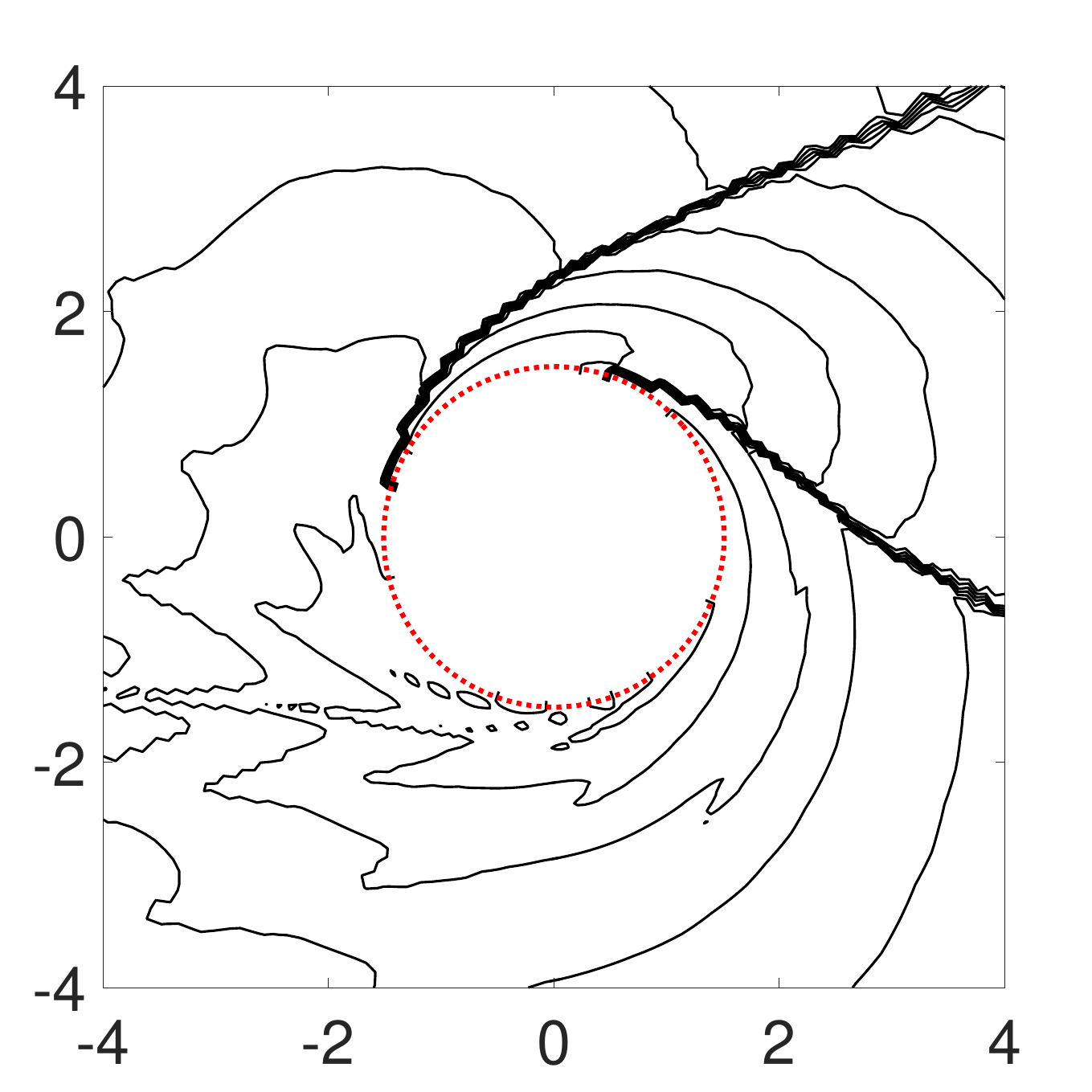}}
\centering
\subfigure[Case 4, $M_{\infty} = 5$]{\includegraphics[width=0.255\textwidth]{figs_2D/case5_10_p2.pdf}}\hspace{5mm}
\subfigure[Case 9, $M_{\infty} = 20$]{\includegraphics[width=0.255\textwidth]{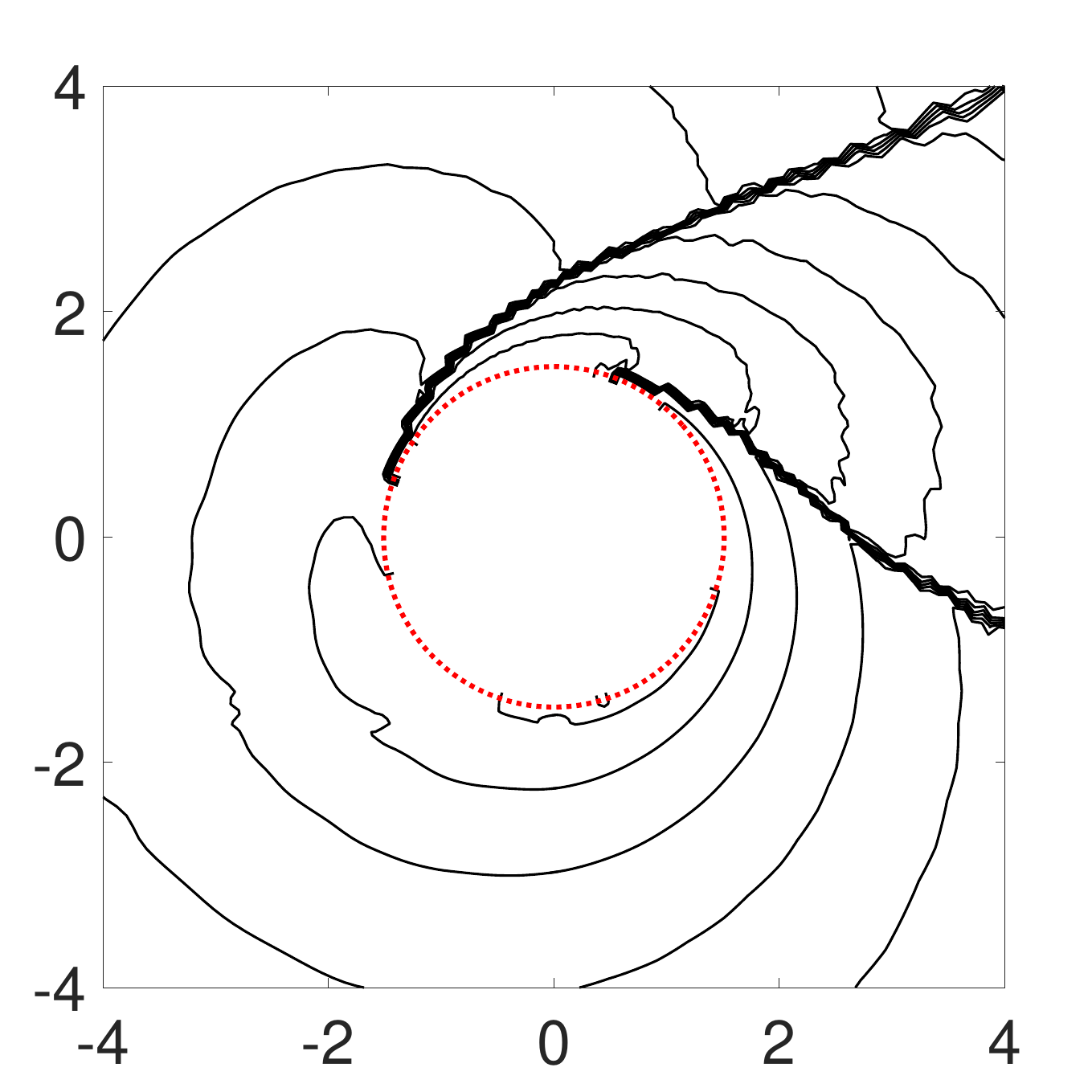}}\hspace{5mm}
\subfigure[Case 10, $M_{\infty} = 50$]{\includegraphics[width=0.255\textwidth]{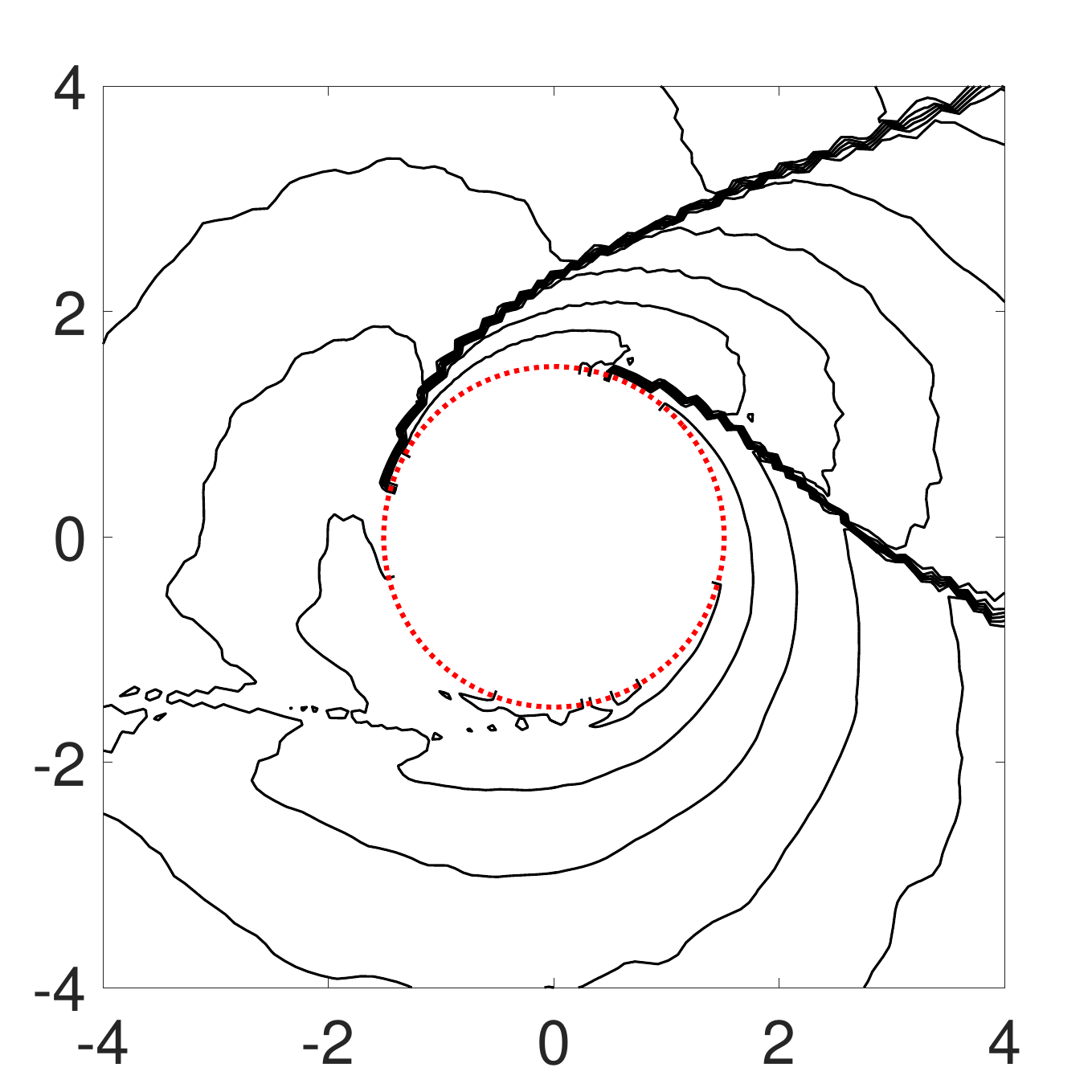}}
\centering
\subfigure[Case 4, $M_{\infty} = 5$]{\includegraphics[width=0.255\textwidth]{figs_2D/case5_10_p3.pdf}}\hspace{5mm}
\subfigure[Case 9, $M_{\infty} = 20$]{\includegraphics[width=0.255\textwidth]{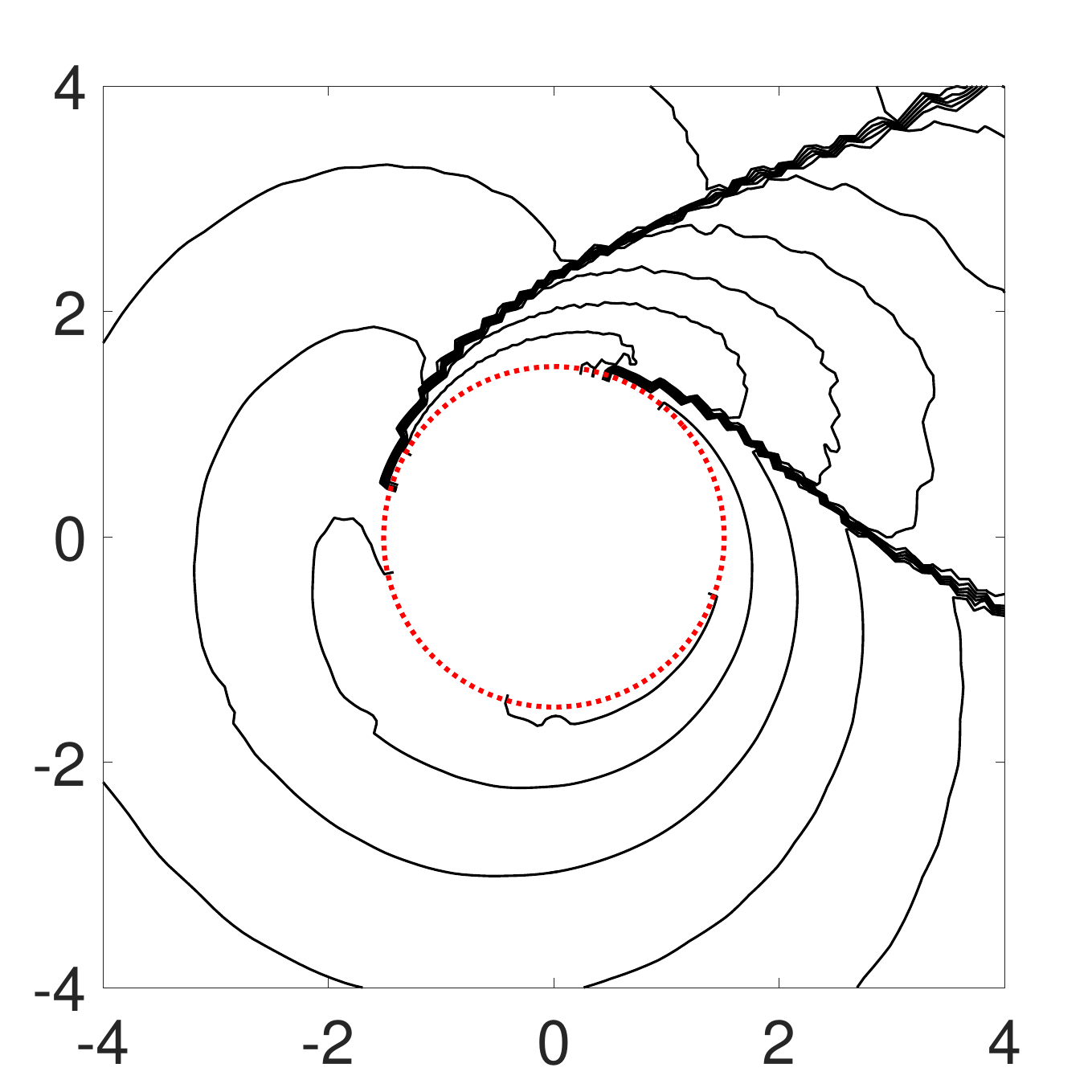}}\hspace{5mm}
\subfigure[Case 10, $M_{\infty} = 50$]{\includegraphics[width=0.255\textwidth]{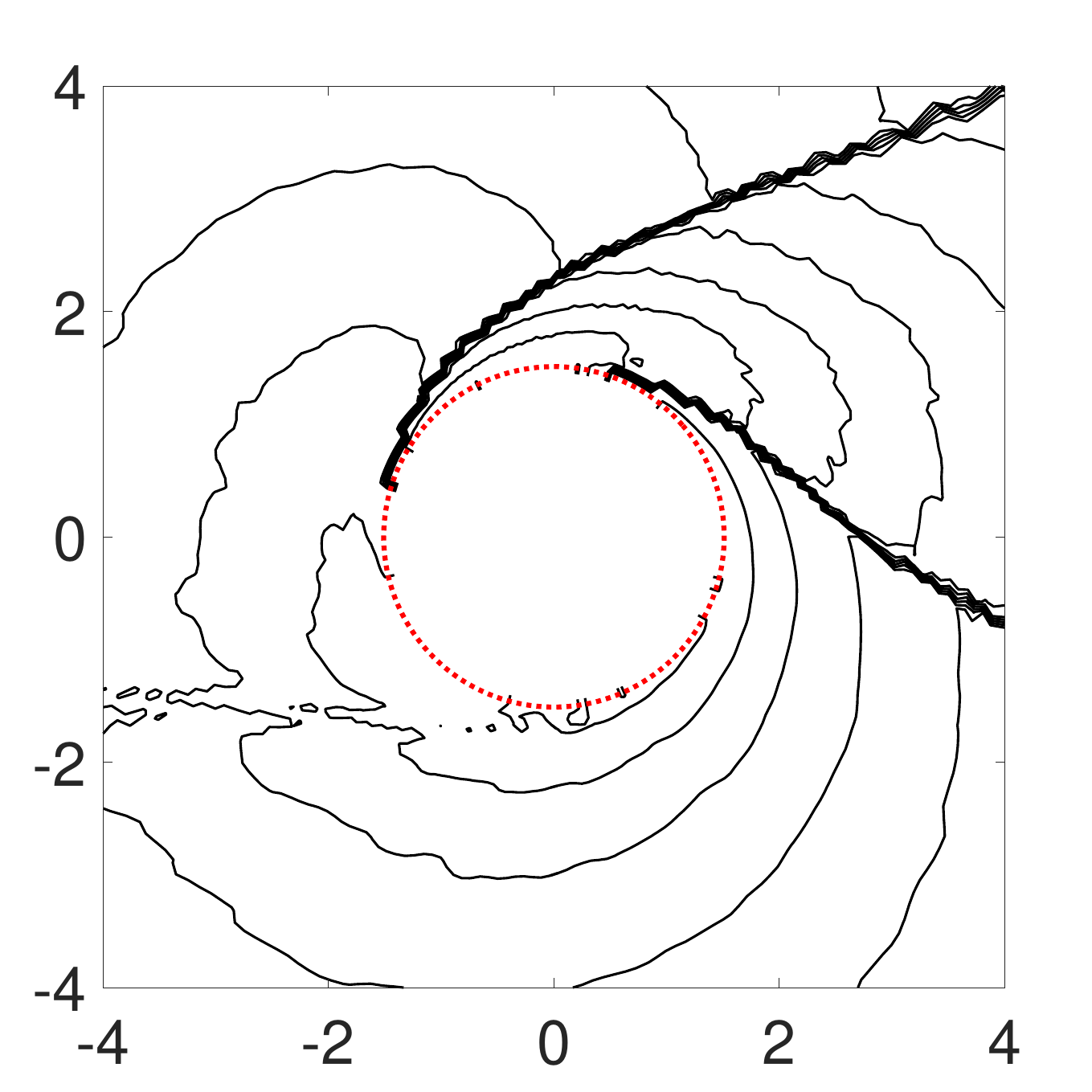}}
\caption{Example \ref{2D:Rie exam10d}: Close-up views of contours of $\log \rho$ in $[-4,4]^2$. From top to bottom: $\mathbb{P}^1$-, $\mathbb{P}^2$-, and $\mathbb{P}^3$-based PCP-OEDG schemes.}\label{fig:Bh p1 case9-10}
\end{figure}

The results clearly indicate that as the Mach number increases, the density contours become less smooth, reflecting the presence of more pronounced discontinuities and oscillations. This suggests that higher Mach numbers contribute to more turbulent and complex flow structures around the black hole, highlighting the role of compressibility in shaping the accretion dynamics.

Due to the large velocity and high Mach number, the simulation of Case 10 is very challenging. The proposed PCP-OEDG schemes perform robustly in this ultra-relativistic test, and the flow patterns are captured with high resolution.

\end{example}

\section{Conclusions}\label{sec6}

In this paper, we have presented the development of high-order accurate, physical-constraint-preserving, oscillation-eliminating discontinuous Galerkin (PCP-OEDG) schemes for the simulation of general relativistic hydrodynamics (GRHD) in arbitrary spacetimes. Our work has addressed several critical challenges inherent in the numerical simulation of GRHD within the DG framework, including handling curved spacetime, achieving high-order accuracy, suppressing spurious oscillations near discontinuities, and preserving key physical constraints such as positive density, positive pressure, and subluminal fluid velocity.

Our PCP-OEDG schemes are designed based on the W-form of the GRHD equations, which are obtained by reformulating the $(3+1)$ Arnowitt--Deser--Misner formalism, leveraging the Cholesky decomposition of the spatial metric. This reformulation has effectively addressed the challenge of non-equivalent admissible state sets at different points in curved spacetime by ensuring a spacetime-independent convex admissible state set, enabling the design of robust, high-order PCP schemes via convexity techniques.

To suppress nonphysical oscillations typically arising near shock waves and contact discontinuities, we have introduced a novel oscillation-eliminating (OE) procedure based on the exact solution of a linear damping equation. This procedure, which avoids using the complex characteristic decomposition, is computationally efficient and non-intrusive to the DG discretization. Its incorporation has proven to be an effective means of enhancing the stability of the DG method while maintaining high-order accuracy, local conservation, and compactness.

Furthermore, we have rigorously proven the PCP property for cell averages using highly technical estimates and the geometric quasi-linearization approach. This approach transforms complex nonlinear constraints into linear ones through the introduction of auxiliary variables, simplifying the analysis. Additionally, a simple but effective PCP limiter has been introduced to enforce the PCP property of DG solution polynomials at certain critical points. We have also developed provably convergent PCP iterative algorithms for the recovery of primitive variables in GRHD, ensuring that these variables satisfy physical constraints during their recovery from evolved variables.

The performance of the proposed PCP-OEDG method has been thoroughly validated through a series of extensive numerical experiments, including benchmark tests in flat spacetime, axisymmetric ultra-relativistic jet flows, and accretion processes around rotating black holes modeled by the Kerr metric. These experiments have demonstrated the method's robustness, accuracy, and efficiency, particularly in extreme GRHD conditions involving strong shocks, low pressure or density, high Lorentz factors, and intense gravitational fields.

In conclusion, the PCP-OEDG schemes developed in this work provide a reliable and efficient numerical framework for solving GRHD equations. These methods have proven to be highly effective in maintaining critical physical constraints while delivering high-order accurate and essentially oscillation-free solutions.

\bibliographystyle{abbrv}
\bibliography{mybib}

\end{document}